\documentclass[12pt,a4, reqno]{amsart}
\usepackage[a4paper, left=26mm, right=26mm, top=28mm, bottom=34mm]{geometry}
\usepackage[all]{xy}
\usepackage{amsmath}
\usepackage{amscd}
\usepackage{amssymb}
\usepackage{amsthm}
\usepackage{url}
\usepackage{mathrsfs}
\usepackage{comment}
\usepackage{bbm}

\newtheorem{thm}{Theorem}[subsection]
\newtheorem{lem}[thm]{Lemma}
\newtheorem{cor}[thm]{Corollary}
\newtheorem{prop}[thm]{Proposition}

\theoremstyle{definition}

\newtheorem{rem}[thm]{Remark}
\newtheorem{defn}[thm]{Definition}

\newtheorem{ex}[thm]{Example}

\def\F{{\mathbb F}}
\def\G{{\mathbb G}}
\def\O{{\mathcal O}}
\def\Q{{\mathbb Q}}

\def\Z{{\mathbb Z}}

\def\L{{\mathbb L}}

\def\Coker{\mathop{\mathrm{Coker}}\nolimits}

\def\Fil{\mathop{\mathrm{Fil}}\nolimits}

\def\Gal{\mathop{\mathrm{Gal}}\nolimits}
\def\Lie{\mathop{\mathrm{Lie}}\nolimits}
\def\Hom{\mathop{\mathrm{Hom}}\nolimits}

\def\Ker{\mathop{\mathrm{Ker}}\nolimits}
\def\id{\mathop{\mathrm{id}}\nolimits}

\def\GL{\mathop{\mathrm{GL}}\nolimits}

\def\Sym{\mathop{\mathrm{Sym}}\nolimits}

\def\Rep{\mathrm{Rep}}
\def\Spa{\mathop{\rm Spa}}
\def\Spec{\mathop{\rm Spec}}
\def\Spf{\mathop{\rm Spf}}

\def\diag{\mathop{\mathrm{diag}}}

\def\Fil{\mathop{\mathrm{Fil}}\nolimits}
\def\dR{\mathop{\mathrm{dR}}}

\newcommand{\plim}[1][]{\mathop{\varprojlim}\limits_{#1}}

\newcommand{\et}{\mathrm{\acute{e}t}}

\newcommand{\ad}{\mathrm{ad}}

\def\Perf{\mathop{\mathrm{Perf}}\nolimits}

\def\arc{\mathop{\mathrm{arc}}\nolimits}

\def\Vect{\mathop{\mathrm{Vect}}\nolimits}

\def\op{\mathop{\mathrm{op}}\nolimits}

\def\Perfd{\mathop{\mathrm{Perfd}}\nolimits}
\def\ori{\mathop{\mathrm{ori}}\nolimits}
\def\perf{\mathop{\mathrm{perf}}\nolimits}

\def\Kos{\mathop{\mathrm{Kos}}\nolimits}

\def\Rees{\mathop{\mathrm{Rees}}\nolimits}
\def\qsyn{\mathop{\mathrm{QSyn}}\nolimits}
\def\qrsperfd{\mathop{\mathrm{QRSPerfd}}\nolimits}

\usepackage{relsize} 
\usepackage[bbgreekl]{mathbbol} 
\usepackage{amsfonts} 
\usepackage{color}
\DeclareSymbolFontAlphabet{\mathbb}{AMSb} %to ensure that the meaning of \mathbb does not change
\DeclareSymbolFontAlphabet{\mathbbl}{bbold} 
\newcommand{\Prism}{{\mathlarger{\mathbbl{\Delta}}}}

\numberwithin{equation}{section}

\usepackage[T1]{fontenc}
\makeatletter
\@namedef{subjclassname@2020}{\textup{2020} Mathematics Subject Classification}
\makeatother

\begin{document}

\title[Prismatic $G$-displays and descent theory]{Prismatic $G$-displays and descent theory}

\author{Kazuhiro Ito}
\address{Mathematical Institute, Tohoku University, 6-3, Aoba, Aramaki, Aoba-Ku, Sendai 980-8578, Japan}
\email{kazuhiro.ito.c3@tohoku.ac.jp}

% \date{\today}


%14G45      Perfectoid spaces and mixed characteristic
%14L05	    Formal groups, $p$-divisible groups
%14F30  	$p$-adic cohomology, crystalline cohomology
% 14J28 	K3 surfaces and Enriques surfaces
% 11G15 	Complex multiplication and moduli of abelian varieties
% 14C25 	Algebraic cycles

\begin{abstract}
For a smooth affine group scheme $G$ over the ring of $p$-adic integers $\Z_p$ and a cocharacter $\mu$ of $G$,
we study $G$-$\mu$-displays over the prismatic site of Bhatt--Scholze.
In particular, we obtain several descent results for them.
If $G=\GL_n$, then our $G$-$\mu$-displays can be thought of as Breuil--Kisin modules with some additional conditions.
The relation between our $G$-$\mu$-displays and prismatic $F$-gauges introduced by Drinfeld and Bhatt--Lurie is also discussed.

In fact, our main results are formulated and proved for smooth affine group schemes over the ring of integers $\O_E$ of any finite extension $E$ of $\Q_p$ by using $\O_E$-prisms, which are $\O_E$-analogues of prisms.
\end{abstract}

\maketitle

\setcounter{tocdepth}{1}
\tableofcontents

\section{Introduction} \label{Section:Introduction}

Bhatt--Scholze introduced the theory of prisms in \cite{BS}.
The category of (bounded) prisms with the flat topology is called the absolute prismatic site.
It has been observed that
\textit{prismatic $F$-crystals} on the absolute prismatic site introduced in \cite{BS2}
play significant roles in various aspects of arithmetic geometry.
In this paper,
for a smooth affine group scheme $G$ over the ring of $p$-adic integers $\Z_p$,
we provide a systematic study of prismatic $F$-crystals with certain $G$-actions, which we call \textit{prismatic $G$-$\mu$-displays}.
The results obtained here will be used to study the deformation theory of prismatic $G$-$\mu$-displays in \cite{Ito-K23-b}.
We also discuss the relation between prismatic $G$-$\mu$-displays and \textit{prismatic $F$-gauges} introduced by Drinfeld and Bhatt--Lurie (cf.\ \cite{Drinfeld22}, \cite{BL}, \cite{BL2}, \cite{BhattGauge}).

\subsection{Prismatic Dieudonn\'e crystals}\label{Subsection:Prismatic Dieudonn\'e crystals Intro}

In \cite{Anschutz-LeBras},
Ansch\"utz--Le Bras
introduced \textit{prismatic Dieudonn\'e crystals},
which are prismatic $F$-crystals with additional conditions, and showed that prismatic Dieudonn\'e crystals can be used to classify $p$-divisible groups in mixed characteristic.
The notion of prismatic $G$-$\mu$-displays can be seen as a generalization of that of prismatic Dieudonn\'e crystals.
Before discussing prismatic $G$-$\mu$-displays,
let us state our main result for prismatic Dieudonn\'e crystals.

Let $k$ be a perfect field of characteristic $p >0$
and let
$W(k)$
be the ring of $p$-typical Witt vectors of $k$.
Let $R$ be a complete regular local ring over $W(k)$ with residue field $k$.
There exists a pair
\[
(A, I)=(W(k)[[t_1, \dotsc, t_n]], (\mathcal{E}))
\]
with an isomorphism
$R \simeq A/I$
over $W(k)$ where $\mathcal{E} \in W(k)[[t_1, \dotsc, t_n]]$ is a formal power series whose constant term is $p$.
Here
$A$ admits a Frobenius endomorphism $\phi \colon A \to A$ such that it acts on $W(k)$ as the usual Frobenius and sends $t_i$ to $t^p_i$ for each $i$.
The pair $(A, I)$ is a typical example of a prism.
Let
$(R)_{\Prism}$
be
the absolute prismatic site of $R$
(where $R$ is equipped with the $p$-adic topology).
We regard $(A, I)$ as an object of $(R)_{\Prism}$.
We will prove (and generalize) the following result.

\begin{thm}[Proposition \ref{Proposition:evaluation prismatic dieudonne crystal equivalence}]\label{Theorem:evaluation prismatic dieudonne crystal equivalence intro}
The category of prismatic Dieudonn\'e crystals on $(R)_{\Prism}$
is equivalent to the category of minuscule Breuil--Kisin modules over $(A, I)$.
\end{thm}

In \cite[Theorem 5.12]{Anschutz-LeBras},
Ansch\"utz--Le Bras proved Theorem \ref{Theorem:evaluation prismatic dieudonne crystal equivalence intro} when the dimension of $R$ is $ \leq 1$ (or equivalently $n \leq 1)$, and stated that their result should be generalized to $R$ of arbitrary dimension.

\begin{rem}\label{Remark:BK module intro}
A \textit{minuscule Breuil--Kisin module} over $(A, I)$
is a free $A$-module $M$ of finite rank equipped with an $A$-linear homomorphism
\[
F_M \colon \phi^*M:=A \otimes_{\phi, A} M \to M
\]
whose cokernel is killed by $I$.
For a prismatic Dieudonn\'e crystal
$\mathcal{M}$
on $(R)_\Prism$,
the value
$\mathcal{M}(A, I)$
at $(A, I) \in (R)_\Prism$
is by definition a minuscule Breuil--Kisin module over $(A, I)$, and the construction $\mathcal{M} \mapsto \mathcal{M}(A, I)$ induces an equivalence between the two categories in Theorem \ref{Theorem:evaluation prismatic dieudonne crystal equivalence intro}.
\end{rem}

Ansch\"utz--Le Bras in \cite[Theorem 4.74]{Anschutz-LeBras} showed that the category of prismatic Dieudonn\'e crystals on $(R)_{\Prism}$
is equivalent to the category of $p$-divisible groups over $R$.
In fact,
such an equivalence is obtained not only for $R$ but also for any quasisyntomic ring (in the sense of \cite[Definition 4.10]{BMS2}), where we need to replace prismatic Dieudonn\'e crystals
by \textit{admissible} prismatic Dieudonn\'e crystals (\cite[Definition 4.5]{Anschutz-LeBras}).

Although (admissible) prismatic Dieudonn\'e crystals are theoretically important,
it is difficult to describe them explicitly in general.
Theorem \ref{Theorem:evaluation prismatic dieudonne crystal equivalence intro}
provides a practical way to study prismatic Dieudonn\'e crystals on $(R)_\Prism$.
For example, this shows that giving a prismatic Dieudonn\'e crystal on $(R)_\Prism$ is equivalent to giving a minuscule Breuil--Kisin module over $(A, I)$.
The latter can be carried out in a much simpler way than the former.

\subsection{Prismatic $G$-$\mu$-displays}\label{Subsection:Prismatic G-mu-displays Introduction}

Let $G$ be a smooth affine group scheme over $\Z_p$ and
$
\mu \colon \G_m \to G_{W(k)}
$
a cocharacter defined over $W(k)$, where $G_{W(k)}:=G \times_{\Spec \Z_p} \Spec W(k)$.
We will generalize Theorem \ref{Theorem:evaluation prismatic dieudonne crystal equivalence intro} to prismatic $G$-$\mu$-displays, or equivalently $G$-Breuil--Kisin modules of type $\mu$, as explained below.

Let $(A, I)$ be a bounded prism in the sense of \cite{BS}, such that $A$ is a $W(k)$-algebra.
A $G$-Breuil--Kisin module over $(A, I)$ is a $G$-torsor $\mathcal{P}$ over $\Spec A$
with 
an isomorphism
\[
F_\mathcal{P} \colon (\phi^*\mathcal{P})[1/I] \overset{\sim}{\to} \mathcal{P}[1/I]
\]
of $G$-torsors over $\Spec A[1/I]$, where $\phi^*\mathcal{P}$ is
the base change of $\mathcal{P}$ along the Frobenius $\phi \colon A \to A$.
We say that
$\mathcal{P}$
is \textit{of type} $\mu$
if, $(p, I)$-completely \'etale locally on $A$,
there exists some trivialization
$\mathcal{P} \simeq G_A$ under which the isomorphism $F_\mathcal{P}$ is given by $g \mapsto Xg$ for an element $X$ in the double coset
\[
G(A)\mu(d)G(A) \subset G(A[1/I]),
\]
where $d \in I$ is a generator.
The notion of $G$-Breuil--Kisin modules of type $\mu$ is important in the study of integral models of (local) Shimura varieties; see Section \ref{Subsection:motivations Introduction}.

In this paper, we will study $G$-Breuil--Kisin modules of type $\mu$ via the theory of higher frames and $G$-$\mu$-displays developed by Lau \cite{Lau21}.
More precisely, we introduce and study the groupoid
\[
G\mathchar`-\mathrm{Disp}_\mu(A, I)
\]
of
\textit{$G$-$\mu$-displays over $(A, I)$}.
It turns out that
$G\mathchar`-\mathrm{Disp}_\mu(A, I)$
is equivalent to the groupoid of
$G$-Breuil--Kisin modules of type $\mu$ over $(A, I)$ (Proposition \ref{Proposition:G-displays and G-BK modules}).
For a $p$-adically complete ring $R$,
the groupoid of \textit{prismatic $G$-$\mu$-displays over $R$} is defined to be
\[
G\mathchar`-\mathrm{Disp}_\mu((R)_{\Prism}):= {2-\varprojlim}_{(A, I) \in (R)_{\Prism}} G\mathchar`-\mathrm{Disp}_\mu(A, I).
\]

The main result of this paper is as follows.
Let $R$ be a complete regular local ring over $W(k)$ with residue field $k$.
As in Section \ref{Subsection:Prismatic Dieudonn\'e crystals Intro},
there exists a prism
$
(W(k)[[t_1, \dotsc, t_n]], (\mathcal{E}))
$
with an isomorphism
$R \simeq W(k)[[t_1, \dotsc, t_n]]/\mathcal{E}$
over $W(k)$.

\begin{thm}[Theorem \ref{Theorem:main result on G displays over complete regular local rings}]\label{Theorem:main result on G displays over complete regular local rings Introduction}
Assume that the cocharacter $\mu$ is 1-bounded.
Then the following natural functor is an equivalence:
\[
 G\mathchar`-\mathrm{Disp}_\mu((R)_{\Prism}) \overset{\sim}{\to} G\mathchar`-\mathrm{Disp}_\mu(W(k)[[t_1, \dotsc, t_n]], (\mathcal{E})).
\]
\end{thm}

We refer to Definition \ref{Definition:1-bounded} for the definition of 1-bounded cocharacters.
If $G$ is reductive, then $\mu$ is 1-bounded if and only if $\mu$ is minuscule.
A minuscule Breuil--Kisin module of rank $N$ over $(W(k)[[t_1, \dotsc, t_n]], (\mathcal{E}))$ can be regarded as
a $\GL_N$-Breuil--Kisin module of type $\mu$ over $(W(k)[[t_1, \dotsc, t_n]], (\mathcal{E}))$
for a minuscule cocharacter $\mu$.
Theorem \ref{Theorem:evaluation prismatic dieudonne crystal equivalence intro} is a special case of Theorem \ref{Theorem:main result on G displays over complete regular local rings Introduction}; see Section \ref{Section:p-divisible groups and prismatic Dieudonn\'e crystals} for details.

We make a few comments on the proof of Theorem \ref{Theorem:main result on G displays over complete regular local rings Introduction}.
To simplify the notation, we set $(A, I):=(W(k)[[t_1, \dotsc, t_n]], (\mathcal{E}))$.
As in the proof of \cite[Theorem 5.12]{Anschutz-LeBras},
the key part of the proof is to show that
every $G$-$\mu$-display
$\mathcal{Q}$
over $(A, I)$
admits a unique descent datum.
More precisely, let $(A^{(2)}, I^{(2)})$ be the coproduct of two copies of $(A, I)$ in $(R)_{\Prism}$ and let
$p_1, p_2 \colon (A, I) \to (A^{(2)}, I^{(2)})$ be the associated morphisms.
Then we will prove that
there exists a unique isomorphism
\[
\epsilon \colon p^*_1\mathcal{Q} \overset{\sim}{\to} p^*_2\mathcal{Q}
\]
of $G$-$\mu$-displays over
$(A^{(2)}, I^{(2)})$
satisfying the usual cocycle condition over
the coproduct $(A^{(3)}, I^{(3)})$ of three copies of $(A, I)$.
In the case where $G=\GL_N$,
the proof of this claim goes along the same line as that of \cite{Anschutz-LeBras}, but it requires some additional arguments when $n \geq 2$.
For general $G$, we will use some techniques from the proof of \cite[Proposition 7.1.5]{Lau21}.

We also give some basic definitions and results on prismatic $G$-$\mu$-displays.
In particular,
we establish several descent results for prismatic $G$-$\mu$-displays, such as flat descent (Proposition \ref{Proposition:flat descent of G display}) and $p$-complete $\arc$-descent
(Corollary \ref{Corollary:arc descent for G-displays}).
We also introduce the \textit{Hodge filtration} and the \textit{underlying $G$-$\phi$-module} of a prismatic $G$-$\mu$-display.
These notions will be needed in
the Grothendieck--Messing deformation theory for prismatic $G$-$\mu$-displays studied in \cite{Ito-K23-b}.

\begin{rem}\label{Remark:OE prisms intro}
In fact, Theorem \ref{Theorem:main result on G displays over complete regular local rings Introduction} will be formulated and proved for a smooth affine group scheme $G$ over the ring of integers $\O_E$ of any finite extension $E$ of $\Q_p$.
For this, we will use $\O_E$-analogues of prisms, called \textit{$\O_E$-prisms}.
This notion was already introduced in the work of Marks \cite{Marks} (in which these objects are called $E$-typical prisms).
Section \ref{Section:Preliminaries on prisms} is devoted to discuss
analogous results to those of \cite[Section 2 and Section 3]{BS} for $\O_E$-prisms.
We will define $G$-$\mu$-displays for bounded $\O_E$-prisms in the same way, and prove the above results for them.
As explained in Remark \ref{Remark:O_E prism motivation} below, it will be convenient to establish our results in this generality,
but the reader (who is familiar with the theory of prisms) may assume that $\O_E=\Z_p$ and skip Section \ref{Section:Preliminaries on prisms} on a first reading.
The arguments for general $\O_E$ are the same as for the case where $\O_E=\Z_p$.
\end{rem}

\begin{rem}\label{Remark:G-mu-displays are techinically better}
    $G$-Breuil--Kisin modules of type $\mu$ may be more familiar to readers than prismatic $G$-$\mu$-displays.
    However, in order to prove Theorem \ref{Theorem:evaluation prismatic dieudonne crystal equivalence intro}, Theorem \ref{Theorem:main result on G displays over complete regular local rings Introduction}, and other descent results (e.g.\ Corollary \ref{Corollary:G-BK module etale banal} and Corollary \ref{Corollary:arc descent for G-displays}),
    it is essential to work with prismatic $G$-$\mu$-displays.
\end{rem}

\begin{rem}\label{Remark:relation to crystalline representations}
    We briefly discuss how our results are related to the theory of $G$-objects in crystalline $\Z_p$-local systems.
    Here we follow the terminology of \cite{IKYtannakian}.
    Let
    $R=\O_K$
    be the ring of integers of a finite totally ramified extension $K$ of $W(k)[1/p]$.
    In \cite{BS2}, Bhatt--Scholze proved that the category of prismatic $F$-crystals on $(\O_K)_\Prism$ is equivalent to the category
    $\mathrm{Loc}^{\mathrm{crys}}_{\Z_p}(K)$
    of
    free $\Z_p$-modules $T$ of finite rank with a continuous $\Gal(\overline{K}/K)$-action such that $T[1/p]$ is crystalline.
    (Here $\Gal(\overline{K}/K)$ is the absolute Galois group of $K$.)
    Using this result, together with \cite{IKYtannakian}, one can prove that there is an equivalence of groupoids
    \begin{equation}\label{equation:equivalence with crystalline local systems}
        G\mathchar`-\mathrm{Disp}_\mu((\O_K)_{\Prism}) \overset{\sim}{\to} G\mathchar`-{\bf{\mathrm{Loc}}}^{\mathrm{crys}}_{\Z_p, \bf{\mu}}(K)
    \end{equation}
    if $G$ is reductive,
    where $G\mathchar`-{\bf{\mathrm{Loc}}}^{\mathrm{crys}}_{\Z_p, \bf{\mu}}(K)$
    is the groupoid of $G$-objects in $\mathrm{Loc}^{\mathrm{crys}}_{\Z_p}(K)$ having cocharacter $\mu$ in the sense of \cite{IKYtannakian}.
    More specifically, this follows from \cite[Theorem 2, Proposition 3.13]{IKYtannakian} and \cite[Proposition 5.1.16]{Ito-K23-b}.
    Let $\mathcal{E} \in W(k)[[t]]$ be the Eisenstein polynomial of a uniformizer $\varpi \in \O_K$.
    Then
    $(\ref{equation:equivalence with crystalline local systems})$
    and
    Theorem \ref{Theorem:main result on G displays over complete regular local rings Introduction} give an equivalence
    \[
    G\mathchar`-\mathrm{Disp}_\mu(W(k)[[t]], (\mathcal{E})) \overset{\sim}{\to} G\mathchar`-{\bf{\mathrm{Loc}}}^{\mathrm{crys}}_{\Z_p, \bf{\mu}}(K).
    \]
    A similar result was previously obtained by Levin \cite[Corollary 4.3.8]{Levin} by a completely different method.
    (The result of Bhatt--Scholze in \cite{BS2} was generalized to higher dimensional smooth $p$-adic formal schemes over $\O_K$; see \cite{DLMS} and \cite{GuoReinecke}.
    However, since a higher dimensional complete regular local ring $R$ as in Theorem \ref{Theorem:main result on G displays over complete regular local rings Introduction}
    is in general not topologically of finite type over $W(k)$ with respect to the \textit{$p$-adic} topology, we can not directly apply those results to obtain an analogue of $(\ref{equation:equivalence with crystalline local systems})$ for $R$.
    We do not pursue this issue here.)
\end{rem}

    We will also discuss (in the case where $\O_E=\Z_p$) the relation between prismatic $G$-$\mu$-displays and prismatic $F$-gauges in vector bundles introduced by Drinfeld and Bhatt--Lurie (cf.\ \cite{Drinfeld22}, \cite{BL}, \cite{BL2}, \cite{BhattGauge}).
    In particular, for a quasisyntomic ring $S$ over $W(k)$, we introduce\footnote{After this work was completed, Gardner--Madapusi--Mathew \cite{GMM} announced that they defined
(certain objects which are essentially equivalent to)
prismatic $G$-$F$-gauges of type $\mu$ for more general $p$-adically complete rings, using the stacky approach of Drinfeld and Bhatt--Lurie.
See also \cite{IKY} for the relation between our
prismatic $G$-$F$-gauges of type $\mu$ and those defined in \cite{GMM}.} the groupoid
    \[
    G\mathchar`-F\mathchar`-\mathrm{Gauge}_\mu(S)
    \]
    of \textit{prismatic $G$-$F$-gauges of type $\mu$ over $S$} and construct a fully faithful functor
    \[
    G\mathchar`-F\mathchar`-\mathrm{Gauge}_\mu(S) \to G\mathchar`-\mathrm{Disp}_\mu((S)_{\Prism}).
    \]
    See Section \ref{Section:Comparison with prismatic $F$-gauges} for details.
    This functor can be thought of as a generalization of
    the fully faithful functor from the category of
    admissible prismatic Dieudonn\'e crystals on $(S)_{\Prism}$ to the category of prismatic Dieudonn\'e crystals on $(S)_{\Prism}$ (cf.\ Example \ref{Example:F-gauge weight [0, 1]}).
    If $S$ is a complete regular local ring over $W(k)$ with residue field $k$, then the above functor is an equivalence (Corollary \ref{Corollary:functor from G-gauge to G-display perfectoid and regular}).
    Thus, we can rephrase Theorem \ref{Theorem:main result on G displays over complete regular local rings Introduction} as follows:

\begin{cor}[Theorem \ref{Theorem:main result on G displays over complete regular local rings}, Corollary \ref{Corollary:functor from G-gauge to G-display perfectoid and regular}]\label{Corollary:G-gauge and G-displays regular case intro}
    Let the notation be as in Theorem \ref{Theorem:main result on G displays over complete regular local rings Introduction}.
    Assume that $\mu$ is 1-bounded.
    Then we have a natural equivalence
    \[
    G\mathchar`-F\mathchar`-\mathrm{Gauge}_\mu(R) \overset{\sim}{\to} G\mathchar`-\mathrm{Disp}_\mu(W(k)[[t_1, \dotsc, t_n]], (\mathcal{E})).
    \]
\end{cor}

\subsection{Motivation}\label{Subsection:motivations Introduction}

The primary motivation behind this work is to understand some classification results on $p$-divisible groups and the local structure of integral local Shimura varieties defined in \cite{Scholze-Weinstein}, by using the theory of prisms.
In the following, we explain this in more detail with a brief review of previous studies.

We first explain the motivation for $G=\GL_N$.
Let
$\O_K$
be the ring of integers of a finite totally ramified extension $K$ of $W(k)[1/p]$.
Let $\mathcal{E} \in W(k)[[t]]$ be the Eisenstein polynomial of a uniformizer $\varpi \in \O_K$.

\begin{rem}
In \cite[Theorem 5.12]{Anschutz-LeBras},
Ansch\"utz--Le Bras
obtained the following equivalence of categories
\begin{equation}\label{equation:Breuil--Kisin classification}
    \left.
\left\{
\begin{tabular}{c}
     $p$-divisible groups \\ over $\O_K$
\end{tabular}
\right\}
\right.
    \overset{\sim}{\to}
    \left.
\left\{
\begin{tabular}{c}
     minuscule Breuil--Kisin modules \\ over $(W(k)[[t]], (\mathcal{E}))$
\end{tabular}
\right\}
\right.
\end{equation}
by combining the classification theorem \cite[Theorem 4.74]{Anschutz-LeBras} with Theorem \ref{Theorem:evaluation prismatic dieudonne crystal equivalence intro} for $R=\O_K$.
This result was conjectured by Breuil (\cite{Breuil}), proved by Kisin (\cite{Kisin06}, \cite{Kisin09}) when $p \geq 3$, and proved in \cite{Kim12}, \cite{LiuTong}, and \cite{Lau14} for all $p>0$.
\end{rem}

We consider the pair $(W(k)[[t]]/t^n, (\mathcal{E}))$, which is naturally a bounded prism for every $n \geq 1$.
In \cite{Lau14},
Lau obtained the following equivalence of categories
\begin{equation}\label{equation:Breuil--Kisin classification torsion case}
    \left.
\left\{
\begin{tabular}{c}
     $p$-divisible groups \\ over $\O_K/\varpi^n$
\end{tabular}
\right\}
\right.
    \overset{\sim}{\to}
    \left.
\left\{
\begin{tabular}{c}
     minuscule Breuil--Kisin modules \\ over $(W(k)[[t]]/t^n, (\mathcal{E}))$
\end{tabular}
\right\}
\right.
\end{equation}
by a deformation-theoretic argument, and then proved $(\ref{equation:Breuil--Kisin classification})$
by taking the limit over $n$; see \cite[Corollary 5.4, Theorem 6.6]{Lau14}.
His proof uses the theory of \textit{displays}, which is initiated by Zink and developed by many authors including Zink and Lau (cf.\ \cite{Zink01}, \cite{Zink02}, \cite{Lau08}, \cite{Lau14}).
This classification result over $\O_K/\varpi^n$ is important in its own right.
For example, this is a key ingredient in the construction of integral canonical models of Shimura varieties of abelian type with hyperspecial level structure in characteristic $p=2$; see \cite{Kim-MadapusiPera} for details.

Contrary to Lau's approach,
it is not clear whether the results in \cite{Anschutz-LeBras} imply $(\ref{equation:Breuil--Kisin classification torsion case})$ since the Grothendieck--Messing deformation theory does not apply directly to prismatic Dieudonn\'e crystals.
This point is discussed in \cite{Ito-K23-b},
where we develop the deformation theory for
prismatic Dieudonn\'e crystals, or more generally for prismatic $G$-$\mu$-displays when $\mu$ is 1-bounded.
In \textit{loc.cit.},
we construct universal deformations of prismatic $G$-$\mu$-displays over $k$ as certain prismatic $G$-$\mu$-displays over complete regular local rings of higher dimension, where Theorem \ref{Theorem:main result on G displays over complete regular local rings Introduction} plays a crucial role.
As a result, we can give an alternative proof of the equivalences $(\ref{equation:Breuil--Kisin classification})$ and $(\ref{equation:Breuil--Kisin classification torsion case})$.

\begin{rem}\label{Remark:p-divisible group Scholze-Weinstein}
In the proof of \cite[Theorem 4.74]{Anschutz-LeBras} (and hence in the proof of the equivalence $(\ref{equation:Breuil--Kisin classification})$ given in \cite{Anschutz-LeBras}),
they use a result of Scholze--Weinstein \cite[Theorem B]{ScholzeWeinsteinModuli},
which says that for an algebraically closed complete extension $C$ of $\Q_p$,
the category of $p$-divisible groups over $\O_C$
is equivalent to the category of free $\Z_p$-modules $T$ of finite rank together with a $C$-subspace of $T \otimes_{\Z_p} C$.
In \cite{Ito-K23-b}, we also give an alternative proof of this result.
\end{rem}

We now explain our motivation for general $G$.
The notion of $G$-$\mu$-displays (``displays with $G$-$\mu$-structures'') were first introduced by B\"ultel \cite{Bultel} and B\"ultel--Pappas \cite{Bultel-Pappas} to study Rapoport--Zink spaces and integral models of Shimura varieties
(where $G$ is reductive and $\mu$ is minuscule).
The theory of $G$-$\mu$-displays has been developed in various settings; see for example \cite{LangerZink}, \cite{Pappas23}, \cite{Lau21}, \cite{Daniels}, and \cite{Bartling}.
In fact,
the notion of $G$-$\mu$-displays over \textit{perfect} prisms is already used in the paper of Bartling \cite{Bartling}.
Bartling used $G$-$\mu$-displays over perfect prisms
to prove
the local representability 
and
the formal smoothness
of integral local Shimura varieties with hyperspecial level structure, under a certain nilpotence assumption (introduced in \cite[Definition 3.4.2]{Bultel-Pappas}).
In \cite{Ito-K23-b}, we prove the same assertion without any nilpotence assumptions, by using the universal deformations of prismatic $G$-$\mu$-displays over $k$.

\begin{rem}\label{Remark:O_E prism motivation}
In \cite{Ito-K23-b},
we establish the above results not only when $G$ is defined over $\Z_p$ but also when $G$ is defined over $\O_E$ where $E$ is any finite extension of $\Q_p$.
For this,
it will be convenient to work with $\O_E$-prisms.
\end{rem}

The theory of $G$-$\mu$-displays also has applications to K3 surfaces and related varieties; see \cite{LangerZink}, \cite{Lau21}, and \cite{inoue}.
In a future work, we plan to employ prismatic $G$-$\mu$-displays to investigate the deformation theory for these varieties.

\subsection{Outline of this paper}\label{Subsection:Outline of this paper}

This paper is organized as follows.
In Section \ref{Section:Preliminaries on prisms}, we collect some basic definitions and facts about $\O_E$-prisms.
In Section \ref{Section:Displayed Breuil--Kisin module},
we discuss the notion of displayed Breuil--Kisin modules (of type $\mu$), which will serve as examples of prismatic $G$-$\mu$-displays.
In Section \ref{Section:display group},
we introduce and study the \textit{display group} $G_\mu(A, I)$, which is used in the definition of prismatic $G$-$\mu$-displays.
The structural results about $G_\mu(A, I)$ obtained there play crucial roles in the study of prismatic $G$-$\mu$-displays.

Section \ref{Section:prismatic G display} and Section \ref{Section:G-displays over complete regular local rings} are the main parts of this paper.
In Section \ref{Section:prismatic G display}, we introduce prismatic $G$-$\mu$-displays, give some basic definitions (e.g.\ Hodge filtrations and underlying $G$-$\phi$-modules), and establish several descent results.
In Section \ref{Section:G-displays over complete regular local rings}, we prove our main result (Theorem \ref{Theorem:main result on G displays over complete regular local rings Introduction}).

In Section \ref{Section:p-divisible groups and prismatic Dieudonn\'e crystals}, we make a few remarks on prismatic Dieudonn\'e crystals, and prove Theorem \ref{Theorem:evaluation prismatic dieudonne crystal equivalence intro}.
Finally, in Section \ref{Section:Comparison with prismatic $F$-gauges}, we provide a comparison between prismatic $G$-$\mu$-displays and prismatic $F$-gauges.
In particular, we introduce the notion of prismatic $G$-$F$-gauges of type $\mu$ for quasisyntomic rings over $W(k)$.

\subsection*{Notation} \label{Subsection:Notation}

In this paper, all rings are commutative and unital.
For a module $M$ over a ring $R$ and a ring homomorphism
$f \colon R \to R'$,
the tensor product $M\otimes_{R}R'$ is denoted by $M_{R'}$ or $f^*M$.
For a scheme $X$ over $R$,
the base change $X \times_{\Spec R} \Spec R'$ is denoted by
$X_{R'}$ or $f^*X$.
We use similar notation for the base change of group schemes,
$p$-divisible groups, etc.
Moreover, all actions of groups will be right actions, unless otherwise stated.

\section{Preliminaries on $\O_E$-prisms} \label{Section:Preliminaries on prisms}

Throughout this paper, we fix a prime number $p$.
Let $E$ be a finite extension of $\Q_p$ with ring of integers $\O_E$ and residue field $\F_q$.
Here $\F_q$ is a finite field with $q$ elements.
We fix a uniformizer $\pi \in \O_E$.

In this section, we study an ``$\O_E$-analogue'' of the notion of prisms.
Such objects are called $\O_E$-prisms in this paper.
This notion was also introduced
in the work of Marks \cite{Marks} (in which $\O_E$-prisms are called \textit{$E$-typical prisms}).
We discuss some properties of $\O_E$-prisms which we need in the sequel.
We hope that this section will also help the reader unfamiliar with \cite{BS} to understand some basic facts about prisms.

\subsection{Prisms} \label{Subsection:prisms}

We first recall the definition of bounded prisms.

Let $A$ be a $\Z_{(p)}$-algebra.
A \textit{$\delta$-structure} on $A$ is a map $\delta \colon A \to A$ of sets with the following properties:
\begin{enumerate}
    \item $\delta(1)=0$.
    \item $\delta(xy)=x^p\delta(y)+y^p\delta(x)+p\delta(x)\delta(y)$.
    \item $\delta(x+y)=\delta(x)+\delta(y)+(x^p+y^p-(x+y)^p)/p$.
\end{enumerate}
A \textit{$\delta$-ring} is a pair $(A, \delta)$ of a $\Z_{(p)}$-algebra $A$ and a $\delta$-structure $\delta \colon A \to A$.
The above equalities imply that
\[
\phi \colon A \to A, \quad x \mapsto x^p+p\delta(x)
\]
is a ring homomorphism which is a lift of
the Frobenius $A/p \to A/p$, $x \mapsto x^p$.

In the following, for a ring $A$ and an ideal $I \subset A$, we say that an $A$-module $M$ is $I$-adically complete (or $x$-adically complete if $I$ is generated by an element $x \in I$) if the natural homomorphism
\[
M \to \widehat{M}:=\plim[n] M/I^nM
\]
is bijective.

\begin{defn}[\cite{BS}]\label{Definition:orientable and bounded prism}
A \textit{bounded prism} is a pair
$(A, I)$ of a $\delta$-ring $A$ and a Cartier divisor $I \subset A$ with the following properties:
\begin{enumerate}
    \item $A$ is $(p, I)$-adically complete.
    \item $A/I$ has bounded $p$-torsion, i.e.\ $(A/I)[p^\infty]=(A/I)[p^n]$ for some integer $n > 0$.
    \item We have $p \in (I, \phi(I))$.
\end{enumerate}
We say that a bounded prism $(A, I)$ is \textit{orientable} if $I$ is principal.
\end{defn}

\begin{rem}\label{Remark:remarks on prisms}
Under the condition that $A/I$ has bounded $p^\infty$-torsion,
the requirement that $A$ is $(p, I)$-adically complete
is equivalent to saying that $A$ is \textit{derived} $(p, I)$-adically complete; see \cite[Lemma 3.7]{BS}.
We refer to \cite[Section 1.2]{BS} and \cite[Tag 091N]{SP} for the notion of derived complete modules (or complexes).
For a ring $A$ and a finitely generated ideal $I \subset A$, if an $A$-module $M$ is $I$-adically complete, then $M$ is derived $I$-adically complete; see \cite[Tag 091T]{SP} or \cite[Lemma 2.3]{Positselski}.
\end{rem}

\subsection{$\delta_E$-rings}\label{Subsection:delta pi ring}

In this subsection, we recall the notion of $\delta_E$-rings, which is an ``$\O_E$-analogue'' of the notion of a $\delta$-ring.
We define
\[
\delta_{\O_E, \pi} \colon \O_E \to \O_E, \quad x \mapsto (x-x^q)/\pi.
\]

\begin{defn}[{\cite[Definition 2.2]{Marks}}]\label{Definition:OE delta}
\
\begin{enumerate}
    \item Let $A$ be an $\O_E$-algebra.
    A \textit{$\delta_E$-structure} on $A$ is a map $\delta_E \colon A \to A$ of sets with the following properties:
\begin{enumerate}
    \item $\delta_E(xy)=x^q\delta_E(y)+y^q\delta_E(x)+\pi\delta_E(x)\delta_E(y)$.
    \item $\delta_E(x+y)=\delta_E(x)+\delta_E(y)+(x^q+y^q-(x+y)^q)/\pi$.
    \item $\delta_E \colon A \to A$ is compatible with $\delta_{\O_E, \pi}$, i.e.\ we have $\delta_E(x)=\delta_{\O_E, \pi}(x)$ for any $x \in \O_E$.
\end{enumerate}
    A \textit{$\delta_E$-ring} is a pair $(A, \delta_E)$ of an $\O_E$-algebra $A$ and a $\delta_E$-structure $\delta_E \colon A \to A$.
\item A homomorphism $f \colon (A, \delta_E) \to (A', \delta'_E)$
of $\delta_E$-rings
is a homomorphism $f \colon A \to A'$ of $\O_E$-algebras such that $f \circ \delta_E=\delta'_E \circ f$.
\end{enumerate}
\end{defn}
The term
$(x^q+y^q-(x+y)^q)/\pi$ in (b) makes sense since we can write it as
\[
(x^q+y^q-(x+y)^q)/\pi= -\sum_{0 < i < q} (\tbinom{q}{i}/\pi)x^iy^{q-i}.
\]
We usually denote a $\delta_E$-ring $(A, \delta_E)$ simply by $A$.

\begin{rem}
The notion of $\delta_E$-rings also appeared in \cite[Remark 1.19]{Borger} and \cite{Li} for example.
In the end of \cite{Li}, Li suggests to use $\delta_E$-structures for the study of prismatic sites of higher level over ramified bases.
\end{rem}

\begin{rem}\label{Remark:change of pi delta structure}
The notion of $\delta_E$-rings is essentially independent of the choice of $\pi$.
More precisely, let $\pi' \in \O_E$ be another uniformizer.
We write $\pi=u\pi'$ for a unique unit $u \in \O^\times_E$.
If an $\O_E$-algebra $A$ is equipped with a
$\delta_E$-structure
$\delta_E \colon A \to A$
with respect to $\pi$, then it also admits
a
$\delta_E$-structure
with respect to $\pi'$, defined by $x \mapsto u\delta_E(x)$.
\end{rem}

For a $\delta_E$-ring $A$, we define
\[
\phi_A \colon A \to A, \quad x \mapsto x^q+\pi\delta_E(x).
\]
We see that $\phi_A$ is a homomorphism of $\O_E$-algebras and is a lift of
the $q$-th power Frobenius $A/\pi \to A/\pi$, $x \mapsto x^q$.
The homomorphism $\phi_A$ is called the \textit{Frobenius} of the $\delta_E$-ring $A$.
When there is no ambiguity, we omit the subscript and simply write $\phi=\phi_A$.

\begin{rem}\label{Remark:torsion free case Frobenius lift}
If $A$ is a $\pi$-torsion free $\O_E$-algebra,
then the construction $\delta_E \mapsto \phi$ gives a bijection between the set of $\delta_E$-structures on $A$ and the set of homomorphisms $\phi \colon A \to A$ over $\O_E$ that are lifts of $A/\pi \to A/\pi$, $x \mapsto x^q$.
\end{rem}

\begin{ex}[Free $\delta_E$-rings]\label{Example:free delta pi rings}
We define an endomorphism $\phi$
of
the polynomial ring
$
\O_E[X_0, X_1, X_2, \dotsc]
$
by $X_i \mapsto X^q_i + \pi X_{i+1}$ $(i \geq 0)$.
By Remark \ref{Remark:torsion free case Frobenius lift},
we get the corresponding $\delta_E$-structure on $\O_E[X_0, X_1, X_2, \dotsc]$, which sends $X_i$ to $X_{i+1}$.
We write
\[
\O_E\{ X \}
\]
for the resulting $\delta_E$-ring.
As in the proof of \cite[Lemma 2.11]{BS}, one can check that $\O_E\{ X \}$ has the following property:
For a $\delta_E$-ring $A$ and an element $x \in A$,
there exists a unique homomorphism 
$f \colon \O_E\{ X \} \to A$ of $\delta_E$-rings with $f(X_0)=x$.
In other words, the $\delta_E$-ring $\O_E\{ X \}$ is a free object with basis $X:=X_0$ in the category of $\delta_E$-rings.

The same argument as in the proof of \cite[Lemma 2.11]{BS} also shows that the Frobenius $\phi \colon \O_E\{ X \} \to \O_E\{ X \}$ is faithfully flat; this fact will be used in Section \ref{Subsection:Prismatic envelopes for regular sequences}.
\end{ex}

%Using $\O_E\{ X \}$, we can prove the following result:

\begin{lem}\label{Lemma:frobenius is compatible with delta}
For a $\delta_E$-ring $A$,
the Frobenius $\phi \colon A \to A$ is a homomorphism of $\delta_E$-rings.
\end{lem}

\begin{proof}
Let $x \in A$ be an element.
We have to show that $\phi(\delta_E(x))=\delta_E(\phi(x))$.
Since there exists a (unique) homomorphism
$f \colon \O_E\{ X \} \to A$
of $\delta_E$-rings with $f(X)=x$, it suffices to prove the assertion for $A=\O_E\{ X \}$, which is clear since $A$ is $\pi$-torsion free and $\phi \colon A \to A$ is $\phi$-equivariant.
\end{proof}

Following \cite[Remark 2.4]{BS},
we shall give a characterization of
$\delta_E$-rings in terms of 
\textit{ramified} Witt vectors.
For an $\O_E$-algebra $A$,
let
\[
W_{\O_E, \pi, 2}(A)
\]
denote the \textit{ring of $\pi$-typical Witt vectors of length $2$}: the underlying set of $W_{\O_E, \pi, 2}(A)$ is $A \times A$, and
for $(x_0, x_1), (y_0, y_1) \in W_{\O_E, \pi, 2}(A)$, we have
\begin{align*}
    (x_0, x_1) + (y_0, y_1)&= (x_0+y_0, x_1+y_1+(x^q_0+y^q_0-(x_0+y_0)^q)/\pi),\\
    (x_0, x_1) \cdot (y_0, y_1)&=(x_0y_0, x^q_0y_1+y^q_0x_1+\pi x_1y_1).
\end{align*}
If $\O_E=\Z_p$ and $\pi=p$, then $W_{\O_E, \pi, 2}(A)$ is the ring $W_2(A)$ of $p$-typical Witt vectors of length $2$.
For a detailed treatment of the rings of $\pi$-typical Witt vectors (of any length), we refer to \cite[Section 1.1]{Schneider} or \cite{Borger}.

\begin{rem}[{cf.\ \cite[Remark 2.4]{BS}}]\label{Remark:ramified Witt vectors of length 2}
The map
\[
\O_E \to W_{\O_E, \pi, 2}(A), \quad x \mapsto (x, \delta_{\O_E, \pi}(x))
\]
is a ring homomorphism for any $\O_E$-algebra $A$.
We regard $W_{\O_E, \pi, 2}(A)$ as an $\O_E$-algebra by this homomorphism.
Let
\[
\epsilon \colon W_{\O_E, \pi, 2}(A) \to A, \quad (x_0, x_1) \mapsto x_0
\]
denote the projection map, which is a homomorphism of $\O_E$-algebras.
For a $\delta_E$-structure $\delta_E \colon A \to A$,
the map
$s \colon A \mapsto W_{\O_E, \pi, 2}(A)$
defined by $x \mapsto (x, \delta_E(x))$
is a homomorphism of $\O_E$-algebras such that $\epsilon \circ s=\id_A$.
By this procedure, we obtain a bijection between
the set of $\delta_E$-structures on $A$
and
the set of homomorphisms $s \colon A \to W_{\O_E, \pi, 2}(A)$ of $\O_E$-algebras satisfying $\epsilon \circ s=\id_A$.
\end{rem}

\begin{rem}[{cf.\ \cite[Remark 2.7]{BS}}]\label{Remark:colimit and limit of delta rings}
It follows from Remark \ref{Remark:ramified Witt vectors of length 2} that the category of $\delta_E$-rings admits all limits and colimits, and they are preserved by the forgetful functor from the category of $\delta_E$-rings to the category of $\O_E$-algebras.
\end{rem}

The following two lemmas will be used frequently in the sequel.

\begin{lem}\label{Lemma:quotient delta pi ring}
Let $A=(A, \delta_E)$ be a $\delta_E$-ring and $I \subset A$ an ideal.
Then $I$ is stable under $\delta_E$ if and only if $A/I$ admits a $\delta_E$-structure that is compatible with the one on $A$.
If such a $\delta_E$-structure on $A/I$ exists, then it is unique.
\end{lem}

\begin{proof}
This follows immediately from the definition of $\delta_E$-structures (see the proof of \cite[Lemma 2.9]{BS}).
\end{proof}

\begin{lem}\label{Lemma:completion of delta ring}
    Let $A$ be a $\delta_E$-ring and
    let $I \subset A$ be a finitely generated ideal containing $\pi$.
    Then, for any integer $n \geq 1$, there exists an integer $m \geq 1$ such that for any $x \in A$, we have
    $\delta_E(x+I^m) \subset \delta_E(x)+I^n$.
    In particular, the $I$-adic completion of $A$ admits a unique $\delta_E$-structure that is compatible with the one on $A$.
\end{lem}

\begin{proof}
    The proof is the same as that of \cite[Lemma 2.17]{BS}.
\end{proof}

We shall discuss some properties of perfect $\delta_E$-rings, which are defined as follows.

\begin{defn}\label{Definition:perfect delta pi ring}
We say that a $\delta_E$-ring $A$ is \textit{perfect} if the Frobenius $\phi \colon A \to A$ is bijective.
\end{defn}

\begin{lem}[{\cite[Lemma 2.11]{Marks}}]\label{Lemma:perfect implies torsion free}
A perfect $\delta_E$-ring $A$ is $\pi$-torsion free.
\end{lem}

\begin{proof}
This is proved in \cite[Lemma 2.11]{Marks}, and follows from the same argument as in the proof of \cite[Lemma 2.28]{BS}.
\end{proof}

\begin{ex}\label{Example:perfect delta pi ring}
Let $R$ be an $\F_q$-algebra.
Assume that $R$ is perfect (i.e.\ $R \to R$, $x \mapsto x^p$ is bijective).
Let
$W(R)$
be the ring of $p$-typical Witt vectors and we define
\[
W_{\O_E}(R):=W(R) \otimes_{W(\F_q)} \O_E.
\]
Let 
$\phi \colon W_{\O_E}(R) \to W_{\O_E}(R)$ denote
the base change of the $q$-th power Frobenius of $W(R)$.
This is a lift of the $q$-th power Frobenius of $W_{\O_E}(R)/\pi=R$.
Since $W_{\O_E}(R)$ is $\pi$-torsion free,
we obtain the corresponding $\delta_E$-structure on $W_{\O_E}(R)$.
It is clear that $W_{\O_E}(R)$ is a perfect $\delta_E$-ring.
\end{ex}

\begin{lem}\label{Lemma:adjoint ramified witt and tilt}
The functor $R \mapsto W_{\O_E}(R)$ from the category of perfect $\F_q$-algebras to the category of $\pi$-adically complete $\O_E$-algebras admits a right adjoint given by
$
A \mapsto \varprojlim_{x \mapsto x^p} A/\pi A.
$
\end{lem}
\begin{proof}
This is well-known in the case where $\O_E=\Z_p$ (see 
\cite[Proposition 3.12]{SzamuelyZabradi} for example).
The general case follows from this special case.
\end{proof}

\begin{cor}[{\cite[Proposition 2.13]{Marks}}]\label{Corollary:equivalence of delta}
The following categories are equivalent:
\begin{itemize}
    \item The category $\mathcal{C}_1$ of $\pi$-adically complete perfect $\delta_E$-rings $(A, \delta_E)$.
    \item The category $\mathcal{C}_2$ of $\pi$-adically complete and $\pi$-torsion free $\O_E$-algebras $A$ such that 
    $A/\pi A$ is perfect.
    \item The category $\mathcal{C}_3$ of perfect $\F_q$-algebras $R$.
\end{itemize}
More precisely, the functors
$
\mathcal{C}_1 \to \mathcal{C}_2 \to \mathcal{C}_3 \to \mathcal{C}_1,
$
defined by $(A, \delta_E) \mapsto A$, $A \mapsto A/\pi$, $R \mapsto W_{\O_E}(R)$, respectively, are equivalences.
\end{cor}

\begin{proof}
Using Lemma \ref{Lemma:adjoint ramified witt and tilt},
one can prove the assertion in exactly the same way as \cite[Corollary 2.31]{BS}.
\end{proof}

\begin{cor}\label{Corollary:homomorphism from perfect delta pi-ring}
Let $A$ be a perfect $\delta_E$-ring
and $B$ a $\pi$-adically complete $\delta_E$-ring.
Then any homomorphism $A \to B$ of $\O_E$-algebras is a homomorphism of $\delta_E$-rings.
\end{cor}

\begin{proof}
We may assume that $A$ is $\pi$-adically complete.
It then follows from Lemma \ref{Lemma:adjoint ramified witt and tilt} and Corollary \ref{Corollary:equivalence of delta} that $A \to B$ is $\phi$-equivariant.
We recall that $\phi \colon B \to B$ is a homomorphism of $\delta_E$-rings by Lemma \ref{Lemma:frobenius is compatible with delta}.
Let $B^{\perf}$ be a limit of the diagram
\[
B \overset{\phi}{\leftarrow} B \overset{\phi}{\leftarrow} B  \leftarrow \cdots
\]
in the category of $\delta_E$-rings, which is a perfect $\delta_E$-ring.
Since $A$ is perfect, $A \to B$ factors through a $\phi$-equivariant homomorphism $A \to B^{\perf}$ of $\O_E$-algebras.
It follows from the $\pi$-torsion freeness of $B^{\perf}$ (see Lemma \ref{Lemma:perfect implies torsion free}) that $A \to B^{\perf}$ is a homomorphism of $\delta_E$-rings.
Since $A \to B$ is the composition $A \to B^{\perf} \to B$, the assertion follows.
\end{proof}

\subsection{$\O_E$-prisms}\label{Subsection:OE prism}

We now introduce $\O_E$-prisms.

\begin{defn}[{\cite[Definition 3.1]{Marks}}]\label{Definition:OE prism}
\
\begin{enumerate}
    \item An \textit{$\O_E$-prism} is a pair $(A, I)$ of a $\delta_E$-ring $A$ and a Cartier divisor $I \subset A$
    such that $A$ is derived $(\pi, I)$-adically complete and $\pi \in I + \phi(I)A$.
    \item We say that an $\O_E$-prism $(A, I)$ is \textit{bounded} if $A/I$ has bounded $p^\infty$-torsion.
    \item We say that an $\O_E$-prism $(A, I)$ is \textit{orientable} if $I$ is principal.
    \item An \textit{oriented and bounded} $\O_E$-prism $(A, d)$ is an orientable and bounded $\O_E$-prism $(A, I)$ with a choice of a generator $d \in I$.
    \item A map $f \colon (A, I) \to (A', I')$ of $\O_E$-prisms is a homomorphism $f \colon A \to A'$ of $\delta_E$-rings such that $f(I) \subset I'$.
\end{enumerate}
\end{defn}

If $\O_E=\Z_p$, then bounded $\O_E$-prisms are nothing but bounded prisms.

\begin{rem}\label{Remark:derived and classical complete}
Let $(A, I)$ be a bounded $\O_E$-prism.
By \cite[Proposition 2.5 (1)]{Tian} (see also Lemma \ref{Lemma:flat maps over prisms} below), we see that $A$ is $(\pi, I)$-adically complete.
Moreover, since $A/I$ is derived $\pi$-adically complete and has bounded $p^\infty$-torsion, it follows that $A/I$ is $\pi$-adically complete (see \cite[Lemma 4.7]{BMS2} for example).
\end{rem}

Let $A$ be a $\delta_E$-ring.
Following \cite[Definition 2.19]{BS}, we say that an element $d \in A$ is \textit{distinguished} if $\delta_E(d) \in A^\times$, i.e.\ $\delta_E(d)$ is a unit.
Since $\delta_{\O_E, \pi}(\pi)=1-\pi^{q-1} \in \O^\times_E$, we see that $\pi \in A$ is distinguished.

\begin{lem}\label{Lemma:distinguished}
Let $A$ be a $\delta_E$-ring and $d \in A$ an element.
Assume that 
$\pi$ is contained in the Jacobson radical $\mathrm{rad}(A)$ of $A$.
\begin{enumerate}
    \item Assume that $d=fh$ for some elements $f, h \in A$ with $f \in \mathrm{rad}(A)$. If $d$ is distinguished, then $f$ is distinguished and $h \in A^\times$.
    \item Assume that $d \in \mathrm{rad}(A)$.
Then $d$ is distinguished if and only if $\pi \in (d, \phi(d))$.
\end{enumerate}
\end{lem}

\begin{proof}
This can be proved exactly in the same way as \cite[Lemma 2.24, Lemma 2.25]{BS}.
See also \cite[Lemma 2.9]{Marks}.
\end{proof}

The following rigidity property plays a fundamental role in the theory of $\O_E$-prisms.

\begin{lem}[{cf.\ \cite[Lemma 3.5]{BS}}]\label{Lemma:rigidity}
Let $(A, I) \to (A', I')$ be a map of $\O_E$-prisms.
Then the natural homomorphism
$
I \otimes_A A' \to  IA'
$
is an isomorphism
and $IA'=I'$.
\end{lem}

\begin{proof}
By using \cite[Lemma 3.4]{Marks}, this follows from the same argument as in the proof of \cite[Lemma 3.5]{BS}.
We recall the argument in the case where both $(A, I)$ and $(A', I')$ are orientable.
It follows from Lemma \ref{Lemma:distinguished} (2) that
any generator $d \in I$ is distinguished.
Let $d' \in I'$ be a generator.
Then Lemma \ref{Lemma:distinguished} (1) implies that $d$ is mapped to
$ud'$ for some $u \in A'^\times$.
In particular, the image of $d$ in $A'$ is a nonzerodivisor, and we obtain
$
I \otimes_A A'\overset{\sim}{\to}  IA'
$
and
$IA'=I'$.
\end{proof}

The following lemma will be used several times in this paper.

\begin{lem}[{cf.\ the proof of \cite[Lemma 4.8]{BS}}]\label{Lemma:maps from perfectoid prisms to prism}
Let $A$ be a perfect $\delta_E$-ring
and
$(B, I)$
a bounded $\O_E$-prism.
Then
any homomorphism $A \to B/I$ of $\O_E$-algebras
lifts uniquely to a homomorphism $A \to B$ of $\delta_E$-rings.
\end{lem}

\begin{proof}
By Corollary \ref{Corollary:homomorphism from perfect delta pi-ring}, it is enough to check that the homomorphism $A \to B/I$ lifts uniquely to a homomorphism $A \to B$ of $\O_E$-algebras.
We may assume that $A$ is $\pi$-adically complete, and then
$A \simeq W_{\O_E}(R)$ for some perfect $\F_q$-algebra $R$
by Corollary \ref{Corollary:equivalence of delta}.
Since $B$ is $(\pi, I)$-adically complete and $B/I$ is $\pi$-adically complete,
it suffices to prove that for every integer $n \geq 1$,
any homomorphism
$W_n(R) \to B/(\pi^n, I)$ of $W_n(\F_q)$-algebras lifts uniquely to a homomorphism $W_n(R) \to B/(\pi^n, I^n)$ of $W_n(\F_q)$-algebras (here $W_n(R)=W(R)/p^n$).
This follows from the fact that
the cotangent complex $L_{W_n(R)/W_n(\F_q)}$
is acyclic
(\cite[Lemma 3.27 (1)]{SzamuelyZabradi}).
\end{proof}

We give some examples of $\O_E$-prisms.

\begin{ex}[{cf.\ \cite[Example 1.3 (1)]{BS}}]\label{Example:crystalline type frame}
Let $A$ be a $\pi$-adically complete and $\pi$-torsion free $\O_E$-algebra.
Let
$\phi \colon A \to A$
be a homomorphism
over $\O_E$ which is a lift of the $q$-th power Frobenius of $A/\pi$.
This homomorphism induces a $\delta_E$-structure on $A$, and the pair $(A, (\pi))$ is a bounded $\O_E$-prism.
\end{ex}

\begin{defn}[{$\O_E$-prism over $\O$}]\label{Definition:prism over O}
Let $k$ be a perfect field containing $\F_q$.
We will write
\[
\O := W(k) \otimes_{W(\F_q)} \O_E
\]
instead of $W_{\O_E}(k)$.
An $\O_E$-prism \textit{over} $\O$ is an $\O_E$-prism $(A, I)$ with a homomorphism $\O \to A$ of $\delta_E$-rings.
%There is an obvious notion of map of $\O_E$-prisms over $\O$.
\end{defn}

Let $\O=W(k) \otimes_{W(\F_q)} \O_E$ be as in Definition \ref{Definition:prism over O}.
Let
\[
\mathfrak{S}_\O:=\O[[t_1, \dotsc, t_n]]
\]
(where $n \geq 0$)
and let
$\phi \colon \mathfrak{S}_\O \to \mathfrak{S}_\O$ be the homomorphism such that $\phi(t_i)=t^q_i$ ($1 \leq i \leq n$) and its restriction to $\O$ agrees with the Frobenius of $\O$.
Since $\mathfrak{S}_\O$ is $\pi$-torsion free, $\phi$ gives rise to a $\delta_E$-structure on $\mathfrak{S}_\O$.

\begin{prop}[{cf.\ \cite[Example 1.3 (3)]{BS}}]\label{Proposition:Breuil-Kisin type frame}
Let $\mathcal{E} \in \mathfrak{S}_\O$ be a formal power series whose constant term is a uniformizer of $\O$.
Then the pair
$
(\mathfrak{S}_\O, (\mathcal{E}))
$
is a bounded $\O_E$-prism over $\O$.
\end{prop}

\begin{proof}
We shall show that $\pi \in (\mathcal{E}, \phi(\mathcal{E}))$; the other required conditions are clearly satisfied.
It is enough to check that $\delta_E(\mathcal{E}) \in \mathfrak{S}^\times_\O$.
For this, it suffices to show that the image of $\delta_E(\mathcal{E})$ in $\mathfrak{S}_\O/(t_1, \dotsc, t_n)=\O$ is a unit, which is clear since the constant term of $\mathcal{E}$ is a uniformizer of $\O$.
\end{proof}

We call $(\mathfrak{S}_\O, (\mathcal{E}))$
an $\O_E$-prism \textit{of Breuil--Kisin type}
in this paper.
Here $n$ could be any non-negative integer.
Such a pair is also considered in \cite{ChengChuangxun}.

\subsection{Perfectoid rings and $\O_E$-prisms}\label{Subsection:Perfectoid rings and pi-prisms}

The notion of (integral) perfectoid rings in the sense of \cite[Definition 3.5]{BMS} plays a central role in the theory of prismatic $G$-$\mu$-displays.
We refer to \cite[Section 3]{BMS} and \cite[Section 2]{CS} for basic properties of perfectoid rings.
We recall the definition of perfectoid rings and some notation from \cite[Section 3]{BMS}.

A ring $R$ is a \textit{perfectoid ring} if
there exists an element $\varpi \in R$
such that $p \in (\varpi)^p$ and $R$ is $\varpi$-adically complete, the Frobenius $R/p \to R/p$, $x \mapsto x^p$ is surjective, and the kernel of $\theta \colon W(R^\flat) \to R$ is principal.
Here
\[
R^\flat:=\varprojlim_{x \mapsto x^p} R/p
\]
and
$
\theta \colon W(R^\flat) \to R
$
is the unique homomorphism whose reduction modulo $p$ is the projection
$R^\flat \to R/p$, $(x_0, x_1, \dotsc) \mapsto x_0$.
The homomorphism $\theta$ is the counit of the adjunction given in Lemma \ref{Lemma:adjoint ramified witt and tilt} (in the case where $\O_E=\Z_p$).
By \cite[Lemma 3.9]{BMS}, there is an element $\varpi^\flat \in R^\flat$ such that $\theta([\varpi^\flat])$ is a unit multiple of $\varpi$, where $[-]$ denotes the Teichm\"uller lift.

%Here are some typical examples of perfectoid rings:

\begin{ex}\label{Example:algebraically closed valuation ring is perfectoid}
\ 
\begin{enumerate}
    \item An $\F_p$-algebra $R$ is a perfectoid ring if and only if it is perfect; see \cite[Example 3.15]{BMS}.
    \item Let $V$ be a $p$-adically complete valuation ring with algebraically closed fraction field.
Then $V$ is a perfectoid ring. This follows from \cite[Lemma 3.10]{BMS}.
\end{enumerate}
\end{ex}

Let $\O$ be as in Definition \ref{Definition:prism over O}.
If $R$ is a perfectoid ring over $\O$
(i.e.\ $R$ is a perfectoid ring with a ring homomorphism $\O \to R$),
then $R^\flat$ is naturally a $k$-algebra, and
$W_{\O_E}(R^\flat)=W(R^\flat) \otimes_{W(\F_q)} \O_E$ is an $\O$-algebra.
Let
\[
\theta_{\O_E} \colon W_{\O_E}(R^\flat) \to R
\]
be
the homomorphism induced from $\theta$.

\begin{lem}[{cf.\ \cite[Proposition II.1.4]{FS}}]\label{Lemma:kernel generator perfectoid}
The kernel $\Ker \theta_{\O_E}$ of $\theta_{\O_E}$ is a principal ideal.
Moreover, any generator of $\Ker \theta_{\O_E}$ is a nonzerodivisor in $W_{\O_E}(R^\flat)$.
\end{lem}

\begin{proof}
Let $\varpi \in R$ be an element such that
$R$ is $\varpi$-adically complete and $p \in (\varpi)^p$.
By \cite[Lemma 3.10 (i)]{BMS}, after replacing $\varpi$ by $\theta([(\varpi^\flat)^{1/p^n}])$ for some integer $n >0$, we have $\pi \in (\varpi)$.
Then we can write $\pi=\theta([\varpi^\flat] x)$ for some element $x \in W(R^\flat)$ since $\theta$ is surjective.
We shall show that $\pi-[\varpi^\flat] x$ generates $\Ker \theta_{\O_E}$.

Let $\mathcal{E}(T) \in W(k)[T]$ be the (monic) Eisenstein polynomial of $\pi \in \O$ so that we have $W(k)[T]/\mathcal{E}(T) \overset{\sim}{\to} \O$, $T \mapsto \pi$.
We see that
\begin{align*}
    W_{\O_E}(R^\flat)/(\pi-[\varpi^\flat] x) &\simeq W(R^\flat)[T]/(\mathcal{E}(T), T-[\varpi^\flat] x) \\
    &\simeq W(R^\flat)/\mathcal{E}([\varpi^\flat] x).
\end{align*}
It thus suffices to show that
$\mathcal{E}([\varpi^\flat] x)$
is a generator of the kernel $\Ker \theta$ of $\theta$.
It is clear that $\mathcal{E}([\varpi^\flat] x) \in \Ker \theta$.
Since $\mathcal{E}([\varpi^\flat] x)$ is a unit multiple of an element of the form
$p + [\varpi^\flat] y$,
the proof of \cite[Lemma 3.10]{BMS} shows that 
$\mathcal{E}([\varpi^\flat] x) \in \Ker \theta$
is a generator.

It remains to prove that any generator $\xi \in \Ker \theta_{\O_E}$ is a nonzerodivisor.
We recall that $\Ker \theta$ is generated by a nonzerodivisor $\xi' \in W(R^\flat)$.
Since $W(\F_q) \to \O_E$ is flat, the element $\xi'$ is a nonzerodivisor in $W_{\O_E}(R^\flat)$.
This implies that $\xi$ is a nonzerodivisor since we have $\xi' \in (\xi)=\Ker \theta_{\O_E}$.
\end{proof}

\begin{prop}[{cf.\ \cite[Example 1.3 (2)]{BS}}]\label{Proposition:perfectoid type}
Let $R$ be a perfectoid ring over $\O$ and we write $I_R:=\Ker \theta_{\O_E}$.
Then the pair
\[
(W_{\O_E}(R^\flat), I_R)
\]
is an orientable and bounded $\O_E$-prism over $\O$.
\end{prop}

\begin{proof}
By the proof of Lemma \ref{Lemma:kernel generator perfectoid},
we know that 
$I_R$
is generated by a nonzerodivisor of the form $\xi=\pi-[(\varpi')^\flat]b$ where $\varpi' \in R$ is such that $R$ is $\varpi'$-adically complete and $p \in (\varpi')^p$.
In order to show that $W_{\O_E}(R^\flat)$ is $(\pi, \xi)$-adically complete, it suffices to show that $W(R^\flat)$ is $(p, [(\varpi')^\flat])$-adically complete, which is easy to check (see also the proof of \cite[Proposition 2.1.11 (b)]{CS}).
Moreover $W_{\O_E}(R^\flat)/\xi=R$ has bounded $p^\infty$-torsion by \cite[Lemma 3.8]{BS}.

It remains to show that $\pi \in (\xi, \phi(\xi))$.
It suffices to prove that $\delta_E(\xi) \in W_{\O_E}(R^\flat)^\times$.
The image of $\delta_E(\xi)$ in $W_{\O_E}(R^\flat)/[(\varpi')^\flat]$
is equal to $1-\pi^{q-1}$ (we note that $W_{\O_E}(R^\flat)/[(\varpi')^\flat]$ is $\pi$-torsion free) and hence is a unit, which in turn implies that $\delta_E(\xi) \in W_{\O_E}(R^\flat)^\times$.
\end{proof}

A homomorphism
$R \to S$
of perfectoid rings over $\O$ induces a map
$
(W_{\O_E}(R^\flat), I_R) \to (W_{\O_E}(S^\flat), I_S)
$
of $\O_E$-prisms over $\O$.

\begin{rem}\label{Remark::perfect O_E prisms and perfectoid rings over O_E}
    We say that an $\O_E$-prism $(A, I)$ is \textit{perfect} if the $\delta_E$-ring $A$ is perfect.
    By \cite[Lemma 3.10]{Marks}, a perfect $\O_E$-prism $(A, I)$ is bounded and orientable.
    Moreover, in \cite[Theorem 3.18]{Marks}, it is proved that
    $A/I$ is a perfectoid ring.
    These facts, together with Lemma \ref{Lemma:maps from perfectoid prisms to prism} and Proposition \ref{Proposition:perfectoid type}, imply that the functor $(A, I) \mapsto A/I$ from the category of perfect $\O_E$-prisms to that of perfectoid rings over $\O_E$ is an equivalence, whose inverse is given by $R \mapsto (W_{\O_E}(R^\flat), I_R)$.
    This is an analogue of \cite[Theorem 3.10]{BS}.
\end{rem}

\subsection{Prismatic sites}\label{Subsection:prismatic sites}

For a ring $A$, let $D(A)$ denote the derived category of $A$-modules.
Let $I \subset A$ be a finitely generated ideal.
We say that a complex $M \in D(A)$
is \textit{$I$-completely flat}
(resp.\ \textit{$I$-completely faithfully flat})
if
$M \otimes^{\L}_A A/I$ 
is concentrated in degree $0$
and it is a flat
(resp.\ faithfully flat) $A/I$-module.
One can easily check that this definition is equivalent to the one introduced in \cite[Section 1.2]{BS}.

\begin{lem}\label{Lemma:flat maps over prisms}
Let $(A, I)$ be a bounded $\O_E$-prism.
\begin{enumerate}
    \item For a complex $M \in D(A)$, the derived $(\pi, I)$-adic completion of $M$ is isomorphic to
    \[
    R\varprojlim_n (M \otimes^{\L}_A A/(\pi, I)^n).
    \]
    In particular, if $M$ is $(\pi, I)$-completely flat, then the derived $(\pi, I)$-adic completion of $M$ is concentrated in degree $0$.
    \item Let M be an $A$-module.
    Assume that $M$ is $(\pi, I)$-completely flat and derived $(\pi, I)$-adically complete.
    Then $M$ is $(\pi, I)$-adically complete.
Moreover
the natural homomorphism $M \otimes_A I \to M$ is injective
and $M/I^nM$ has bounded $p^\infty$-torsion for any $n$.
\end{enumerate}
\end{lem}

\begin{proof}
(1) The assertion follows from \cite[Proposition 2.5 (1)]{Tian} or the proof of \cite[Lemma 3.7 (1)]{BS}.
This can also be deduced from the results discussed in \cite{Yekutieli21}; by \cite[Corollary 3.5, Theorem 3.11]{Yekutieli21}, it suffices to prove that the ideal $(\pi, I) \subset A$ is weakly proregular in the sense of \cite[Definition 3.2]{Yekutieli21}, which follows from the same argument as in the proof of \cite[Theorem 7.3]{Yekutieli21}.

(2) It follows from (1) that 
$M$ is $(\pi, I)$-adically complete.
The second statement can be proved in the same way as \cite[Lemma 3.7 (2)]{BS}.
(In \textit{loc.cit.}, we should assume that $M$ is derived $(p, I)$-adically complete.)
\end{proof}

We say that
a map
$f \colon (A, I) \to (A', I')$ of bounded $\O_E$-prisms
is a \textit{(faithfully) flat map}
if $A \to A'$ is 
$(\pi,I)$-completely (faithfully) flat.
If $f$ is a faithfully flat map, then we say that $(A', I')$ is a flat covering of $(A, I)$.

\begin{rem}\label{Remark:flat maps of prisms}
For a map 
$f \colon (A, I) \to (A', I')$ of bounded $\O_E$-prisms,
we have
$A' \otimes^{\L}_A A/I \simeq A'/I'$ by Lemma \ref{Lemma:rigidity}, which in turn implies that
\begin{equation}\label{equation:reduction of map of prisms}
    A' \otimes^{\L}_A A/(\pi, I) \simeq A'/I' \otimes^{\L}_{A/I} A/(\pi, I).
\end{equation}
In particular $f$ is a (faithfully) flat map
if and only if $A/I \to A'/I'$ is 
$\pi$-completely (faithfully) flat.
\end{rem}

\begin{defn}\label{Definition:category of bounded prisms}
Let $R$ be a $\pi$-adically complete $\O_E$-algebra.
Let
\[
(R)_{\Prism, \O_E}
\]
denote the category of bounded $\O_E$-prisms $(A, I)$ together with a homomorphism $R \to A/I$ of $\O_E$-algebras.
The morphisms $f \colon (A, I) \to (A', I')$ in $(R)_{\Prism, \O_E}$ are the maps of $\O_E$-prisms such that $A/I \to A'/I'$ is a homomorphism of $R$-algebras.
We endow
the opposite category
$
(R)^{\op}_{\Prism, \O_E}
$
with the topology generated by the faithfully flat maps.
This topology is called the \textit{flat topology}.
\end{defn}

We note that $(\O_E)_{\Prism, \O_E}$ is just the category of bounded $\O_E$-prisms.
If $\O_E=\Z_p$, then $(R)_{\Prism, \O_E}$ is the category $(R)_{\Prism}$ introduced in \cite[Remark 4.7]{BS}.
The category $(R)_{\Prism}$ (or its opposite) is called the \textit{absolute prismatic site of $R$}.

\begin{rem}\label{Remark:pushout in prismatic site}
A diagram
    \[
    (A_2, I_2) \overset{g}{\leftarrow} (A_1, I_1) \overset{f}{\rightarrow} (A_3, I_3)
    \]
    in $(R)_{\Prism, \O_E}$ such that $g$ is a flat map, admits a colimit (i.e.\ a pushout).
    Indeed, by Lemma \ref{Lemma:flat maps over prisms} (1), the $(\pi, I_3)$-adic completion
    $
    A:=(A_2 \otimes_{A_1} A_3)^\wedge
    $
    is isomorphic to the derived $(\pi, I_3)$-adic completion of $A_2 \otimes^{\L}_{A_1} A_3$.
    In particular $A$ is $(\pi, I_3)$-completely flat over $A_3$.
    (Here we use that $J$-complete flatness is preserved under base change and taking derived $J$-adic completions.)
    It follows from Remark \ref{Remark:colimit and limit of delta rings}
    and Lemma \ref{Lemma:completion of delta ring}
    that
    $A$ admits a unique $\delta_E$-structure that is compatible with the $\delta_E$-structures on $A_2$ and $A_3$.
    By Lemma \ref{Lemma:flat maps over prisms} (2),
    we see that $(A, I_3A)$ is a bounded $\O_E$-prism.
    By construction, $(A, I_3A)$ is a colimit of the above diagram.
    As a result, it follows that
    $
    (R)^{\op}_{\Prism, \O_E}
    $
    is indeed a site.
\end{rem}

\begin{rem}[{cf.\ \cite[Corollary 3.12]{BS}}]\label{Remark:structure sheaf}
A faithfully flat map
$(A, I) \to (A', I')$
induces
faithfully flat homomorphisms
$A/(\pi, I)^n \to A'/(\pi, I')^n$
and $A/(\pi^n, I) \to A'/(\pi^n, I')$
for any $n$.
It follows that the functors
\begin{align*}
    \O_{\Prism} &\colon (R)_{\Prism, \O_E} \to \mathrm{Set}, \quad (A, I) \mapsto A, \\
    \O_{\overline{\Prism}} &\colon
(R)_{\Prism, \O_E} \to \mathrm{Set}, \quad (A, I) \mapsto A/I
\end{align*}
form sheaves with respect to the flat topology.
Here $\mathrm{Set}$ is the category of sets.
\end{rem}

More generally, we have the following descent result.

\begin{prop}\label{Proposition:flat descent for finite projective modules}
The fibered category
over
$(\O_E)^{\op}_{\Prism, \O_E}$
which associates to each $(A, I) \in (\O_E)_{\Prism, \O_E}$
the category of finite projective $A$-modules
satisfies descent with respect to the flat topology.
The same holds for finite projective $A/I$-modules.
\end{prop}

\begin{proof}
For a ring $B$ and an ideal $J \subset B$ such that
$B$ is $J$-adically complete,
the natural functor
\[
\left.
\left\{
\begin{tabular}{c}
     finite projective $B$-modules
\end{tabular}
\right\}
\right.
\to {2-\varprojlim}_{n}
\left\{
\begin{tabular}{c}
     finite projective $B/J^n$-modules
\end{tabular}
\right\}
\]
is an equivalence; see for example \cite[Lemma 4.11]{Bhatt16}.
The assertions of the proposition follow from this fact and faithfully flat descent for finite projective modules over $A/(\pi, I)^n$ and $A/(\pi^n, I)$, respectively.
See also \cite[Lemma A.1, Proposition A.3]{Anschutz-LeBras}.
\end{proof}

\begin{defn}\label{Definition:category of prisms with flat topology}
For a bounded $\O_E$-prism $(A, I)$,
let
\[
(A, I)_\Prism
\]
be the category of bounded $\O_E$-prisms $(B, J)$ with a map $(A, I) \to (B, J)$.
%(Here we omit $\O_E$ from the notation.)
We endow $(A, I)^{\op}_\Prism$ with the flat topology induced from $(\O_E)^{\op}_{\Prism, \O_E}$.
\end{defn}

\begin{ex}\label{Example:prismatic sites}
\ 
\begin{enumerate}
    \item
    Let $\O$ be as in Definition \ref{Definition:prism over O}.
    The category $(\O)_{\Prism, \O_E}$ is the same as the category of bounded $\O_E$-prisms over $\O$ by Lemma \ref{Lemma:maps from perfectoid prisms to prism}.
    \item Let $R$ be a perfectoid ring over $\O_E$.
    It follows from Lemma \ref{Lemma:maps from perfectoid prisms to prism} that
    $(R)_{\Prism, \O_E}$ is the same as the category
    $(W_{\O_E}(R^\flat), I_R)_\Prism$.
\end{enumerate}
\end{ex}

Let $A \to B$ be a ring homomorphism
and $I \subset A$ a finitely generated ideal.
We say that
$A \to B$
is \textit{$I$-completely \'etale}
if
$B$ is derived $I$-adically complete,
$B \otimes^{\L}_A A/I$ 
is concentrated in degree $0$,
and $B/IB$ is \'etale over $A/I$.
We write
$A_{I\mathchar`-\et}$
for the category of $I$-completely \'etale $A$-algebras.
If $I=0$, then $A_{I\mathchar`-\et}$ is just the category $A_{\et}$ of \'etale $A$-algebras.

\begin{lem}\label{Lemma:etale morphism and reduction}
Let $(A, I)$ be a bounded $\O_E$-prism.
\begin{enumerate}
    \item A ring homomorphism $A/I \to C$ is $\pi$-completely \'etale if and only if $C$ is $\pi$-adically complete and $C/\pi^n$ is \'etale over $A/(\pi^n, I)$ for every integer $n \geq 1$.
    If this is the case, then $C$ has bounded $p^{\infty}$-torsion.
    \item A ring homomorphism $A \to B$ is $(\pi, I)$-completely \'etale if and only if $B$ is $(\pi, I)$-adically complete and $B/(\pi, I)^n$ is \'etale over $A/(\pi, I)^n$ for every $n \geq 1$.
    If this is the case, then
    $B \otimes^\L_A A/I \overset{\sim}{\to} B/IB$ and $A/I \to B/IB$ is $\pi$-completely \'etale.
    \item
   The functors
   \[
   A_{(\pi, I)\mathchar`-\et} \to (A/I)_{\pi\mathchar`-\et} \to (A/(\pi, I))_{\et},
   \]
   where the first one is defined by $B \mapsto B/IB$ and the second one is defined by $C \mapsto C/\pi$,
   are equivalences.
\end{enumerate}
\end{lem}

\begin{proof}
This result is well-known to specialists, but we include a proof for the convenience of the reader.

(1) Assume that $A/I \to C$ is $\pi$-completely \'etale.
Then, since $A/I$ has bounded $p^{\infty}$-torsion,
\cite[Lemma 4.7]{BMS2} implies that $C$ is $\pi$-adically complete and has bounded $p^{\infty}$-torsion.
Since $C/\pi^n$ is flat over $A/(\pi^n, I)$ 
and $C/\pi$ is \'etale over $A/(\pi, I)$,
it follows that $C/\pi^n$ is \'etale over $A/(\pi^n, I)$ for any $n$.

We next prove the ``if'' direction, so we assume that $C$ is $\pi$-adically complete and $C/\pi^n$ is \'etale over $A/(\pi^n, I)$ for any $n$.
We want to show that $C \otimes^{\L}_{A/I} A/(\pi, I)$ is concentrated in degree $0$.
There exists an \'etale $A/I$-algebra $C_0$ such that $C_0/\pi \simeq C/\pi$ over $A/(\pi, I)$; see for example \cite[Section 1.1]{Arabia} or \cite[Tag 04D1]{SP} (this is known as a special case of Elkik's result \cite{Elkik}).
One easily sees that
the derived $\pi$-adic completion of $C_0$, the $\pi$-adic completion of $C_0$,
and $C$ are isomorphic to each other.
Then we obtain
\[
C \otimes^{\L}_{A/I} A/(\pi, I) \simeq C_0 \otimes^{\L}_{A/I} A/(\pi, I) \simeq C_0/\pi.
\]
This proves the assertion.

(2) Assume that $A \to B$ is $(\pi, I)$-completely \'etale.
We easily see that $B/(\pi, I)^n$ is \'etale over $A/(\pi, I)^n$.
It follows from Lemma \ref{Lemma:flat maps over prisms} that $B$ is $(\pi, I)$-adically complete
and
we have
$B \otimes^\L_A A/I \overset{\sim}{\to} B/IB$.
It is then clear that $A/I \to B/IB$ is $\pi$-completely \'etale.

The ``if'' direction can be proved by the same argument as in (1).
Suppose that $B$ is $(\pi, I)$-adically complete and $B/(\pi, I)^n$ is \'etale over $A/(\pi, I)^n$ for any $n$.
As above, 
there exists an \'etale $A$-algebra $B_0$ such that
the $(\pi, I)$-adic completion of $B_0$ is isomorphic to $B$.
It follows from Lemma \ref{Lemma:flat maps over prisms} (1)
that
$B$ is isomorphic to the derived 
$(\pi, I)$-adic completion of $B_0$, which in turn implies that $B$ is $(\pi, I)$-completely \'etale over $A$.

(3) This follows from the proofs of (1) and (2).
\end{proof}

\begin{lem}[{cf.\ \cite[Lemma 2.18]{BS}}]\label{Lemma:etale morphism prism}
Let $(A, I)$ be a bounded $\O_E$-prism
and $A \to B$ a $(\pi, I)$-completely \'etale homomorphism.
Then $B$ admits a unique $\delta_E$-structure compatible with that on $A$.
Moreover, the pair $(B, IB)$ is a bounded $\O_E$-prism.
\end{lem}

\begin{proof}
It suffices to prove the first statement by Lemma \ref{Lemma:flat maps over prisms}.
For this, we proceed as in the proof of \cite[Lemma 2.18]{BS}.

We regard
$W_{\O_E, \pi, 2}(B)$
as an $A$-algebra via the composition $A \to W_{\O_E, \pi, 2}(A) \to W_{\O_E, \pi, 2}(B)$,
where $A \to W_{\O_E, \pi, 2}(A)$ is the homomorphism corresponding to the $\delta_E$-structure on $A$ (Remark \ref{Remark:ramified Witt vectors of length 2}).
Then $W_{\O_E, \pi, 2}(B)$ is $(\pi, I)$-adically complete.
Indeed, we have an exact sequence of $A$-modules
\[
0 \to \phi_*B \overset{V}{\to} W_{\O_E, \pi, 2}(B) \overset{\epsilon}{\to} B \to 0,
\]
where we write $\phi_*B$ for $B$ regarded as an $A$-algebra via the composition $A \overset{\phi}{\to} A \to B$, and
$V \colon \phi_*B \to W_{\O_E, \pi, 2}(B)$ is defined by $x \mapsto (0, x)$.
Since $B \otimes^{\L}_A A/(\pi, I)^n$ is concentrated in degree $0$
and both $\phi_*B$ and $B$ are $(\pi, I)$-adically complete, we can conclude that $W_{\O_E, \pi, 2}(B)$ is $(\pi, I)$-adically complete.

As in the proof of Lemma \ref{Lemma:etale morphism and reduction},
there exists an \'etale $A$-algebra $B_0$ such that
the $(\pi, I)$-adic completion of $B_0$ is isomorphic to $B$.
Since $W_{\O_E, \pi, 2}(B)$ is $(\Ker \epsilon)$-adically complete and $A \to B_0$ is \'etale,
there exists a unique homomorphism $s_0 \colon B_0 \to W_{\O_E, \pi, 2}(B)$ of $A$-algebras such that $\epsilon \circ s_0$ coincides with $B_0 \to B$.
Then, since $W_{\O_E, \pi, 2}(B)$ is $(\pi, I)$-adically complete, we see that
$s_0 \colon B_0 \to W_{\O_E, \pi, 2}(B)$
extends to a unique homomorphism $s \colon B \to W_{\O_E, \pi, 2}(B)$ of $A$-algebras such that $\epsilon \circ s= \id_B$, which corresponds to a unique $\delta_E$-structure on $B$ compatible with that on $A$ by virtue of Remark \ref{Remark:ramified Witt vectors of length 2}.
\end{proof}

\begin{ex}\label{Example:perfectoid ring etale morphism}
Let $R$ be a perfectoid ring over $\O_E$
and let $R \to S$ be a $\pi$-completely \'etale (or equivalently, $p$-completely \'etale) homomorphism.
By \cite[Corollary 2.10]{Anschutz-LeBras} or \cite[Lemma 8.11]{Lau18}, we see that $S$ is a perfectoid ring.
Moreover, the isomorphism $(\ref{equation:reduction of map of prisms})$ implies that
$W_{\O_E}(R^\flat) \to W_{\O_E}(S^\flat)$ is $(\pi, I_R)$-completely \'etale.
\end{ex}

Let $(A, I)$ be a bounded $\O_E$-prism.
We say that
a homomorphism $B \to B'$
of $(\pi, I)$-completely \'etale $A$-algebras
is a \textit{$(\pi, I)$-completely \'etale covering} if
\[
\Spec B'/(\pi, I) \to \Spec B/(\pi, I)
\]
is surjective, or equivalently, the homomorphism $B \to B'$ is $(\pi, I)$-completely faithfully flat.
We note that $B \to B'$ is automatically $(\pi, I)$-completely \'etale.

\begin{defn}\label{Definition:etale prismatic site}
We write
\[
(A, I)_\et
\]
for the category of $(\pi, I)$-completely \'etale $A$-algebras instead of $A_{(\pi, I)\mathchar`-\et}$.
We endow
the opposite category
$
(A, I)^{\op}_\et
$
with the topology generated by the $(\pi, I)$-completely \'etale coverings, which is called the \textit{$(\pi, I)$-completely \'etale topology}.
\end{defn}

The category
$(A, I)^{\op}_\et$
admits fiber products.
Indeed, a colimit of the diagram
$
C \leftarrow B \rightarrow D
$
in $(A, I)_\et$
is given by the $(\pi, I)$-adic completion of $C \otimes_B D$; see Remark \ref{Remark:pushout in prismatic site}.
It follows that $(A, I)^{\op}_\et$ is a site.

\begin{rem}\label{Remark:etale category is a subcategory}
Recall that, for a $(\pi, I)$-completely \'etale $A$-algebra $B \in (A, I)_\et$,
the pair $(B, IB)$ is naturally a bounded $\O_E$-prism
by Lemma \ref{Lemma:etale morphism prism}.
We can regard $(A, I)_\et$ as a full subcategory of the category $(A, I)_\Prism$.
The $(\pi, I)$-completely \'etale topology on $(A, I)^{\op}_\et$ coincides with the one induced from the flat topology.
\end{rem}

\begin{rem}\label{Remark:flat locally orientable}
    Any bounded $\O_E$-prism $(A, I)$ admits a $(\pi, I)$-completely \'etale covering $A \to B$
    such that $(B, IB)$ is orientable.
    Indeed, there exists an \'etale and faithfully flat homomorphism $A \to A'$ such that $IA'$ is principal.
    The $(\pi, I)$-adic completion $B$ of $A'$ is a $(\pi, I)$-completely \'etale covering of $A$ (by Lemma \ref{Lemma:flat maps over prisms}).
    Since $IA'$ is principal, the bounded $\O_E$-prism $(B, IB)$ is orientable.
\end{rem}

\subsection{Prismatic envelopes for regular sequences}\label{Subsection:Prismatic envelopes for regular sequences}

The existence and the flatness of prismatic envelopes for regular sequences are proved in \cite[Proposition 3.13]{BS}.
In this subsection, 
we give an analogous result for $\O_E$-prisms.
We will freely use the formalism of \textit{animated rings} here.
For the definition and properties of animated rings, see for example \cite[Section 5]{CS} and \cite[Appendix A]{BL}.
(See also \cite[Chapter 25]{SAG}, where animated rings are called simplicial rings.)

We recall some terminology from \cite{CS} and \cite{BS}.
To an animated ring $A$, we can attach a graded ring of homotopy groups
$\oplus_{n \geq 0}\pi_n(A)$.
For an animated ring $A$, the \textit{derived quotient} of $A$ with respect to a sequence $x_1, \dotsc, x_n \in \pi_0(A)$ is defined by
\[
A/^{\L}(x_1, \dotsc, x_n):= A \otimes^{\L}_{\Z[X_1, \dotsc, X_n]}  \Z[X_1, \dotsc, X_n]/(X_1, \dotsc, X_n).
\]
Here $\Z[X_1, \dotsc, X_n] \to A$ is a morphism such that the induced ring homomorphism
$\Z[X_1, \dotsc, X_n] \to \pi_0(A)$ is given by $X_i \mapsto x_i$.
In \cite{BS}, the derived quotient is denoted by $\Kos(A; x_1, \dotsc, x_n)$.
We say that a morphism $A \to B$ of animated rings is \textit{flat} (resp.\ \textit{faithfully flat}) if
$\pi_0(B)$ is flat (resp.\ faithfully flat) over $\pi_0(A)$ and we have
$\pi_n(A) \otimes_{\pi_0(A)} \pi_0(B) \overset{\sim}{\to} \pi_n(B)$ for any $n \geq 0$.

Before stating the result, let us quickly recall the definition of an $\O_E$-PD structure, and its relation to $\delta_E$-structures.

\begin{defn}[{\cite[Section 10]{HopkinsGross}, \cite[Definition 14]{Faltings02}}]\label{Definition:pi PD strudture}
    Let $A$ be an $\O_E$-algebra and $I \subset A$ an ideal.
    An \textit{$\O_E$-PD structure} on $I$ is a map $\gamma_\pi \colon I \to I$ of sets with the following properties:
    \begin{enumerate}
    \item $\pi\gamma_\pi(x)=x^q$.
    \item $\gamma_\pi(ax)=a^q\gamma_\pi(x)$ where $a \in A$.
    \item $\gamma_\pi(x+y)=\gamma_\pi(x)+\gamma_\pi(y)+(x^q+y^q-(x+y)^q)/\pi$.
    \end{enumerate}
\end{defn}

\begin{ex}[{\cite[Section 7]{Faltings02}}]\label{Example:pi PD polynomial ring}
    Let $n \geq 0$ be an integer
    and
    let
    $\O_E[ (X_{i, j}) ]$
    be the polynomial ring with variables $X_{i, j}$ indexed by integers $i, j$ with $1 \leq i \leq n$ and $j \geq 0$.
    We write
    \[
    \O_E[X_1, \dotsc, X_n]^{\mathrm{PD}}
    \]
    for the quotient of $\O_E[ (X_{i, j}) ]$ by the ideal generated by the elements $X^q_{i, j}-\pi X_{i, j+1}$ for all $i, j$.
    The image of $X_{i, 0}$ in $\O_E[X_1, \dotsc, X_n]^{\mathrm{PD}}$ is denoted by $X_i$.
    We see that $\O_E[X_1, \dotsc, X_n]^{\mathrm{PD}}$ is canonically isomorphic to the $\O_E$-subalgebra of $E[X_1, \dotsc, X_n]$ generated by $X^{q^j}_i/\pi^{1 + q + \cdots + q^{j-1}}$ ($1 \leq i \leq n$ and $j \geq 0$).
    The ideal $I^{\mathrm{PD}} \subset \O_E[X_1, \dotsc, X_n]^{\mathrm{PD}}$ generated by the elements $X_{i, j}$ admits an $\O_E$-PD structure $\gamma_\pi$ such that $\gamma_\pi(X_{i, j})=X_{i, j+1}$.
    In fact, the pair $(\O_E[X_1, \dotsc, X_n]^{\mathrm{PD}}, I^{\mathrm{PD}})$ is the $\O_E$-PD envelope (in the usual sense) of the polynomial ring $\O_E[X_1, \dotsc, X_n]$ with respect to the ideal $(X_1, \dotsc, X_n)$.
\end{ex}

\begin{lem}[{cf.\ \cite[Lemma 2.38]{BS}}]\label{Lemma:pd envelope regular sequence}
    Let $B$ be a $\pi$-torsion free $\O_E$-algebra.
    Let $x_1, \dotsc, x_n \in B$ be a sequence
    such that
    $(B/\pi)/^\L(\overline{x}_1, \dotsc, \overline{x}_n)$
    is concentrated in degree $0$, where $\overline{x}_1, \dotsc, \overline{x}_n \in B/\pi$ are the images of $x_1, \dotsc, x_n \in B$.
    We set
    \[
    C:= B \otimes^\L_{\O_E[X_1, \dotsc, X_n]} \O_E[X_1, \dotsc, X_n]^{\mathrm{PD}}
    \]
    where
    $\O_E[X_1, \dotsc, X_n] \to B$ is defined by $X_i \mapsto x_i$.
    Then $C$ is concentrated in degree $0$ and $\pi_0(C)$ is $\pi$-torsion free.
    Moreover the pair $(\pi_0(C), I^{\mathrm{PD}}\pi_0(C))$ is the $\O_E$-PD envelope of $B$ with respect to the ideal $(x_1, \dotsc, x_n)$.
\end{lem}

In the following, we will write
\[
D_{(x_1, \dotsc, x_n)}(B):=\pi_0(C)=B \otimes_{\O_E[X_1, \dotsc, X_n]} \O_E[X_1, \dotsc, X_n]^{\mathrm{PD}}.
\]

\begin{proof}
The proof is identical to that of \cite[Lemma 2.38]{BS}.
\end{proof}

We also need the following construction.
Let $B$ be a $\delta_E$-ring.
Let $d \in B$ be an element, and $x_1, \dotsc, x_n \in B$ a sequence.
We set
$B\{ X \}:=B \otimes_{\O_E} \O_E\{ X \}$
and
let
\[
B \{ X_1, \dotsc, X_n \}
\]
be the $n$-th tensor power of $B\{ X \}$ over $B$.
We consider the following diagram of $\delta_E$-rings
\[
B \overset{f}{\leftarrow} B \{ X_1, \dotsc, X_n \} \overset{g}{\rightarrow} B \{ Y_1, \dotsc, Y_n \},
\]
where $f$ is defined by $X_i \mapsto x_i$ and $g$ is defined by $X_i \mapsto dY_i$.
Let
\[
B\{ x_1/d, \dotsc, x_n/d \}
\]
denote
the pushout of this diagram, which is a $\delta_E$-ring over $B$ with the following property:
For a homomorphism $B \to C$ of $\delta_E$-rings such that the image of $d$ is a nonzerodivisor in $C$ and 
$x_i \in dC$ for all $i$, there exists a unique homomorphism
$B\{ x_1/d, \dotsc, x_n/d \} \to C$
of $\delta_E$-rings over $B$.
We let $x_i/d \in B\{ x_1/d, \dotsc, x_n/d \}$ denote the image of $Y_i \in B \{ Y_1, \dotsc, Y_n \}$.

Using this construction, we can relate $\delta_E$-structures to $\O_E$-PD structures.

\begin{lem}[{cf.\ \cite[Lemma 2.36]{BS}}]\label{Lemma:pd envelope and delta structure basic case}
We have a natural isomorphism
\[
(\O_E \{ X_1, \dotsc, X_n \}) \{ \phi(X_1)/\pi, \dotsc, \phi(X_n)/\pi \} \simeq D_{(X_1, \dotsc, X_n)}(\O_E \{ X_1, \dotsc, X_n \})
\]
of $\O_E \{ X_1, \dotsc, X_n \}$-algebras.
\end{lem}

\begin{proof}
    This can be proved in the same way as \cite[Lemma 2.36]{BS}.
    We include a sketch of the proof.

    We set $B:=\O_E \{ X_1, \dotsc, X_n \}$.
    Since the Frobenius
    $\phi \colon B \to B$
    is faithfully flat by Example \ref{Example:free delta pi rings},
    it follows that
    $C:=B \{ \phi(X_1)/\pi, \dotsc, \phi(X_n)/\pi \}$
    is $\pi$-torsion free.
    Since $\phi(X_i)/\pi=X^q_i/\pi + \delta_E(X_i)$,
    then $C$ can be regarded as
    the smallest $\delta_E$-subring of $B[1/\pi]$ which contains $B$ and $X^q_i/\pi$ ($1 \leq i \leq n$).
    On the other hand, since $D:=D_{(X_1, \dotsc, X_n)}(B)$ is $\pi$-torsion free by Lemma \ref{Lemma:pd envelope regular sequence}, we see that $D$ is the $\O_E$-subalgebra of $B[1/\pi]$ generated by
    $B$ and $X^{q^j}_i/\pi^{1 + q + \cdots + q^{j-1}}$ ($1 \leq i \leq n$ and $j \geq 1$).

    We shall prove that $C=D$ in $B[1/\pi]$.
    Let us first show that $D \subset C$.
    For this, it suffices to prove that for every $1 \leq i \leq n$, we have
    $
    X^{q^j}_i/\pi^{1 + q + \cdots + q^{j-1}} \in C
    $
    for all $j \geq 1$.
    We proceed by induction on $j$.
    The assertion is clear for $j=1$.
    Assume that the assertion is true for some $j \geq 1$.
    Then we have
    \begin{equation*}
\begin{split}
\delta_E(X^{q^j}_i/\pi^{1 + q + \cdots + q^{j-1}}) &= \phi(X^{q^j}_i/\pi^{1 + q + \cdots + q^{j-1}})/\pi-(X^{q^j}_i/\pi^{1 + q + \cdots + q^{j-1}})^q/\pi \\
&=\phi(X_i)^{q^j}/\pi^{2 + q + \cdots + q^{j-1}} - X^{q^{j+1}}_i/\pi^{1 + q + \cdots + q^{j}} \in C.
\end{split}
\end{equation*}
To prove that $X^{q^{j+1}}_i/\pi^{1 + q + \cdots + q^{j}} \in C$, it is enough to show that $\phi(X_i)^{q^j}/\pi^{2 + q + \cdots + q^{j-1}} \in C$.
Since
$
\phi(X_i)/\pi=X^q_i/\pi+ \delta_E(X_i)
$
is contained in $C$,
the assertion now follows from the inequality
$q^j \geq 2 + q + \cdots + q^{j-1}$.

It remains to prove that $C \subset D$.
Since
$
\phi(X_i)/\pi=X^q_i/\pi+ \delta_E(X_i)
$
is contained in $D$,
the inequality
$q^j \geq 2 + q + \cdots + q^{j-1}$ for $j \geq 1$
implies that
\[
\phi(X^{q^j}_i/\pi^{1 + q + \cdots + q^{j-1}})=\phi(X_i)^{q^j}/\pi^{1 + q + \cdots + q^{j-1}} \in \pi D
\]
for $j \geq 1$.
Since
\[
(X^{q^j}_i/\pi^{1 + q + \cdots + q^{j-1}})^q=\pi(X^{q^{j+1}}_i/\pi^{1 + q + \cdots + q^{j}}) \in \pi D,
\]
we see that $\phi$ preserves $D$ and that the reduction modulo $\pi$ of $\phi \colon D \to D$ is the $q$-th power Frobenius.
This implies that $C \subset D$.
\end{proof}

\begin{cor}[{cf.\ \cite[Corollary 2.39]{BS}}]\label{Corollary:pd envelope and delta structure}
    Let $B$ be a $\pi$-torsion free $\delta_E$-ring.
    Let $x_1, \dotsc, x_n \in B$ be a sequence
    such that
    $(B/\pi)/^\L(\overline{x}_1, \dotsc, \overline{x}_n)$
    is concentrated in degree $0$.
    We set
    $D:=B \otimes^\L_{\O_E \{ X_1, \dotsc, X_n \}} \O_E \{ Y_1, \dotsc, Y_n \}$
    where $\O_E \{ X_1, \dotsc, X_n \} \to B$ is defined by $X_i \mapsto \phi(x_i)$ and
    $\O_E \{ X_1, \dotsc, X_n \} \to \O_E \{ Y_1, \dotsc, Y_n \}$ is defined by $X_i \mapsto \pi Y_i$.
    Then $D$ is concentrated in degree $0$.
    Moreover
    \[
    \pi_0(D)=B \{ \phi(x_1)/\pi, \dotsc, \phi(x_n)/\pi \}
    \]
    is $\pi$-torsion free, and is isomorphic to $D_{(x_1, \dotsc, x_n)}(B)$ as a $B$-algebra.
\end{cor}

\begin{proof}
    Since
    $\phi \colon \O_E \{ X_1, \dotsc, X_n \} \to \O_E \{ X_1, \dotsc, X_n \}$
    is faithfully flat by Example \ref{Example:free delta pi rings},
    we have an identification
    \[
    D = B \otimes^\L_{\O_E \{ X_1, \dotsc, X_n \}} (\O_E \{ X_1, \dotsc, X_n \})\{ \phi(X_1)/\pi, \dotsc, \phi(X_n)/\pi \}
    \]
    where $\O_E \{ X_1, \dotsc, X_n \} \to B$ is defined by $X_i \mapsto x_i$.
    Then Lemma \ref{Lemma:pd envelope and delta structure basic case} implies that
    \[
    D \simeq B \otimes^\L_{\O_E \{ X_1, \dotsc, X_n \}} D_{(X_1, \dotsc, X_n)}(\O_E \{ X_1, \dotsc, X_n \}).
    \]
    By Lemma \ref{Lemma:pd envelope regular sequence}, we have
    \[
    D_{(X_1, \dotsc, X_n)}(\O_E \{ X_1, \dotsc, X_n \}) \simeq \O_E \{ X_1, \dotsc, X_n \} \otimes^\L_{\O_E [X_1, \dotsc, X_n ]}\O_E[X_1, \dotsc, X_n]^{\mathrm{PD}}
    \]
    and thus
    $
    D \simeq B \otimes^\L_{\O_E [X_1, \dotsc, X_n ]}\O_E[X_1, \dotsc, X_n]^{\mathrm{PD}}.
    $
    The assertion then follows by applying Lemma \ref{Lemma:pd envelope regular sequence} again.
\end{proof}

Now we can state the desired result:

\begin{prop}[{cf.\ \cite[Proposition 3.13]{BS}}]\label{Proposition:prismatic envelope}
    Assume that $(A, I)$ is an orientable and bounded $\O_E$-prism.
    Let $d \in I$ be a generator.
    Let $B$ be a $\delta_E$-ring over $A$.
    Let $x_1, \dotsc, x_n \in B$ be a sequence such that
    the induced morphism
    \begin{equation}\label{equation:regular sequence derived quotient}
        A/^\L(\pi, d) \to B/^\L(\pi, d, x_1, \dotsc, x_n)
    \end{equation}
    of animated rings is flat.
    (In other words, the sequence $x_1, \dotsc, x_n \in B$ is $(\pi, I)$-completely regular relative to $A$ in the sense of \cite[Definition 2.42]{BS}.)
    We set $J:=(d, x_1, \dotsc, x_n) \subset B$.
    Then the following assertions hold:
    \begin{enumerate}
        \item The $(\pi, I)$-adic completion $B\{ J/I \}^{\wedge}$ of $B\{ x_1/d, \dotsc, x_n/d \}$ is $(\pi, I)$-completely flat over $A$.
        In particular,
        the pair
        \[
        (B\{ J/I \}^{\wedge}, IB\{ J/I \}^{\wedge})
        \]
        is an orientable and bounded $\O_E$-prism.
        Moreover
        $B\{ J/I \}^{\wedge}$ is $(\pi, I)$-completely faithfully flat over $A$
        if the morphism
        $(\ref{equation:regular sequence derived quotient})$
        is faithfully flat.
        \item For a bounded $\O_E$-prism $(D, ID)$ over $(A, I)$
        and a homomorphism $B \to D$ of $\delta_E$-rings over $A$ such that $JD \subset ID$,
        there exists a unique map of $\O_E$-prisms
        \[
        (B\{ J/I \}^{\wedge}, IB\{ J/I \}^{\wedge}) \to (D, ID)
        \]
        over $B$.
        Moreover,
        the formation of $B\{ J/I \}^{\wedge}$ commutes with base change along any map $(A, I) \to (A', I')$ of bounded $\O_E$-prisms, and also commutes with base change along any $(\pi, I)$-completely flat homomorphism $B \to B'$ of $\delta_E$-rings.
    \end{enumerate}
\end{prop}

See \cite[Proposition 3.13 (3)]{BS} for the precise meaning of the last statement.

\begin{proof}
Let $C:=B \otimes^\L_{B \{ X_1, \dotsc, X_n \}} B \{ Y_1, \dotsc, Y_n \}$ be the pushout of the diagram
$
B \leftarrow B \{ X_1, \dotsc, X_n \} \rightarrow B \{ Y_1, \dotsc, Y_n \}
$
in the $\infty$-category of animated rings, where the first map is defined by $X_i \mapsto x_i$ and the second one is defined by $X_i \mapsto dY_i$.
It suffices to prove that
if the morphism
$(\ref{equation:regular sequence derived quotient})$
is flat (resp.\ faithfully flat), then $C$ is $(\pi, I)$-completely flat (resp.\ $(\pi, I)$-completely faithfully flat) over $A$.
Indeed, if this is true, then the derived $(\pi, I)$-adic completion of $C$ is isomorphic to $B\{ J/I \}^{\wedge}$,
and in particular $B\{ J/I \}^{\wedge}$ is $(\pi, I)$-completely flat (resp.\ $(\pi, I)$-completely faithfully flat) over $A$.
It is then easy to see that $B\{ J/I \}^{\wedge}$ has the desired properties.

In order to prove that $C$ is $(\pi, I)$-completely flat (resp.\ $(\pi, I)$-completely faithfully flat) over $A$, one can argue as in the proof of \cite[Proposition 3.13]{BS}.
(The faithful flatness is not discussed in \textit{loc.cit.}, but the same argument works.)
The only difference is that we have to use $\O_E$-PD structures, instead of usual PD structures.
Here we need the results established above (e.g.\ Corollary \ref{Corollary:pd envelope and delta structure}).
The details are left to the reader.
\end{proof}

The bounded $\O_E$-prism
$(B\{ J/I \}^{\wedge}, IB\{ J/I \}^{\wedge})$
is called the \textit{prismatic envelope} of $B$ over $(A, I)$ with respect to the ideal $J$.

\begin{rem}\label{Remark:animated pi delta rings}
As in \cite[Proposition 3.13]{BS}, we need to use animated $\delta_E$-rings in the proof of Proposition \ref{Proposition:prismatic envelope}.
(For example, in the proof of \cite[Proposition 3.13]{BS}, the notion of animated $\delta_E$-rings is used to obtain the description of the bottom right vertex $C''$ of the diagram appearing there.)
One can define the notion of animated $\delta_E$-rings in the same way as in \cite[Section 5]{Mao} (i.e.\ by animating $\delta_E$-rings).
Alternatively, we can follow the approach employed in \cite[Appendix A]{BL2}.
\end{rem}

\section{Displayed Breuil--Kisin modules}\label{Section:Displayed Breuil--Kisin module}

In this section, we study Breuil--Kisin modules for bounded $\O_E$-prisms.
We introduce the notions of a displayed Breuil--Kisin module and of a minuscule Breuil--Kisin module.
These objects serve as examples of prismatic $G$-$\mu$-displays introduced in Section \ref{Section:prismatic G display}.

\subsection{Displayed Breuil--Kisin modules}\label{Subsection:Displayed Breuil--Kisin modules}

We use the following notation.
Let $A$ be a ring.
For $A$-modules $M$ and $N$,
the set of $A$-linear homomorphisms $M \to N$ is denoted by $\Hom_A(M, N)$.
Let $I \subset A$ be a Cartier divisor.
For an integer $n \geq 1$, we define
$I^{-n}:= \Hom_A(I^{n}, A)$.
We have a natural injection $I^{-n} \hookrightarrow I^{-n-1}$ for any integer $n$.
We then define
\[
A[1/I]:=\varinjlim_n I^{-n},
\]
which is an $A$-algebra.
For an $A$-module $M$, we set
$M[1/I]:=M \otimes_A A[1/I]$.
If $I$ is generated by a nonzerodivisor $d$, then we have $A[1/I]=A[1/d]$ and $I^{-n}=d^{-n}A$.

\begin{lem}\label{Lemma:Brauil-Kisin module, filtration and height}
Let $M, N$ be finite projective $A$-modules and
let $F \colon N[1/I] \overset{\sim}{\to} M[1/I]$ be an $A[1/I]$-linear isomorphism.
For an integer $i$, we set
\[
\Fil^i(N):= \{ \, x \in N \, \vert \, F(x) \in M \otimes_A I^{i} \, \},
\]
where we view $M \otimes_A I^{i}$ as a subset of $M[1/I]$.
Let $m$ be an integer.
Then the following are equivalent:
    \begin{enumerate}
        \item $\Fil^{m+1}(N) \subset IN$.
        \item $M \otimes_A I^{m} \subset F(N)$.
    \end{enumerate}
    If these equivalent conditions are satisfied, then $F$ restricts to an isomorphism
    $\Fil^m(N) \overset{\sim}{\to} M \otimes_A I^{m}$, and in particular $\Fil^m(N)$ is a finite projective $A$-module.
\end{lem}

\begin{proof}
    The final statement clearly follows from (2).
    We shall prove that (1) and (2) are equivalent.
    For this, we can reduce to the case where $I=(d)$ is principal.
    
    Assume that (1) holds.
    Let $x \in M$.
    We want to show that $d^m x \in F(N)$.
    For a large enough integer $n$, we have $d^n x \in F(N)$.
    Let $y \in N$ be an element such that $F(y)=d^n x$.
    If $n \geq m+1$, then we have $y \in \Fil^{m+1}(N) \subset IN$, which in turn implies that $d^{n-1} x \in F(N)$.
    From this observation, we can conclude that $d^m x \in F(N)$.

    Assume that (2) holds.
    Let $y \in \Fil^{m+1}(N)$.
    There exists an element $x \in M$ such that $F(y)=d^{m+1}x$.
    The condition (2) implies that $d^mx=F(z)$ for some $z \in N$.
    It then follows that $y = dz \in IN$.
\end{proof}

Let $(A, I)$ be a bounded $\O_E$-prism.

\begin{defn}\label{Definition:Breuil-Kisin module}
    A \textit{Breuil--Kisin module} over $(A, I)$ is a pair $(M, F_M)$ consisting of a finite projective $A$-module $M$ and an $A[1/I]$-linear isomorphism
    \[
    F_M \colon (\phi^*M)[1/I] \overset{\sim}{\to} M[1/I],
    \]
    where $\phi^*M:=A \otimes_{\phi, A}{M}$.
    When there is no possibility of confusion, we simply write $M$ instead of $(M, F_M)$.
    For an integer $i$, we set
    \[
    \Fil^i(\phi^*M):= \{ \, x \in \phi^*M \, \vert \, F_M(x) \in M \otimes_A I^{i} \, \}.
    \]
    Let $P^i \subset (\phi^*M)/I(\phi^*M)$ be the image of $\Fil^i(\phi^*M)$.
    We often write
    \[
    M_{\dR}:=(\phi^*M)/I(\phi^*M).
    \]
\end{defn}

\begin{rem}\label{Remark:effective Breuil-Kisin module}
If $F_M(\phi^*M) \subset M$,
then we say that $M$ is \textit{effective}.
In this case, the induced homomorphism $\phi^*M \to M$ is again denoted by $F_M$.
The cokernel of $F_M \colon \phi^*M \to M$ is killed by some power of $I$.

Conversely,
for a finite projective $A$-module $M$ and a homomorphism $F_M \colon \phi^*M \to M$
of $A$-modules
whose cokernel is killed by some power of $I$,
the induced homomorphism
$(\phi^*M)[1/I] \to M[1/I]$ is an isomorphism.
Indeed, it is clear that $(\phi^*M)[1/I] \to M[1/I]$ is surjective, which in turn implies that it is an isomorphism since (the vector bundles on $\Spec A$ associated with) $\phi^* M$ and $M$ have the same rank.
In particular, it follows that $F_M \colon \phi^*M \to M$ is injective.
\end{rem}

\begin{rem}\label{Remark:filtration of Breuil-Kisin module}
For any integer $i$, we have $I\Fil^{i-1}(\phi^*M)=\Fil^i(\phi^*M) \cap I(\phi^*M)$.
In other words,
the natural homomorphism
$\Fil^{i}(\phi^*M)/I\Fil^{i-1}(\phi^*M) \to P^{i}$
is bijective.
We have $P^i=M_{\dR}$ for small enough $i$ and $P^i=0$ for large enough $i$.
\end{rem}

\begin{defn}\label{Definition:displayed and minuscule Breuil-Kisin module}
    Let $M$ be a Breuil--Kisin module over $(A, I)$.
    We say that $M$ is \textit{displayed} if the $A/I$-submodule $P^i \subset M_{\dR}$ is a direct summand for every $i$.
    In this case, the filtration $\{ P^i \}_{i \in \Z}$ is called the \textit{Hodge filtration}.
    We say that $M$ is \textit{minuscule} if it is displayed, and if we have $P^i=M_{\dR}$ for any $i \leq 0$ and $P^i=0$ for any $i \geq 2$.
\end{defn}

The following proposition, which is basically a consequence of \cite[Remark 4.25]{Anschutz-LeBras}, shows that the definition of a minuscule Breuil--Kisin module given in Definition \ref{Definition:displayed and minuscule Breuil-Kisin module} agrees with the usual one employed in the literature (for example in \cite[Section 2.2]{Kisin06} and \cite[Definition 4.24]{Anschutz-LeBras}).

\begin{prop}\label{Proposition:minuscule equivalent condition}
     Let $M$ be a Breuil--Kisin module over $(A, I)$.
     The following statements are equivalent:
     \begin{enumerate}
         \item $M$ is minuscule.
         \item $M$ is effective, and the cokernel $\Coker F_M$ of $F_M \colon \phi^*M \to M$ is killed by $I$.
     \end{enumerate}
\end{prop}

\begin{proof}
    Assume that (1) holds.
    It follows from $P^0=M_{\dR}$ and Nakayama's lemma that $\Fil^0(\phi^*M)=\phi^*M$.
    Moreover, we have $IM \subset F_M(\phi^*M)$ by Lemma \ref{Lemma:Brauil-Kisin module, filtration and height}.
    This proves that (1) implies (2).

    Assume that (2) holds.
    It follows from Lemma \ref{Lemma:Brauil-Kisin module, filtration and height} that $\Fil^{2}(\phi^*M) \subset I(\phi^*M)$, and hence $P^{i}=0$ for any $i \geq 2$.
    Since $M$ is effective, we have $P^{i}=M_{\dR}$ for any $i \leq 0$.
    It remains to prove that $P^1$ is a direct summand of $M_{\dR}$.
    For this, it suffices to show that $(\phi^*M)/\Fil^{1}(\phi^*M)$ is projective as an $A/I$-module.
    Since
    we have
    the following exact sequence of $A/I$-modules
    \[
    0 \to  (\phi^*M)/\Fil^1 (\phi^*M) \to M/IM \to \Coker F_M \to 0,
    \]
    is suffices to prove that 
    $\Coker F_M$
    is a projective $A/I$-module.
    With Lemma \ref{Lemma:morphism to crystalline prisms} below, this follows from the same argument as in \cite[Remark 4.25]{Anschutz-LeBras}.
\end{proof}

\begin{lem}\label{Lemma:morphism to crystalline prisms}
    Let $(A, I)$ be an $\O_E$-prism.
    For a perfect field $k$ containing $\F_q$ and a homomorphism $g \colon A/I \to k$ of $\O_E$-algebras, there exists a map
    $(A, I) \to (\O, (\pi))$
    of $\O_E$-prisms which induces $g$, where $\O:=W(k)\otimes_{W(\F_q)} \O_E$.
\end{lem}

\begin{proof}
    Let
    $A_{\perf}:=\varinjlim_\phi A$ be a colimit of the diagram
$
A \overset{\phi}{\rightarrow} A \overset{\phi}{\rightarrow} A  \rightarrow \cdots,
$
which is a perfect $\delta_E$-ring.
Since $k$ is perfect,
the homomorphism
$A/\pi \to k$ induced by the composition $A \to A/I \to k$ factors through a homomorphism $A_{\perf}/\pi \to k$.
This homomorphism lifts uniquely to a homomorphism $A_{\perf} \to \O$ of $\delta_E$-rings by Lemma \ref{Lemma:maps from perfectoid prisms to prism}.
The composition $A \to A_{\perf} \to \O$ gives a map
$(A, I) \to (\O, (\pi))$
which induces $g$, as desired.
\end{proof}

We shortly discuss the relation between
the notion of minuscule Breuil--Kisin modules and that of windows introduced by Zink and Lau.
We recall the notion of windows, adapted to our context.
Let $(A, d)$ be an oriented and bounded $\O_E$-prism.

\begin{defn}\label{Definition:window}
A \textit{window} over $(A, d)$ is a quadruple
\[
\underline{N}=(N, \Fil^1(N), \Phi, \Phi_1)
\]
where $N$ is a finite projective $A$-module,
$\Fil^1(N) \subset N$ is an $A$-submodule,
$\Phi \colon N \to N$ and
$\Phi_1 \colon \Fil^1(N) \to N$
are $\phi$-linear homomorphisms, such that the following conditions hold:
\begin{enumerate}
    \item We have $dN \subset \Fil^1(N)$, and $\Phi(x)=\Phi_1(dx)$ for every $x \in N$.
    \item The image $P^1 \subset N/dN$ of $\Fil^1(N)$ is a direct summand of $N/dN$.
    \item The linearization $1 \otimes \Phi_1 \colon \phi^*\Fil^1(N) \to N$ of $\Phi_1$ is an isomorphism.
\end{enumerate}
\end{defn}

\begin{prop}[cf.\ {\cite[Lemma 2.1.16]{Cais-Lau}, \cite[Proposition 4.26]{Anschutz-LeBras}}]\label{Proposition:windows and BK modules}
For a window
$\underline{N}$ over $(A, d)$,
the pair $(\Fil^1(N), F)$, where $F \colon \phi^*\Fil^1(N) \to \Fil^1(N)$ is defined by $F=d(1 \otimes \Phi_1)$, is a minuscule Breuil--Kisin module over $(A, (d))$.
This construction gives an equivalence between the category of windows over $(A, d)$ and the category of minuscule Breuil--Kisin modules over $(A, (d))$.
\end{prop}

\begin{proof}
By virtue of Proposition \ref{Proposition:minuscule equivalent condition}, the same argument as in the proof of \cite[Lemma 2.1.16]{Cais-Lau} works.
\end{proof}

We study the structure of displayed Breuil--Kisin modules.
For this, we introduce the following definition.
Let $(A, I)$ be a bounded $\O_E$-prism.

\begin{defn}\label{Definition:normal decomposition}
    Let $M$ be a Breuil--Kisin module over $(A, I)$.
    A decomposition
$
\phi^*M = \bigoplus_{j \in \Z} L_j
$
is called a \textit{normal decomposition} if the isomorphism $F_M \colon (\phi^*M)[1/I] \overset{\sim}{\to} M[1/I]$ restricts to an isomorphism
\[
\bigoplus_{j \in \Z} (L_j \otimes_A I^{-j}) \overset{\sim}{\to} M
\]
of $A$-modules.
\end{defn}

\begin{rem}\label{Remark:normal decomposition and displayed BK}
    If $
\phi^*M = \bigoplus_{j \in \Z} L_j
$
is a normal decomposition, then we have
\[
\Fil^i(\phi^*M)= (\bigoplus_{j \geq i} L_j) \oplus (\bigoplus_{j < i} I^{i-j}L_j)
\]
for every $i \in \Z$.
In particular, a Breuil--Kisin module over $(A, I)$ which admits a normal decomposition is displayed.
In the next lemma, we shall prove that the converse is also true.
\end{rem}

\begin{lem}\label{Lemma:normal decomposition for displayed BK module}
Let $M$ be a displayed Breuil--Kisin module over $(A, I)$.
Then there exists a normal decomposition
$
\phi^*M = \bigoplus_{j \in \Z} L_j.
$
\end{lem}

\begin{proof}
We choose a decomposition
$
(\phi^*M)/I(\phi^*M)=\bigoplus_{j\in \Z} K_j
$
such that $P^i=\bigoplus_{j \geq i} K_{j}$ for every $i$.
Since
$K_{j}$
is a finite projective $A/I$-module
and
$A$ is $I$-adically complete,
there exists a finite projective $A$-module $L_j$
such that $L_j/IL_j \simeq K_j$ for every $j$; see \cite[Tag 0D4A]{SP} or \cite[Theorem 5.1]{Greco} for example.
Moreover we have $L_j=0$ for all but finitely many $j$.
Since $\Fil^{i}(\phi^*M) \to P^i$ is surjective,
there exists a homomorphism $L_i \to \Fil^{i}(\phi^*M)$
which fits into the following commutative diagram:
\[
\xymatrix{
L_i \ar^-{}[r]  \ar[d]_-{} & K_i  \ar[d]_-{} \\
\Fil^{i}(\phi^*M) \ar[r]^-{} & P^i.
}
\]
The induced homomorphism
$\bigoplus_{j\in \Z} L_j \to \phi^*M$
is an isomorphism
since it is a lift of the isomorphism $\bigoplus_{j\in \Z} K_j \overset{\sim}{\to}
(\phi^*M)/I(\phi^*M)$.
We shall prove that, under this isomorphism,
$\Fil^i(\phi^*M)$ coincides with $(\bigoplus_{j \geq i} L_j) \oplus (\bigoplus_{j < i} I^{i-j}L_j)$
for any $i \in \Z$.
This implies that $\bigoplus_{j\in \Z} L_j$ is a normal decomposition.

We proceed by induction on $i$.
The assertion clearly holds for small enough $i$.
Let us assume that the assertion holds for an integer $i$.
Since
\[
(\bigoplus_{j \geq i} IL_j) \oplus (\bigoplus_{j < i} I^{i+1-j}L_j)=I\Fil^i(\phi^*M) \subset \Fil^{i+1}(\phi^*M)
\]
and $\bigoplus_{j \geq i+1} L_j \subset \Fil^{i+1}(\phi^*M)$ by construction, we obtain
\[
(\bigoplus_{j \geq i+1} L_j) \oplus (\bigoplus_{j < i+1} I^{i+1-j}L_j) \subset \Fil^{i+1}(\phi^*M).
\]
The left hand side contains $I\Fil^i(\phi^*M)$ and the quotient by $I\Fil^i(\phi^*M)$ is equal to $P^{i+1}$.
The same holds for the right hand side by Remark \ref{Remark:filtration of Breuil-Kisin module}.
Therefore, this inclusion is actually an equality.
\end{proof}

Let $f \colon (A, I) \to (A', I')$ be a map of bounded $\O_E$-prisms.
For a Breuil--Kisin module $(M, F_M)$ over $(A, I)$,
let
$
F_{M_{A'}} \colon (\phi^*(M_{A'}))[1/I'] \overset{\sim}{\to} M_{A'}[1/I']
$
be the base change of $F_M$, where $M_{A'}:= M \otimes_A A'$.
We also write $f^*M$ for $(M_{A'}, F_{M_{A'}})$.

\begin{prop}\label{Proposition:displayed condition base change}
Let $(M, F_M)$ be a Breuil--Kisin module over $(A, I)$.
\begin{enumerate}
    \item Assume that $(M, F_M)$ is displayed.
    Then $(M_{A'}, F_{M_{A'}})$
    is a displayed Breuil--Kisin module over $(A', I')$, and we have
    $\Fil^i(\phi^*M) \otimes_A A' \overset{\sim}{\to} \Fil^i(\phi^*(M_{A'}))$ for any integer $i$.
    \item Assume that $(M_{A'}, F_{M_{A'}})$ is displayed and $f \colon (A, I) \to (A', I')$ is a faithfully flat map of $\O_E$-prisms. Then $(M, F_M)$ is displayed.
\end{enumerate}
\end{prop}

\begin{proof}
    (1) This follows from Remark \ref{Remark:normal decomposition and displayed BK}, Lemma \ref{Lemma:normal decomposition for displayed BK module}, and the fact that normal decompositions are preserved under base change.

    (2) We note that $\Fil^i(\phi^*(M_{A'}))$ is a finite projective $A'$-module for any $i$ by Lemma \ref{Lemma:normal decomposition for displayed BK module} and Remark \ref{Remark:normal decomposition and displayed BK}.
    Since $\Fil^i(\phi^*(M_{A'}))$ is stable under the natural descent datum of $\phi^*(M_{A'})$ (with respect to the flat covering $(A, I) \to (A', I')$) by (1),
    it follows from Proposition \ref{Proposition:flat descent for finite projective modules} that there is a descending filtration
$\{ \Fil^i \}_{i \in \Z}$ of $\phi^*M$ by finite projective $A$-submodules such that
$\Fil^i \otimes_A A' \to \phi^*(M_{A'})$ induces an isomorphism
$\Fil^i \otimes_A A' \overset{\sim}{\to} \Fil^i(\phi^*(M_{A'}))$ for any $i$.

Let $m$ be an integer such that $M \otimes_A I^{m} \subset F_M(\phi^*M)$.
Then $\Fil^m=\Fil^m(\phi^*M)$ (see Lemma \ref{Lemma:Brauil-Kisin module, filtration and height}).
Moreover, we have $I\Fil^{i-1} \subset \Fil^i$ for any $i$, and $I\Fil^{i-1} = \Fil^i$ for $i \geq m+1$.
In particular, we obtain $\Fil^i=\Fil^i(\phi^*M)$ for $i \geq m$.

Let $i$ be any integer.
We claim that the natural homomorphism of $A/I$-modules
    \[
    \iota \colon \Fil^i/I\Fil^{i-1} \to (\phi^*M)/I(\phi^*M)
    \]
is injective and its cokernel is a finite projective $A/I$-module.
Indeed, it suffices to show that for every closed point $x \in \Spec A/I$, the base change of $\iota$ to the residue field $k(x)$ is injective.
Since $x$ is contained in $\Spec A/(\pi, I)$ and $\Spec A'/(\pi, I') \to \Spec A/(\pi, I)$ is surjective,
it is enough to prove that the base change of $\iota$ along $A/I \to A'/I'$ is injective and its cokernel is a finite projective $A'/I'$-module. 
This follows from the assumption that 
$(M_{A'}, F_{M_{A'}})$ is displayed.

It follows from the claim that
$I\Fil^{i-1}= I(\phi^*M) \cap \Fil^{i}$, or equivalently, $\Fil^{i-1}= \phi^*M \cap (\Fil^{i} \otimes_A I^{-1})$.
Since $\Fil^i=\Fil^i(\phi^*M)$ for $i \geq m$,
we can conclude that $\Fil^i=\Fil^i(\phi^*M)$ for any $i$.
This, together with the claim, shows that $(M, F_M)$ is displayed.
\end{proof}

\begin{rem}\label{Remark:flat descent for BK modules}
The functor
\[
(\O_E)_{\Prism, \O_E} \to \mathrm{Set}, \quad (B, J) \mapsto B[1/J]
\]
forms a sheaf (with respect to the flat topology) by
Lemma \ref{Lemma:rigidity} and Remark \ref{Remark:structure sheaf}.
Thus for finite projective $A$-modules $M, M'$,
the functor $(A, I)_{\Prism} \to \mathrm{Set}$ which associates to each $(B, J) \in (A, I)_{\Prism}$ the set of isomorphisms
$M_B[1/J] \overset{\sim}{\to} M'_B[1/J]$
forms a sheaf.
This fact, together with
Proposition \ref{Proposition:flat descent for finite projective modules}, implies that the fibered category over $(\O_E)^{\op}_{\Prism, \O_E}$
which associates to a bounded $\O_E$-prism $(A, I)$ the category of Breuil--Kisin modules over $(A, I)$ satisfies descent with respect to the flat topology.
\end{rem}

\begin{cor}\label{Corollary:flat descent for dispalyed BK modules}
The fibered category over $(\O_E)^{\op}_{\Prism, \O_E}$
which associates to a bounded $\O_E$-prism $(A, I)$ the category of displayed Breuil--Kisin modules over $(A, I)$
satisfies descent with respect to the flat topology.
\end{cor}

\begin{proof}
    This follows from Proposition \ref{Proposition:displayed condition base change} and Remark \ref{Remark:flat descent for BK modules}.
\end{proof}

We finish this subsection by giving an example of a Breuil--Kisin module which is not displayed.

\begin{ex}\label{Example:non-displayed BK module}
Let $(A, I)$ be an orientable and bounded $\O_E$-prism.
We assume that $A/I$ is $\pi$-torsion free and $A/I \neq 0$.
We set $M:=A^2$ and let
$F_M \colon \phi^*M \to M$
be the homomorphism defined by
the matrix
$\begin{pmatrix}
\pi & d \\
d & d^2 \\
\end{pmatrix}
$
for a generator $d \in I$.
The pair $(M, F_M)$ is a Breuil--Kisin module over $(A, I)$.
We claim that $P^1/P^2$ is not $\pi$-torsion free, and thus $(M, F_M)$ is not displayed.
Indeed, since $(d, 1) \in \Fil^1(\phi^*M)$,
we have $(0, 1) \in P^1 \subset M_{\dR}$.
One can check that
the image of $(0, 1)$ in $P^1/P^2$ is not zero and is killed by $\pi$.
\end{ex}

\subsection{Breuil--Kisin modules of type $\mu$}\label{Subsection:BK module of type mu}

Here we introduce the notion of Breuil--Kisin modules \textit{of type $\mu$}.
Let
\[
\mu \colon \G_m \to \GL_{n, \O}
\]
be a cocharacter defined over $\O$, where $\O=W(k)\otimes_{W(\F_q)} \O_E$ is as in Definition \ref{Definition:prism over O}.
There is a unique tuple $(m_1, \dotsc, m_n)$ of integers $m_1 \geq \cdots \geq m_n$ such that
$\mu$ is conjugate to the cocharacter defined by
$
t \mapsto \diag{(t^{m_1}, \dotsc, t^{m_n})}.
$
By abuse of notation,
the tuple $(m_1, \dotsc, m_n)$ is also denoted by $\mu$.
Let $r_i \in \Z_{\geq 0}$ be the number of occurrences of $i$ in $(m_1, \dotsc, m_n)$.
We set $L:=\O^n_E$ and $L_{\O} := L \otimes_{\O_E} \O$.
The cocharacter $\mu$ induces an action of $\G_m$ on $L_\O$.
We have the weight decomposition
\[
L_{\O} = \bigoplus_{j \in\Z}  L_{\mu, j}
\]
where an element $t \in \G_m(\O)=\O^\times$ acts on $L_{\mu, j}$ by $x \mapsto t^{j}x$ for every $j \in \Z$.
(See for example \cite[Lemma A.8.8]{CGP} for the existence of the weight decomposition over a ring.)
The rank of $L_{\mu, j}$ is equal to $r_j$.

Let $(A, I)$ be a bounded $\O_E$-prism over $\O$.

\begin{defn}\label{Definition:type of displayed BK module}
Let
$M$
be a displayed Breuil--Kisin module over $(A, I)$.
We say that $M$ is \textit{of type} $\mu$
if, for the Hodge filtration $\{ P^i \}_{i \in \Z}$, the successive quotient $P^i/P^{i+1}$ is of rank $r_i$ (i.e.\ the corresponding vector bundle on $\Spec A/I$ has constant rank $r_i$) for every $i$.
We say that $M$ is \textit{banal} if
all successive quotients $P^i/P^{i+1}$ are free $A/I$-modules.
\end{defn}

We write
$\mathrm{BK}_\mu(A, I)$
(resp.\ 
$\mathrm{BK}_\mu(A, I)_{\mathrm{banal}}$)
for the category of Breuil--Kisin modules over $(A, I)$ of type $\mu$
(resp.\ banal Breuil--Kisin modules over $(A, I)$ of type $\mu$).

%\begin{ex}\label{Example:minuscule BK module and minuscule cocharacter}
    %Assume that $\mu=(1, \dotsc, 1, 0, \dotsc , 0)$ where $1$ appears $s$ times and $0$ appears $n-s$ times.
    %Then any Breuil--Kisin module over $(A, I)$ of type $\mu$ is minuscule.
%\end{ex}

\begin{rem}\label{Remark:standard filtration}
We set
\[
\Fil^i_\mu:= (\bigoplus_{j \geq i} (L_{\mu, j})_A) \oplus (\bigoplus_{j < i} I^{i-j}(L_{\mu, j})_A) \subset A^n,
\]
where $(L_{\mu, j})_A:=L_{\mu, j} \otimes_{\O} A$.
Let $M \in \mathrm{BK}_\mu(A, I)_{\mathrm{banal}}$.
Let $\phi^*M = \bigoplus_{j \in \Z} L_j$ be a normal decomposition.
Then, each $L_j$ is a free $A$-module of rank $r_j$.
Thus there is an isomorphism
$A^n \simeq \phi^*M$ such that
the filtration
$\{ \Fil^i_\mu \}_{i \in \Z}$
coincides with
$\{ \Fil^i(\phi^*M) \}_{i \in \Z}$.
\end{rem}

\begin{rem}\label{Remark:etale locally banal of type mu}
    Let
    $M$
    be a displayed Breuil--Kisin module over $(A, I)$.
    Then there exists a $(\pi, I)$-completely \'etale covering
    $A \to A_1 \times \cdots \times A_m$
    such that for each $1 \leq j \leq m$,
    the base change
    of $M$ to the bounded $\O_E$-prism $(A_j, IA_j)$ (see Lemma \ref{Lemma:etale morphism prism}) is of type $\mu$ for some $\mu$ and banal.
    Indeed, by Lemma \ref{Lemma:etale morphism and reduction}, it suffices to prove that
    there exists an \'etale and faithfully flat homomorphism
    $A/(\pi, I) \to B_1 \times \cdots \times B_m$
    such that for all $1 \leq j \leq m$ and $i$,
    the base change $P^{i}/P^{i+1} \otimes_{A/I} B_j$ is free over $B_j$, which is clear.
\end{rem}

\section{Display group} \label{Section:display group}

Let $G$ be a smooth affine group scheme over $\O_E$.
Let $k$ be a perfect field containing $\F_q$
and
we set
$\O := W(k)\otimes_{W(\F_q)} \O_E$.
Let
\[
\mu \colon \G_m \to G_{\O}:=G \times_{\Spec \O_E} \Spec \O
\]
be a cocharacter defined over $\O$.
In this section, we introduce the display group $G_\mu(A, I)$
for an orientable and bounded $\O_E$-prism $(A, I)$ over $\O$.
The display group will be used in the definition of prismatic $G$-$\mu$-displays.

\subsection{Definition of the display group}\label{Subsection:Definition of the Display group}

Let $A$ be an $\O$-algebra with
an ideal $I \subset A$ which is generated by a nonzerodivisor $d \in A$.

\begin{defn}\label{Definition:display group}
We define
\[
G_\mu(A, I):=\{ \, g \in G(A) \, \vert \, \mu(d)g\mu(d)^{-1} \, \, \text{lies in} \, \, G(A) \subset G(A[1/I]) \, \}.
\]
The group $G_\mu(A, I)$ is called the \textit{display group}.
We note that $G_\mu(A, I)$ does not depend on the choice of $d$.
\end{defn}

\begin{rem}\label{Remark:definition of display group is simpler}
    The definition of the display group given here is a translation of the one given in \cite{Lau21} to our setting; see Remark \ref{Remark:analogue of Lau's definitions display version} for details.
    If $G$ is reductive and $\mu$ is minuscule, such a group was also considered in \cite{Bultel-Pappas}.
\end{rem}

For the cocharacter
$
\mu \colon \G_m \to G_{\O},
$
we endow $G_\O$ with the action of $\G_m$ defined by
\begin{equation}\label{equation:action cocharacter adjoint}
    G_\O(R) \times \G_m(R) \to G_\O(R), \quad (g, t) \mapsto \mu(t)^{-1}g\mu(t)
\end{equation}
for every $\O$-algebra $R$.
We note that this action is the inverse of the one used in Definition \ref{Definition:display group}.
We write $G=\Spec A'_G$ and $A_G:=A'_G \otimes_{\O_E} \O$, so that $G_\O=\Spec A_G$.
Let
\[
A_G= \bigoplus_{i\in\Z} A_{G, i}
\]
be the weight decomposition with respect to the action of $\G_m$.
An element $t \in \G_m(R)=R^\times$ acts on $A_{G, i} \otimes_\O R$ by $x \mapsto t^{i}x$.

\begin{rem}\label{Remark:formula action of cocharacter}
Let $R$ be an $\O$-algebra.
For any $t \in \G_m(R)$ and any $g \in G_\O(R)$ with corresponding homomorphism $g^* \colon A_G \to R$,
the homomorphism
\[
(\mu(t)^{-1}g\mu(t))^* \colon A_G \to R
\]
corresponding to $\mu(t)^{-1}g\mu(t) \in G_\O(R)$ 
sends an element $x \in A_{G, i}$ to $t^ig^*(x) \in R$.
\end{rem}

\begin{lem}\label{Lemma:Gmu iff condition}
Let $g \in G(A)$ be an element.
Then $g \in G_\mu(A, I)$ if and only if $g^*(x) \in I^iA$ for every $i > 0$ and every $x \in A_{G, i}$.
\end{lem}

\begin{proof}
This follows from Remark \ref{Remark:formula action of cocharacter}.
\end{proof}

\begin{ex}\label{Example:display group GLn case}
Assume that $G=\GL_n$.
Let
$\mu \colon \G_m \to \GL_{n, \O}$
be a cocharacter and let $(m_1, \dotsc, m_n)$ be the corresponding tuple of integers $m_1 \geq \cdots \geq m_n$ as in Section \ref{Subsection:BK module of type mu}.
Let $\{ \Fil^i_\mu \}_{i \in \Z}$ be the filtration of $M:=A^n$ defined as in Remark \ref{Remark:standard filtration}.
Then we have
\[
(\GL_n)_\mu(A, I)=\{ \, g \in \GL_n(A) \, \vert \, g(\Fil^i_\mu)=\Fil^i_\mu \, \, \text{for every $i \in \Z$} \, \}.
\]
Let $d \in I$ be a generator.
For any $g \in (\GL_n)_\mu(A, I)$,
the following diagram commutes:
\[
\xymatrix{
\Fil^{m_1}_\mu \ar^-{g}[rr]  \ar[d]_-{\simeq} & & \Fil^{m_1}_\mu  \ar[d]^-{\simeq} \\
M \ar[rr]^-{\mu(d)g\mu(d)^{-1}} & & M,
}
\]
where
$\Fil^{m_1}_\mu \overset{\sim}{\to} M$ is defined by $d^{-m_1}\mu(d)$.
\end{ex}

\subsection{Properties of the display group}\label{Subsection:Properties of the display group}

For the cocharacter $\mu  \colon \G_m \to G_{\O}$,
we consider the closed subgroup schemes $P_\mu, U^{-}_{\mu} \subset G_\O$ over $\O$ defined by, for every $\O$-algebra $R$,
\begin{align*}
    P_\mu(R)&=\{ \, g \in G(R) \, \vert \, \lim_{t \to 0} \mu(t)g\mu(t)^{-1} \, \text{exists} \, \},\\
    U^{-}_{\mu}(R)&=\{ \, g \in G(R) \, \vert \, \lim_{t \to 0} \mu(t)^{-1}g\mu(t)=1 \, \}.
\end{align*}
(We refer to \cite[Lemma 2.1.4]{CGP} for the definition of $\lim_{t \to 0} \mu(t)g\mu(t)^{-1}$.)
We see that $P_\mu$ and $U^{-}_{\mu}$ are stable under the action of $\G_m$ on $G_\O$ given by $(\ref{equation:action cocharacter adjoint})$.
The group schemes $P_\mu$ and $U^{-}_{\mu}$ are smooth over $\O$.
Moreover, the multiplication map
\[
U^{-}_{\mu} \times_{\Spec \O} P_\mu \to G_\O
\]
is an open immersion.
See \cite[Section 2.1]{CGP}, especially \cite[Proposition 2.1.8]{CGP}, for details.

\begin{rem}\label{Remark:change of notation from Lau}
We have employed slightly different notation than in \cite{Lau21}.
For example, in \textit{loc.cit.}, the subgroup $P_\mu$ (resp.\ $U^{-}_{\mu}$) is denoted by $P^{-}$ (resp.\ $U^{+}$).
\end{rem}

\begin{lem}\label{Lemma:Pmu structure}
\
\begin{enumerate}
    \item Let $R$ be an $\O$-algebra
    and $g \in G_\O(R)$ an element.
    Then $g \in P_\mu(R)$ if and only if $g^*(x)=0$ for every $i > 0$ and every $x \in A_{G, i}$.
    \item We have
    $
    P_\mu(A) \subset G_\mu(A, I),
    $
    and the image of $G_\mu(A, I)$ in $G(A/I)$
    under the projection
    $G(A) \to G(A/I)$ is contained in $P_\mu(A/I)$.
    Moreover
    $
    \mu(d)P_\mu(A)\mu(d)^{-1}$
    is contained in $P_\mu(A)$.
\end{enumerate}
\end{lem}

\begin{proof}
Remark \ref{Remark:formula action of cocharacter} immediately implies (1).
The second assertion (2) follows from (1) and Lemma \ref{Lemma:Gmu iff condition}.
\end{proof}

\begin{defn}[{\cite[Definition 6.3.1]{Lau21}}]\label{Definition:1-bounded}
The action of $\G_m$ on $G_\O$ given in $(\ref{equation:action cocharacter adjoint})$ induces an action of $\G_m$ on
the Lie algebra $\Lie(G_\O)$.
Let
\[
\Lie(G_\O)=\bigoplus_{i \in \Z} \mathfrak{g}_i
\]
be the weight decomposition with respect to the action of $\G_m$.
We say that the cocharacter
$\mu \colon \G_m \to G_{\O}$
is \textit{1-bounded} if $\mathfrak{g}_i=0$ for $i \geq 2$.
\end{defn}

In general,
the Lie algebra $\Lie(U^-_{\mu})$ of $U^-_{\mu}$
coincides with
$\bigoplus_{i \geq 1} \mathfrak{g}_i$.
(We also note that
$\Lie(P_{\mu})= \bigoplus_{i \leq  0} \mathfrak{g}_i$.)
Thus $\mu$ is 1-bounded if and only if
$\Lie(U^-_{\mu})=\mathfrak{g}_1$.

\begin{rem}\label{Remark:minuscule cocharacter}
    If $G$ is a reductive group scheme over $\O_E$, then $\mu$ is 1-bounded if and only if $\mu$ is minuscule, that is, the equality $\Lie(G_\O)=\mathfrak{g}_{-1} \oplus \mathfrak{g}_{0} \oplus \mathfrak{g}_{1}$ holds.
\end{rem}

\begin{ex}\label{Example:GLn minuscule case}
    Assume that $G=\GL_n$.
    The cocharacter
    $\G_m \to \GL_{n, \O}$
    defined by
\[
t \mapsto \diag{(\underbrace{t^m, \dotsc, t^m}_{s}, \underbrace{t^{m-1}, \dotsc, t^{m-1}}_{n-s})}
\]
for some integers $m$ and $s$ $(0 \leq s \leq n)$ is 1-bounded.
In fact, any 1-bounded cocharacter of $\GL_{n, \O}$ is conjugate to a cocharacter of this form.
\end{ex}

For a free $\O$-module $M$ of finite rank,
we let $V(M)$ denote the group scheme over $\O$ defined by $R \mapsto M \otimes_\O R$ for every $\O$-algebra $R$.

\begin{lem}\label{Lemma:Umu structure}
There exists a $\G_m$-equivariant isomorphism
    \[
    \log \colon U^-_{\mu} \overset{\sim}{\to} V(\Lie(U^-_{\mu}))
    \]
    of schemes over $\O$ which induces the identity on the Lie algebras.
    If $\mu$ is 1-bounded, then the isomorphism $\log$ is unique, and it is an isomorphism of group schemes over $\O$.
\end{lem}

\begin{proof}
The same arguments as in the proofs of \cite[Lemma 6.1.1, Lemma 6.3.2]{Lau21} work here.
\end{proof}

\begin{rem}\label{Remark:Umu identify}
\ 
\begin{enumerate}
    \item An isomorphism
    $
    \log \colon U^-_{\mu} \overset{\sim}{\to} V(\Lie(U^-_{\mu}))
    $
    as in Lemma \ref{Lemma:Umu structure} induces a bijection
\[
U^-_{\mu}(A) \cap G_\mu(A, I) \overset{\sim}{\to} \bigoplus_{i \geq 1} I^i(\mathfrak{g}_i \otimes_\O A).
\]
    \item If $\mu$ is 1-bounded,
then
we identify $U^-_{\mu}$ with $V(\Lie(U^-_{\mu}))$
by the unique isomorphism $\log$.
In particular, we view
$\Lie(U^-_{\mu}) \otimes_\O A$
as a subgroup of $G(A)$.
We then obtain
\[
I(\Lie(U^-_{\mu}) \otimes_\O A)=(\Lie(U^-_{\mu}) \otimes_\O A) \cap G_\mu(A, I).
\]
Moreover, the following diagram commutes:
\[
\xymatrix{
I(\Lie(U^-_{\mu}) \otimes_\O A)  \ar@{^{(}->}[r]^-{} \ar_-{dv \mapsto v}[d] & G_\mu(A, I) \ar[d]^-{g \mapsto \mu(d)g\mu(d)^{-1}}  \\
\Lie(U^-_{\mu}) \otimes_\O A
 \ar@{^{(}->}[r]^-{}  & G(A).
}
\]
\end{enumerate}
\end{rem}

\begin{prop}[{cf.\ \cite[Lemma 6.2.2]{Lau21}}]\label{Proposition:decomposition of display group}
Assume that $A$ is $I$-adically complete.
Then the multiplication map
\begin{equation}\label{equation:multiplication map Gmu}
    (U^-_{\mu}(A) \cap G_\mu(A, I)) \times P_\mu(A) \to G_\mu(A, I)
\end{equation}
is bijective.
\end{prop}

\begin{proof}
Since $P_\mu(A) \subset G_\mu(A, I)$ by Lemma \ref{Lemma:Pmu structure}, the map $(\ref{equation:multiplication map Gmu})$ is well-defined.
Since
the map
$
U^-_{\mu} \times_{\Spec \O} P_\mu \to G_\O
$
is an open immersion, the map $(\ref{equation:multiplication map Gmu})$ is injective.

We shall show that the map $(\ref{equation:multiplication map Gmu})$ is surjective.
Let $g \in G_\mu(A, I)$ be an element.
By Lemma \ref{Lemma:Pmu structure}, the image of $g$ in $G(A/I)$ is contained in $P_\mu(A/I)$.
Since $P_\mu$ is smooth and $A$ is $I$-adically complete,
there exists an element $t \in P_\mu(A)$ whose image in $P_\mu(A/I)$ coincides with the image of $g$.
The restriction of 
the morphism $gt^{-1} \colon \Spec A \to G_\O$
to $\Spec A/I$
factors through the open subscheme
$U^-_{\mu} \times_{\Spec \O} P_\mu$.
Since $I \subset \mathrm{rad}(A)$,
it follows that
$gt^{-1} \colon \Spec A \to G_\O$
itself
factors through
$U^-_{\mu} \times_{\Spec \O} P_\mu$.
In other words, there are elements $u \in U^-_{\mu}(A)$ and $t' \in P_\mu(A)$ such that
$g=ut't$.
We note that $u \in G_\mu(A, I)$.
In conclusion, we have shown that $g$ is contained in the image of the map (\ref{equation:multiplication map Gmu}).
\end{proof}

\begin{prop}\label{Proposition:BB isomorphism}
Assume that $\mu \colon \G_m \to G_{\O}$ is 1-bounded.
Assume further that $A$ is $I$-adically complete.
Then the multiplication map
\begin{equation}\label{equation:multiplication map Gmu 1-bounded case}
    I(\Lie(U^-_{\mu}) \otimes_\O A) \times P_\mu(A) \to G_\mu(A, I)
\end{equation}
is bijective.
Moreover $G_\mu(A, I)$ coincides with the inverse image of $P_\mu(A/I)$ in $G(A)$ under the projection $G(A) \to G(A/I)$, and we have the following bijection:
\begin{equation}\label{equation:BB isomorphism}
    G(A)/G_\mu(A, I) \overset{\sim}{\to} G(A/I)/P_\mu(A/I).
\end{equation}
\end{prop}

\begin{proof}
It follows from Remark \ref{Remark:Umu identify} and Proposition \ref{Proposition:decomposition of display group} that
the map (\ref{equation:multiplication map Gmu 1-bounded case}) is bijective.
Let $G'_\mu \subset G(A)$ be the inverse image of $P_\mu(A/I)$.
We have
$G_\mu(A, I) \subset G'_\mu$.
By the same argument as in the proof of Proposition \ref{Proposition:decomposition of display group},
one can show that
$
G'_\mu \subset I(\Lie(U^-_{\mu}) \otimes_\O A) \times P_\mu(A).
$
Thus, we obtain $G_\mu(A, I) = G'_\mu$.

It remains to prove that the map (\ref{equation:BB isomorphism}) is bijective.
Since $G$ is smooth and $A$ is $I$-adically complete,
the projection $G(A) \to G(A/I)$ is surjective, which in turn implies the surjectivity of (\ref{equation:BB isomorphism}).
The injectivity follows from the equality $G_\mu(A, I)=G'_\mu$.
\end{proof}

For an integer $m \geq 0$,
let
$
G^{\geq m}(A)
$
be the kernel of $G(A) \to G(A/I^m)$.
We set
\[
G^{\geq m}_\mu(A, I):=G_\mu(A, I) \cap G^{\geq m}(A).
\]
We record a structural result about the quotient $G^{\geq m}_\mu(A, I)/G^{\geq m+1}_\mu(A, I)$.

\begin{lem}\label{Lemma:congruent subgroup of Gmu}
Assume that $A$ is $I$-adically complete.
Then we have the following isomorphisms of groups:
\[
G^{\geq m}(A)/G^{\geq m+1}(A)
\simeq
\begin{cases}
G(A/I) \quad &(m=0) \\ 
\Lie(G_\O) \otimes_\O I^m/I^{m+1} \quad &(m \geq 1),
\end{cases}
\]
\[
G^{\geq m}_\mu(A, I)/G^{\geq m+1}_\mu(A, I)
\simeq
\begin{cases}
P_\mu(A/I) \quad &(m=0) \\ 
(\bigoplus_{i \leq m} \mathfrak{g}_i)
\otimes_\O I^m/I^{m+1} \quad &(m \geq 1).
\end{cases}
\]
\end{lem}

\begin{proof}
Since $A$ is $I$-adically complete and $G$ is smooth, the map $G(A) \to G(A/I^m)$ is surjective.
It follows that
$
G^{\geq m}(A)/G^{\geq m+1}(A)
$
is isomorphic to the kernel
$\Ker(G(A/I^{m+1}) \to G(A/I^m))$
of $G(A/I^{m+1}) \to G(A/I^m)$.
This is equal to $G(A/I)$ when $m=0$.
If $m \geq 1$, then we have a canonical identification
\[
\Ker(G(A/I^{m+1}) \to G(A/I^m)) = \Lie(G_\O) \otimes_\O I^m/I^{m+1}
\]
since $I^m/I^{m+1} \subset A/I^{m+1}$ is a square zero ideal.
This proves the first assertion.

Since $P_\mu(A) \to P_\mu(A/I)$ is surjective (as $A$ is $I$-adically complete and $P_\mu$ is smooth), it follows from Lemma \ref{Lemma:Pmu structure} that
$G_\mu(A, I)/G^{\geq 1}_\mu(A, I) \simeq P_\mu(A/I)$.
To prove the second assertion,
it then suffices to show that
the image of the natural homomorphism
\[
G^{\geq m}_\mu(A, I) \to \Ker(G(A/I^{m+1}) \to G(A/I^m)) = \Lie(G_\O) \otimes_\O I^m/I^{m+1}
\]
is $(\bigoplus_{i \leq m} \mathfrak{g}_i)
\otimes_\O I^m/I^{m+1}$ for any $m \geq 1$.
By Proposition \ref{Proposition:decomposition of display group},
we may identify $G^{\geq m}_\mu(A, I)$ with
\[
(U^{-}_{\mu}(A) \cap G^{\geq m}_\mu(A, I)) \times P^{\geq m}_\mu(A)
\]
where
$P^{\geq m}_\mu(A):=P_\mu(A) \cap G^{\geq m}(A)$.
By the same argument as above,
we have
\[
P^{\geq m}_\mu(A)/P^{\geq m+1}_\mu(A) \simeq \Lie(P_{\mu}) \otimes_\O I^m/I^{m+1}= (\bigoplus_{i \leq  0} \mathfrak{g}_i) \otimes_\O I^m/I^{m+1}.
\]
It now suffices to prove that
the image of the natural homomorphism
\begin{equation}\label{equation:U_mu lie algebra image}
    U^{-}_{\mu}(A) \cap G^{\geq m}_\mu(A, I) \to \Ker(U^{-}_{\mu}(A/I^{m+1}) \to U^{-}_{\mu}(A/I^m)) = \Lie(U^{-}_{\mu}) \otimes_\O I^m/I^{m+1}
\end{equation}
is
$(\bigoplus_{1 \leq i \leq m} \mathfrak{g}_i)
\otimes_\O I^m/I^{m+1}$.
For this, we fix an isomorphism
$\log \colon U^-_{\mu} \overset{\sim}{\to} V(\Lie(U^-_{\mu}))$
as in Lemma \ref{Lemma:Umu structure}.
Since $\log$ induces the identity on the Lie algebras,
the isomorphism
\[
\Ker(U^{-}_{\mu}(A/I^{m+1}) \to U^{-}_{\mu}(A/I^m)) \overset{\sim}{\to} \Lie(U^{-}_{\mu}) \otimes_\O I^m/I^{m+1}
\]
induced by $\log$ is the same as the one in (\ref{equation:U_mu lie algebra image}).
Since $\log$ induces
\[
U^{-}_{\mu}(A) \cap G^{\geq m}_\mu(A, I) \overset{\sim}{\to} (\bigoplus_{1 \leq i \leq m} \mathfrak{g}_i) \otimes_\O I^m)  \oplus  (\bigoplus_{i \geq m+1} \mathfrak{g}_i) \otimes_\O I^i
\]
by
Remark \ref{Remark:Umu identify} (1), the result follows.
\end{proof}

\subsection{Display groups on prismatic sites}\label{Subsection:Display groups on prismatic sites}

In this subsection,
for a bounded $\O_E$-prism $(A, I)$ over $\O$,
we define the display group sheaf 
$G_{\mu, A, I}$
on the site
$
(A, I)^{\op}_\et
$
and discuss some basic results on $G_{\mu, A, I}$-torsors.

Let $(A, I)$ be a bounded $\O_E$-prism over $\O$.
We begin with a comparison result between torsors over $\Spec A$ (or $\Spec A/I$) with respect to the usual \'etale topology, and torsors on the sites $(A, I)^{\op}_\et$ and $(A, I)^{\op}_\Prism$ from Section \ref{Subsection:prismatic sites}.
To an affine scheme $X$ over $\O$ (or $A$),
we attach
a functor
\[
X_{\Prism, A} \colon (A, I)_\Prism \to \mathrm{Set},\quad (B, J) \mapsto X(B).
\]
This forms a sheaf (with respect to the flat topology) by Remark \ref{Remark:structure sheaf}
since $X(B)$ can be regarded as the set of homomorphisms
$R \to B$ of $\O$-algebras (or $A$-algebras) where $X=\Spec R$.
Similarly, to an affine scheme $X$ over $\O$ (or $A/I$),
we attach
a sheaf
\[
X_{\overline{\Prism}, A} \colon (A, I)_\Prism \to \mathrm{Set}, \quad (B, J) \mapsto X(B/J).
\]
The restrictions of these sheaves to $(A, I)_\et$ are denoted by the same notation (see also Remark \ref{Remark:etale category is a subcategory}).

\begin{prop}\label{Proposition:equivalences of pi completely etale torsors}
Let $H$ be a smooth affine group scheme over $\O$.
\begin{enumerate}
    \item
    For an $H_{A/I}$-torsor $\mathcal{P}$ over $\Spec A/I$ with respect to the \'etale topology, which is an affine scheme over $A/I$, the sheaf $\mathcal{P}_{\overline{\Prism}, A}$ on $(A, I)^{\op}_\Prism$ is an $H_{\overline{\Prism}, A}$-torsor with respect to the flat topology.
The functor
\[
\mathcal{P} \mapsto \mathcal{P}_{\overline{\Prism}, A}
\]
is an equivalence from the groupoid of $H_{A/I}$-torsors over $\Spec A/I$ to
the groupoid of $H_{\overline{\Prism}, A}$-torsors on $(A, I)^{\op}_\Prism$.
The same holds if we replace $(A, I)^{\op}_\Prism$ by $(A, I)^{\op}_\et$.
\item The construction
\[
\mathcal{P} \mapsto \mathcal{P}_{\Prism, A}
\]
gives an equivalence from the groupoid of $H_A$-torsors over $\Spec A$ to
the groupoid of $H_{\Prism, A}$-torsors on $(A, I)^{\op}_\Prism$.
The same holds if we replace $(A, I)^{\op}_\Prism$ by $(A, I)^{\op}_\et$.
\end{enumerate}
\end{prop}

\begin{proof}
(1)
It follows from Lemma \ref{Lemma:etale morphism and reduction} that
$\mathcal{P}_{\overline{\Prism}, A}$ is trivialized by a $(\pi, I)$-completely \'etale covering of $A$.
Thus $\mathcal{P}_{\overline{\Prism}, A}$ is an $H_{\overline{\Prism}, A}$-torsor on both $(A, I)^{\op}_\Prism$ and $(A, I)^{\op}_\et$.
It then suffices to prove that
the fibered category over $(A, I)^{\op}_\Prism$
which associates to each $(B, J) \in (A, I)_\Prism$ the groupoid of $H_{B/J}$-torsors over $\Spec B/J$
is a stack with respect to the flat topology.

It is known that, for any $\O$-algebra $R$,
the groupoid of $H_R$-torsors over $\Spec R$
is equivalent to the groupoid of
exact tensor functors
$
\Rep_\O(H) \to \Vect(R),
$
where $\Rep_\O(H)$ is the category of algebraic representations of $H$ on free $\O$-modules of finite rank,
and $\Vect(R)$ is the category of finite projective $R$-modules; see \cite[Theorem 19.5.1]{Scholze-Weinstein} and \cite[Theorem 1.2]{Broshi}.
(Although this result is stated only for the case where $\O=\Z_p$ in \cite[Theorem 19.5.1]{Scholze-Weinstein}, the proof also works for general $\O$.)
Using this Tannakian perspective, the desired claim follows from Proposition \ref{Proposition:flat descent for finite projective modules} and the following fact:
For a $\pi$-completely faithfully flat homomorphism $C \to C'$ of $\pi$-adically complete $\O$-algebras, a complex
\[
0 \to M_1 \to M_2 \to M_3 \to 0
\]
of finite projective $C$-modules
is exact if the base change
\[
0 \to M_1 \otimes_{C} C'  \to M_2 \otimes_{C} C' \to M_3 \otimes_{C} C' \to 0
\]
is exact.
(This fact follows from the following criterion: A complex $0 \to M_1 \to M_2 \to M_3 \to 0$ of finite projective modules over a ring $C$ is exact if for every closed point $x \in \Spec C$, its base change to the residue field $k(x)$ is exact.)
%(For the proof of this fact, see \cite[Lemma 2.3.9]{PappasRapoport21} for example. In fact, one can easily prove that the complex $0 \to M_1 \to M_2 \to M_3 \to 0$ is exact if for every closed point $x \in \Spec C$, its base change to the residue field at $x$ is exact. This implies the fact.)

(2) This can be proved in the same way as (1).
\end{proof}

\begin{defn}\label{Definition:category of orientable bounded prisms}
Let
$(A, I)_{\Prism, \ori}$
be the category of \textit{orientable} and bounded $\O_E$-prisms $(B, J)$ with a map $(A, I) \to (B, J)$.
We endow $(A, I)^{\op}_{\Prism, \ori}$ with the flat topology.
If $(A, I)$ is orientable, then we have $(A, I)_{\Prism, \ori}=(A, I)_{\Prism}$.
\end{defn}

\begin{rem}\label{Remark:sheaves on oriented prisms}
By Remark \ref{Remark:flat locally orientable},
the objects in $(A, I)^{\op}_{\Prism, \ori}$ form a basis for $(A, I)^{\op}_{\Prism}$.
We may identify sheaves on $(A, I)^{\op}_{\Prism, \ori}$ with sheaves on
$(A, I)^{\op}_{\Prism}$. 
\end{rem}

\begin{defn}\label{Definition:display group sheaf}
We define the following functor:
\[
G_{\Prism, A} \colon (A, I)_{\Prism} \to \mathrm{Set}, \quad (B, J) \mapsto G(B).
\]
As explained above, the functor $G_{\Prism, A}$ forms a group sheaf.
We also define the following functor:
\[
G_{\mu, A, I} \colon (A, I)_{\Prism, \ori} \to \mathrm{Set}, \quad (B, J) \mapsto G_\mu(B, J).
\]
Since the functor
$(A, I)_\Prism \to \mathrm{Set}$, $(B, J) \mapsto G(B[1/J])$ forms a group sheaf (Remark \ref{Remark:flat descent for BK modules}),
it follows that $G_{\mu, A, I}$ forms a group sheaf.
We regard $G_{\mu, A, I}$ as a group sheaf on $(A, I)^{\op}_{\Prism}$.
The restrictions of $G_{\Prism, A}$ and $G_{\mu, A, I}$ to $(A, I)_\et$ will be denoted by the same notation.
\end{defn}

We remark that Proposition \ref{Proposition:equivalences of pi completely etale torsors} can not be applied directly to $G_{\mu, A, I}$-torsors.
However,
it is still useful for analyzing $G_{\mu, A, I}$-torsors in several places below, since we have the following lemma.
For the notation used below, see Lemma \ref{Lemma:congruent subgroup of Gmu}.
%See also Remark \ref{Remark:push out} below for the notion of pushout.

\begin{lem}\label{Lemma:Gmu torsor successive quotient}
\ 
\begin{enumerate}
    \item For an integer $m \geq 0$, the functor
    \[
    G^{=m}_{\mu, A, I} \colon (A, I)_{\Prism, \ori} \to \mathrm{Set}, \quad
    (B, J) \mapsto G^{\geq m}_\mu(B, J)/G^{\geq m+1}_\mu(B, J)
    \]
    forms a group sheaf, and it is isomorphic to
    $(P_\mu)_{\overline{\Prism}, A}$
    (resp.\ $V(\bigoplus_{i \leq m} \mathfrak{g}_i)_{\overline{\Prism}, A}$)
    if $m=0$ 
    (resp.\ $m \geq 1$).
    Moreover, the functor
    \[
    G^{<m}_{\mu, A, I} \colon (A, I)_{\Prism, \ori} \to \mathrm{Set}, \quad
    (B, J) \mapsto G_\mu(B, J)/G^{\geq m}_\mu(B, J)
    \]
    forms a group sheaf.
    \item 
    For a $G_{\mu, A, I}$-torsor $\mathcal{Q}$ on $(A, I)^{\op}_\Prism$, we write
    $\mathcal{Q}^{<m}$
    for the pushout of $\mathcal{Q}$ along $G_{\mu, A, I} \to G^{<m}_{\mu, A, I}$ (see Remark \ref{Remark:push out} below).
    Then we have
    $
    \mathcal{Q} \overset{\sim}{\to} \varprojlim_m \mathcal{Q}^{<m}.
    $
    The same holds for $G_{\mu, A, I}$-torsors on $(A, I)^{\op}_\et$.
\end{enumerate}
\end{lem}

\begin{proof}
(1) The statement about $G^{=m}_{\mu, A, I}$ follows from Lemma \ref{Lemma:congruent subgroup of Gmu}.
Using the exact sequence
\[
1 \mapsto G^{=m}_{\mu, A, I}(B) \to G^{<m+1}_{\mu, A, I}(B) \to G^{<m}_{\mu, A, I}(B) \to 1,
\]
the statement about $G^{<m}_{\mu, A, I}$ then follows by induction on $m$.

(2) 
We may assume that $\mathcal{Q}$ is a trivial $G_{\mu, A, I}$-torsor and $(A, I)$ is orientable.
Then it is enough to prove that
$
G_{\mu, A, I} \overset{\sim}{\to} \varprojlim_m G^{<m}_{\mu, A, I}
$
on $(A, I)^{\op}_{\Prism}$.
By Proposition \ref{Proposition:decomposition of display group}, the multiplication map
$(U^-_{\mu}(A) \cap G_\mu(A, I)) \times P_\mu(A) \to G_\mu(A, I)$
is bijective.
Note that
$G_\mu(A, I)/G^{\geq m}_\mu(A, I)$
can be identified with the image of $G_\mu(A, I)$ in  $G(A/I^m)$.
Let $U^{<m}$ be the image of $U^-_{\mu}(A) \cap G_\mu(A, I)$ in $U^-_{\mu}(A/I^m)$.
Then the multiplication map induces a bijection
\[
U^{<m} \times P_\mu(A/I^m) \overset{\sim}{\to} G_\mu(A, I)/G^{\geq m}_\mu(A, I).
\]
We have
$P_\mu(A) \overset{\sim}{\to} \varprojlim P_\mu(A/I^m)$.
Moreover, using Lemma \ref{Lemma:Umu structure},
one can check that
$U^-_{\mu}(A) \cap G_\mu(A, I) \overset{\sim}{\to} \varprojlim U^{<m}$.
Thus, we obtain
$G_{\mu}(A, I) \overset{\sim}{\to} \varprojlim G_\mu(A, I)/G^{\geq m}_\mu(A, I)$.
The same assertion holds for any $(B, J) \in (A, I)_\Prism$, and hence
$G_{\mu, A, I} \overset{\sim}{\to} \varprojlim  G^{<m}_{\mu, A, I}$.
\end{proof}

\begin{rem}\label{Remark:push out}
Let $f \colon H \to H'$ be a homomorphism of groups and $Q$ a set with an $H$-action.
We can attach to $Q$ a set $Q^f$ with an $H'$-action
and an $H$-equivariant map $Q \to Q^f$ with the following universal property: For any set $Q'$ with an $H'$-action and any $H$-equivariant map $Q \to Q'$,
the map $Q \to Q'$ factors through a unique $H'$-equivariant map $Q^f \to Q'$.
Explicitly, we can define $Q^f$ as the contracted product
\[
Q^f= (Q \times H')/H,
\]
where the action of an $h \in H$ on $Q \times H'$ is defined by $(x, h') \mapsto (xh, f(h)^{-1}h')$.
We call $Q^f$ the pushout of $Q$ along $f \colon H \to H'$.

Similarly, for a homomorphism
$f \colon H \to H'$ of group sheaves on a site and
a sheaf $Q$ with an action of $H$, we can form the pushout $Q^f$ with the same properties as above.
If $Q$ is an $H$-torsor, then $Q^f$ is an $H'$-torsor.
\end{rem}

%We note that $G^{<1}_{\mu, A, I}=G^{=0}_{\mu, A, I}=(P_\mu)_{\overline{\Prism}, A}$ by Lemma \ref{Lemma:Gmu torsor successive quotient} (1).

We will use the following notation.
Let us denote
the inclusion
$G_{\mu, A, I} \hookrightarrow G_{\Prism, A}$
by $\tau$.
The composition of $\tau$ with the projection map
$G_{\Prism, A} \to G_{\overline{\Prism}, A}$
is denoted by $\overline{\tau}$.
(Here $G_{\overline{\Prism}, A}:=(G_\O)_{\overline{\Prism}, A}$.)
By Lemma \ref{Lemma:Pmu structure}, the homomorphism $\overline{\tau}$ factors through
a homomorphism
$\overline{\tau}_P \colon G_{\mu, A, I} \to (P_{\mu})_{\overline{\Prism}, A}$.
In summary, we have the following commutative diagram of group sheaves on $(A, I)^{\op}_{\Prism}$ (or on $(A, I)^{\op}_{\et}$):
\begin{equation}\label{equation:diagram of group sheaves}
    \vcenter{\xymatrix{
G_{\mu, A, I} \ar^-{\tau}[r]  \ar[d]_-{\overline{\tau}_P} \ar^-{\overline{\tau}}[rd] &  G_{\Prism, A}   \ar[d]_-{} \\
(P_{\mu})_{\overline{\Prism}, A} \ar[r]^-{} & G_{\overline{\Prism}, A}.
}}
\end{equation}

\begin{cor}\label{Corollary:trivial Gmu torsor}
A $G_{\mu, A, I}$-torsor $\mathcal{Q}$ on $(A, I)^{\op}_\Prism$ is trivial if the pushout of $\mathcal{Q}$ along $\overline{\tau}_P \colon G_{\mu, A, I} \to (P_\mu)_{\overline{\Prism}, A}$ is trivial as a $(P_\mu)_{\overline{\Prism}, A}$-torsor on $(A, I)^{\op}_\Prism$.
The same holds if we replace $(A, I)^{\op}_\Prism$ by $(A, I)^{\op}_\et$.
\end{cor}

\begin{proof}
We prove the assertion for $(A, I)^{\op}_\Prism$; the argument for $(A, I)^{\op}_\et$ is similar.
By Lemma \ref{Lemma:Gmu torsor successive quotient} (2), it suffices to show that
$\mathcal{Q}^{<m}$
is trivial as a $G^{<m}_{\mu, A, I}$-torsor
for any $m$.
We proceed by induction on $m$.
The assertion is true for $m=1$ by our assumption.
We assume that $\mathcal{Q}^{<m}$ is trivial for an integer $m \geq 1$, so that there exists an element $x \in \mathcal{Q}^{<m}(A)$.
The fiber of the morphism $\mathcal{Q}^{<m+1} \to \mathcal{Q}^{<m}$ at $x$ is a $G^{=m}_{\mu, A, I}$-torsor.
Lemma \ref{Lemma:Gmu torsor successive quotient} (1) shows that $G^{=m}_{\mu, A, I} \simeq V(\bigoplus_{i \leq m} \mathfrak{g}_i)_{\overline{\Prism}, A}$.
By Proposition \ref{Proposition:equivalences of pi completely etale torsors},
the fiber arises from a $V(\bigoplus_{i \leq m} \mathfrak{g}_i)_{A/I}$-torsor over $\Spec A/I$, which is trivial since $\Spec A/I$ is affine.
This implies that
the $G^{<m+1}_{\mu, A, I}$-torsor
$\mathcal{Q}^{<m+1}$
is trivial, as desired.
\end{proof}

\begin{cor}\label{Corollary:flat and etale Gmu torsors}
    A $G_{\mu, A, I}$-torsor $\mathcal{Q}$ on $(A, I)^{\op}_{\Prism}$ is trivialized by a $(\pi, I)$-completely \'etale covering $A \to B$, i.e.\ the restriction of $\mathcal{Q}$ to $(B, IB)^{\op}_{\Prism}$ is trivial.
    Moreover, the restriction functor induces an equivalence from the groupoid of $G_{\mu, A, I}$-torsors on $(A, I)^{\op}_{\Prism}$ to the groupoid of $G_{\mu, A, I}$-torsors on $(A, I)^{\op}_{\et}$.
\end{cor}

\begin{proof}
    The first assertion follows from Corollary \ref{Corollary:trivial Gmu torsor} since any $(P_\mu)_{\overline{\Prism}, A}$-torsor $\mathcal{P}$ on $(A, I)^{\op}_{\Prism}$ arises from a $(P_\mu)_{A/I}$-torsor over $\Spec A/I$ with respect to the \'etale topology, which in turn implies that $\mathcal{P}$ is trivialized by a $(\pi, I)$-completely \'etale covering $A \to B$ (see also Lemma \ref{Lemma:etale morphism and reduction}).
    The second assertion is a formal consequence of the first one.
\end{proof}

\begin{rem}\label{Remark:alternative description of Gmu torsors for 1-bounded}
    To a $G_{\mu, A, I}$-torsor $\mathcal{Q}$ on $(A, I)^{\op}_{\Prism}$,
    we can associate the $G_{\Prism, A}$-torsor $\mathcal{Q}^{\tau}$
    and the $(P_{\mu})_{\overline{\Prism}, A}$-torsor
    $\mathcal{Q}^{\overline{\tau}_P}$
    on $(A, I)^{\op}_\Prism$, and there is a canonical isomorphism between the $G_{\overline{\Prism}, A}$-torsors associated with $\mathcal{Q}^{\tau}$ and $\mathcal{Q}^{\overline{\tau}_P}$.
    We assume that $\mu$ is 1-bounded.
    Then, by Proposition \ref{Proposition:BB isomorphism},
    this construction induces an equivalence from the groupoid of $G_{\mu, A, I}$-torsors on $(A, I)^{\op}_{\Prism}$
    to the groupoid of triples consisting of a $G_{\Prism, A}$-torsor, a $(P_{\mu})_{\overline{\Prism}, A}$-torsor, and an isomorphism between the $G_{\overline{\Prism}, A}$-torsors associated with them.
    The same holds if we replace $(A, I)^{\op}_\Prism$ by $(A, I)^{\op}_\et$.
    Corollary \ref{Corollary:trivial Gmu torsor} and Corollary \ref{Corollary:flat and etale Gmu torsors} also follow from this fact and Proposition \ref{Proposition:equivalences of pi completely etale torsors} when $\mu$ is 1-bounded.
\end{rem}

\section{Prismatic $G$-$\mu$-displays} \label{Section:prismatic G display}

In this section,
we come to the heart of this paper, namely prismatic $G$-$\mu$-displays.
We first discuss the notion of $G$-Breuil--Kisin modules of type $\mu$ in Section \ref{Subsection:G-BK module}.
Then we introduce and study prismatic $G$-$\mu$-displays in Sections \ref{Subsection:prismatic G-display}-\ref{Subsection:Examples}.
Our prismatic $G$-$\mu$-displays are essentially equivalent to $G$-Breuil--Kisin modules of type $\mu$, and the latter may be more familiar to readers.
Nevertheless, in many cases, such as the proof of the main result (Theorem \ref{Theorem:main result on G displays over complete regular local rings}) of this paper, it will be crucial to work with prismatic $G$-$\mu$-displays.

We retain the notation of Section \ref{Section:display group}.
Recall that $G$ is a smooth affine group scheme over $\O_E$
and $\mu \colon \G_m \to G_{\O}$ is a cocharacter defined over $\O=W(k)\otimes_{W(\F_q)} \O_E$.

\subsection{$G$-Breuil--Kisin modules of type $\mu$}\label{Subsection:G-BK module}

Let $(A, I)$ be a bounded $\O_E$-prism over $\O$.

\begin{defn}\label{Definition:G-BK module}
A \textit{$G$-Breuil--Kisin module} over $(A, I)$
is a pair
$(\mathcal{P}, F_\mathcal{P})$ consisting of
a $G_A$-torsor $\mathcal{P}$ over $\Spec A$ (with respect to the \'etale topology) and
an isomorphism
\[
F_\mathcal{P} \colon (\phi^*\mathcal{P})[1/I] \overset{\sim}{\to} \mathcal{P}[1/I]
\]
of $G_{A[1/I]}$-torsors over $\Spec A[1/I]$.
\end{defn}

Here, for a $G_A$-torsor $\mathcal{P}$ over $\Spec A$,
we let
$\phi^*\mathcal{P}$ denote
the base change of $\mathcal{P}$ along the Frobenius $\phi \colon A \to A$.
Since $\phi$ is $\O_E$-linear and $G$ is defined over $\O_E$,
we have $\phi^*G_A=G_A$, and hence $\phi^*\mathcal{P}$ is a $G_A$-torsor over $\Spec A$.
Moreover, we write
$\mathcal{P}[1/I]:=\mathcal{P} \times_{\Spec A} \Spec A[1/I]$.
When there is no ambiguity, we simply write $\mathcal{P}=(\mathcal{P}, F_\mathcal{P})$.

\begin{ex}\label{Example:GLn BK module}
Assume that $G=\GL_n$.
Let
$(M, F_M)$ be a Breuil--Kisin module of rank $n$ over $(A, I)$.
Let
\[
\mathcal{P}(M):=\underline{\mathrm{Isom}}(A^n, M)
\]
be the $\GL_{n, A}$-torsor over $\Spec A$
defined by sending an $A$-algebra $B$
to
the set of isomorphisms $B^n \simeq M_B$.
Together with the isomorphism
$(\phi^*\mathcal{P}(M))[1/I] \overset{\sim}{\to} \mathcal{P}(M)[1/I]$ induced by $F_M$, we regard $\mathcal{P}(M)$
as a $\GL_n$-Breuil--Kisin module over $(A, I)$.
This construction
$M \mapsto \mathcal{P}(M)$
induces an equivalence
between the groupoid of Breuil--Kisin modules of rank $n$ over $(A, I)$ and the groupoid of $\GL_n$-Breuil--Kisin modules over $(A, I)$.
\end{ex}

\begin{rem}\label{Remark:flat descent for G-BK modules}
    Let $\mathcal{P}$ and $\mathcal{P}'$ be $G_A$-torsors over $\Spec A$.
    Using that the functor
    $
    (\O_E)_{\Prism, \O_E} \to \mathrm{Set}, (B, J) \mapsto B[1/J]
    $
    forms a sheaf (see Remark \ref{Remark:flat descent for BK modules}) and that $\mathcal{P}, \mathcal{P}'$ are affine and flat over $\Spec A$, one can show that the functor $(A, I)_{\Prism} \to \mathrm{Set}$ which associates to each $(B, J) \in (A, I)_{\Prism}$ the set of isomorphisms
    $\mathcal{P}_B[1/J] \overset{\sim}{\to} \mathcal{P}'_B[1/J]$
    of $G_{B[1/J]}$-torsors forms a sheaf.
    This, together with Proposition \ref{Proposition:equivalences of pi completely etale torsors}, implies that
    the fibered category over $(A, I)^{\op}_{\Prism}$ which associates to each $(B, J) \in (A, I)_\Prism$ the groupoid of $G$-Breuil--Kisin modules over $(B, J)$ is a stack with respect to the flat topology.
\end{rem}

We introduce $G$-Breuil--Kisin modules of type $\mu$.
Recall that for a $(\pi, I)$-completely \'etale $A$-algebra $B \in (A, I)_\et$,
the pair $(B, IB)$ is naturally a bounded $\O_E$-prism; see Lemma \ref{Lemma:etale morphism prism}.

\begin{defn}[{$G$-Breuil--Kisin module of type $\mu$}]\label{Definition:G-BK module of type mu}
We say that a $G$-Breuil--Kisin module
$(\mathcal{P}, F_\mathcal{P})$
over $(A, I)$ is \textit{of type $\mu$}
if there exists a $(\pi, I)$-completely \'etale covering
$A \to B$ such that $(B, IB)$ is orientable, the base change $\mathcal{P}_B$ is a trivial $G_B$-torsor, and via some (and hence any) trivialization
    $\mathcal{P}_B \simeq G_B$,
    the isomorphism $F_\mathcal{P}$ is given by $g \mapsto Yg$ for an element $Y$ in the double coset
    \[
    G(B)\mu(d)G(B) \subset G(B[1/IB])
    \]
    where $d \in IB$ is a generator.
    If these conditions are satisfied for $B=A$, then we say that $(\mathcal{P}, F_\mathcal{P})$ is \textit{banal}.
\end{defn}

We write
\[
G\mathchar`-\mathrm{BK}_\mu(A, I) \quad \text{and} \quad G\mathchar`-\mathrm{BK}_\mu(A, I)_{\mathrm{banal}}
\]
for the groupoid of $G$-Breuil--Kisin modules of type $\mu$ over $(A, I)$ and the groupoid of banal $G$-Breuil--Kisin modules of type $\mu$ over $(A, I)$ (when $(A, I)$ is orientable), respectively.

\begin{rem}\label{Remark:etale descent for G-BK of type mu}
    By Remark \ref{Remark:flat descent for G-BK modules}, the fibered category over $(A, I)^{\op}_{\et}$ which associates to each $B \in (A, I)_\et$ the groupoid of $G$-Breuil--Kisin modules of type $\mu$ over $(B, IB)$ is a stack with respect to the $(\pi, I)$-completely \'etale topology.
    We will prove that the same result holds for the flat topology in Corollary \ref{Corollary:G-BK module etale banal} below, using $G$-$\mu$-displays introduced in the next subsection.
\end{rem}

\begin{ex}\label{Example:G-BK of type mu implies BK of type mu}
    Let $M$ be a Breuil--Kisin module of rank $n$ over $(A, I)$ and
    let $\mathcal{P}(M)$ be the associated $\GL_n$-Breuil--Kisin module over $(A, I)$ (see Example \ref{Example:GLn BK module}).
    If $\mathcal{P}(M)$ is of type $\mu$, then $M$ is of type $\mu$ in the sense of Definition \ref{Definition:type of displayed BK module} by Proposition \ref{Proposition:displayed condition base change}.
    We will prove that the converse is also true in Example \ref{Example:GLn displays}.
\end{ex}

\subsection{$G$-$\mu$-displays}\label{Subsection:prismatic G-display}

We now introduce prismatic $G$-$\mu$-displays.
To an orientable and bounded $\O_E$-prism $(A, I)$ over $\O$, we attach the display group $G_\mu(A, I)$ as in Definition \ref{Definition:display group}.
Since $G$ is defined over $\O_E$, the Frobenius $\phi$ of $A$ induces a homomorphism $\phi \colon G(A) \to G(A)$.
For each generator $d \in I$, we define the following homomorphism:
\begin{equation}\label{equation:sigma map of sets}
    \sigma_{\mu, d} \colon G_\mu(A, I) \to G(A), \quad g \mapsto \phi(\mu(d)g\mu(d)^{-1}).
\end{equation}
We endow $G(A)$ with the following action of $G_\mu(A, I)$:
\begin{equation}\label{equation:action of display group}
    G(A) \times G_\mu(A, I) \to G(A), \quad (X, g) \mapsto X \cdot g:=g^{-1}X\sigma_{\mu, d}(g).
\end{equation}
We write $G(A)=G(A)_d$ when we regard $G(A)$ as a set with this action of $G_\mu(A, I)$.
For another generator $d' \in I$, we have $d=ud'$ for a unique element $u \in A^\times$.
The map $G(A)_d \to G(A)_{d'}$ defined by $X \mapsto X\phi(\mu(u))$ is $G_\mu(A, I)$-equivariant.
Then we define the set
\[
G(A)_I := \varprojlim_{d} G(A)_d
\]
equipped with a natural action of $G_\mu(A, I)$,
where $d$ runs over the set of generators $d \in I$.
The projection map $G(A)_I \to G(A)_d$ is an isomorphism.
For an element $X \in G(A)_I$, let
\[
X_d \in G(A)_d
\]
denote the image of $X$.
Although $G(A)_I$ depends on the cocharacter $\mu$, we omit it from the notation.
We hope that this will not cause any confusion.

Let $(A, I)$ be a bounded $\O_E$-prism over $\O$.
We recall the category $(A, I)_{\Prism, \ori}$ from Definition \ref{Definition:category of orientable bounded prisms}.
We define the following functor
\[
G_{\Prism, A, I} \colon (A, I)_{\Prism, \ori} \to \mathrm{Set}, \quad (B, J) \mapsto G(B)_{J}.
\]
This forms a sheaf.
We regard $G_{\Prism, A, I}$ as a sheaf on $(A, I)^{\op}_{\Prism}$ (see Remark \ref{Remark:sheaves on oriented prisms}).
The sheaf $G_{\Prism, A, I}$ is equipped with a natural action of the group sheaf $G_{\mu, A, I}$ on $(A, I)^{\op}_{\Prism}$ defined in Definition \ref{Definition:display group sheaf}.

The restriction of $G_{\Prism, A, I}$ to $(A, I)_\et$ is denoted by the same notation.
We define prismatic $G$-$\mu$-displays, using the $(\pi, I)$-completely \'etale topology, as follows.

\begin{defn}[{$G$-$\mu$-display}]\label{Definition:G mu display over oriented prisms}
Let $(A, I)$ be a bounded $\O_E$-prism over $\O$.
\begin{enumerate}
    \item A \textit{$G$-$\mu$-display} over
    $(A, I)$ is a pair
    \[
    (\mathcal{Q}, \alpha_\mathcal{Q})
    \]
    where $\mathcal{Q}$ is a $G_{\mu, A, I}$-torsor on $(A, I)^{\op}_\et$ and $\alpha_\mathcal{Q} \colon \mathcal{Q} \to G_{\Prism, A, I}$ is a $G_{\mu, A, I}$-equivariant map of sheaves.
    The $G_{\mu, A, I}$-torsor $\mathcal{Q}$ is called the \textit{underlying $G_{\mu, A, I}$-torsor} of $(\mathcal{Q}, \alpha_\mathcal{Q})$.
    We say that $(\mathcal{Q}, \alpha_\mathcal{Q})$ is \textit{banal} if $\mathcal{Q}$ is trivial as a $G_{\mu, A, I}$-torsor.
    When there is no possibility of confusion, we write $\mathcal{Q}$ instead of $(\mathcal{Q}, \alpha_\mathcal{Q})$.
    
    \item An isomorphism
    $g \colon (\mathcal{Q}, \alpha_\mathcal{Q}) \to (\mathcal{R}, \alpha_\mathcal{R})$ of $G$-$\mu$-displays over 
    $(A, I)$
    is an isomorphism $g \colon \mathcal{Q} \overset{\sim}{\to} \mathcal{R}$ of $G_{\mu, A, I}$-torsors such that $\alpha_{\mathcal{R}} \circ g=\alpha_\mathcal{Q}$.
\end{enumerate}
\end{defn}

We write
\[
G\mathchar`-\mathrm{Disp}_\mu(A, I) \quad \text{and} \quad G\mathchar`-\mathrm{Disp}_\mu(A, I)_{\mathrm{banal}}
\]
for the groupoid of $G$-$\mu$-displays over $(A, I)$ and the groupoid of banal $G$-$\mu$-displays over $(A, I)$, respectively.

\begin{rem}\label{Remark:compare with Lau, and Bartling perfectoid case}
The notion of $G$-$\mu$-displays was originally introduced by B\"ultel \cite{Bultel}, B\"ultel--Pappas \cite{Bultel-Pappas}, and Lau \cite{Lau21} in different settings.
The definition given here is an adaptation of Lau's approach to the context of ($\O_E$-)prisms; see also Remark \ref{Remark:analogue of Lau's definitions display version} below.
If
$\O_E=\Z_p$ and $\mu$ is $1$-bounded,
the notion of $G$-$\mu$-displays for 
an oriented perfect prism has already appeared in the work of Bartling \cite{Bartling}.
He also claimed that the same construction should work for more general oriented prisms in \cite[Remark 14]{Bartling}.
\end{rem}

\begin{rem}\label{Remark:analogue of Lau's definitions display version}
Assume that $(A, I)$ is orientable.
We consider the graded ring
\[
\Rees(I^\bullet):= (\bigoplus_{i \geq 0} I^i t^{-i}) \oplus (\bigoplus_{i < 0} A t^{-i}) \subset A[t, t^{-1}],
\]
where the degree of $t$ is $-1$.
Let $\tau \colon \Rees(I^\bullet) \to A$
be the homomorphism of $A$-algebras defined by $t \mapsto 1$.
For a generator $d \in I$, let
$\sigma_d \colon \Rees(I^\bullet) \to A$
be the homomorphism
defined by
$a_it^{-i} \mapsto \phi(a_id^{-i})$
for any $i \in \Z$.
The triple
\[
(\Rees(I^\bullet), \sigma_d, \tau)
\]
can be viewed as an analogue of a higher frame introduced in \cite[Definition 2.0.1]{Lau21}.
We note that by Lemma \ref{Lemma:Gmu iff condition},
the homomorphism $\tau$ induces an isomorphism
between
the display group $G_\mu(A, I)$
and the subgroup
\[
G(\Rees(I^\bullet))^{0} \subset G(\Rees(I^\bullet))
\]
consisting of
homomorphisms
$g^* \colon A_G \to \Rees(I^\bullet)$
of graded $\O$-algebras.
Under this isomorphism,
the homomorphism
$\sigma_{\mu, d}$ agrees with the one
$G(\Rees(I^\bullet))^{0} \to G(A)$ induced by $\sigma_d$.
Therefore, the action $(\ref{equation:action of display group})$ is consistent with the one considered in \cite[(5-2)]{Lau21}.
\end{rem}

\begin{rem}\label{Remark:comparison base field}
    Let $\widetilde{k}$ be a perfect field containing $k$.
    We set $\widetilde{\O}:=W(\widetilde{k}) \otimes_{W(\F_q)} \O_E$.
    Let $\widetilde{\mu} \colon \G_m \to G_{\widetilde{\O}}$ be the base change of $\mu$.
    Then, for a bounded $\O_E$-prism $(A, I)$ over $\widetilde{\O}$,
    a $G$-$\widetilde{\mu}$-display over $(A, I)$ is the same as a $G$-$\mu$-display over $(A, I)$.
\end{rem}

We have the following alternative description of banal $G$-$\mu$-displays, which we will use frequently in the sequel.

\begin{rem}\label{Remark:quotient groupoid banal G-displays}
Assume that $(A, I)$ is orientable.
Let
\[
[G(A)_I/G_\mu(A, I)]
\]
denote the groupoid whose objects are the elements $X \in G(A)_I$ and whose morphisms are defined by
$
\Hom(X, X')=\{\, g \in G_\mu(A, I) \, \vert \, X'\cdot g=X  \, \}.
$
Here $(-)\cdot g$ denotes the action of $g \in G_\mu(A, I)$ on $G(A)_I$.
To each
$X \in G(A)_I$,
we attach a banal $G$-$\mu$-display
\[
\mathcal{Q}_X:=(G_{\mu, A, I}, \alpha_X)
\]
over $(A, I)$ where $\alpha_X \colon G_{\mu, A, I} \to G_{\Prism, A, I}$ is given by $1 \mapsto X$.
We obtain an equivalence
\[
[G(A)_I/G_\mu(A, I)] \overset{\sim}{\to} G\mathchar`-\mathrm{Disp}_\mu(A, I)_{\mathrm{banal}}, \quad X \mapsto \mathcal{Q}_X
\]
of groupoids.
\end{rem}

We discuss the notion of base change for $G$-$\mu$-displays.
Let $f \colon (A, I) \to (A', I')$ be a map of orientable and bounded $\O_E$-prisms over $\O$.
We have natural homomorphisms
$f \colon G(A) \to G(A')$
and $f \colon G_\mu(A, I) \to G_\mu(A', I')$.
Let $d \in I$ and $d' \in I'$ be generators and let $u \in A'^\times$ be the unique element satisfying $f(d)=ud'$.
Then the following composition of $G_\mu(A, I)$-equivariant maps
\[
G(A)_I \simeq G(A)_d \to G(A')_{d'} \simeq G(A')_{I'},
\]
where the second map is defined by $X \mapsto f(X)\phi(\mu(u))$, is independent of the choices of $d$ and $d'$, and is also denoted by $f$.

We now consider
a map
$f \colon (A, I) \to (A', I')$
of (not necessarily orientable) bounded $\O_E$-prisms over $\O$.
The functor
$(A, I)^{\op}_\et \to (A', I')^{\op}_\et$
sending $B \in (A, I)^{\op}_\et$ to the $(\pi, I)$-adic completion $B'$ of $B \otimes_A A'$
induces
a morphism of the associated topoi
\[
f \colon ((A', I')^{\op}_\et)^\sim \to ((A, I)^{\op}_\et)^\sim
\]
(since it sends $(\pi, I)$-completely \'etale coverings to $(\pi, I')$-completely \'etale coverings, sends final objects to final objects, and commutes with fiber products).
We have a natural homomorphism
$
f \colon f^{-1}G_{\mu, A, I} \to G_{\mu, A', I'}
$
of group sheaves.
Moreover, the maps 
$G(A)_I \to G(A')_{I'}$ 
defined in the orientable case glue together to a morphism
$
f \colon f^{-1}G_{\Prism, A, I} \to G_{\Prism, A', I'}
$
of sheaves.

\begin{defn}\label{Definition:base change for G-displays}
Let
$(\mathcal{Q}, \alpha_{\mathcal{Q}})$ be a $G$-$\mu$-display over $(A, I)$.
Let
$f^*\mathcal{Q}$
be the pushout of
the $f^{-1}G_{\mu, A, I}$-torsor
$f^{-1}\mathcal{Q}$ along
$f \colon f^{-1}G_{\mu, A, I} \to G_{\mu, A', I'}$.
By the universal property of $f^*\mathcal{Q}$, 
the composition 
\[
f^{-1}\mathcal{Q} \overset{f^{-1}(\alpha_{\mathcal{Q}})}{\longrightarrow} f^{-1}G_{\Prism, A, I} \to G_{\Prism, A', I'}
\]
factors through a unique $G_{\mu, A', I'}$-equivariant map
$f^*(\alpha_{\mathcal{Q}}) \colon f^*\mathcal{Q} \to G_{\Prism, A', I'}$.
The base change of $(\mathcal{Q}, \alpha_{\mathcal{Q}})$ along $f \colon (A, I) \to (A', I')$ is defined to be
$
(f^*\mathcal{Q}, f^*(\alpha_{\mathcal{Q}})).
$
\end{defn}

\begin{ex}\label{Example:base change of banal displays}
    Assume that $(A, I)$ is orientable.
    For the banal $G$-$\mu$-display $\mathcal{Q}_X$ associated with an element $X \in G(A)_I$ (see Remark \ref{Remark:quotient groupoid banal G-displays}), we have
    $f^*(\mathcal{Q}_X)=\mathcal{Q}_{f(X)}$.
\end{ex}

By definition, it is clear that $G$-$\mu$-displays form a stack with respect to the $(\pi, I)$-completely \'etale topology.
In fact, we can prove the following flat descent result, which is an analogue of \cite[Lemma 5.4.2]{Lau21}.

\begin{prop}[Flat descent]\label{Proposition:flat descent of G display}
The fibered category over
$
(\O)^{\op}_{\Prism, \O_E}
$
which associates to each $(A, I) \in (\O)_{\Prism, \O_E}$
the groupoid
$
G\mathchar`-\mathrm{Disp}_\mu(A,I)
$
is a stack with respect to the flat topology.
\end{prop}

\begin{proof}
It suffices to prove that
$
G\mathchar`-\mathrm{Disp}_\mu(A, I)
$
is equivalent to the groupoid of pairs 
$(\mathcal{Q}, \alpha_{\mathcal{Q}})$
where $\mathcal{Q}$ is a $G_{\mu, A, I}$-torsor on
$(A, I)^{\op}_{\Prism}$ (with respect to the flat topology)
and $\alpha_{\mathcal{Q}} \colon \mathcal{Q} \to G_{\Prism, A, I}$ is a $G_{\mu, A, I}$-equivariant map of sheaves on $(A, I)^{\op}_{\Prism}$.
This follows from Corollary \ref{Corollary:flat and etale Gmu torsors}.
\end{proof}

\subsection{$G$-$\mu$-displays and $G$-Breuil--Kisin modules of type $\mu$}\label{Subsection:G-mu-displays and G-Breuil--Kisin modules of type mu}

Here we shall show that $G$-$\mu$-displays are essentially equivalent to $G$-Breuil--Kisin modules of type $\mu$.
Let $(A, I)$ be a bounded $\O_E$-prism over $\O$.

\begin{defn}\label{Definition:G-torsor Q_BK}
    To a $G_{\mu, A, I}$-torsor $\mathcal{Q}$ on $(A, I)^{\op}_\et$,
    we attach a $G_A$-torsor $\mathcal{Q}_{\mathrm{BK}}$ over $\Spec A$ as follows.
    We first assume that $(A, I)$ is orientable.
    Let $d \in I$ be a generator.
Let
    $
    \mathcal{Q}_{\mathrm{BK}, d}
    $
    be the pushout of
$\mathcal{Q}$
along the homomorphism
\[
G_{\mu, A, I} \hookrightarrow G_{\Prism, A}, \quad g \mapsto \mu(d)g\mu(d)^{-1}.
\]
Let $d' \in I$ be another generator
and let $u \in A^\times$ be the unique element such that $d=ud'$.
We define
$\ad(\mu(u)) \colon G_{\Prism, A} \overset{\sim}{\to} G_{\Prism, A}$
by $g \mapsto \mu(u)g\mu(u)^{-1}$.
The pushout $(\mathcal{Q}_{\mathrm{BK}, d'})^{\ad(\mu(u))}$ can be identified with $\mathcal{Q}_{\mathrm{BK}, d}$.
The following composition
\begin{equation}\label{equation:identification along inner auto}
    \mathcal{Q}_{\mathrm{BK}, d'} \to (\mathcal{Q}_{\mathrm{BK}, d'})^{\ad(\mu(u))}=\mathcal{Q}_{\mathrm{BK}, d} \overset{x \mapsto x \cdot \mu(u)}{\longrightarrow}\mathcal{Q}_{\mathrm{BK}, d}
\end{equation}
is an isomorphism of $G_{\Prism, A}$-torsors.
(See Remark \ref{Remark:push out} for the first map.)
Then we define
\[
\mathcal{Q}_{\mathrm{BK}}:=\varprojlim_{d} \mathcal{Q}_{\mathrm{BK}, d}
\]
where $d$ runs over the set of generators $d \in I$.

In general, the sheaves constructed in the banal case glue together to a $G_{\Prism, A}$-torsor $\mathcal{Q}_{\mathrm{BK}}$ on 
$(A, I)^{\op}_\et$.
By Proposition \ref{Proposition:equivalences of pi completely etale torsors},
we regard $\mathcal{Q}_{\mathrm{BK}}$ as a $G_A$-torsor over $\Spec A$.
\end{defn}

\begin{rem}\label{Remark:isomorphism between Q_BK and Q_A}
    Recall that $\tau \colon G_{\mu, A, I} \hookrightarrow G_{\Prism, A}$ is the natural inclusion.
    For a $G_{\mu, A, I}$-torsor $\mathcal{Q}$ on $(A, I)^{\op}_\et$, let
    \[
    \mathcal{Q}_A:=\mathcal{Q}^\tau
    \]
    be the pushout of
    $\mathcal{Q}$
    along $\tau$, regarded as a $G_A$-torsor over $\Spec A$
    (by Proposition \ref{Proposition:equivalences of pi completely etale torsors}).
    There exists a canonical isomorphism
    \[
    \mathcal{Q}_A[1/I] \overset{\sim}{\to} \mathcal{Q}_{\mathrm{BK}}[1/I]
    \]
    of $G_{A[1/I]}$-torsors over $\Spec A[1/I]$ obtained as follows.
    We first assume that $(A, I)$ is orientable.
    Let $d \in I$ be a generator.
    Similarly to (\ref{equation:identification along inner auto}), the composition
    \[
    \mathcal{Q}_A[1/I] \to (\mathcal{Q}_A[1/I])^{\ad(\mu(d))}=\mathcal{Q}_{\mathrm{BK}, d}[1/I] \overset{x \mapsto x \cdot \mu(d)}{\longrightarrow}\mathcal{Q}_{\mathrm{BK}, d}[1/I]
    \]
    is an isomorphism of $G_{A[1/I]}$-torsors, where
    $\ad(\mu(d)) \colon G_{A[1/I]} \overset{\sim}{\to} G_{A[1/I]}$ is defined by $g \mapsto \mu(d)g\mu(d)^{-1}$.
    We then obtain the desired isomorphism as
    \[
    \mathcal{Q}_A[1/I] \simeq \mathcal{Q}_{\mathrm{BK}, d}[1/I] \simeq \mathcal{Q}_{\mathrm{BK}}[1/I],
    \]
    which does not depend on the choice of $d \in I$.
    By Remark \ref{Remark:flat descent for G-BK modules},
    the isomorphisms in the banal case glue together to an isomorphism $\mathcal{Q}_A[1/I] \overset{\sim}{\to} \mathcal{Q}_{\mathrm{BK}}[1/I]$.
\end{rem}

\begin{ex}\label{Example:GLn case Q_BK}
    Assume that $G=\GL_n$.
    Let the notation be as in Example \ref{Example:display group GLn case}.
    Let
    $M$
    be a Breuil--Kisin module of type $\mu$ over $(A, I)$.
    Recall the filtration $\{ \Fil^i(\phi^*M) \}_{i \in \Z}$ of $\phi^*M$ from Definition \ref{Definition:Breuil-Kisin module}.
Let $\{ \Fil^i_\mu \}_{i \in \Z}$ be the filtration of $A^n$ defined in Remark \ref{Remark:standard filtration}.
The functor
\[
\mathcal{Q}(M):= \underline{\mathrm{Isom}}_{\Fil}(A^n, \phi^*M)
\colon (A, I)_\et \to \mathrm{Set}
\]
sending $B \in (A, I)_\et$ to the set of isomorphisms
$h \colon B^n \overset{\sim}{\to} (\phi^*M)_B$ preserving the filtrations is a $(\GL_n)_{\mu, A, I}$-torsor by
Remark \ref{Remark:standard filtration},
Example \ref{Example:display group GLn case},
and the fact that $M$ is $(\pi, I)$-completely \'etale locally on $A$ banal.
We note that
\[
\mathcal{Q}(M)_A = \underline{\mathrm{Isom}}(A^n, \phi^*M).
\]
We set $\widetilde{M}:=\Fil^{m_1}(\phi^*M) \otimes_A I^{-m_1}$.
Then we have a canonical identification
\[
\mathcal{Q}(M)_{\mathrm{BK}}= \underline{\mathrm{Isom}}(A^n, \widetilde{M}).
\]
If $A$ is orientable and $d \in I$ is a generator, then $\mathcal{Q}(M)_{\mathrm{BK}, d}= \underline{\mathrm{Isom}}(A^n, \widetilde{M})$ and the natural map
$\mathcal{Q}(M) \to \mathcal{Q}(M)_{\mathrm{BK}, d}$ sends $h \in \mathcal{Q}(M)(A)$ to the composition of isomorphisms
\[
A^n \overset{\mu(d)^{-1}}{\to} \Fil^{m_1}_\mu \otimes_A I^{-m_1} \overset{h}{\to} \widetilde{M}.
\]
If $d' \in I$ is another generator, then the isomorphism 
$\mathcal{Q}(M)_{\mathrm{BK}, d'} \overset{\sim}{\to} \mathcal{Q}(M)_{\mathrm{BK}, d}$
from (\ref{equation:identification along inner auto}) is the identity $\underline{\mathrm{Isom}}(A^n, \widetilde{M}) \to \underline{\mathrm{Isom}}(A^n, \widetilde{M})$.

The isomorphism
$\mathcal{Q}(M)_A[1/I] \overset{\sim}{\to} \mathcal{Q}(M)_{\mathrm{BK}}[1/I]$
defined in Remark \ref{Remark:isomorphism between Q_BK and Q_A}
agrees with the one induced from the equality
$(\phi^*M)[1/I]=\widetilde{M}[1/I]$.
\end{ex}

To construct
$G$-Breuil--Kisin modules of type $\mu$ from $G$-$\mu$-displays, we use the following proposition, which also gives an alternative description of $G$-$\mu$-displays.

\begin{prop}\label{Proposition:alternative description of G-mu-displays}
    Let $\mathcal{Q}$ be a $G_{\mu, A, I}$-torsor on $(A, I)^{\op}_\et$.
    Then there is a natural bijection $\alpha \mapsto \alpha'$ from the set of $G_{\mu, A, I}$-equivariant maps $\alpha \colon \mathcal{Q} \to G_{\Prism, A, I}$
    to the set of isomorphisms
    $\alpha' \colon \phi^*(\mathcal{Q}_{\mathrm{BK}}) \overset{\sim}{\to} \mathcal{Q}_A$
    of $G_A$-torsors over $\Spec A$.
\end{prop}

\begin{proof}
    We shall construct the bijection when
    $(A, I)$ is orientable and $\mathcal{Q}$ is a trivial $G_{\mu, A, I}$-torsor; the general case follows by gluing.
    Let
    $\alpha \colon \mathcal{Q} \to G_{\Prism, A, I}$
    be a $G_{\mu, A, I}$-equivariant map.
    We choose a trivialization
    $\mathcal{Q} \simeq G_{\mu, A, I}$.
    Then $\alpha$ can be regarded as a $G_{\mu, A, I}$-equivariant map
    $G_{\mu, A, I} \to G_{\Prism, A, I}$, which is determined by the image $X \in G(A)_I$ of $1 \in G_\mu(A, I)$.
    We may also identify $\mathcal{Q}_A$ with $G_A$.
    Let $d \in I$ be a generator.
    Then we may identify
$\mathcal{Q}_{\mathrm{BK}}$
with $G_A$ by
\[
\mathcal{Q}_{\mathrm{BK}} \simeq \mathcal{Q}_{\mathrm{BK}, d} \simeq (G_{\mu, A, I})_{\mathrm{BK}, d}= G_A.
\]
(See Definition \ref{Definition:G-torsor Q_BK} for $\mathcal{Q}_{\mathrm{BK}, d}$.)
Via these identifications, we define $\alpha' \colon \phi^*(\mathcal{Q}_{\mathrm{BK}}) \overset{\sim}{\to} \mathcal{Q}_A$ by
\[
\phi^*(\mathcal{Q}_{\mathrm{BK}})=\phi^*G_A=G_A \overset{\sim}{\to} G_A=\mathcal{Q}_A, \quad g \mapsto X_d \cdot g
\]
where $X_d \in G(A)=G(A)_d$ is the image of $X \in G(A)_I$.
One can check that the resulting isomorphism
$\alpha'$
does not depend on the choices of $\mathcal{Q} \simeq G_{\mu, A, I}$
and $d \in I$.
It is clear that the map $\alpha \mapsto \alpha'$ is a bijection.
\end{proof}

\begin{rem}\label{Remark:alternative description of G-mu-displays}
By Proposition \ref{Proposition:alternative description of G-mu-displays},
a $G$-$\mu$-display over $(A, I)$ can be thought of as a pair
$(\mathcal{Q}, \alpha')$
of a $G_{\mu, A, I}$-torsor $\mathcal{Q}$ on $(A, I)^{\op}_\et$ and an isomorphism
$\alpha' \colon \phi^*(\mathcal{Q}_{\mathrm{BK}}) \overset{\sim}{\to} \mathcal{Q}_A$
of $G_A$-torsors over $\Spec A$.
\end{rem}

\begin{defn}\label{Definition:functor from G-displays to G-BK}
    Let $(\mathcal{Q}, \alpha_\mathcal{Q})$ be a $G$-$\mu$-display over $(A, I)$
    and let $(\alpha_\mathcal{Q})' \colon \phi^*(\mathcal{Q}_{\mathrm{BK}}) \overset{\sim}{\to} \mathcal{Q}_A$
    be the corresponding isomorphism.
    We denote by $F$ the following composition
    \[
    (\phi^*(\mathcal{Q}_{\mathrm{BK}}))[1/I] \overset{(\alpha_\mathcal{Q})'}{\to} \mathcal{Q}_A[1/I] \overset{\sim}{\to} \mathcal{Q}_{\mathrm{BK}}[1/I]
    \]
    where the second isomorphism is constructed in Remark \ref{Remark:isomorphism between Q_BK and Q_A}.
    By construction, we see that
    $\mathcal{Q}_{\mathrm{BK}}$,
    together with the isomorphism $F$,
    is a $G$-Breuil--Kisin module of type $\mu$.
    (See also Example \ref{example:associated G-BK module banal case} below.)
    We have a functor
    \begin{equation}\label{equation:functor from G-displays to G-BK}
G\mathchar`-\mathrm{Disp}_\mu(A, I) \to G\mathchar`-\mathrm{BK}_\mu(A, I), \quad \mathcal{Q} \mapsto \mathcal{Q}_{\mathrm{BK}}.
\end{equation}
\end{defn}

\begin{ex}\label{example:associated G-BK module banal case}
    Assume that $(A, I)$ is orientable
    and let $d \in I$ be a generator.
    Let $\mathcal{Q}_X$ be the banal $G$-$\mu$-display associated with an element $X \in G(A)_I$ (Remark \ref{Remark:quotient groupoid banal G-displays}).
    The trivial $G_A$-torsor $G_A$ with the isomorphism
    \[
    (\phi^*G_A)[1/I]=G_A[1/I] \overset{\sim}{\to} G_A[1/I], \quad g \mapsto (\mu(d)X_d) g
    \]
    is a banal $G$-Breuil--Kisin module of type $\mu$ over $(A, I)$, which is denoted by $\mathcal{P}_{X_d}$.
    By construction, we have
$(\mathcal{Q}_X)_{\mathrm{BK}} \overset{\sim}{\to} \mathcal{P}_{X_d}$.
\end{ex}

\begin{prop}\label{Proposition:G-displays and G-BK modules}
    Let $(A, I)$ be a bounded $\O_E$-prism over $\O$.
    The functor $(\ref{equation:functor from G-displays to G-BK})$
    \[
    G\mathchar`-\mathrm{Disp}_\mu(A, I) \to G\mathchar`-\mathrm{BK}_\mu(A, I), \quad \mathcal{Q} \mapsto \mathcal{Q}_{\mathrm{BK}}
    \]
    is an equivalence.
\end{prop}

\begin{proof}
By Remark \ref{Remark:flat locally orientable},
Remark \ref{Remark:etale descent for G-BK of type mu}, and $(\pi, I)$-completely \'etale descent for $G$-$\mu$-displays, it suffices to prove that the functor
\[
G\mathchar`-\mathrm{Disp}_\mu(A, I)_{\mathrm{banal}} \to G\mathchar`-\mathrm{BK}_\mu(A, I)_{\mathrm{banal}}, \quad \mathcal{Q} \mapsto \mathcal{Q}_{\mathrm{BK}}
\]
is an equivalence when $(A, I)$ is orientable.

We shall prove that the functor is fully faithful.
It suffices to prove that for all $X, X' \in G(A)_I$ and the associated banal $G$-$\mu$-displays
$\mathcal{Q}_X, \mathcal{Q}_{X'}$
over $(A, I)$,
we have
\begin{equation}\label{equation:hom functor}
    \Hom(\mathcal{Q}_X, \mathcal{Q}_{X'}) \overset{\sim}{\to} \Hom((\mathcal{Q}_X)_{\mathrm{BK}}, (\mathcal{Q}_{X'})_{\mathrm{BK}}).
\end{equation}
We fix a generator $d \in I$.
The left hand side can be identified with
\[
\{\, g \in G_\mu(A, I) \, \vert \, g^{-1}X'_d \phi(\mu(d)g\mu(d)^{-1})=X_d  \, \}.
\]
(See Remark \ref{Remark:quotient groupoid banal G-displays}.)
By Example \ref{example:associated G-BK module banal case}, we have
$(\mathcal{Q}_X)_{\mathrm{BK}} \overset{\sim}{\to} \mathcal{P}_{X_d}$ and
$(\mathcal{Q}_{X'})_{\mathrm{BK}} \overset{\sim}{\to} \mathcal{P}_{X'_d}$.
Thus the right hand side of $(\ref{equation:hom functor})$ can be identified with
\[
\{\, h \in G(A) \, \vert \, h^{-1}\mu(d)X'_d \phi(h)=\mu(d)X_d  \, \}.
\]
The map $(\ref{equation:hom functor})$ is given by $g \mapsto \mu(d)g\mu(d)^{-1}$ under these identifications.
In particular, the map is injective.
For surjectivity, let $h \in G(A)$ be an element such that 
$h^{-1}\mu(d)X'_d \phi(h)=\mu(d)X_d$.
The element $g:=\mu(d)^{-1}h \mu(d)=X'_d\phi(h)X^{-1}_d$ belongs to $G(A)$, and hence $g \in G_\mu(A, I)$.
It follows that $g \in \Hom(\mathcal{Q}_X, \mathcal{Q}_{X'})$, and $g$ is mapped to $h$.

It remains to prove that the functor is essentially surjective.
It is enough to show that a banal $G$-Breuil--Kisin module
$\mathcal{P}$ of type $\mu$
over $(A, I)$, such that $\mathcal{P}=G_{A}$ and $F_\mathcal{P}$
corresponds to an element $Y \in G(A)\mu(d)G(A),
$
is isomorphic to $\mathcal{P}_{X_d}$ for some $X \in G(A)_I$.
After changing the trivialization $\mathcal{P}=G_{A}$, we may assume that $Y \in \mu(d)G(A)$.
Then the result is clear.
\end{proof}

\begin{cor}\label{Corollary:G-BK module etale banal}
    The fibered category over $(\O)^{\op}_{\Prism, \O_E}$ which associates to each $(A, I) \in (\O)_{\Prism, \O_E}$ the groupoid
    $G\mathchar`-\mathrm{BK}_\mu(A, I)$
    of $G$-Breuil--Kisin modules of type $\mu$ over $(A, I)$ is a stack with respect to the flat topology.
\end{cor}

\begin{proof}
    This follows from
    Proposition \ref{Proposition:flat descent of G display}
    and
    Proposition \ref{Proposition:G-displays and G-BK modules}.
\end{proof}

\begin{ex}\label{Example:GLn displays}
    Assume that $G=\GL_n$.
    We retain the notation of Example \ref{Example:GLn case Q_BK}.
    Let
    $M$
    be a Breuil--Kisin module of type $\mu$ over $(A, I)$.
    Since $M$ is of type $\mu$, it follows from Lemma \ref{Lemma:Brauil-Kisin module, filtration and height} that $F_M$ restricts to an isomorphism
    $\widetilde{M} \overset{\sim}{\to} M$.
    The base change
    $\phi^*(F_M) \colon \phi^*\widetilde{M} \overset{\sim}{\to} \phi^*M$
    induces an isomorphism
    \[
    \alpha' \colon \phi^*(\mathcal{Q}(M)_{\mathrm{BK}}) \overset{\sim}{\to} \mathcal{Q}(M)_A
    \]
    of $\GL_{n, A}$-torsors over $\Spec A$.
    The $(\GL_n)_{\mu, A, I}$-torsor $\mathcal{Q}(M)$ with $\alpha'$ is a $\GL_n$-$\mu$-display over $(A, I)$.
    
    By construction, the $\GL_n$-Breuil--Kisin module
    $\mathcal{Q}(M)_{\mathrm{BK}}$
    agrees with the one $\mathcal{P}(\widetilde{M})$ associated with the Breuil--Kisin module $(\widetilde{M}, F_{\widetilde{M}})$ where
    the isomorphism $F_{\widetilde{M}}$ is
    \[
    (\phi^*\widetilde{M})[1/I] \overset{\phi^*(F_M)}{\longrightarrow} (\phi^*M)[1/I]=\widetilde{M}[1/I].
    \]
    (See $\mathcal{P}(\widetilde{M})$ for Example \ref{Example:GLn BK module}.)
    We note that $F_M \colon \widetilde{M} \overset{\sim}{\to} M$ is an isomorphism of Breuil--Kisin modules.
    Since $\mathcal{Q}(M)_{\mathrm{BK}}$ is of type $\mu$, it follows that $\mathcal{P}(M)$ is of type $\mu$.
\end{ex}

\begin{cor}\label{Corollary:GLn BK module of type mu}
Let $(A, I)$ be a bounded $\O_E$-prism over $\O$.
We have equivalences of groupoids
    \begin{align*}
        \mathrm{BK}_\mu(A, I)^{\simeq} &\overset{\sim}{\to} \GL_n\mathchar`-\mathrm{BK}_\mu(A, I), \quad M \mapsto \mathcal{P}(M), \\
        \mathrm{BK}_\mu(A, I)^{\simeq} &\overset{\sim}{\to} \GL_n\mathchar`-\mathrm{Disp}_\mu(A, I), \quad M \mapsto \mathcal{Q}(M).
    \end{align*}
Here $\mathrm{BK}_\mu(A, I)^{\simeq}$ is the groupoid of Breuil--Kisin modules of type $\mu$ over $(A, I)$.
\end{cor}

\begin{proof}
    The first equivalence follows from Example \ref{Example:GLn BK module}, Example \ref{Example:G-BK of type mu implies BK of type mu}, and Example \ref{Example:GLn displays}.
    We shall prove that the functor $M \mapsto \mathcal{Q}(M)$ is an equivalence.
    It follows from Example \ref{Example:GLn displays} that the composition of this functor with the functor (\ref{equation:functor from G-displays to G-BK}) is isomorphic to the functor 
$M \mapsto \mathcal{P}(M)$.
Since (\ref{equation:functor from G-displays to G-BK}) is an equivalence by Proposition \ref{Proposition:G-displays and G-BK modules}, the result follows.
\end{proof}

\subsection{Hodge filtrations}\label{Subsection:Hodge filtrations}

We define the Hodge filtrations for $G$-$\mu$-displays, following \cite[Section 7.4]{Lau21}.
Let
$(A, I)$
be a bounded $\O_E$-prism over $\O$.
We recall the commutative diagram (\ref{equation:diagram of group sheaves}) from Section \ref{Subsection:Display groups on prismatic sites}.

\begin{defn}[Hodge filtration]\label{Definition:Hodge filtration of G-displays}
Let
$\mathcal{Q}$ be a $G$-$\mu$-display over $(A, I)$.
We write
\[
\mathcal{Q}_{A/I}:=\mathcal{Q}^{\overline{\tau}} \quad (\text{resp.}\ P(\mathcal{Q})_{A/I}:=\mathcal{Q}^{\overline{\tau}_P})
\]
for the pushout of 
the underlying $G_{\mu, A, I}$-torsor $\mathcal{Q}$ on $(A, I)^{\op}_\et$ along $\overline{\tau}$ (resp.\ $\overline{\tau}_P$), which is a $G_{\overline{\Prism}, A}$-torsor
(resp.\ a $(P_{\mu})_{\overline{\Prism}, A}$-torsor) on $(A, I)^{\op}_\et$.
There is a natural $(P_{\mu})_{\overline{\Prism}, A}$-equivariant injection
\[
P(\mathcal{Q})_{A/I} \hookrightarrow \mathcal{Q}_{A/I}.
\]
We call $P(\mathcal{Q})_{A/I}$
(or the injection $P(\mathcal{Q})_{A/I} \hookrightarrow \mathcal{Q}_{A/I}$)
the \textit{Hodge filtration} of $\mathcal{Q}_{A/I}$.
If there is no risk of confusion,
we also say that
$P(\mathcal{Q})_{A/I}$
is the Hodge filtration of $\mathcal{Q}$.
\end{defn}

\begin{ex}\label{Example:Hodge filtration of GLn-display}
Assume that $G=\GL_n$ and let the notation be as in Example \ref{Example:GLn case Q_BK}.
Let
$M$
be a Breuil--Kisin module over $(A, I)$ of type $\mu$
and
let
$\mathcal{Q}=\mathcal{Q}(M)$
be
the associated $\GL_n$-$\mu$-display over $(A, I)$ given in Example \ref{Example:GLn displays}.
Recall that the filtration $\{ \Fil^i(\phi^*M) \}_{i \in \Z}$
defines
the Hodge filtration $\{ P^i \}_{i \in \Z}$ of
$M_{\dR}=(\phi^*M)/I(\phi^*M)$.
Similarly, the filtration $\{ \Fil^i_\mu \}_{i \in \Z}$ of $A^n$ induces a filtration
of $(A/I)^n$.
Let $\underline{\mathrm{Isom}}((A/I)^n, M_{\dR})$
(resp.\ $\underline{\mathrm{Isom}}_{\Fil}((A/I)^n, M_{\dR})$)
be the functor sending $B \in (A, I)_\et$ to the set of isomorphisms
$(B/IB)^n \overset{\sim}{\to} (M_{\dR})_{B/IB}$
(resp.\ the set of isomorphisms $(B/IB)^n \overset{\sim}{\to} (M_{\dR})_{B/IB}$ preserving the filtrations).
Since
$M$ is of type $\mu$,
we see that
$\underline{\mathrm{Isom}}_{\Fil}((A/I)^n, M_{\dR})$ is naturally a $(P_{\mu})_{\overline{\Prism}, A}$-torsor.
It follows that the natural morphism
\[
\mathcal{Q}= \underline{\mathrm{Isom}}_{\Fil}(A^n, \phi^*M) \to \underline{\mathrm{Isom}}_{\Fil}((A/I)^n, M_{\dR})
\]
induces an isomorphism
\[
P(\mathcal{Q})_{A/I} \overset{\sim}{\to} \underline{\mathrm{Isom}}_{\Fil}((A/I)^n, M_{\dR}).
\]
%(Note that a morphism between two $(P_{\mu})_{\overline{\Prism}, A}$-torsors is automatically an isomorphism.)
Similarly, we obtain
$
\mathcal{Q}_{A/I} \overset{\sim}{\to} \underline{\mathrm{Isom}}((A/I)^n, M_{\dR}).
$
\end{ex}

\begin{rem}\label{Remark:Hodge filtration schematic}
Let
$\mathcal{Q}$ be a $G$-$\mu$-display over $(A, I)$.
By Proposition \ref{Proposition:equivalences of pi completely etale torsors}, the $G_{\overline{\Prism}, A}$-torsor
$\mathcal{Q}_{A/I}$
(resp.\ the $(P_{\mu})_{\overline{\Prism}, A}$-torsor $P(\mathcal{Q})_{A/I}$)
corresponds to a
$G_{A/I}$-torsor
(resp.\ a $(P_\mu)_{A/I}$-torsor) over $\Spec A/I$, which will be denoted by the same symbol.
\end{rem}

\begin{ex}\label{Example:Hodge filtration banal case}
Assume that $(A, I)$ is orientable.
Let $X \in G(A)_I$ be an element.
Then the Hodge filtration associated with $\mathcal{Q}_X$
can be identified with the natural inclusion
$(P_{\mu})_{A/I} \hookrightarrow G_{A/I}$.
\end{ex}

\begin{prop}\label{Proposition:G display with trivial Hodge filtration is banal}
A $G$-$\mu$-display
$\mathcal{Q}$
over $(A, I)$ is banal if and only if the Hodge filtration
$P(\mathcal{Q})_{A/I}$
is a trivial $(P_{\mu})_{A/I}$-torsor over $\Spec A/I$.
\end{prop}

\begin{proof}
This is a restatement of Corollary \ref{Corollary:trivial Gmu torsor} in the current context.
\end{proof}

\subsection{Underlying $G$-$\phi$-modules}\label{Subsection:phi G torsor of type mu}

Let $(A, I)$ be a bounded $\O_E$-prism over $\O$ and let $(M, F_M)$ be a Breuil--Kisin module over $(A, I)$.
Since $\{ \Fil^i(\phi^*M) \}_{i \in \Z}$ is the filtration of $\phi^*M$, it is sometimes reasonable to consider $\phi^*M$ (rather than $M$) as ``the underlying $A$-module'' of the Breuil--Kisin module $(M, F_M)$.
The same applies to $G$-Breuil--Kisin modules $\mathcal{P}$ over $(A, I)$.
In fact, the Frobenius of $\phi^*\mathcal{P}$ will also be important.
For example, this can be observed in the Grothendieck--Messing deformation theory studied in \cite{Ito-K23-b}.

It will be convenient to make the following definition.
We assume that $(A, I)$ is orientable for simplicity.
We set $A[1/\phi(I)]:=A[1/\phi(d)]$ for a generator $d \in I$, which does not depend on the choice of $d$.

\begin{defn}\label{Definition:phi-G-torsor}
A \textit{$G$-$\phi$-module} over $(A, I)$
is a pair
$(\mathcal{P}, \phi_\mathcal{P})$ consisting of
a $G_A$-torsor $\mathcal{P}$ over $\Spec A$ and
an isomorphism
\[
\phi_\mathcal{P} \colon (\phi^*\mathcal{P})[1/\phi(I)] \overset{\sim}{\to} \mathcal{P}[1/\phi(I)]
\]
of $G_{A[1/\phi(I)]}$-torsors over $\Spec A[1/\phi(I)]$.
(Here $\mathcal{P}[1/\phi(I)]:=\mathcal{P} \times_{\Spec A} \Spec A[1/\phi(I)]$.)
If there is no possibility of confusion, we write $\mathcal{P}=(\mathcal{P}, \phi_\mathcal{P})$.
\end{defn}

Here we explain how to attach a $G$-$\phi$-module over $(A, I)$ to a $G$-$\mu$-display $\mathcal{Q}$ over $(A, I)$.
Recall
$
\mathcal{Q}_A:=\mathcal{Q}^\tau
$
from Remark \ref{Remark:isomorphism between Q_BK and Q_A}, which we regard as a $G_A$-torsor over $\Spec A$.
We define
\[
\phi_{\mathcal{Q}_A} \colon (\phi^*(\mathcal{Q}_A))[1/\phi(I)] \overset{\sim}{\to} \mathcal{Q}_A[1/\phi(I)]
\]
as the composition
\[
(\phi^*(\mathcal{Q}_A))[1/\phi(I)] \overset{\sim}{\to} (\phi^*(\mathcal{Q}_{\mathrm{BK}}))[1/\phi(I)] \overset{\sim}{\to} \mathcal{Q}_A[1/\phi(I)]
\]
where the first isomorphism is the base change of $\mathcal{Q}_A[1/I] \overset{\sim}{\to} \mathcal{Q}_{\mathrm{BK}}[1/I]$
given in Remark \ref{Remark:isomorphism between Q_BK and Q_A} along $\phi \colon A[1/I] \to A[1/\phi(I)]$, and the second one is the base change of
$(\alpha_\mathcal{Q})' \colon \phi^*(\mathcal{Q}_{\mathrm{BK}}) \overset{\sim}{\to} \mathcal{Q}_A$
given in Proposition \ref{Proposition:alternative description of G-mu-displays} along the natural homomorphism $A \to A[1/\phi(I)]$.

\begin{defn}[{Underlying $G$-$\phi$-module}]\label{Definition:underlying phi-G-torsor}
Let
$\mathcal{Q}$
be a $G$-$\mu$-display over $(A, I)$.
The $G$-$\phi$-module
\[
\mathcal{Q}_{\phi}:=(\mathcal{Q}_A, \phi_{\mathcal{Q}_A})
\]
over $(A, I)$
is called the \textit{underlying $G$-$\phi$-module} of $\mathcal{Q}$.
\end{defn}

\begin{ex}\label{Example:underlying G-module banal case}
    Let $\mathcal{Q}_X$ be the banal $G$-$\mu$-display associated with an element $X \in G(A)_I$.
    The underlying $G$-$\phi$-module $(\mathcal{Q}_X)_\phi$ of $\mathcal{Q}_X$ is the trivial $G_A$-torsor $G_A$ with the isomorphism
    \[
    (\phi^*G_A)[1/\phi(I)]=G_A[1/\phi(I)] \overset{\sim}{\to} G_A[1/\phi(I)], \quad g \mapsto X_d\phi(\mu(d)) g
    \]
    for a generator $d \in I$.
    We note that the element
    $
    X_d\phi(\mu(d)) \in G(A[1/\phi(I)])
    $
    is independent of the choice of $d \in I$.
\end{ex}

\begin{rem}\label{Remark:underlying phi G torsor and underlying G-BK module}
    Let $\mathcal{Q}$ be a $G$-$\mu$-display over $(A, I)$.
    The base change $\phi^*(\mathcal{Q}_{\mathrm{BK}})$ of the associated $G$-Breuil--Kisin module $\mathcal{Q}_{\mathrm{BK}}$ is naturally a $G$-$\phi$-module over $(A, I)$.
    We note that $(\alpha_\mathcal{Q})'$ gives an isomorphism
$\phi^*(\mathcal{Q}_{\mathrm{BK}}) \overset{\sim}{\to} \mathcal{Q}_\phi$
of $G$-$\phi$-modules.
Therefore, one can also define the underlying $G$-$\phi$-module of $\mathcal{Q}$ as $\phi^*(\mathcal{Q}_{\mathrm{BK}})$.
However, the construction of $\mathcal{Q}_\phi$ is more natural and will be useful in \cite{Ito-K23-b}.
\end{rem}

\subsection{$G$-$\mu$-displays for perfectoid rings}\label{Subsection:G displays for perfectoid rings}

Let $R$ be a perfectoid ring over $\O$.
We discuss $p$-complete $\arc$-descent results for $G$-$\mu$-displays over the $\O_E$-prism
$
(W_{\O_E}(R^\flat), I_R).
$

\begin{rem}\label{Remark:G-BK of type mu over perdectoid}
Assume that $\O_E=\Z_p$.
In \cite{Bartling}, the notion of $G$-Breuil--Kisin modules over $(W(R^\flat), I_R)$ \textit{of type $\mu$} was introduced in a different way.
Namely, in \cite{Bartling},
a $G$-Breuil--Kisin module $\mathcal{P}$ over $(W(R^\flat), I_R)$ is said to be of type $\mu$ if for any homomorphism $R \to V$ with $V$ a
$p$-adically complete valuation ring of rank $\leq 1$
whose fraction field is algebraically closed,
the base change $\mathcal{P}_{W(V^\flat)}$
is of type $\mu$ in the sense of Definition \ref{Definition:G-BK module of type mu}.
%(More precisely, in \cite{Bartling}, Bartling works with $G$-$\phi$-modules over $(W_{\O_E}(R^\flat), I_R)$ instead of $G$-Breuil--Kisin modules.)
In Proposition \ref{Proposition:comparison of notions of type mu} below, we will prove that this notion agrees with the one introduced in Definition \ref{Definition:G-BK module of type mu}.
\end{rem}

Let
$\Perfd_{R}$
be the category
of perfectoid rings over $R$.
We endow $\Perfd^{\op}_{R}$ with
the topology generated by the \textit{$\pi$-complete $\arc$-coverings} (or equivalently, the $p$-complete $\arc$-coverings) in the sense of \cite[Section 2.2.1]{CS}.
This topology is called the \textit{$\pi$-complete $\arc$-topology}.

\begin{rem}\label{Remark:arc topology}
We quickly review the notion of a $\pi$-complete $\arc$-covering.
\begin{enumerate}
    \item We say that a homomorphism
$R \to S$
of perfectoid rings over $\O$
is a $\pi$-complete $\arc$-covering if
for any homomorphism $R \to V$ with $V$ a $\pi$-adically complete valuation ring of rank $\leq 1$, there exist an extension $V \hookrightarrow W$ of $\pi$-adically complete valuation rings of rank $\leq 1$ and a homomorphism $S \to W$ such that the following diagram commutes:
\[
\xymatrix{
R \ar^-{}[r]  \ar[d]_-{} & S \ar[d]_-{}  \\
V \ar[r]^-{} & W.
}
\]
    \item The category $\Perfd^{\op}_R$ admits fiber products; a colimit of the diagram
$
S_2 \leftarrow S_1 \rightarrow S_3
$
in $\Perfd_R$
is given by the $\pi$-adic completion of $S_2 \otimes_{S_1} S_3$ (cf.\ \cite[Proposition 2.1.11]{CS}).
We see that $\Perfd^{\op}_R$ is indeed a site.
    \item
    Let $R \to S$ be a $\pi$-completely \'etale covering.
    Then $S$ is perfectoid as explained in 
    Example \ref{Example:perfectoid ring etale morphism},
    and $R \to S$ is a $\pi$-complete $\arc$-covering; see \cite[Section 2.2.1]{CS}.
    \item There exists a $\pi$-complete $\arc$-covering of the form $R \to \prod_{i \in I} V_i$ where $V_i$ are $\pi$-adically complete valuation rings of rank $\leq 1$ with algebraically closed fraction fields; see \cite[Lemma 2.2.3]{CS}.
\end{enumerate}
\end{rem}

We recall the following result from \cite{Ito-K21}.

\begin{prop}[{\cite[Corollary 4.2]{Ito-K21}}]\label{Proposition:arc descent for finite projective modules}
The fibered category over $\Perfd^{\op}_R$
which associates to a perfectoid ring $S$ over $R$
the category of finite projective $S$-modules satisfies
descent with respect to the $\pi$-complete $\arc$-topology.
In particular, the functor
$
\Perfd_R \to \mathrm{Set}$,
$S \mapsto S$
forms a sheaf.
\end{prop}

\begin{proof}
See \cite[Corollary 4.2]{Ito-K21}.
The second assertion was previously proved in \cite[Proposition 8.10]{BS}.
\end{proof}

\begin{rem}\label{Remark:arc descent for finite projective modules}
In fact, it is proved in \cite[Theorem 1.2]{Ito-K21} that
the functor on $\Perfd_R$ associating to each $S \in \Perfd_R$ the $\infty$-category $\Perf(S)$ of perfect complexes over $S$ satisfies $\pi$-complete $\arc$-hyperdescent.
    Using this,
    we can prove that for any integer $n \geq 1$,
    the functor
    $S \mapsto \Perf(W_{\O_E}(S^\flat)/I^n_S)$
    on $\Perfd_R$
    satisfies $\pi$-complete $\arc$-hyperdescent, by induction on $n$.
    This implies that the functor
    $
    S \mapsto \Perf(W_{\O_E}(S^\flat))
    $
    satisfies $\pi$-complete $\arc$-hyperdescent as well.
    See the discussion in \cite[Section 4.1]{Ito-K21}.
\end{rem}

%As a corollary of these results, we have:

\begin{cor}\label{Corollary:descent witt vectors}
The fibered category over $\Perfd^{\op}_R$
which associates to a perfectoid ring $S$ over $R$
the category of finite projective $W_{\O_E}(S^\flat)$-modules
satisfies
descent with respect to the $\pi$-complete $\arc$-topology.
The same holds for finite projective $W_{\O_E}(S^\flat)/I^n_S$-modules.
\end{cor}

\begin{proof}
By the same argument as in the proof of \cite[Corollary 4.2]{Ito-K21},
we can deduce the assertion from Remark \ref{Remark:arc descent for finite projective modules}.
\end{proof}

In particular, the functor
$\Perfd_R \to \mathrm{Set}$, $S \mapsto W_{\O_E}(S^\flat)$
forms a sheaf.
This fact also follows from \cite[Lemma 4.2.6]{CS} or the proof of \cite[Proposition 8.10]{BS} (using that $W(\F_q) \to \O_E$ is flat).

\begin{rem}\label{Remark:alternative proof of arc descent for finite projective modules} 
    In the case where $\O_E=\Z_p$, the first assertion of Corollary \ref{Corollary:descent witt vectors} is proved in \cite[Corollary 4.2]{Ito-K21}.
    The general case can also be deduced from this special case, using that a module over $W_{\O_E}(S^\flat)$ is finite projective if and only if it is finite projective over $W(S^\flat)$.
\end{rem}

For an affine scheme $X$ over $\O$ (or $R$),
we define a functor
$
X_{\overline{\Prism}} \colon \Perfd_R \to \mathrm{Set},
$
$S \mapsto X(S).
$
By Proposition \ref{Proposition:arc descent for finite projective modules}, this forms a sheaf.
Similarly, 
for an affine scheme $X$ over $\O$ (or $W_{\O_E}(R^\flat)$),
we define a functor
$
X_{\Prism} \colon \Perfd_R \to \mathrm{Set},
$
$S \mapsto X(W_{\O_E}(S^\flat))$,
which forms a sheaf by Corollary \ref{Corollary:descent witt vectors}.
We have the following analogue of Proposition \ref{Proposition:equivalences of pi completely etale torsors}.

\begin{prop}\label{Proposition:equivalence of arc torsors}
Let $H$ be a smooth affine group scheme over $\O$.
\begin{enumerate}
    \item The functor $\mathcal{P} \mapsto \mathcal{P}_{\overline{\Prism}}$ from the groupoid of $H_R$-torsors over $\Spec R$ to the groupoid of $H_{\overline{\Prism}}$-torsors on $\Perfd^{\op}_R$ is an equivalence.
    \item The functor
    $\mathcal{P} \mapsto \mathcal{P}_{\Prism}$
    from the groupoid of $H_{W_{\O_E}(R^\flat)}$-torsors over $\Spec W_{\O_E}(R^\flat)$ to
the groupoid of $H_{\Prism}$-torsors on $\Perfd^{\op}_R$ is an equivalence.
\end{enumerate}
\end{prop}

\begin{proof}
This can be proved by the same argument as in the proof of Proposition \ref{Proposition:equivalences of pi completely etale torsors}, using Proposition \ref{Proposition:arc descent for finite projective modules} and Corollary \ref{Corollary:descent witt vectors}.
\end{proof}

\begin{rem}\label{Remark:arc descent for G-BK}
    Arguing as in Remark \ref{Remark:flat descent for G-BK modules},
    we see that the fibered category over $\Perfd^{\op}_R$ which associates to each $S \in \Perfd_R$ the groupoid of $G$-Breuil--Kisin modules over $(W_{\O_E}(S^\flat), I_S)$ is a stack with respect to the $\pi$-complete $\arc$-topology.
\end{rem}

As in Section \ref{Subsection:prismatic G-display},
the following functors
\begin{align*}
    G_{\mu, I} &\colon \Perfd_R \to \mathrm{Set}, \quad S \mapsto G_\mu(W_{\O_E}(S^\flat), I_S), \\
    G_{\Prism, I} &\colon \Perfd_R \to \mathrm{Set}, \quad S \mapsto G(W_{\O_E}(S^\flat))_{I_S}
\end{align*}
form sheaves, and the group sheaf $G_{\mu, I}$ acts on $G_{\Prism, I}$.

\begin{lem}\label{Lemma:Gmu arc torsor is etale torsor}
Let $\mathcal{Q}$ be a $G_{\mu, I}$-torsor with respect to the $\pi$-complete $\arc$-topology.
Then $\mathcal{Q}$ is trivialized by a $\pi$-completely \'etale covering $R \to S$.
\end{lem}

\begin{proof}
We claim that if the pushout of $\mathcal{Q}$ along the homomorphism
$G_{\mu, I} \to (P_\mu)_{\overline{\Prism}}$
is trivial as a $(P_\mu)_{\overline{\Prism}}$-torsor, then $\mathcal{Q}$ is itself trivial.
Indeed, one can prove the analogue of Lemma \ref{Lemma:Gmu torsor successive quotient} for $G_{\mu, I}$, and then the argument as in the proof of Corollary \ref{Corollary:trivial Gmu torsor} works.

By the claim, it suffices to prove that any $(P_\mu)_{\overline{\Prism}}$-torsor
with respect to the $\pi$-complete $\arc$-topology
can be trivialized by a $\pi$-completely \'etale covering $R \to S$.
This is a consequence of Proposition \ref{Proposition:equivalence of arc torsors}.
\end{proof}

\begin{cor}\label{Corollary:arc descent for G-displays}
The fibered category over $\Perfd^{\op}_R$
which associates to a perfectoid ring $S$ over $R$
the groupoid of $G$-$\mu$-displays over $(W_{\O_E}(S^\flat), I_S)$ is a stack with respect to the $\pi$-complete $\arc$-topology.
The same holds for $G$-Breuil--Kisin modules of type $\mu$ over $(W_{\O_E}(S^\flat), I_S)$.
\end{cor}

\begin{proof}
The first assertion can be deduced from Lemma \ref{Lemma:Gmu arc torsor is etale torsor} by the same argument as in the proof of Proposition \ref{Proposition:flat descent of G display}.
The second assertion follows from the first one, together with Proposition \ref{Proposition:G-displays and G-BK modules}.
\end{proof}

Now we are ready to prove the following result:

\begin{prop}\label{Proposition:comparison of notions of type mu}
For
a $G$-Breuil--Kisin module $\mathcal{P}$ over $(W_{\O_E}(R^\flat), I_R)$, the following conditions are equivalent.
\begin{enumerate}
    \item $\mathcal{P}$ is of type $\mu$ (in the sense of Definition \ref{Definition:G-BK module of type mu}).
    \item There exists a $\pi$-complete $\arc$-covering $R \to S$ such that the base change of
    $\mathcal{P}$
    along $(W_{\O_E}(R^\flat), I_R) \to (W_{\O_E}(S^\flat), I_S)$
    is of type $\mu$.
    \item For any homomorphism $R \to V$ with $V$ a
$\pi$-adically complete valuation ring of rank $\leq 1$
whose fraction field is algebraically closed,
the base change
of $\mathcal{P}$ along
$(W_{\O_E}(R^\flat), I_R) \to (W_{\O_E}(V^\flat), I_V)$
is of type $\mu$.
\end{enumerate}
\end{prop}

\begin{proof}
It is clear that (1) implies (2) and (3).
By Remark \ref{Remark:arc descent for G-BK} and Corollary \ref{Corollary:arc descent for G-displays}, we see that (2) implies (1).

Assume that the condition (3) is satisfied.
We want to show that this implies (2), which will conclude the proof of the proposition.
By Remark \ref{Remark:arc topology} (4),
there exists a $\pi$-complete $\arc$-covering 
$R \to S=\prod_{i} V_i$ where $V_i$ are $\pi$-adically complete valuation rings of rank $\leq 1$ with algebraically closed fraction fields.
Since $W_{\O_E}(V^\flat_i)$ is strictly henselian, 
the base change
$\mathcal{P}_{W_{\O_E}(V^\flat_i)}$
is a trivial $G_{W_{\O_E}(V^\flat_i)}$-torsor.
Since $\mathcal{P}_{W_{\O_E}(S^\flat)}$ is affine
and $W_{\O_E}(S^\flat)=\prod_{i} W_{\O_E}(V^\flat_i)$, it follows that $\mathcal{P}_{W_{\O_E}(S^\flat)}$ has a $W_{\O_E}(S^\flat)$-valued point, and hence is a trivial $G_{W_{\O_E}(S^\flat)}$-torsor.
We fix a trivialization
$\mathcal{P}_{W_{\O_E}(S^\flat)} \simeq G_{W_{\O_E}(S^\flat)}$.
Let $\xi \in I_S$ be a generator.
The condition (3) implies that
for each $i$,
the base change of $F_\mathcal{P}$
along $W_{\O_E}(R^\flat) \to W_{\O_E}(V^\flat_i)$
corresponds to
an element of $G(W_{\O_E}(V^\flat_i)[1/\xi])$
which is of the form
$Z_i\mu(\xi)Z'_i$
for some $Z_i, Z'_i \in G(W_{\O_E}(V^\flat_i))$, via the induced trivialization
$\mathcal{P}_{W_{\O_E}(V^\flat_i)} \simeq G_{W_{\O_E}(V^\flat_i)}$.
We set
\[
Z:=(Z_i)_i \in G(W_{\O_E}(S^\flat))=\prod_i G(W_{\O_E}(V^\flat_i))
\]
and similarly let $Z':=(Z'_i)_i \in G(W_{\O_E}(S^\flat))$.
Then
the base change of
$F_\mathcal{P}$
along $W_{\O_E}(R^\flat) \to W_{\O_E}(S^\flat)$
corresponds to the element
$Z\mu(\xi)Z'$.
This means that
the condition (2) is satisfied.
\end{proof}

\subsection{Examples}\label{Subsection:Examples}

We discuss some examples of $G$-$\mu$-displays and $G$-Breuil--Kisin modules of type $\mu$ for certain pairs $(G, \mu)$.

We first discuss a pair $(G, \mu)$ of Hodge type.
Let $G$ be a connected reductive group scheme over $\O_E$ and
$\mu \colon \G_m \to G_{\O}$ a cocharacter.
We assume that there exists a closed immersion
$G \hookrightarrow \GL_n$ over $\O_E$ such that the composition
$
\G_m \to  G_{\O} \to \GL_{n, \O}
$
is conjugate to the cocharacter defined by
\[
t \mapsto \diag{(\underbrace{t, \dotsc, t}_s, \underbrace{1, \dotsc, 1}_{n-s})}
\]
for some $s$.
In particular $\mu$ is 1-bounded.
We set $L:=\O^n_E$.
By \cite[Proposition 1.3.2]{Kisin10}, there exists a finite set of tensors
$\{  s_\alpha \}_{\alpha \in \Lambda} \subset L^{\otimes}$ such that
$G \hookrightarrow \GL_n$ is the pointwise stabilizer of $\{  s_\alpha \}_{\alpha \in \Lambda}$, where $L^{\otimes}$ is the direct sum of all $\O_E$-modules obtained from $L$ by taking tensor products, duals, symmetric powers, and exterior powers.
Let
\[
L_\O = L_{\mu, 1} \oplus L_{\mu, 0}
\]
be the weight decomposition with respect to $\mu$.
(Here the composition
$\G_m \to  G_{\O} \to \GL_{n, \O}$
is also denoted by $\mu$.)

Let $(A, I)$ be a bounded $\O_E$-prism over $\O$.
Let $M$ be a Breuil--Kisin module of type $\mu$ over $(A, I)$.
We note that $M$ is minuscule in the sense of Definition \ref{Definition:displayed and minuscule Breuil-Kisin module}, and that the rank of the Hodge filtration
$P^1 \subset (\phi^*M)/I(\phi^*M)$ is $s$.
For a finite set of tensors
$\{  s_{\alpha, M} \}_{\alpha \in \Lambda} \subset M^{\otimes}$
which are $F_M$-invariant,
we say that the pair $(M, \{  s_{\alpha, M} \}_{\alpha \in \Lambda})$ is
\textit{$G$-$\mu$-adapted} if there exist a $(\pi, I)$-completely \'etale covering $A \to B$ and an isomorphism
$
\psi \colon L_B \overset{\sim}{\to} M_B 
$
such that $\psi$ carries $s_\alpha$ to $s_{\alpha, M}$ for each $\alpha \in \Lambda$ and the reduction modulo $I$ of $\phi^*\psi$ identifies
$(L_{\mu, 1})_{B/IB} \subset L_{B/IB}$
with the Hodge filtration $(P^1)_{B/IB}$.

\begin{prop}\label{Proposition:Hodge type case}
    With the notation above,
    the groupoid $G\mathchar`-\mathrm{Disp}_\mu(A, I)$
    is equivalent to
    the groupoid of $G$-$\mu$-adapted pairs
    $(M, \{  s_{\alpha, M} \}_{\alpha \in \Lambda})$ over $(A, I)$.
\end{prop}

\begin{proof}
    We shall construct a functor
    from the groupoid of $G$-$\mu$-adapted pairs over $(A, I)$ to $G\mathchar`-\mathrm{Disp}_\mu(A, I)$.
    Let
    $(M, \{  s_{\alpha, M} \}_{\alpha \in \Lambda})$
    be a $G$-$\mu$-adapted pair over $(A, I)$.
    Let
    \[
    \mathcal{Q}:= \underline{\mathrm{Isom}}_{\Fil, \{ s_\alpha \}}(L_A, \phi^*M)
    \colon (A, I)_\et \to \mathrm{Set}
    \]
    be the functor sending $B \in (A, I)_\et$ to the set of isomorphisms
    $h \colon L_B \overset{\sim}{\to} (\phi^*M)_B$
    preserving the filtrations and carrying $s_\alpha$ to $1 \otimes s_{\alpha, M}$ for each $\alpha \in \Lambda$.
    Here $L_A$ is equipped with the filtration
    $\{ \Fil^i_\mu \}_{i \in \Z}$ given in Remark \ref{Remark:standard filtration}.
    We claim that $\mathcal{Q}$ is a $G_{\mu, A, I}$-torsor.
    For this, we may assume that there exists an isomorphism
    $\psi \colon L_A \overset{\sim}{\to} M$
    such that $\psi$ carries $s_{\alpha}$ to $s_{\alpha, M}$ for each $\alpha \in \Lambda$ and the reduction modulo $I$ of $h:=\phi^*\psi$ identifies
    $(L_{\mu, 1})_{A/I}$ with $P^1$.
    Under the isomorphism $h \colon L_A \overset{\sim}{\to} \phi^*M$, we have 
    \[
    \{ \Fil^i_\mu \}_{i \in \Z}=\{ \Fil^i(\phi^*M) \}_{i \in \Z},
    \]
    which in turn implies that $h \in \mathcal{Q}(A)$.
    To see this, it suffices to prove that $\Fil^1_\mu=\Fil^1(\phi^*M)$ since $M$ is minuscule.
    We observe that $\Fil^1_\mu$ and $\Fil^1(\phi^*M)$ are the inverse images of $(L_{\mu, 1})_{A/I} \subset L_{A/I}$ and $P^1 \subset (\phi^*M)/I(\phi^*M)$, respectively.
    It then follows that $\Fil^1_\mu=\Fil^1(\phi^*M)$.

    We define $\widetilde{M}:=\Fil^1(\phi^*M) \otimes_A I^{-1}$.
    Since $M$ is of type $\mu$, we see that $F_M$ restricts to an isomorphism
    $\widetilde{M} \overset{\sim}{\to} M$, and we may regard 
    $\{  1 \otimes s_{\alpha, M} \}_{\alpha \in \Lambda}$
    as tensors of $\widetilde{M}$.
    Similarly to Example \ref{Example:GLn case Q_BK},
    we have
    \[
    \mathcal{Q}_{\mathrm{BK}}=\underline{\mathrm{Isom}}_{\{ s_\alpha \}}(L_A, \widetilde{M}),
    \]
    where
    $\underline{\mathrm{Isom}}_{\{ s_\alpha \}}(L_A, \widetilde{M})$
    is the $G_A$-torsor over $\Spec A$ defined by sending an $A$-algebra $B$ to the set of isomorphisms $L_B \overset{\sim}{\to} \widetilde{M}_B$
    carrying $s_\alpha$ to $1 \otimes s_{\alpha, M}$ for each $\alpha \in \Lambda$.
    Moreover, we have
    \[
    \mathcal{Q}_{A}=\underline{\mathrm{Isom}}_{\{ s_\alpha \}}(L_A, \phi^*{M}).
    \]
    The base change
    $\phi^*(F_M) \colon \phi^*\widetilde{M} \overset{\sim}{\to} \phi^*M$
    induces an isomorphism
    $
    \alpha' \colon \phi^*(\mathcal{Q}_{\mathrm{BK}}) \overset{\sim}{\to} \mathcal{Q}_A
    $
    of $G_A$-torsors.
    The $G_{\mu, A, I}$-torsor $\mathcal{Q}$ with $\alpha'$ is a $G$-$\mu$-display over $(A, I)$; see Remark \ref{Remark:alternative description of G-mu-displays}.
    In this way, we obtain a functor
    from the groupoid of $G$-$\mu$-adapted pairs over $(A, I)$ to $G\mathchar`-\mathrm{Disp}_\mu(A, I)$.

    One can prove that this functor is an equivalence in the same way as in the case of $G=\GL_n$; see Section \ref{Subsection:G-mu-displays and G-Breuil--Kisin modules of type mu}.
\end{proof}

\begin{rem}\label{Remark:Hodge type integral model of Shimura variety}
    In \cite[Proposition 1.3.4]{Kisin10}, \cite[Theorem 2.5]{Kim-MadapusiPera}, and \cite{IKY},
    it is observed that $G$-$\mu$-adapted pairs naturally arise from crystalline Galois representations associated with integral canonical models of Shimura varieties of Hodge type with hyperspecial level structure.
    The notion of $G$-$\mu$-adapted pairs plays a central role in the construction of integral canonical models in \cite{Kisin10} and \cite{Kim-MadapusiPera} (see especially \cite[Proposition 1.5.8]{Kisin10} and \cite[Section 3]{Kim-MadapusiPera}).
\end{rem}

We include the following two important examples.
The details will be presented elsewhere.

\begin{ex}[$G$-shtuka]\label{Example:G-shutuka}
    Let $G$ be a connected reductive group scheme over $\O_E$.
    We assume that $k$ is an algebraic closure of $\F_q$ and let
    $\mu \colon \G_m \to G_\O$ be a 1-bounded (or equivalently, minuscule) cocharacter.
    Let $C$ be an algebraically closed nonarchimedean field over $\O[1/\pi]$ with ring of integers $\O_C$.
    We consider the perfectoid space $S= \Spa(C, \O_C)$ and its tilt $S^\flat=\Spa(C^\flat, \O^\flat_C)$.
    We can show that
    the groupoid
    of $G$-$\mu$-displays over
    $(W_{\O_E}(\O^\flat_C), I_{\O_C})$
    is equivalent to the groupoid of \textit{$G$-shtukas} over $S^\flat$ with one leg at $S$ which are bounded by $\mu$ (or bounded by $\mu^{-1}$, depending on the sign convention) introduced by Scholze \cite{Scholze-Weinstein}.
    See \cite[Section 5.1]{Ito-K23-b} for details.
\end{ex}

\begin{ex}[Orthogonal Breuil--Kisin module]\label{example:orthogonal display}
    Let $n=2m$ be an even positive integer and we set $L:=\O^n_E$.
    We define the quadratic form
    \[
    Q \colon L \to \O_E
    \]
    by
    $
    (a_1, \dotsc, a_{2m}) \mapsto \sum^m_{i=1} a_i a_{2m-i+1}.
    $
    The quadratic form $Q$ is perfect in the sense that
    the bilinear form on $L$ defined by $(x, y) \mapsto Q(x+y)-Q(x)-Q(y)$ is perfect.
    Let $G:=\mathrm{O}(Q) \subset \GL_n$ be the orthogonal group of $Q$, which is a smooth affine group scheme over $\O_E$.
    Let $\mu \colon \G_m \to G \subset \GL_n$ be the cocharacter defined by
    \[
    t \mapsto \diag{(t, 1, \dotsc, 1, t^{-1})}.
    \]
    Let $(A, I)$ be a bounded $\O_E$-prism.
    An \textit{orthogonal Breuil--Kisin module of type $\mu$} over $(A, I)$ is a Breuil--Kisin module $M$ of type $\mu$ over $(A, I)$ together with a perfect quadratic form
    $Q_M \colon M \to A$ which is compatible with $F_M$ in the sense that for every $x \in M$, we have
    $
    \phi(Q_M(x))=Q_M(F_M(1 \otimes x))
    $
    in $A[1/I]$.
    Let
    \[
    \mathcal{P}:=\underline{\mathrm{Isom}}_Q(L_A, M)
    \]
    be the $G_A$-torsor over $\Spec A$ defined by sending an $A$-algebra $B$ to the set of isomorphisms $L_B \simeq M_B$ of quadratic spaces.
    One can show that $\mathcal{P}$, together with the isomorphism
$F_\mathcal{P} \colon (\phi^*\mathcal{P})[1/I] \overset{\sim}{\to} \mathcal{P}[1/I]$ induced by $F_M$, forms a $G$-Breuil--Kisin module of type $\mu$ over $(A, I)$.
This construction gives an equivalence
between the groupoid of orthogonal Breuil--Kisin modules of type $\mu$ over $(A, I)$ and the groupoid $G\mathchar`-\mathrm{BK}_\mu(A, I)$.
Thus, by Proposition \ref{Proposition:G-displays and G-BK modules}, the groupoid of orthogonal Breuil--Kisin modules of type $\mu$ over $(A, I)$ is equivalent to the groupoid $G\mathchar`-\mathrm{Disp}_\mu(A, I)$.
The details will be presented in a forthcoming paper. 
    \end{ex}

\begin{rem}\label{Remark:minuscule observation}
   Let the notation be as in Example \ref{example:orthogonal display}.
   Our main result (Theorem \ref{Theorem:main result on G displays over complete regular local rings}) can not be applied to Breuil--Kisin modules of type $\mu$ since $\mu$ is not 1-bounded as a cocharacter of $\GL_n$.
   However, since $\mu$ is 1-bounded as a cocharacter of $G=\mathrm{O}(Q)$, the result can be applied to \textit{orthogonal} Breuil--Kisin modules of type $\mu$.
   Such an observation was made by Lau \cite{Lau21}
   in the context of the deformation theory of K3 surfaces.
\end{rem}

\section{Prismatic $G$-$\mu$-displays over complete regular local rings}\label{Section:G-displays over complete regular local rings}

In this section, we prove the main result (Theorem \ref{Theorem:main result on G displays over complete regular local rings}) of this paper, which we state in Section \ref{Subsection:G displays on absolute prismatic sites}.
The proof will be given in Section \ref{Subsection:proof of main result}.
In Sections \ref{Subsection:Coproducts of Breuil--Kisin prisms}-\ref{Subsection:Deformations of isomorphisms},
we discuss a few technical results that will be used in the proof.

\subsection{$G$-$\mu$-displays on absolute prismatic sites}\label{Subsection:G displays on absolute prismatic sites}

In this paper, we use the following definition.

\begin{defn}\label{Definition:G displays over prismatic sites}
Let $R$ be a $\pi$-adically complete $\O$-algebra.
A \textit{prismatic $G$-$\mu$-display over} $R$ is defined to be an object of the following groupoid
\[
G\mathchar`-\mathrm{Disp}_\mu((R)_{\Prism, \O_E}):= {2-\varprojlim}_{(A, I) \in (R)_{\Prism, \O_E}} G\mathchar`-\mathrm{Disp}_\mu(A, I).
\]
\end{defn}

\begin{rem}\label{Remark:alternative definition of G displays over prismatic sites}
Giving a prismatic $G$-$\mu$-display $\mathfrak{Q}$ over $R$
is equivalent to giving a $G$-$\mu$-display $\mathfrak{Q}_{(A, I)}$ over $(A, I)$ for each $(A, I) \in (R)_{\Prism, \O_E}$
and an isomorphism
\[
\gamma_f \colon f^*(\mathfrak{Q}_{(A, I)}) \overset{\sim}{\to} \mathfrak{Q}_{(A', I')}
\]
for each morphism $f \colon (A, I) \to (A', I')$ in $(R)_{\Prism, \O_E}$,
such that
$\gamma_{f'} \circ ({f'}^*\gamma_f) = \gamma_{f' \circ f}$
(via the natural identification ${f'}^* \circ f^* \simeq (f' \circ f)^*$)
for two morphisms
$f \colon (A, I) \to (A', I')$ and $f' \colon (A', I') \to (A'', I'')$.
We call $\mathfrak{Q}_{(A, I)}$ the \textit{value} of $\mathfrak{Q}$ at $(A, I) \in (R)_{\Prism, \O_E}$.
\end{rem}

Assume that $R$ is a complete regular local ring over $\O$ with residue field $k$.
Let
$
(\O[[t_1, \dotsc, t_n]], (\mathcal{E}))
$
be
an $\O_E$-prism
of Breuil--Kisin type
with an isomorphism
$
R \simeq \O[[t_1, \dotsc, t_n]]/\mathcal{E}
$ over $\O$
(where $n \geq 0$ is the dimension of $R$).
Such an $\O_E$-prism exists; see for example \cite[Section 3.3]{ChengChuangxun}.
We set
$\mathfrak{S}_\O:=\O[[t_1, \dotsc, t_n]]$.
Our goal is to prove the following result.

\begin{thm}\label{Theorem:main result on G displays over complete regular local rings}
Assume that the cocharacter $\mu$ is 1-bounded.
Then the functor
    \[
 G\mathchar`-\mathrm{Disp}_\mu((R)_{\Prism, \O_E}) \to G\mathchar`-\mathrm{Disp}_\mu(\mathfrak{S}_\O, (\mathcal{E})), \quad \mathfrak{Q} \mapsto \mathfrak{Q}_{(\mathfrak{S}_\O, (\mathcal{E}))}
 \]
given by evaluation at $(\mathfrak{S}_\O, (\mathcal{E}))$ is an equivalence.
\end{thm}

The rest of this section is devoted to the proof of Theorem \ref{Theorem:main result on G displays over complete regular local rings}.

\subsection{Coproducts of Breuil--Kisin prisms}\label{Subsection:Coproducts of Breuil--Kisin prisms}

In this subsection, we establish some properties of the object
\[
(\mathfrak{S}_\O, (\mathcal{E})) \in (R)_{\Prism, \O_E}.
\]
%If $\O_E=\Z_p$, $n=1$, and $R$ is $p$-torsion free,
%then all the results in this subsection follow from (the arguments given in) \cite[Section 5.2]{Anschutz-LeBras}.
%In the general case, we need some additional arguments, as we will see below.
We begin with the following result.

\begin{prop}\label{Proposition:weakely initial object}
For any $(A, I) \in (R)_{\Prism, \O_E}$, there exists a flat covering $(A, I) \to (A', I')$ in $(R)_{\Prism, \O_E}$ such that $(A', I')$ admits a morphism
$(\mathfrak{S}_\O, (\mathcal{E})) \to (A', I')$ in $(R)_{\Prism, \O_E}$.
\end{prop}

\begin{proof}
We may assume that $(A, I)$ is orientable by Remark \ref{Remark:flat locally orientable}.
Let $d \in I$ be a generator.
Let $v_1, \dotsc, v_n \in A$ be elements such that each $v_i$ is a lift of the image of
$t_i \in \mathfrak{S}_\O$ under the composition $\mathfrak{S}_\O \to R \to A/I$.
Let $B:=A \otimes_\O \mathfrak{S}_\O$.
We set
\[
x_i:=1\otimes t_i - v_i \otimes 1 \in B.
\]
Then the morphism
$
A/^\L(\pi, d) \to B/^\L(\pi, d, x_1, \dotsc, x_n)
$
of animated rings is faithfully flat.
Indeed, using that the natural homomorphism
$\O[t_1, \dotsc, t_n] \to \mathfrak{S}_\O$ is flat, we see that $A/^\L(\pi, d) \to B/^\L(\pi, d, x_1, \dotsc, x_n)$ is flat.
Since the composition $\mathfrak{S}_\O \to R \to A/(\pi, d)$
induces a homomorphism $B/(\pi, d, x_1, \dotsc, x_n) \to A/(\pi, d)$ over $A/(\pi, d)$,
it follows that 
$
A/^\L(\pi, d) \to B/^\L(\pi, d, x_1, \dotsc, x_n)
$
is faithfully flat.

Let
$\mathfrak{S}_{\O, \infty}$ be the $(\pi, \mathcal{E})$-adic completion of a colimit
$\varinjlim_\phi \mathfrak{S}_\O$
of the diagram
\[
\mathfrak{S}_\O \overset{\phi}{\rightarrow} \mathfrak{S}_\O \overset{\phi}{\rightarrow} \mathfrak{S}_\O  \rightarrow \cdots.
\]
Since $\phi \colon \mathfrak{S}_\O \to \mathfrak{S}_\O$ is faithfully flat, we see that
$\mathfrak{S}_\O \to \mathfrak{S}_{\O, \infty}$ is $(\pi, \mathcal{E})$-completely faithfully flat.
In fact, it is faithfully flat by \cite[Theorem 1.5]{Yekutieli18}.
We set
$B':=B \otimes_{\mathfrak{S}_\O} \mathfrak{S}_{\O, \infty}$.
Then 
$
A/^\L(\pi, d) \to B'/^\L(\pi, d, x_1, \dotsc, x_n)
$
is faithfully flat as well.
Thus, by Proposition \ref{Proposition:prismatic envelope},
we can consider the prismatic envelope
\[
(A', I'):=(B'\{ J/I \}^{\wedge}, IB'\{ J/I \}^{\wedge})
\]
of $B'$ over $(A, I)$ with respect to the ideal
$J:=(d, x_1, \dotsc, x_n) \subset B'$.
The map $(A, I) \to (A', I')$
is a flat covering.

We shall construct a morphism $(\mathfrak{S}_\O, (\mathcal{E})) \to (A', I')$ in $(R)_{\Prism, \O_E}$.
We remark that since $A'/I'$ is not necessarily $\mathfrak{m}$-adically complete for the maximal ideal $\mathfrak{m} \subset R$, it is not clear that the natural homomorphism $\mathfrak{S}_\O \to A'$ induces a morphism $(\mathfrak{S}_\O, (\mathcal{E})) \to (A', I')$ in $(R)_{\Prism, \O_E}$.
Instead, we construct a morphism $(\mathfrak{S}_\O, (\mathcal{E})) \to (A', I')$ as follows.
Since
$\mathfrak{S}_{\O, \infty}$ can be identified with the $(\pi, \mathcal{E})$-adic completion of
\[
\cup_{m \geq 0} \mathfrak{S}_\O[t^{1/q^m}_1, \dotsc, t^{1/q^m}_n],
\]
the quotient
$R_{\infty}:=\mathfrak{S}_{\O, \infty}/\mathcal{E}$ is the $\pi$-adic completion of
$\cup_{m \geq 0} R[\overline{t}^{1/q^m}_1, \dotsc, \overline{t}^{1/q^m}_n]$, where
$\overline{t}_i \in R$ is the image of $t_i$.
Here
\[
\mathfrak{S}_\O[t^{1/q^m}_1, \dotsc, t^{1/q^m}_n]=\mathfrak{S}_\O[X_1, \dotsc, X_n]/(X^{q^m}_1-t_1, \dotsc, X^{q^m}_n-t_n)
\]
and similarly for $R[\overline{t}^{1/q^m}_1, \dotsc, \overline{t}^{1/q^m}_n]$.
The composition $R \to A/I \to A'/I'$ factors through the homomorphism
\[
g \colon R_\infty \to A'/I'
\]
defined by sending $\overline{t}^{1/q^m}_i$ to the image of $1 \otimes t^{1/q^m}_i \in A'$, which is well-defined since $1 \otimes t_i = v_i \otimes 1$ in $A'/I'$.
By Lemma \ref{Lemma:maps from perfectoid prisms to prism}, there exists a unique map
$(\mathfrak{S}_{\O, \infty}, (\mathcal{E})) \to (A', I')$
of bounded $\O_E$-prisms which induces $g$.
By construction, the composition
\[
(\mathfrak{S}_\O, (\mathcal{E})) \to (\mathfrak{S}_{\O, \infty}, (\mathcal{E})) \to (A', I')
\]
is a morphism in $(R)_{\Prism, \O_E}$.
\end{proof}

\begin{rem}\label{Remark:weakely initial object}
    Assume that $\O_E=\Z_p$.
    In this case, Proposition \ref{Proposition:weakely initial object} is proved in 
    \cite[Example 7.13]{BS} and \cite[Lemma 5.14]{Anschutz-LeBras}, using \cite[Proposition 7.11]{BS}.
    Moreover, if $\mathfrak{S}_\O=W(k)[[t]]$ and $\mathcal{E} \in W(k)[[t]]$ is an Eisenstein polynomial, then an alternative argument using prismatic envelopes is given in \cite[Example 2.6]{BS2}.
    Our argument is similar to the one given in \cite[Example 2.6]{BS2},
    but we have to modify it slightly in order to treat the case where $A/I$ is not $\mathfrak{m}$-adically complete.
\end{rem}

In a similar way, we obtain the following result:

\begin{lem}\label{Lemma:two morphisms from Breuil-Kisin coincide}
    Let $(A, I) \in (R)_{\Prism, \O_E}$ and let
    $f_1, f_2 \colon (\mathfrak{S}_\O, (\mathcal{E})) \to (A, I)$
    be two morphisms in $(R)_{\Prism, \O_E}$.
    If $f_1(t_i)=f_2(t_i)$ in $A$ for any $i$, then we have $f=g$.
\end{lem}

\begin{proof}
    As in the proof of Proposition \ref{Proposition:weakely initial object}, let $\mathfrak{S}_{\O, \infty}$ be the $(\pi, \mathcal{E})$-adic completion of 
    $\varinjlim_\phi \mathfrak{S}_\O$, which is faithfully flat over $\mathfrak{S}_\O$.
    After replacing $(A, I)$ by a flat covering, we may assume that
    $f_1 \colon (\mathfrak{S}_\O, (\mathcal{E})) \to (A, I)$
    factors through a map $\widetilde{f}_1 \colon (\mathfrak{S}_{\O, \infty}, (\mathcal{E})) \to (A, I)$ of $\O_E$-prisms.
    Since
    $\mathfrak{S}_{\O, \infty}$ is the $(\pi, \mathcal{E})$-adic completion of
    $
    \cup_{m \geq 0} \mathfrak{S}_\O[t^{1/q^m}_1, \dotsc, t^{1/q^m}_n],
    $
    there exists a map
    $\widetilde{f}_2 \colon \mathfrak{S}_{\O, \infty} \to A$
    extending $f_2$ such that
    $\widetilde{f}_2(t^{1/q^m}_i)=\widetilde{f}_1(t^{1/q^m}_i)$
    for all $m$ and $i$.
    The map $\widetilde{f}_2$ preserves the $\delta_E$-structures by Corollary \ref{Corollary:homomorphism from perfect delta pi-ring}.
    It suffices to prove that $\widetilde{f}_1=\widetilde{f}_2$.
    We note that both $f_1$ and $f_2$ induce the same homomorphism
    $R \to A/I$.
    Since $R_{\infty}=\mathfrak{S}_{\O, \infty}/\mathcal{E}$ is the $\pi$-adic completion of
    $\cup_{m \geq 0} R[\overline{t}^{1/q^m}_1, \dotsc, \overline{t}^{1/q^m}_n]$,
    it follows that
    the homomorphism $R_{\infty} \to A/I$ induced by $\widetilde{f}_1$ agrees with the one induced by $\widetilde{f}_2$.
    Then, by Lemma \ref{Lemma:maps from perfectoid prisms to prism}, we conclude that $\widetilde{f}_1=\widetilde{f}_2$.
\end{proof}

We next study a coproduct of two copies of $(\mathfrak{S}_\O, (\mathcal{E}))$
in the category $(R)_{\Prism, \O_E}$.
To simplify the notation, we write
\[
(A, I):=(\mathfrak{S}_\O, (\mathcal{E}))
\]
in the rest of this section.
We set
\[
B:=A[[x_1, \dotsc, x_n]]
\]
and let $p'_1 \colon A \to B$ be the natural homomorphism.
There exists a unique $\delta_E$-structure on $B$ such that $p'_1$ is a homomorphism of $\delta_E$-rings and the associated Frobenius $\phi \colon B \to B$ sends $x_i$ to $(x_i+t_i)^q-t^q_i$ for every $i$.
We consider the prismatic envelope
\[
(A^{(2)}, I^{(2)})
\]
of $B$ over $(A, I)$ with respect to the ideal $(\mathcal{E}, x_1, \dotsc, x_n) \subset B$
as in Proposition \ref{Proposition:prismatic envelope}.
Let
$p_1 \colon (A, I) \to (A^{(2)}, I^{(2)})$ denote the natural map.
We view $(A^{(2)}, I^{(2)})$ as an object of $(R)_{\Prism, \O_E}$ via 
the homomorphism $\overline{p}_1 \colon R \to A^{(2)}/I^{(2)}$ induced by $p_1$.

The homomorphism $p'_2 \colon A \to B$ over $\O$ defined by $t_i \mapsto x_i+t_i$ is a homomorphism of $\delta_E$-rings.
Let $p_2 \colon A \to A^{(2)}$ be the composition of $p'_2$ with the natural homomorphism $B \to A^{(2)}$.

\begin{lem}\label{Lemma:coproduct and prismatic envelope}
Let the notation be as above.
\begin{enumerate}
    \item We have $p_2(I) \subset I^{(2)}$, and the induced map
    $p_2 \colon (A, I) \to (A^{(2)}, I^{(2)})$ is a morphism in $(R)_{\Prism, \O_E}$.
    \item The object $(A^{(2)}, I^{(2)}) \in (R)_{\Prism, \O_E}$ with the morphisms
    $p_1, p_2 \colon (A, I) \to (A^{(2)}, I^{(2)})$ is a coproduct of two copies of $(A, I)$ in the category $(R)_{\Prism, \O_E}$.
\end{enumerate}
\end{lem}

\begin{proof}
    (1) It suffices to show that the composition of $p_2 \colon A \to A^{(2)}$ with $A^{(2)} \to A^{(2)}/I^{(2)}$ coincides with the composition of $A \to R$ with $\overline{p}_1 \colon R \to A^{(2)}/I^{(2)}$.
    For any $h \in A$, the element $p'_2(h)-p'_1(h) \in B$ is contained in the ideal $(x_1, \dotsc, x_n) \subset B$.
    Since the image of $x_i$ in $A^{(2)}$ is contained in $I^{(2)}$, we have
    $p_2(h)-p_1(h) \in I^{(2)}$, which implies the assertion.

    (2) We have to show that for any $(A', I') \in (R)_{\Prism, \O_E}$ and two morphisms $f_1, f_2 \colon (A, I) \to (A', I')$ in $(R)_{\Prism, \O_E}$, there exists a unique morphism
    \[
    f \colon (A^{(2)}, I^{(2)}) \to (A', I')
    \]
    in $(R)_{\Prism, \O_E}$ such that $f \circ p_1=f_1$ and $f \circ p_2=f_2$.

    We first prove the uniqueness of $f$.
    Let $f' \colon B \to A'$ be the composition of $f$ with $B \to A^{(2)}$.
    Then we have $f' \circ p'_j=f_j$ ($j=1, 2$), and $f'$ sends $x_i=p'_2(t_i)-p'_1(t_i)$ to
    \[
    f_2(t_i)-f_1(t_i) \in I' \subset A'
    \]
    for any $i$.
    Since $A'$ is $I'$-adically complete, such a homomorphism $f'$ of $\delta_E$-rings is uniquely determined (if it exists).
    The uniqueness of $f$ now follows from the universal property of the prismatic envelope $(A^{(2)}, I^{(2)})$.

    We next prove the existence of $f$.
    Since $f_2(t_i)-f_1(t_i) \in I' \subset A'$ and $A'$ is $I'$-adically complete, there exists a unique homomorphism
    $f' \colon B \to A'$ over $\O$ such that $f' \circ p'_1=f_1$ and $f'(x_i)=f_2(t_i)-f_1(t_i)$ for every $i$.
    
    We claim that $f'$ is a homomorphism of $\delta_E$-rings.
    Indeed,
    as in the proof of Proposition \ref{Proposition:weakely initial object}, let $A_\infty$ be the $(\pi, \mathcal{E})$-adic completion of 
    $\varinjlim_\phi A$.
    Then $A_\infty$ is faithfully flat over $A$.
    After replacing $(A', I')$ by a flat covering,
    we may assume that $f_j$ factors through a morphism
    $\widetilde{f}_j \colon (A_\infty, IA_\infty) \to (A', I')$ in $(R)_{\Prism, \O_E}$ for each $j=1, 2$.
    For an integer $m \geq 0$ and $i$, we set
    \[
    x_{i, m}:=\widetilde{f}_2(t^{1/q^m}_i)-\widetilde{f}_1(t^{1/q^m}_i) \in A'.
    \]
    Since we have $x^{q^m}_{i, m} \in (\pi, I')$, it follows that $A'$ is $(x_{1, m}, \dotsc, x_{n, m})$-adically complete for every $m$ (see \cite[Tag 090T]{SP} for example).
    Thus, for each $m \geq 0$, there exists a unique homomorphism
    $f'(m) \colon B \to A'$
    such that $f'(m) \circ p'_1$ is the composition
    \[
    A \to A_\infty \overset{\phi^{-m}}{\to} A_\infty \overset{\widetilde{f}_1}{\to} A'
    \]
    and $f'(m)(x_i)=x_{i, m}$ for any $i$.
    Since $f'(m)=f'(m+1) \circ \phi$, they give rise to a homomorphism
    $
    \widetilde{f}' \colon \varinjlim_\phi B \to A'.
    $
    By Corollary \ref{Corollary:homomorphism from perfect delta pi-ring}, $\widetilde{f}'$ is a homomorphism of $\delta_E$-rings.
    Since $f'$ is the composition $B \to \varinjlim_\phi B \to A'$, we conclude that $f'$ is a homomorphism of $\delta_E$-rings.
    
    By the universal property of the prismatic envelope $(A^{(2)}, I^{(2)})$,
    the homomorphism $f'$ extends to a unique morphism
    $f \colon (A^{(2)}, I^{(2)}) \to (A', I')$ in $(R)_{\Prism, \O_E}$.
    By construction, we have $f \circ p_1=f_1$.
    It follows from Lemma \ref{Lemma:two morphisms from Breuil-Kisin coincide} that $f \circ p_2=f_2$.

    The proof of Lemma \ref{Lemma:coproduct and prismatic envelope} is complete.
\end{proof}

\begin{rem}\label{Remark:coproduct of BK prisms}
    It might be more natural to expect that
    the prismatic envelope $(C, IC)$ of $A \otimes_\O A$
    over $(A, I)$
    with respect to the ideal
    $(\mathcal{E} \otimes 1, t_1 \otimes 1 - 1 \otimes t_1, \dotsc, t_n \otimes 1 - 1 \otimes t_n)$ is a coproduct of two copies of $(A, I)$ in the category $(R)_{\Prism, \O_E}$, where we regard $A \otimes_\O A$ as an $A$-algebra via the homomorphism $a \mapsto a \otimes 1$.
    However, this does not seem to be the case in general.
    For example, it is not clear that the homomorphism
    $A \to C$, $a \mapsto 1 \otimes a$ induces a morphism
    $(A, I) \to (C, IC)$
    in $(R)_{\Prism, \O_E}$ (cf.\ the proof of Proposition \ref{Proposition:weakely initial object}).
\end{rem}

Let
\[
m \colon (A^{(2)}, I^{(2)}) \to (A, I)
\]
be the unique morphism in $(R)_{\Prism, \O_E}$ such that $m \circ p_1 = m \circ p_2= \id_{(A, I)}$.
Let $K$ be the kernel of $m \colon A^{(2)} \to A$.
Let $d \in I^{(2)}$ be a generator.

\begin{lem}[{cf.\ \cite[Lemma 5.15]{Anschutz-LeBras}}]\label{Lemma:kernel of two fold coproduct I}
    We have $\phi(K) \subset  dK$.
\end{lem}

\begin{proof}
    It suffices to show that $\phi(K) \subset  dA^{(2)}$.
    Indeed, let $x \in K$, and we assume that $\phi(x)=dy$ for some $y \in A^{(2)}$.
    Then, since $m(\phi(x))=0$ and $m(d) \in A$ is a nonzerodivisor,
    we have $y \in K$.

    We shall prove that $\phi(K) \subset  dA^{(2)}$.
    We may assume that $d=p_1(\mathcal{E})$.
    The image $p'_1(\mathcal{E}) \in B$ is also denoted by $d$.
    It follows from Proposition \ref{Proposition:prismatic envelope} that $A^{(2)}$ can be identified with the $(\pi, d)$-adic completion of
    \[
    C:=B\{ x_1/d, \dotsc, x_n/d \}.
    \]
    We write $y_i:=x_i/d$.
    The composition $C \to A^{(2)} \to A$ sends $\delta^j_E(y_i)$ to $0$ for any $j \geq 0$ and any $i$.
    Here $\delta^j_E$ is the $j$-th iterate of the map $\delta_E \colon C \to C$.
    Since the kernel of the homomorphism $B \to A$ defined by $x_i \mapsto 0$ ($1 \leq i \leq n$) coincides with $(x_1, \dotsc, x_n)$, it follows that
    the kernel $K_0$ of $C \to A$ is generated by
    \[
    \{ \delta^j_E(y_i) \}_{1 \leq i \leq n, j \geq 0}.
    \]
    We note that $K$ can be identified with the $(\pi, d)$-adic completion of $K_0$.
    We also note that
    $dA^{(2)} = \cap_{l \geq 0}(dA^{(2)} + (\pi, d)^lA^{(2)})$
    since $A^{(2)}/d$ is $\pi$-adically complete (see Remark \ref{Remark:derived and classical complete}).
    It then suffices to show that
    $\phi(\delta^j_E(y_i)) \in dA^{(2)}$ for any $j \geq 0$ and any $i$.
    This can be proved by the same argument as in the first paragraph of the proof of \cite[Lemma 5.15]{Anschutz-LeBras} when $\mathcal{E} \in A$ is not contained in $\pi A$.
    A similar argument holds when $\mathcal{E} \in \pi A$.
    We include the argument in this case for the convenience of the reader.

    We may assume that $\mathcal{E}=\pi$.
    In fact, we prove a more general statement: For any $j \geq 0$, we have
    $\phi^l(\delta^j_E(y_i)) \in \pi^{l} A^{(2)}$ for any $l \geq 1$ and any $i$.
    We proceed by induction on $j$.
    Let $u_i:=x_i+t_i=\pi y_i + t_i \in A^{(2)}$.
    Then we have
    \[
    \phi^l(x_i)=u^{q^l}_i-t^{q^l}_i= \sum_{0 \leq h \leq q^l -1} \dbinom{q^l}{h} (\pi y_i)^{q^l-h} t^{h}_i
    \in \pi^{l+1} A^{(2)}.
    \]
    Thus, we obtain
    $\phi^l(y_i) \in \pi^{l} A^{(2)}$, which proves the assertion in the case where $j=0$.
    Suppose that the assertion holds for some $j \geq 0$.
    Since
    \[
    \pi \phi^l(\delta^{j+1}_E(y_i)) = \phi^l(\pi \delta^{j+1}_E(y_i)) = \phi^l(\phi(\delta^{j}_E(y_i))-\delta^{j}_E(y_i)^{q})
    = \phi^{l+1}(\delta^{j}_E(y_i))-\phi^l(\delta^{j}_E(y_i))^q,
    \]
    the induction hypothesis implies that $\pi \phi^l(\delta^{j+1}_E(y_i)) \in \pi^{l+1} A^{(2)}$, whence $\phi^l(\delta^{j+1}_E(y_i)) \in \pi^l A^{(2)}$.
\end{proof}

The following lemma plays a crucial role in the proof of Theorem \ref{Theorem:main result on G displays over complete regular local rings} (especially in the proof of Proposition \ref{Proposition:deformation of isomorphism} below).
As in the proof of Lemma \ref{Lemma:kernel of two fold coproduct I},
we set $y_i:=x_i/d \in A^{(2)}$.

\begin{lem}\label{Lemma:kernel of two fold coproduct refined}
        Let
        $M \subset A^{(2)}$
        be the ideal generated by
        $\phi(y_1)/d, \dotsc, \phi(y_n)/d \in K$.
        Then we have inclusions
        \[
        \phi(K) \subset dM + d(\pi, d)K \quad \text{and} \quad \phi(M) \subset d(t_1, \dotsc, t_n)M + d(\pi, d)K.
        \]
\end{lem}

The proof of Lemma \ref{Lemma:kernel of two fold coproduct refined} will be given in Section \ref{Subsection:proof of Lemma kernel of two fold coproduct refined}.

\begin{rem}\label{Remark:three coproduct}
    There exists a coproduct $(A^{(3)}, I^{(3)})$ of three copies of $(A, I)$
    in the category $(R)_{\Prism, \O_E}$.
    Indeed, one can define $(A^{(3)}, I^{(3)})$ as a pushout of the diagram
    \[
    (A^{(2)}, I^{(2)}) \overset{p_2}{\leftarrow} (A, I) \overset{p_1}{\rightarrow} (A^{(2)}, I^{(2)}),
    \]
    which exists since $p_1$ is a flat map (see Remark \ref{Remark:pushout in prismatic site}).
    Let $q_1, q_2, q_3 \colon (A, I) \to (A^{(3)}, I^{(3)})$ denote the associated three morphisms.
    For $1 \leq i < j \leq 3$, let
    $p_{ij} \colon (A^{(2)}, I^{(2)}) \to (A^{(3)}, I^{(3)})$ be the unique morphism such that $p_{ij} \circ p_1=q_i$ and $p_{ij} \circ p_2=q_j$.
\end{rem}

\begin{cor}\label{Corollary:kernel of three fold coproduct}
    Let $m \colon (A^{(3)}, I^{(3)}) \to (A, I)$ be the unique morphism in $(R)_{\Prism, \O_E}$ such that $m \circ q_i = \id_{(A, I)}$ for $i=1, 2, 3$.
    Let $L$ be the kernel of $m \colon A^{(3)} \to A$.
    Let $d \in I^{(3)}$ be a generator.
    Then the following assertions hold:
    \begin{enumerate}
        \item We have $\phi(L) \subset dL$.
        \item
        Let
        $N \subset A^{(3)}$
        be the ideal generated by
        $
        \{ \phi(p_{12}(y_l))/d, \phi(p_{23}(y_l))/d \}_{1 \leq l \leq n} \subset L.
        $
        Then we have inclusions
        \[
        \phi(L) \subset dN + d(\pi, d)L \quad \text{and} \quad \phi(N) \subset d(q_1(t_1), \dotsc, q_1(t_n))N + d(\pi, d)L.
        \]
    \end{enumerate}
\end{cor}

\begin{proof}
    We may assume that $d$ is the image of a generator of $I^{(2)}$, again denoted by $d$, under the homomorphism $p_{12}$.
    As in Remark \ref{Remark:three coproduct}, we identify
    $A^{(3)}$
    with
    the $(\pi, d)$-adic completion of
    $A^{(3)}_0:= A^{(2)} \otimes_{p_2, A, p_1} A^{(2)}$.
    Under this identification,
    the homomorphism $p_{12}$ (resp.\ $p_{23}$)
    is induced by 
    the homomorphism $A^{(2)} \to A^{(3)}_0$ defined by $a \mapsto a \otimes 1$ (resp.\ $a \mapsto 1 \otimes a$).
    The kernel $L_0$ of the natural homomorphism
    $A^{(3)}_0 \to A$
    coincides with
    $K \otimes_A A^{(2)} + A^{(2)} \otimes_A K$, and $L$ is the $(\pi, d)$-adic completion of $L_0$.

    In order to prove the assertion (1), it suffices to show that for any element $x \in L$ which lies in the image of $L_0 \to L$, we have
    $\phi(x) \in d A^{(3)}$.
    (Note that $A^{(3)}/d$ is $\pi$-adically complete by Remark \ref{Remark:derived and classical complete}.)
    This follows from Lemma \ref{Lemma:kernel of two fold coproduct I}.
    Similarly, the assertion (2) follows from Lemma \ref{Lemma:kernel of two fold coproduct refined}.
    We note here that, since
    $q_j(t_l)-q_i(t_l)=p_{ij}(x_l) \in d L$ for $1 \leq i < j \leq 3$,
    the ideal
    $d(q_1(t_1), \dotsc, q_1(t_n))N + d(\pi, d)L$
    is unchanged if we replace $q_1$ by $q_i$ ($1 \leq i \leq 3$).
\end{proof}

\begin{rem}\label{Remark:compare with Anschutz-Le Bras}
    Assume that $\O_E=\Z_p$.
    Under the assumption that $n=1$ and $R$ is $p$-torsion free,
    Ansch\"utz--Le Bras gave a proof of the analogue of Theorem \ref{Theorem:main result on G displays over complete regular local rings} for minuscule Breuil--Kisin modules in \cite[Section 5.2]{Anschutz-LeBras}.
    (We will come back to this result in Section \ref{Subsection:A remark on prismatic Dieudonne crystals}.)
    In the proof, they use that the map $K \to K$, $x \mapsto \phi(x)/d$
    is topologically nilpotent with respect to the $(p, d)$-adic topology (\cite[Lemma 5.15]{Anschutz-LeBras}).
    This topological nilpotence may not be true if $n \geq 2$ or $p=0$ in $R$.
    We will use
    Lemma \ref{Lemma:kernel of two fold coproduct refined}, Corollary \ref{Corollary:kernel of three fold coproduct}, and the fact that the local ring $A$ is complete and noetherian to overcome this issue; see Section \ref{Subsection:Deformations of isomorphisms}.
\end{rem}

\subsection{Proof of Lemma \ref{Lemma:kernel of two fold coproduct refined}}\label{Subsection:proof of Lemma kernel of two fold coproduct refined}

The proof of Lemma \ref{Lemma:kernel of two fold coproduct refined} will require some preliminary results.
We first introduce some notation.

If $\mathcal{E}$ is not contained in $\pi A$,
then the image of $\mathcal{E}$ in $A/\pi$ is a nonzerodivisor (since $A/\pi$ is an integral domain).
In this case, the $\delta_E$-ring
$A \{ \phi(\mathcal{E}) / \pi \}$
is $\pi$-torsion free, and is isomorphic to
the
$\O_E$-PD envelope
$D_{(\mathcal{E})}(A)$
of $A$ with respect to the ideal $(\mathcal{E})$; see Corollary \ref{Corollary:pd envelope and delta structure}.
Let $A''$ be the $\pi$-adic completion of
$A \{ \phi(\mathcal{E}) / \pi \}$, and let $g \colon A \to A''$ be the natural homomorphism.
We note that $A''$ is also $\pi$-torsion free.
We consider the following pushout squares of $\delta_E$-rings:
\[
\xymatrix{
A \ar^-{p'_1}[r]  \ar[d]_-{\phi} & B \ar^-{}[r] \ar[d]_-{} & A^{(2)} \ar^-{m}[r] \ar[d]_-{} & A \ar[d]_-{\phi} \\
A \ar[r]^-{} \ar[d]_-{g} & B' \ar^-{}[r] \ar[d]_-{} & A^{(2)'} \ar^-{}[r] \ar[d]_-{} & A \ar[d]_-{g} \\
A'' \ar[r]^-{} & B''_0 \ar^-{}[r] & A^{(2)''}_0 \ar^-{}[r] & A''.
}
\]
Let $A^{(2)''}$ be the $\pi$-adic completion
of $A^{(2)''}_0$ and $K''$ the kernel of
the induced homomorphism $A^{(2)''} \to A''$.
Since $A \to A^{(2)}$ is flat (by Proposition \ref{Proposition:prismatic envelope} and \cite[Theorem 1.5]{Yekutieli18}),
so is
$A'' \to A^{(2)''}_0$.
It follows that $A^{(2)''}$ is $\pi$-torsion free.
In the case where $\mathcal{E} \in \pi A$, we set $A^{(2)''}:= A^{(2)}$ and $K'':=K$.

\begin{lem}\label{Lemma:K'' case I}
Let the notation be as above. Then the following assertions hold:
\begin{enumerate}
    \item We have $\phi(K'') \subset \pi K''$.
    \item 
    We have $x_i \in \pi K''$ for any $1 \leq i \leq n$.
    (Here we denote the image of $x_i \in B$ in $A^{(2)''}$ again by $x_i$.)
    We set $w_i:=x_i/\pi \in K''$.
    Then $K''/\pi K''$ is generated by the images of
    $\{ \delta^j_E(w_i) \}_{1 \leq i \leq n, j \geq 0}$ as an $A^{(2)''}$-module.
\end{enumerate}
\end{lem}

\begin{proof}
    If $\mathcal{E} \in \pi A$, then the assertions follow from Lemma \ref{Lemma:kernel of two fold coproduct I} and its proof.
    Thus, we may assume that $\mathcal{E}$ is not contained in $\pi A$.
    Let $h \colon A^{(2)} \to A^{(2)''}$ denote the natural homomorphism.
    Since $g(\phi(\mathcal{E}))/\pi \in A''$ is a unit by Lemma \ref{Lemma:distinguished} (1),
    it follows that $h(d) \in A^{(2)''}$ is a unit multiple of $\pi$.
    The kernel $K''$ of $A^{(2)''} \to A''$ can be identified with
    the $\pi$-adic completion of $h^*K$.
    Therefore, the assertion (1) follows from Lemma \ref{Lemma:kernel of two fold coproduct I}.

    Using that the image of $g(\phi(\mathcal{E}))$ in $B''_0$ is a unit multiple of $\pi$, we see that
    $A^{(2)''}$ agrees with the $\pi$-adic completion of
    $B''_0\{ x_1/\pi, \dotsc, x_n/\pi \}$.
    Since the kernel of $B''_0 \to A''$ is generated by $x_1, \dotsc, x_n$,
    it follows that the kernel of
    $B''_0\{ x_1/\pi, \dotsc, x_n/\pi \} \to A''$
    is generated by $\{ \delta^j_E(x_i/\pi) \}_{1 \leq i \leq n, j \geq 0}$, which implies the assertion (2).
\end{proof}

\begin{lem}\label{Lemma:K'' case II}
We define
\[
\phi_1 \colon K'' \to K'', \quad x \mapsto \phi(x)/\pi.
\]
The induced $\phi$-linear homomorphism
$K''/\pi K'' \to K''/ \pi K''$
is denoted by the same symbol $\phi_1$.
Let
$\overline{M}'' \subset K''/\pi K''$
be the
$A^{(2)''}$-submodule generated by the images of
$
\phi_1(w_1), \dotsc, \phi_1(w_n) \in K''.
$
Then we have inclusions
\[
\phi_1(K''/\pi K'') \subset \overline{M}'' \quad \text{and} \quad \phi_1(\overline{M}'') \subset (t_1, \dotsc, t_n)\overline{M}'',
\]
where we denote the image of $t_i \in A^{(2)}$ in $A^{(2)''}$ again by $t_i$.
\end{lem}

\begin{proof}
We have
$x^q=\phi(x)- \pi \delta_E(x) \in \pi K''$
for every $x \in K''$ by Lemma \ref{Lemma:K'' case I}.
Let $J \subset K''$ be the ideal generated by $\{ x^q/\pi \}_{x \in K''}$.
For any $x \in K''$, we have
\[
\phi_1(x^q/\pi)= \phi(x^q/\pi)/\pi = \phi(x)^q/\pi^2 = \pi^{q-1} (\phi_1(x)^q/\pi) \in \pi J,
\]
and thus we obtain $\phi_1(J) \subset \pi J$.

We shall prove that
$K''/(J+ \pi K'')$
is generated by the images of $w_1, \dotsc, w_n$ as an $A^{(2)''}$-module.
By Lemma \ref{Lemma:K'' case I}, it suffices to show that for any $j \geq 0$ and any $i$,
the image of $\delta^j_E(w_i)$
in $K''/(J+ \pi K'')$
is contained in
the 
$A^{(2)''}$-submodule of $K''/(J+ \pi K'')$ generated by the images of $w_1, \dotsc, w_n$.
We proceed by induction on $j$.
If $j=0$, then the assertion holds trivially.
Assume that the assertion is true for some $j \geq 0$.
Since
\[
\phi_1(w_i)=\phi(x_i)/\pi^2=((x_i+t_i)^q-t^q_i)/\pi^2=((\pi w_i+t_i)^q-t^q_i)/\pi^2,
\]
we can write $\phi_1(w_i)$ as
\begin{equation}\label{equation:formula phi_1}
    \phi_1(w_i)= \pi^{q-2} w^{q}_i + (q/\pi)t^{q-1}_i w_i + \pi b_i
\end{equation}
for some element $b_i \in K''$.
For any $x \in K''$, we have
$\delta_E(x)=\phi_1(x)$
in $K''/J$.
Thus
the image of
$\delta^{j+1}_E(w_i)$
in $K''/(J+ \pi K'')$
agrees with the one of
$\phi_1(\delta^{j}_E(w_i))$,
which is contained in the 
$A^{(2)''}$-submodule of $K''/(J+ \pi K'')$ generated by the images of $\phi_1(w_1), \dotsc, \phi_1(w_n)$ by the induction hypothesis.
Then $(\ref{equation:formula phi_1})$ implies the assertion for $j+1$.

We have shown that every $x \in K''$ can be written as
\[
x= (\sum_{1 \leq i \leq n} a_i w_i) + b + \pi c
\]
for some $a_i \in A^{(2)''}$ ($1 \leq i \leq n$), $b \in J$, and $c \in K''$.
Since $\phi_1(b) \in \pi J$,
the image
of $\phi_1(x)$ in $K''/\pi K''$
coincides with that of 
$\sum_{1 \leq i \leq n} \phi(a_i) \phi_1(w_i)$.
This proves that
$\phi_1(K''/\pi K'') \subset \overline{M}''$.
Moreover, since
$\phi_1(w^q_i)=\phi(w_i)\phi_1(w^{q-1}_i)$
is contained in $\pi K''$,
it follows from (\ref{equation:formula phi_1}) that
the image
of $\phi_1(\phi_1(w_i))$ in $K''/\pi K''$
is equal to
that of
$\phi_1((q/\pi)t^{q-1}_i w_i)=(q/\pi)t^{q(q-1)}_i \phi_1(w_i)$.
This proves that
$\phi_1(\overline{M}'') \subset (t_1, \dotsc, t_n)\overline{M}''$.
\end{proof}

We now prove Lemma \ref{Lemma:kernel of two fold coproduct refined}.

\begin{proof}[Proof of Lemma \ref{Lemma:kernel of two fold coproduct refined}]
    We first treat the case where $\mathcal{E} \in \pi A$.
    In this case $d$ is a unit multiple of $\pi$.
    Thus, the assertion follows from Lemma \ref{Lemma:K'' case II}.

    We now assume that $\mathcal{E}$ is not contained in $\pi A$.
    We define
    \[
    \phi_1 \colon K \to K, \quad x \mapsto \phi(x)/d.
    \]
    The induced $\phi$-linear homomorphism
    $K/(\pi, d)K \to K/(\pi, d)K$
    is also denoted by $\phi_1$.
    Let
    $
    \overline{M} \subset K/(\pi, d)K
    $
    be the $A^{(2)}$-submodule generated by the images of
    $
    \phi_1(y_1), \dotsc, \phi_1(y_n) \in K.
    $
    It suffices to prove that
    $\phi_1(K/(\pi, d)K) \subset \overline{M}$
    and
    $\phi_1(\overline{M}) \subset (t_1, \dotsc, t_n)\overline{M}$.

    Let $f \colon A^{(2)} \to A^{(2)'}$ denote the natural homomorphism.
    Let $K'$ be the kernel of the homomorphism $A^{(2)'} \to A$, which can be identified with $f^*K$.
    We define $\phi'_1 \colon K' \to K'$ by $x \mapsto \phi(x)/f(d)$, and
    let
    $
    \overline{M}' \subset K'/(\pi, f(d))K'
    $
    be the $A^{(2)'}$-submodule generated by the images of
    $
    \phi'_1(f(y_1)), \dotsc, \phi'_1(f(y_n)) \in K'.
    $
    Since $\phi \colon A \to A$ is faithfully flat, so is $f$.
    Therefore, in order to prove the assertion, it is enough to prove that
    \begin{equation}\label{equation:goal}
        \phi'_1(K'/(\pi, f(d))K') \subset \overline{M}' \quad \text{and} \quad \phi'_1(\overline{M}') \subset (f(t_1), \dotsc, f(t_n))\overline{M}'.
    \end{equation}

    The homomorphism $A^{(2)'} \to A^{(2)''}$
    induced by $g \colon A \to A''$ is again denoted by $g$.
    The element $g(f(d))$ is a unit multiple of $\pi$ in $A^{(2)''}$.
    Thus, for $\phi_1 \colon K'' \to K''$ defined in Lemma \ref{Lemma:K'' case II},
    the element $g(\phi'_1(x))$ is a unit multiple of
    $\phi_1(g(x))$ for any $x \in K'$.
    Also, the induced homomorphism
    $A^{(2)'}/(\pi, f(d)) \to A^{(2)''}/\pi$,
    again denoted by $g$, sends $\overline{M}'$ into $\overline{M}''$.
    It follows from Lemma \ref{Lemma:K'' case II} that, for any
    $x \in K'/(\pi, f(d))K'$
    (resp.\ $x \in \overline{M}'$),
    we have
    \begin{equation}\label{equation:key inclusions phi'_1}
        g(\phi'_1(x)) \in \overline{M}''
    \quad (\text{resp.\ } g(\phi'_1(x)) \in (g(f(t_1)), \dotsc, g(f(t_n)))\overline{M}'').
    \end{equation}

    Since $A''/\pi \simeq D_{(\mathcal{E})}(A)/\pi$, 
    we can find a homomorphism
    \[
    s \colon A''/\pi \to A/(\pi, \phi(\mathcal{E}))
    \]
    of $\O_E$-algebras such that
    the composition
    $A/(\pi, \phi(\mathcal{E})) \overset{g}{\to} A''/\pi \overset{s}{\to} A/(\pi, \phi(\mathcal{E}))$
    is the identity; see Example \ref{Example:pi PD polynomial ring} and Lemma \ref{Lemma:pd envelope regular sequence}.
    We consider the following pushout squares of $\O_E$-algebras:
    \[
    \xymatrix{
    A''/\pi \ar^-{}[r]  \ar[d]_-{s} & A^{(2)''}/\pi \ar^-{}[r] \ar[d]_-{\widetilde{s}} & A''/\pi  \ar[d]_-{s}  \\
    A/(\pi, \phi(\mathcal{E})) \ar[r]^-{}  & A^{(2)'}/(\pi, f(d)) \ar^-{}[r]  & A/(\pi, \phi(\mathcal{E})).
    }
    \]
    The homomorphism $g \colon A^{(2)'}/(\pi, f(d)) \to A^{(2)''}/\pi$
    is a section of $\widetilde{s}$.
    We observe that $\widetilde{s}(K''/\pi K'') \subset K'/(\pi, f(d))K'$ and $\widetilde{s}(\overline{M}'') \subset \overline{M}'$.
    It follows from (\ref{equation:key inclusions phi'_1}) that,
    for any
    $x \in K'/(\pi, f(d))K'$
    (resp.\ $x \in \overline{M}'$),
    its image $\phi'_1(x)=\widetilde{s}(g(\phi'_1(x)))$
    belongs to $\overline{M}'$ (resp.\ $(f(t_1), \dotsc, f(t_n))\overline{M}'$).
    This proves (\ref{equation:goal}),
    and the proof of Lemma \ref{Lemma:kernel of two fold coproduct refined} is now complete.
\end{proof}

\subsection{Deformations of isomorphisms}\label{Subsection:Deformations of isomorphisms}

As in Section \ref{Subsection:Coproducts of Breuil--Kisin prisms},
we write $(A, I)=(\mathfrak{S}_\O, (\mathcal{E}))$.
In this subsection, as a preparation for the proof of Theorem \ref{Theorem:main result on G displays over complete regular local rings}, we study deformations of isomorphisms of $G$-$\mu$-displays over $(A, I)$ along the morphisms
$m \colon (A^{(2)}, I^{(2)}) \to (A, I)$ and $m \colon (A^{(3)}, I^{(3)}) \to (A, I)$ defined in Section \ref{Subsection:Coproducts of Breuil--Kisin prisms}.
Throughout this subsection, we assume that $\mu$ is 1-bounded.

Our setup is as follows.
Let
$(A', I'):=(A^{(2)}, I^{(2)})$
(resp.\ $(A', I'):=(A^{(3)}, I^{(3)})$).
Let
$m \colon (A', I') \to (A, I)$
denote
$m \colon (A^{(2)}, I^{(2)}) \to (A, I)$
(resp.\ $m \colon (A^{(3)}, I^{(3)}) \to (A, I)$).
Let $f_1, f_2 \in \{ p_1, p_2 \}$
(resp.\ $f_1, f_2 \in \{ q_1, q_2, q_3 \}$).
We do not exclude the case where $f_1=f_2$.

The purpose of this subsection is to prove the following result:

\begin{prop}\label{Proposition:deformation of isomorphism}
Assume that $\mu$ is 1-bounded.
    Let $\mathcal{Q}_1$ and $\mathcal{Q}_2$ be $G$-$\mu$-displays over $(A, I)$.
    Then
    the map
    \begin{equation}\label{equation:reduction map for isomorphisms}
        m^{*} \colon \Hom_{G\mathchar`-\mathrm{Disp}_\mu(A', I')}(f^*_1(\mathcal{Q}_1), f^*_2(\mathcal{Q}_2)) \to \Hom_{G\mathchar`-\mathrm{Disp}_\mu(A, I)}(\mathcal{Q}_1, \mathcal{Q}_2)
    \end{equation}
    induced by the base change functor
    $m^* \colon G\mathchar`-\mathrm{Disp}_\mu(A', I') \to G\mathchar`-\mathrm{Disp}_\mu(A, I)$ is bijective.
\end{prop}

We need some preliminary results for the proof of Proposition \ref{Proposition:deformation of isomorphism}.
We will use the following notation.
Let $H$ be a group scheme over $\O$.
For an ideal $J \subset A'$,
we write
\[
H(J):= \Ker (H(A') \to H(A'/J))
\]
for the kernel of the homomorphism $H(A') \to H(A'/J)$.
If $H=G_\O$, then we simply write $G(J):=G_\O(J)$.

Let $K$ denote the kernel of $m \colon A' \to A$.
(We note that if $A'=A^{(3)}$, then this kernel was denoted by $L$ in Corollary \ref{Corollary:kernel of three fold coproduct}.)
Let $d \in I'$ be a generator.

\begin{lem}\label{Lemma:stability of G(J)}
    Let $J \subset A'$ be an ideal such that
    $J \subset dK$ and, for any $x \in J$, we have $\phi(x/d) \in J$.
    Then the homomorphism
    $\sigma_{\mu, d} \colon G_\mu(A', I') \to G(A')$
    (see $(\ref{equation:sigma map of sets})$)
    sends $G(J) \subset G_\mu(A', I')$ into itself.
\end{lem}

\begin{proof}
    We note that, by Proposition \ref{Proposition:BB isomorphism}, we have
    $G(J) \subset G(dK) \subset G_\mu(A', I')$, and
    the multiplication map
    $U^-_{\mu} \times_{\Spec \O} P_\mu \to G_\O$
    induces a bijection
    \[
    (\Lie(U^-_{\mu}) \otimes_\O J)
    \times P_\mu(J) \overset{\sim}{\to} G(J).
    \]
    Thus, it suffices to prove that
    $\sigma_{\mu, d}(P_\mu(J)) \subset G(J)$ and 
    $\sigma_{\mu, d}(\Lie(U^-_{\mu}) \otimes_\O J) \subset G(J)$.

    %We first prove that $\sigma_{\mu, d}(P_\mu(J)) \subset G(J)$.
    By Remark \ref{Remark:formula action of cocharacter} and Lemma \ref{Lemma:Pmu structure}, we have
    $\mu(d)P_\mu(J)\mu(d)^{-1} \subset P_\mu(J)$.
    (In fact, this holds for any ideal $J \subset A'$.)
    Since $\phi(J) \subset J$, we have $\phi(G(J)) \subset G(J)$.
    It follows that $\sigma_{\mu, d}(P_\mu(J)) \subset G(J)$.

    %We shall show that $\sigma_{\mu, d}(\Lie(U^-_{\mu}) \otimes_\O J) \subset G(J)$.
    Since $\mu$ is 1-bounded, the homomorphism
    $G_\mu(A', I') \to G(A')$, $g \mapsto \mu(d)g\mu(d)^{-1}$ restricts to a homomorphism
    \[
    \Lie(U^-_{\mu}) \otimes_\O J \to \Lie(U^-_{\mu}) \otimes_\O \frac{1}{d} J, \quad v \mapsto v/d.
    \]
    (See Remark \ref{Remark:Umu identify}.)
    Since $\phi((1/d)J) \subset J$, we obtain
    $\sigma_{\mu, d}(\Lie(U^-_{\mu}) \otimes_\O J) \subset G(J)$.
\end{proof}

\begin{defn}\label{Definition:UdX FdX}
    Let $J \subset A'$ be an ideal as in Lemma \ref{Lemma:stability of G(J)}. For an element $X \in G(A')$,
we define a homomorphism
\[
\mathcal{U}_{d, X} \colon G(J) \to  G(J), \quad g \mapsto X\sigma_{\mu, d}(g)X^{-1}.
\]
We also define a map of sets
\[
\mathcal{V}_{d, X} \colon G(J) \to  G(J), \quad g \mapsto \mathcal{U}_{d, X}(g)g^{-1}.
\]
\end{defn}

Let $J_2 \subset J_1 \subset A'$ be two ideals which satisfy the assumption of Lemma \ref{Lemma:stability of G(J)}.
Then $\mathcal{U}_{d, X} \colon G(J_1) \to G(J_1)$ induces a homomorphism
\[
G(J_1)/G(J_2) \to G(J_1)/G(J_2),
\]
which we denote by the same symbol $\mathcal{U}_{d, X}$.
Let
$
\mathcal{V}_{d, X} \colon G(J_1)/G(J_2) \to G(J_1)/G(J_2)
$
be the map of sets defined by
$g \mapsto  \mathcal{U}_{d, X}(g)g^{-1}$.

By Lemma \ref{Lemma:kernel of two fold coproduct I} and Corollary \ref{Corollary:kernel of three fold coproduct}, we have
$\phi(K) \subset dK$.
Thus, the ideal $dK \subset A'$ satisfies the assumption of Lemma \ref{Lemma:stability of G(J)}.
We shall prove
(in Proposition \ref{Proposition:bijectivity of VdX for G(dK)} below)
that
$
\mathcal{V}_{d, X} \colon G(dK) \to  G(dK)
$
is bijective for any $X \in G(A')$, from which we will deduce Proposition \ref{Proposition:deformation of isomorphism}.
For this purpose, we need the following lemmas.

\begin{lem}\label{Lemma:square zero case}
 Let $J_2 \subset J_1 \subset A'$ be two ideals which satisfy the assumption of Lemma \ref{Lemma:stability of G(J)}.
 Assume that for any $x \in J_1$, we have $\phi(x/d) \in J_2$.
 Then we have
 \[
 \sigma_{\mu, d}(G(J_1)) \subset G(J_2).
 \]
 In particular, the map
 $
\mathcal{V}_{d, X} \colon G(J_1)/G(J_2) \to G(J_1)/G(J_2)
 $
 is equal to the map $g \mapsto g^{-1}$ for any $X \in G(A')$.
\end{lem}

\begin{proof}
    The same argument as in the proof of Lemma \ref{Lemma:stability of G(J)} shows that $\sigma_{\mu, d}(G(J_1)) \subset G(J_2)$.
    The second assertion immediately follows from the first one.
\end{proof}

\begin{lem}\label{Lemma:five lemma for VdX}
    Let $J_3 \subset J_2 \subset J_1 \subset A'$ be three ideals which satisfy the assumption of Lemma \ref{Lemma:stability of G(J)}.
    Let $X \in G(A')$.
    If the maps
 \[
\mathcal{V}_{d, X} \colon G(J_1)/G(J_2) \to G(J_1)/G(J_2) \quad \text{and} \quad \mathcal{V}_{d, X} \colon G(J_2)/G(J_3) \to G(J_2)/G(J_3)
 \]
 are bijective, then
$
\mathcal{V}_{d, X} \colon G(J_1)/G(J_3) \to G(J_1)/G(J_3)
$
is also bijective.
\end{lem}

\begin{proof}
    Let us prove the surjectivity.
    Let $h \in G(J_1)/G(J_3)$ be an element.
    The image $h' \in G(J_1)/G(J_2)$ of $h$ can be written as
    $h'=\mathcal{V}_{d, X}(g')$
    for some element $g' \in G(J_1)/G(J_2)$.
    We choose some $g \in G(J_1)/G(J_3)$ which is a lift of $g'$.
    Then we see that
    $\mathcal{U}_{d, X}(g)^{-1}hg$
    is contained in $G(J_2)/G(J_3)$, so that there exists an element
    $g'' \in G(J_2)/G(J_3)$
    such that
    \[
    \mathcal{V}_{d, X}(g'')=\mathcal{U}_{d, X}(g'')g''^{-1}=\mathcal{U}_{d, X}(g)^{-1}hg.
    \]
    It follows that $h=\mathcal{V}_{d, X}(gg'')$.
    This proves that
    $
    \mathcal{V}_{d, X} \colon G(J_1)/G(J_3) \to G(J_1)/G(J_3)
    $
    is surjective.
    The proof of the injectivity is similar.
\end{proof}

\begin{lem}\label{Lemma:VdX bijective primitive case}
    Let $l \geq 0$ be an integer.
    For any $X \in G(A')$,
    the map
    \[
    \mathcal{V}_{d, X} \colon G((\pi, d)^l dK)/G((\pi, d)^{l+1} dK) \to  G((\pi, d)^l dK)/G((\pi, d)^{l+1} dK)
    \]
    is bijective.
\end{lem}

\begin{proof}
    \textit{Step 1.}
    We set $K_l:=(\pi, d)^lK$.
    We consider the ideal
    $K^{-}:=K^2+(\pi, d)K$
    and let
    $
    K^{-}_l:=(\pi, d)^l K^{-}.
    $
    All of $dK_l$, $dK_{l+1}$, $dK^{-}_l$
    satisfy the assumption of Lemma \ref{Lemma:stability of G(J)}.
    Since
    \[
    \phi(K^2) \subset d^2K^2 \subset d(\pi, d)K,
    \]
    we have $\phi(K^{-}_l) \subset dK_{l+1}$.
    Therefore, it follows from Lemma \ref{Lemma:square zero case} that
    $\mathcal{V}_{d, X}$ is bijective for $G(dK^{-}_l)/G(dK_{l+1})$.
    By Lemma \ref{Lemma:five lemma for VdX},
    it now suffices to show that $\mathcal{V}_{d, X}$ is bijective for
    $G(dK_l)/G(dK^{-}_l)$.
    
    \textit{Step 2.} 
    By Lemma \ref{Lemma:kernel of two fold coproduct refined} and Corollary \ref{Corollary:kernel of three fold coproduct},
    there exists a finitely generated ideal $M \subset A'$ which is contained in $K$ such that
    $\phi(K) \subset dM + dK^{-}$
    and
    $\phi(M) \subset (t_1, \dotsc, t_n)dM + dK^{-}$,
    where we abuse notation and denote the image of $t_i \in A$ under the morphism $p_1 \colon A \to A'$
    (resp.\ $q_1 \colon A \to A'$)
    if $A'=A^{(2)}$
    (resp.\ if $A'=A^{(3)}$)
    by the same symbol.
    We set
    $M_l := (\pi, d)^l M \subset K_l$.
    Then we have inclusions
    \begin{equation}\label{equation:key inclusions}
        \phi(K_l) \subset dM_l + dK^{-}_l \quad \text{and} \quad \phi(M_l) \subset (t_1, \dotsc, t_n)dM_l + dK^{-}_l.
    \end{equation}
    In particular, the ideals
    $dM_l + dK^{-}_l \subset dK_l$ satisfy the assumption of Lemma \ref{Lemma:square zero case}, so that
    $\mathcal{V}_{d, X}$ is bijective for $G(dK_l)/G(dM_l + dK^{-}_l)$.
    By Lemma \ref{Lemma:five lemma for VdX},
    it is enough to prove that $\mathcal{V}_{d, X}$ is bijective for $G(dM_l + dK^{-}_l)/G(dK^{-}_l)$.

     \textit{Step 3.} 
     We shall prove that
     \begin{equation}\label{equation:inverse limit I}
         G(dM_l + dK^{-}_l)/G(dK^{-}_l) \overset{\sim}{\to} \varprojlim_{r \geq 0} G(dM_l + dK^{-}_l)/G((t_1, \dotsc, t_n)^r dM_l + dK^{-}_l).
     \end{equation}
     To simplify the notation, we set $C_1:=A'/(dM_l + dK^{-}_l)$ and $C_2:=A'/dK^{-}_l$.
     Let $N \subset C_2$ be the image of $dM_l + dK^{-}_l$.
     We first claim that
     \begin{equation}\label{equation:ring limit part 1}
         A' \overset{\sim}{\to} \varprojlim_{r \geq 0} A'/dK_r
     \end{equation}
     and 
     \begin{equation}\label{equation:ring limit part 2}
         C_2 \overset{\sim}{\to} \varprojlim_{r \geq 0} C_2/(t_1, \dotsc, t_n)^r N.
     \end{equation}
     Since $d$ is a nonzerodivisor and $K$ is $(\pi, d)$-adically complete, it follows that
    $dK \overset{\sim}{\to} \varprojlim_{r \geq 0} dK/dK_r$.
    This implies (\ref{equation:ring limit part 1}).
    Since $N$ is killed by $K$, we see that $N$ is a finitely generated module over $A'/K \overset{\sim}{\to} A$.
     Since $A$ is noetherian and is $(t_1, \dotsc, t_n)$-adically complete, it follows that $N$ is also $(t_1, \dotsc, t_n)$-adically complete, which means that $N \overset{\sim}{\to} \varprojlim_{r \geq 0} N/(t_1, \dotsc, t_n)^r N$.
     This implies (\ref{equation:ring limit part 2}).

     We next show that $G(A') \to G(C_2)$ is surjective.
     Indeed, by (\ref{equation:ring limit part 1}) and the fact that
     $G(A'/dK_{r+1}) \to G(A'/dK_r)$ is surjective (as $G$ is smooth), it follows that $G(A') \to G(A'/dK_r)$ is surjective for any $r$.
     Since we have $(dK^{-}_l)^2 \subset dK_{l+1} \subset dK^{-}_l$, we see that
     $G(A'/dK_{l+1}) \to G(C_2)$ is surjective (again by the smoothness of $G$).
     Therefore $G(A') \to G(C_2)$ is surjective, as desired.
     Similarly, it follows from 
     (\ref{equation:ring limit part 2})
     that
     $G(C_2) \to G(C_2/(t_1, \dotsc, t_n)^r N)$ is surjective.

     Using the results obtained in the previous paragraph,
     we see that
     \[
     G(dM_l + dK^{-}_l)/G(dK^{-}_l) \overset{\sim}{\to} \Ker(G(C_2) \to G(C_1))
     \]
     and
     \[
     G(dM_l + dK^{-}_l)/G((t_1, \dotsc, t_n)^r dM_l + dK^{-}_l) \overset{\sim}{\to} \Ker(G(C_2/(t_1, \dotsc, t_n)^r N) \to G(C_1)).
     \]
     Now (\ref{equation:inverse limit I}) follows from 
     (\ref{equation:ring limit part 2}).

    \textit{Step 4.}
    We claim that
    $\mathcal{V}_{d, X}$
    is bijective for
    \[
    G((t_1, \dotsc, t_n)^r dM_l + dK^{-}_l)/G((t_1, \dotsc, t_n)^{r+1} dM_l + dK^{-}_l)
    \]
    for any $r \geq 0$.
    Indeed,
    the second inclusion of (\ref{equation:key inclusions}) shows that
    the assumption of Lemma \ref{Lemma:square zero case} is satisfied in this case, and hence the assertion follows.
    
    Using Lemma \ref{Lemma:five lemma for VdX} repeatedly, we see that
    $\mathcal{V}_{d, X}$
    is bijective for
    \[
    G(dM_l + dK^{-}_l)/G((t_1, \dotsc, t_n)^r dM_l + dK^{-}_l)
    \]
    for any $r \geq 0$.
    It then follows from (\ref{equation:inverse limit I}) that
    $\mathcal{V}_{d, X}$
    is bijective for $G(dM_l + dK^{-}_l)/G(dK^{-}_l)$ as well.
    This completes the proof.
\end{proof}

Let us now prove the desired result.

\begin{prop}\label{Proposition:bijectivity of VdX for G(dK)}
    For any $X \in G(A')$, the map
    $
    \mathcal{V}_{d, X} \colon G(dK) \to  G(dK)
    $
    is bijective.
\end{prop}

\begin{proof}
    By (\ref{equation:ring limit part 1}) in the proof of Lemma \ref{Lemma:VdX bijective primitive case},
    we have
    \[
    G(dK) \overset{\sim}{\to} \varprojlim_{l \geq 0} G(dK)/G((\pi, d)^l dK).
    \]
    In order to show that
    $
    \mathcal{V}_{d, X} \colon G(dK) \to  G(dK)
    $
    is bijective, it suffices to check that
    $
    \mathcal{V}_{d, X} \colon G(dK)/G((\pi, d)^l dK) \to  G(dK)/G((\pi, d)^l dK)
    $
    is bijective for any $l \geq 0$.
    This follows from 
    Lemma \ref{Lemma:VdX bijective primitive case}
    by using Lemma \ref{Lemma:five lemma for VdX} repeatedly.
\end{proof}

We also need the following lemma:

\begin{lem}\label{Lemma:banal after finite field extension}
    Let $\mathcal{Q}$ be a $G$-$\mu$-display over $(A, I)$.
    Then there exists a finite extension $\widetilde{k}$ of $k$ such that the base change of
    $\mathcal{Q}$ to $(A_{\widetilde{\O}}, IA_{\widetilde{\O}})$ is banal, where $\widetilde{\O}:= W(\widetilde{k}) \otimes_{W(\F_q)} \O_E$ and $A_{\widetilde{\O}}:=A \otimes_\O \widetilde{\O} = \widetilde{\O}[[t_1, \dotsc, t_n]]$.
\end{lem}

\begin{proof}
    The Hodge filtration
    $P(\mathcal{Q})_{A/I}=P(\mathcal{Q})_{R}$ of $\mathcal{Q}$
    is a $(P_\mu)_R$-torsor over $\Spec R$.
    There exists a finite extension
    $\widetilde{k}$ of $k$ such that
    $P(\mathcal{Q})_{R} \times_{\Spec R} \Spec \widetilde{k}$
    is trivial.
    Since $P_\mu$ is smooth and
    $R \otimes_{\O} \widetilde{\O}$ is a complete local ring,
    it follows that $P(\mathcal{Q})_{R}$ is trivial over $R \otimes_{\O} \widetilde{\O}$.
    By Proposition \ref{Proposition:G display with trivial Hodge filtration is banal}, the base change of
    $\mathcal{Q}$ to $(A_{\widetilde{\O}}, IA_{\widetilde{\O}})$ is banal.
\end{proof}

\begin{proof}[Proof of Proposition \ref{Proposition:deformation of isomorphism}]
    By Lemma \ref{Lemma:banal after finite field extension}, there exists a finite Galois extension $\widetilde{k}$ of $k$ such that the base changes of
    $\mathcal{Q}_1$  
    and $\mathcal{Q}_2$
    to $(A_{\widetilde{\O}}, IA_{\widetilde{\O}})$ are banal.
    Here $\widetilde{\O}:= W(\widetilde{k}) \otimes_{W(\F_q)} \O_E$ and $A_{\widetilde{\O}}:=A \otimes_\O \widetilde{\O}$; we use the same notation for $\O$-algebras.
    We can identify
    $(A'_{\widetilde{\O}}, I'A'_{\widetilde{\O}})$
    with a coproduct of two (resp.\ three) copies of
    $(A_{\widetilde{\O}}, IA_{\widetilde{\O}})$ in $(R_{\widetilde{\O}})_{\Prism, \O_E}$
    if $A'=A^{(2)}$
    (resp.\ if $A'=A^{(3)}$).
    By Galois descent for $G$-$\mu$-displays,
    it suffices to prove the same statement for banal
    $G$-$\mu$-displays over $(A_{\widetilde{\O}}, IA_{\widetilde{\O}})$.
    We may thus assume without loss of generality that $\mathcal{Q}_1$ and $\mathcal{Q}_2$ are banal $G$-$\mu$-displays over $(A, I)$.

    If
    $\mathcal{Q}_1$  
    and $\mathcal{Q}_2$ are not isomorphic to each other, then the assertion holds trivially.
    Thus, we may assume that 
    $\mathcal{Q}_1=\mathcal{Q}_2=\mathcal{Q}_{Y}$
    for some $Y \in G(A)_I$.
    Let $d:=f_2(\mathcal{E})$.
    We have $f_1(\mathcal{E})=u d$ for some $u \in A'^\times$.
    With the choice of $d \in I'$,
    the $G$-$\mu$-displays
    $f^*_1(\mathcal{Q}_Y)$ and
    $f^*_2(\mathcal{Q}_Y)$
    correspond to the elements
    $f_1(Y_{\mathcal{E}})\phi(\mu(u)), f_2(Y_{\mathcal{E}}) \in G(A')_d$, respectively.
    Thus we can identify
    $\Hom_{G\mathchar`-\mathrm{Disp}_\mu(A', I')}(f^*_1(\mathcal{Q}_Y), f^*_2(\mathcal{Q}_Y))$
    with the set
    \[
    \{\, g \in G_\mu(A', I') \, \vert \, g^{-1}f_2(Y_{\mathcal{E}}) \sigma_{\mu, d}(g)=f_1(Y_{\mathcal{E}})\phi(\mu(u)) \, \}.
    \]

    We set $X:=f_2(Y_{\mathcal{E}})$.
    We shall prove that the map
    (\ref{equation:reduction map for isomorphisms})
    is injective.
    Let $g, h \in G_\mu(A', I')$ be two elements in $\Hom_{G\mathchar`-\mathrm{Disp}_\mu(A', I')}(f^*_1(\mathcal{Q}_1), f^*_2(\mathcal{Q}_2))$
    such that $m(g)=m(h)$ in $G_\mu(A, I)$.
    We set $\beta:=gh^{-1}$.
    Since $m(\beta)=1$, we have
    $\mu(d)\beta\mu(d)^{-1} \in G(K)$.
    It then follows from $\phi(K) \subset dK$ that $\sigma_{\mu, d}(\beta) \in G(dK)$.
    The equalities
    \[
    g^{-1}X \sigma_{\mu, d}(g) = f_1(Y_{\mathcal{E}})\phi(\mu(u)) = h^{-1}X \sigma_{\mu, d}(h)
    \]
    imply that $\beta=X \sigma_{\mu, d}(\beta) X^{-1}$.
    It follows that $\beta \in G(dK)$, and we have $\mathcal{V}_{d, X}(\beta)=1$ for the map $\mathcal{V}_{d, X} \colon G(dK) \to G(dK)$.
    Since $\mathcal{V}_{d, X}$ is bijective by Proposition \ref{Proposition:bijectivity of VdX for G(dK)}, we obtain $\beta=1$.
    
    It remains to prove that the map
    (\ref{equation:reduction map for isomorphisms})
    is surjective.
    For this, it suffices to prove that
    $\Hom_{G\mathchar`-\mathrm{Disp}_\mu(A', I')}(f^*_1(\mathcal{Q}_Y), f^*_2(\mathcal{Q}_Y))$
    is not empty.
    (Once we have obtained an isomorphism
    $g \colon f^*_1(\mathcal{Q}_Y) \overset{\sim}{\to} f^*_2(\mathcal{Q}_Y)$, we can write any isomorphism
    $h \colon \mathcal{Q}_Y \overset{\sim}{\to} \mathcal{Q}_Y$
    as $m^{*}(f^{*}_2(h \circ m^{*}(g^{-1})) \circ g)$.)
    We claim that $\phi(\mu(u)) \in G(dK)$ and $\gamma:=f_2(Y_\mathcal{E})^{-1}f_1(Y_\mathcal{E}) \in G(dK)$.
    Indeed, since $m(u)=1$, we have $\mu(u) \in G(K)$, which in turn implies that $\phi(\mu(u)) \in G(dK)$.
    Since the morphisms $f_1$ and $f_2$ induce the same homomorphism
    $R \to A'/I'$, we see that $\gamma \in G(I')$.
    Using that $I' \cap K=dK$, we then obtain
    $\gamma \in G(dK)$.
    Since $\mathcal{V}_{d, X} \colon G(dK) \to G(dK)$ is bijective, there exists an element $g \in G(dK)$ such that
    $\mathcal{V}_{d, X}(g^{-1})=X\phi(\mu(u))^{-1}\gamma^{-1} X^{-1}$, or equivalently
    \[
    g^{-1}f_2(Y_{\mathcal{E}}) \sigma_{\mu, d}(g)=f_1(Y_{\mathcal{E}})\phi(\mu(u)).
    \]
    In other words, the element $g$ gives an isomorphism
    $f^*_1(\mathcal{Q}_Y) \overset{\sim}{\to} f^*_2(\mathcal{Q}_Y)$.
\end{proof}

\subsection{Proof of Theorem \ref{Theorem:main result on G displays over complete regular local rings}}\label{Subsection:proof of main result}

In this section, we prove Theorem \ref{Theorem:main result on G displays over complete regular local rings} using our previous results.

As in Section \ref{Subsection:Coproducts of Breuil--Kisin prisms},
we write $(A, I)=(\mathfrak{S}_\O, (\mathcal{E}))$.
Let
\[
G\mathchar`-\mathrm{Disp}^\mathrm{DD}_\mu(A, I)
\]
be the groupoid of pairs $(\mathcal{Q}, \epsilon)$ consisting of a $G$-$\mu$-display
$\mathcal{Q}$
over $(A, I)$
and an isomorphism
$\epsilon \colon p^*_1\mathcal{Q} \overset{\sim}{\to} p^*_2\mathcal{Q}$
of $G$-$\mu$-displays
over $(A^{(2)}, I^{(2)})$
satisfying the cocycle condition $p^*_{13}\epsilon=p^*_{23}\epsilon \circ p^*_{12}\epsilon$.
An isomorphism $(\mathcal{Q}, \epsilon) \overset{\sim}{\to} (\mathcal{Q}', \epsilon')$ is an isomorphism $f \colon \mathcal{Q} \overset{\sim}{\to} \mathcal{Q}'$ of $G$-$\mu$-displays
over $(A, I)$ such that $\epsilon' \circ (p^*_1 f)=(p^*_2 f) \circ \epsilon$.

For a prismatic $G$-$\mu$-display $\mathfrak{Q}$ over $R$,
we have the associated isomorphism
\[
\gamma_{p_i} \colon p^*_i(\mathfrak{Q}_{(A, I)}) \overset{\sim}{\to} \mathfrak{Q}_{(A^{(2)}, I^{(2)})}
\]
for $i=1, 2$.
Let $\epsilon:=\gamma^{-1}_{p_2} \circ \gamma_{p_1}$.
Then $\epsilon$ satisfies the cocycle condition, so that the pair $(\mathfrak{Q}_{(A, I)}, \epsilon)$
is an object of $\mathrm{Disp}^\mathrm{DD}_\mu(A, I)$.
This construction induces a functor
\[
G\mathchar`-\mathrm{Disp}_\mu((R)_{\Prism, \O_E}) \to G\mathchar`-\mathrm{Disp}^\mathrm{DD}_\mu(A, I), \quad \mathfrak{Q} \mapsto (\mathfrak{Q}_{(A, I)}, \epsilon).
\]

\begin{prop}\label{Proposition:crystal and descent datum}
    The functor
    $G\mathchar`-\mathrm{Disp}_\mu((R)_{\Prism, \O_E}) \to G\mathchar`-\mathrm{Disp}^\mathrm{DD}_\mu(A, I)$
    is an equivalence.
\end{prop}

\begin{proof}
    This is a formal consequence of Proposition \ref{Proposition:flat descent of G display} and Proposition \ref{Proposition:weakely initial object}.
    %We briefly explain how to construct an inverse to the above functor.
    %Let $(\mathcal{Q}, \epsilon) \in G\mathchar`-\mathrm{Disp}^\mathrm{DD}_\mu(A, I)$.
    %For an object $(C, J) \in (R)_{\Prism, \O_E}$, we choose a flat cover
    %$(C, J) \to (C', J')$ such that $(C', J')$ admits a morphism
    %$h \colon (A, I) \to (C', J')$ in $(R)_{\Prism, \O_E}$; see Proposition \ref{Proposition:weakely initial object}.
    %Let $(C'', J'')$ be a pushout of the diagram
    %$(C', J') \leftarrow (C, J) \rightarrow (C', J')$, and let
    %$p'_1, p'_2 \colon (C', J') \to (C'', J'')$ denote the two natural maps.
    %Then $\epsilon$ gives rise to an isomorphism
    %$\epsilon' \colon (p'_1)^*h^*\mathcal{Q} \overset{\sim}{\to} (p'_2)^*h^*\mathcal{Q}$, which satisfies the usual cocycle condition.
    %Thus, it follows from Proposition \ref{Proposition:flat descent of G display} that
    %the pair $(h^*\mathcal{Q}, \epsilon')$ arises from
    %a $G$-$\mu$-display
    %$\mathfrak{Q}_{(C, J)}$ over $(C, J)$.
    %One can show that
    %$\mathfrak{Q}_{(C, J)}$ does not depend on the choices of $(C', J')$ and $h$ up to canonical isomorphism, and for each morphism
    %$f \colon (C_1, J_1) \to (C_2, J_2)$ in $(R)_{\Prism, \O_E}$,
    %we have a canonical isomorphism
    %$\gamma_f \colon f^*(\mathfrak{Q}_{(C_1, J_1)}) \overset{\sim}{\to} \mathfrak{Q}_{(C_2, J_2)}$
    %such that
    %$\gamma_{f'} \circ ({f'}^*\gamma_f) = \gamma_{f' \circ f}$.
    %In this way, we obtain a $G$-$\mu$-display over %$(R)_{\Prism, \O_E}$.
    %It is straightforward to show that this construction gives an inverse to the above functor.
\end{proof}

\begin{proof}[Proof of Theorem \ref{Theorem:main result on G displays over complete regular local rings}]
We assume that $\mu$ is 1-bounded.
By virtue of Proposition \ref{Proposition:crystal and descent datum},
it suffices to show that the forgetful functor
\[
G\mathchar`-\mathrm{Disp}^\mathrm{DD}_\mu(A, I) \to G\mathchar`-\mathrm{Disp}_\mu(A, I)
\]
is an equivalence.
Let $m \colon (A^{(2)}, I^{(2)}) \to (A, I)$
be the unique morphism in $(R)_{\Prism, \O_E}$ such that
$m \circ p_i = \id_{(A, I)}$ for $i=1, 2$, and
let
$m' \colon (A^{(3)}, I^{(3)}) \to (A, I)$ be the unique morphism in $(R)_{\Prism, \O_E}$ such that $m \circ q_i = \id_{(A, I)}$ for $i=1, 2, 3$.
Let $\mathcal{Q}$ be a $G$-$\mu$-display over $(A, I)$.
We claim that an isomorphism $\epsilon \colon p^*_1\mathcal{Q} \overset{\sim}{\to} p^*_2\mathcal{Q}$ satisfies the cocycle condition $p^*_{13}\epsilon=p^*_{23}\epsilon \circ p^*_{12}\epsilon$
if and only if $m^*\epsilon=\id_{\mathcal{Q}}$.
Indeed, since the map
\[
m'^* \colon \Hom_{G\mathchar`-\mathrm{Disp}_\mu(A^{(3)}, I^{(3)})}(q^*_1\mathcal{Q}, q^*_3\mathcal{Q}) \to \Hom_{G\mathchar`-\mathrm{Disp}_\mu(A, I)}(\mathcal{Q}, \mathcal{Q})
\]
is bijective by Proposition \ref{Proposition:deformation of isomorphism},
we see that $\epsilon$ satisfies the cocycle condition $p^*_{13}\epsilon=p^*_{23}\epsilon \circ p^*_{12}\epsilon$
if and only if $m^*\epsilon=m^*\epsilon \circ m^*\epsilon$, which is equivalent to saying that $m^*\epsilon=\id_{\mathcal{Q}}$.

By Proposition \ref{Proposition:deformation of isomorphism},
the map
\[
m^* \colon \Hom_{G\mathchar`-\mathrm{Disp}_\mu(A^{(2)}, I^{(2)})}(p^*_1\mathcal{Q}, p^*_2\mathcal{Q}) \to \Hom_{G\mathchar`-\mathrm{Disp}_\mu(A, I)}(\mathcal{Q}, \mathcal{Q})
\]
is bijective.
Therefore,
for any $G$-$\mu$-display $\mathcal{Q}$ over $(A, I)$,
there exists a unique isomorphism
$\epsilon \colon p^*_1\mathcal{Q} \overset{\sim}{\to} p^*_2\mathcal{Q}$
satisfying the cocycle condition $p^*_{13}\epsilon=p^*_{23}\epsilon \circ p^*_{12}\epsilon$, and $\epsilon$ is characterized by the condition that $m^*\epsilon=\id_{\mathcal{Q}}$.
It follows that the forgetful functor
$G\mathchar`-\mathrm{Disp}^\mathrm{DD}_\mu(A, I) \to G\mathchar`-\mathrm{Disp}_\mu(A, I)$
is an equivalence.
\end{proof}

\section{$p$-divisible groups and prismatic Dieudonn\'e crystals}\label{Section:p-divisible groups and prismatic Dieudonn\'e crystals}

In this section, we make a few remarks on prismatic Dieudonn\'e crystals, which are introduced in \cite{Anschutz-LeBras}.

\subsection{A remark on prismatic Dieudonn\'e crystals}\label{Subsection:A remark on prismatic Dieudonne crystals}

Let $R$ be a $\pi$-adically complete $\O_E$-algebra.
Recall the sheaf
$\O_\Prism$
on the site $(R)^{\op}_{\Prism, \O_E}$
from Remark \ref{Remark:structure sheaf}.

We say that an $\O_\Prism$-module $\mathcal{M}$ on $(R)^{\op}_{\Prism, \O_E}$ is a
\textit{prismatic crystal in vector bundles}
if $\mathcal{M}(A, I)$
is a finite projective $A$-module for any $(A, I) \in (R)_{\Prism, \O_E}$, and for any morphism $(A, I) \to (A', I')$ in $(R)_{\Prism, \O_E}$, the natural homomorphism
\[
\mathcal{M}(A, I) \otimes_A A' \to \mathcal{M}(A', I')
\]
is bijective.
A
\textit{prismatic Dieudonn\'e crystal}
on $(R)^{\op}_{\Prism, \O_E}$ (or on $(R)_{\Prism, \O_E}$)
is a prismatic crystal $\mathcal{M}$ in vector bundles on $(R)^{\op}_{\Prism, \O_E}$ equipped with
a $\phi$-linear homomorphism
\[
\varphi_\mathcal{M} \colon \mathcal{M} \to \mathcal{M}
\]
such that for any $(A, I) \in (R)_{\Prism, \O_E}$,
the finite projective $A$-module $\mathcal{M}(A, I)$
with
the linearization
$1 \otimes \varphi_\mathcal{M} \colon \phi^*(\mathcal{M}(A, I)) \to \mathcal{M}(A, I)$
is a minuscule Breuil--Kisin module over $(A, I)$ in the sense of Definition \ref{Definition:displayed and minuscule Breuil-Kisin module} (see also Proposition \ref{Proposition:minuscule equivalent condition}).
For a bounded $\O_E$-prism $(A, I)$, let
\[
\mathrm{BK}_{\mathrm{min}}(A, I)
\]
be the category of minuscule Breuil--Kisin modules over $(A, I)$.
Then the category of prismatic Dieudonn\'e crystals on $(R)_{\Prism, \O_E}$ is equivalent to the category
\[
{2-\varprojlim}_{(A, I) \in (R)_{\Prism, \O_E}} \mathrm{BK}_{\mathrm{min}}(A, I).
\]

As in Section \ref{Section:G-displays over complete regular local rings},
let $R$ be a complete regular local ring 
over $\O$ with residue field $k$.
Let
$
(\mathfrak{S}_\O, (\mathcal{E}))
$
be
an $\O_E$-prism
of Breuil--Kisin type, where $\mathfrak{S}_\O=\O[[t_1, \dotsc, t_n]]$,
with an isomorphism
$
R \simeq \mathfrak{S}_\O/\mathcal{E}
$ over $\O$.
By using the results of Section \ref{Section:G-displays over complete regular local rings}, we can prove the following proposition, which is obtained in the proof of \cite[Theorem 5.12]{Anschutz-LeBras} if $n \leq 1$ (and $\O_E=\Z_p$).

\begin{prop}\label{Proposition:evaluation prismatic dieudonne crystal equivalence}
    The functor $\mathcal{M} \mapsto \mathcal{M}(\mathfrak{S}_\O, (\mathcal{E}))$ 
    from the category of prismatic Dieudonn\'e crystals on $(R)_{\Prism, \O_E}$ to the category $\mathrm{BK}_{\mathrm{min}}(\mathfrak{S}_\O, (\mathcal{E}))$
    is an equivalence.
\end{prop}

\begin{proof}
    This follows from Corollary \ref{Corollary:GLn BK module of type mu}, Theorem \ref{Theorem:main result on G displays over complete regular local rings}, and the following fact: A functor of additive categories is an equivalence if and only if it induces an equivalence of the associated groupoids.
    (This fact follows since homomorphisms $f \colon X \to Y$ in an additive category can be completely described in terms of automorphisms of $X \oplus Y$ by considering $\begin{pmatrix}
\id_X & 0 \\
f & \id_Y \\
\end{pmatrix}$.)
\end{proof}

\subsection{Quasisyntomic rings}\label{Subsection:Quasisyntomic rings}

In the following (and in Section \ref{Section:Comparison with prismatic $F$-gauges} below), we will need the notions of
\textit{quasisyntomic rings} in the sense of \cite[Definition 4.10]{BMS2}
and
\textit{quasiregular semiperfectoid rings} in the sense of \cite[Definition 4.20]{BMS2}.
Let
\[
\mathrm{QSyn}
\]
be the category of quasisyntomic rings
and
let
\[
\mathrm{QRSPerfd} \subset \mathrm{QSyn}
\]
be the full subcategory spanned by quasiregular semiperfectoid rings.
We endow both
$\mathrm{QSyn}^{\op}$
and
$\mathrm{QRSPerfd}^{\op}$
with the quasisyntomic topology, i.e.\ the topology generated by the quasisyntomic coverings; see \cite[Definition 4.10]{BMS2}.
We will assume that the reader is familiar with basic properties of
$\mathrm{QSyn}$
and
$\mathrm{QRSPerfd}$ discussed in \cite[Section 4]{BMS2}.
Here we just recall that
quasiregular semiperfectoid rings form a basis for $\mathrm{QSyn}$; see \cite[Lemma 4.28]{BMS2}.

\begin{ex}
    A $p$-adically complete regular local ring is a quasisyntomic ring (see \cite[Example 3.17]{Anschutz-LeBras}).
    A perfectoid ring is a quasiregular semiperfectoid ring (see \cite[Example 4.24]{BMS2}).
\end{ex}

\begin{rem}\label{Remark:definition of prismatic Dieudonne crystal in ALB}
Let $R \in \qsyn$ be a quasisyntomic ring.
In \cite[Definition 4.5]{Anschutz-LeBras}, Ansch\"utz--Le Bras defined prismatic Dieudonn\'e crystals over $R$ as sheaves on the quasisyntomic site of $R$.
By virtue of \cite[Proposition 4.4]{Anschutz-LeBras},
the category of prismatic Dieudonn\'e crystals on $(R)_{\Prism}$ in our sense is equivalent to the category of prismatic Dieudonn\'e crystals over $R$ in the sense of \cite[Definition 4.5]{Anschutz-LeBras}.
\end{rem}

\subsection{$p$-divisible groups and minuscule Breuil-Kisin modules}\label{Subsection:p-divisible groups and minuscule Breuil-Kisin modules}

In this subsection, we consider the case where $\O_E=\Z_p$.
Let $R$ be a $p$-adically complete ring,
and
let $\mathcal{G}$ be a $p$-divisible group over $\Spec R$.
We define the following functors
\begin{align*}
    \underline{\mathcal{G}} &\colon (R)_\Prism \to \mathrm{Set}, \quad (A, I) \mapsto \mathcal{G}(A/I), \\
    \underline{\mathcal{G}[p^n]} &\colon (R)_\Prism \to \mathrm{Set}, \quad (A, I) \mapsto \mathcal{G}[p^n](A/I).
\end{align*}
These functors form sheaves on the site $(R)^{\op}_\Prism$.
In \cite[Proposition 4.69]{Anschutz-LeBras}, it is proved that
the $\O_\Prism$-module
\[
\mathcal{E}xt^1_{(R)_\Prism}(\underline{\mathcal{G}}, \O_\Prism)
\]
on $(R)^{\op}_\Prism$
is a prismatic crystal in vector bundles.
(Here we simply write $\mathcal{E}xt^1_{(R)_\Prism}(\underline{\mathcal{G}}, \O_\Prism)$ rather than $\mathcal{E}xt^1_{(R)^{\op}_\Prism}(\underline{\mathcal{G}}, \O_\Prism)$.)

\begin{rem}\label{Remark:evaluation local Ext groups}
\
\begin{enumerate}
    \item For an integer $n \geq 1$, the map $[p^{n}] \colon \underline{\mathcal{G}} \to \underline{\mathcal{G}}$ induced by multiplication by $p^n$ is surjective.
    This follows from \cite[Corollary 3.25]{Anschutz-LeBras}.
    \item We have $\mathcal{H}om_{(R)_\Prism}(\underline{\mathcal{G}}, \O_\Prism)=0$.
    Indeed, since $[p] \colon \underline{\mathcal{G}} \to \underline{\mathcal{G}}$ is surjective and the topos associated with $(R)^{\op}_\Prism$ is replete in the sense of \cite[Definition 3.1.1]{Bhatt-ScholzeProetale},
    the projection $\varprojlim_{[p]} \underline{\mathcal{G}} \to \underline{\mathcal{G}}$ is surjective.
    Since
    $\O_\Prism(A, I)=A$ is $p$-adically complete for any $(A, I) \in (R)_\Prism$, we can conclude that $\mathcal{H}om_{(R)_\Prism}(\underline{\mathcal{G}}, \O_\Prism)=0$.
    As a consequence,
    the local-to-global spectral sequence implies that
    \[
    \mathrm{Ext}^1_{(A, I)_\Prism}(\underline{\mathcal{G}}, \O_\Prism) \overset{\sim}{\to} \mathcal{E}xt^1_{(R)_\Prism}(\underline{\mathcal{G}}, \O_\Prism)(A, I)
    \]
    for any $(A, I) \in (R)_\Prism$.
    Here we regard the site 
    $(A, I)^{\op}_\Prism$
    as the localization of $(R)^{\op}_\Prism$ at $(A, I)$, and the restriction of $\underline{\mathcal{G}}$ to $(A, I)^{\op}_{\Prism}$ is denoted by the same symbol.
    In particular
    $\mathrm{Ext}^1_{(A, I)_\Prism}(\underline{\mathcal{G}}, \O_\Prism)$
    is a finite projective $A$-module and its formation commutes with base change along any morphism $(A, I) \to (A', I')$ in $(R)_\Prism$.
\end{enumerate}
\end{rem}

We assume that $R$ is quasisyntomic.
In \cite[Theorem 4.71]{Anschutz-LeBras}, it is proved that
$\mathcal{E}xt^1_{(R)_\Prism}(\underline{\mathcal{G}}, \O_\Prism)$
with
the $\phi$-linear homomorphism
$\mathcal{E}xt^1_{(R)_\Prism}(\underline{\mathcal{G}}, \O_\Prism) \to \mathcal{E}xt^1_{(R)_\Prism}(\underline{\mathcal{G}}, \O_\Prism)$
induced by the Frobenius $\phi \colon \O_\Prism \to \O_\Prism$
is a prismatic Dieudonn\'e crystal.
More precisely, they showed that $\mathcal{E}xt^1_{(R)_\Prism}(\underline{\mathcal{G}}, \O_\Prism)$ is \textit{admissible} in the sense of \cite[Definition 4.5]{Anschutz-LeBras}.
(See also Remark \ref{Remark:definition of prismatic Dieudonne crystal in ALB}.)
We shall recall the argument.

\begin{prop}[{\cite[Theorem 4.71]{Anschutz-LeBras}}]\label{Proposition:prismatic Dieudonne crystal of p-divisivle group}
Let $R \in \qsyn$ and $(A, I) \in (R)_\Prism$.
We write $M:=\mathrm{Ext}^1_{(A, I)_\Prism}(\underline{\mathcal{G}}, \O_\Prism)$.
Then $M$ with the induced homomorphism
$
F_M \colon \phi^*M \to M
$
is a minuscule Breuil--Kisin module over $(A, I)$.
\end{prop}

\begin{proof}
    By the fact that the formation of $\mathrm{Ext}^1_{(A, I)_\Prism}(\underline{\mathcal{G}}, \O_\Prism)$ commutes with base change along any morphism $(A, I) \to (A', I')$ (see Remark \ref{Remark:evaluation local Ext groups}) and Corollary \ref{Corollary:flat descent for dispalyed BK modules}, the assertion can be checked $(p, I)$-completely flat locally.
    Let $R \to R'$ be a quasisyntomic covering with $R'$ a quasiregular semiperfectoid ring.
    Applying \cite[Proposition 7.11]{BS} to the $p$-adic completion of $A/I \otimes_{R} R'$, which is a quasisyntomic covering of $A/I$,
    we can find a flat covering
    $(A, I) \to (A', I')$
    in $(R)_\Prism$ such that there exists a homomorphism $R' \to A'/I'$ over $R$.
    After replacing $R$ by $R'$ and replacing $(A, I)$ by $(A', I')$,
    we may assume that $R$ is a quasiregular semiperfectoid ring.
    Then, by choosing a surjective homomorphism from a perfectoid ring to $R$ and using \cite[Corollary 2.10, Lemma 4.70]{Anschutz-LeBras}, we may assume that $R$ is a perfectoid ring and $(A, I)=(W(R^\flat), I_R)$.
    (Here we regard $(W(R^\flat), I_R)$ as an object of $(R)_\Prism$ via the homomorphism $\theta \colon W(R^\flat) \to R$.
    In \cite{Anschutz-LeBras}, the composition $\theta \circ \phi^{-1}$ is used instead.)
    
    Let $\xi \in I_R$ be a generator.
    By Proposition \ref{Proposition:minuscule equivalent condition}, it suffices to prove that the cokernel of $F_M$ is killed by $\xi$.
    By Remark \ref{Remark:arc topology} (4) and $p$-complete $\arc$-descent (Proposition \ref{Proposition:arc descent for finite projective modules}),
    we may further assume that $R$ is a $p$-adically complete valuation ring of rank $\leq 1$ with algebraically closed fraction field.

    If $p=0$ in $R$, then $R$ is perfect by Example \ref{Example:algebraically closed valuation ring is perfectoid}.
    In this case, the Frobenius $F_M$ can be identified with
    the homomorphism
    \[
    \mathrm{Ext}^1_{(A, I)_\Prism}(\underline{(\phi^*\mathcal{G})}, \O_\Prism) \to \mathrm{Ext}^1_{(A, I)_\Prism}(\underline{\mathcal{G}}, \O_\Prism)
    \]
    induced by the relative Frobenius $\mathcal{G} \to \phi^*\mathcal{G}$.
    Thus,
    the Verschiebung homomorphism
    $\phi^*\mathcal{G} \to \mathcal{G}$
    induces a $W(R)$-linear homomorphism
    $V_M \colon M \to \phi^*M$
    such that $F_M \circ V_M=p$, which in turn implies the assertion.

    It remains to treat the case where $p$ is a nonzerodivisor in $R$, so that $R$ is the ring of integers $\O_C$ of an algebraically closed nonarchimedean extension $C$ of $\Q_p$.
    We set
    \[
    M_n:=\mathrm{Ext}^1_{(W(\O^\flat_C), I_{\O_C})_\Prism}(\underline{\mathcal{G}[p^n]}, \O_\Prism).
    \]
    By the proof of \cite[Proposition 4.69]{Anschutz-LeBras}, the natural homomorphism $M \to M_n$ induces an isomorphism
    $M/p^n \overset{\sim}{\to} M_n$ for any $n \geq 1$.
    In particular, we obtain
    \[
    M \overset{\sim}{\to} \varprojlim_n M_n
    \]
    and $M_n$ is a free $W_n(\O^\flat_C)$-module of finite rank.
    We claim that the cokernel of the Frobenius
    $
    F_{M_n} \colon \phi^*M_n \to M_n
    $
    is killed by $\xi$.
    Indeed, there is an embedding $\mathcal{G}[p^n] \hookrightarrow X$ into an abelian scheme $X$ over $\Spec \O_C$;
    see \cite[Th\'eor\`eme 3.1.1]{BBM}.
    Let $Y$ be the $p$-adic completion of $X$, which is a smooth $p$-adic formal scheme over $\Spf \O_C$.
    It follows from the proofs of \cite[Theorem 4.62, Proposition 4.66]{Anschutz-LeBras} that
    there exists a surjection
    \[
    H^1_\Prism(Y/W(\O^\flat_C)) \to M_n
    \]
    which is compatible with Frobenius homomorphisms.
    Here $H^1_\Prism(Y/W(\O^\flat_C))$ is the first prismatic cohomology of $Y$ (with respect to $(W(\O^\flat_C), I_{\O_C})$) defined in \cite{BS}.
    By \cite[Theorem 1.8 (6)]{BS}, the cokernel of the Frobenius
    \[
    \phi^*H^1_\Prism(Y/W(\O^\flat_C)) \to H^1_\Prism(Y/W(\O^\flat_C))
    \]
    is killed by $\xi$, which in turn implies the claim.
    Since the image of $\xi$ in $W_n(\O^\flat_C)$ is a nonzerodivisor,
    it follows that
    $
    F_{M_n}
    $
    is injective.
    Since $F_M= \varprojlim_n F_{M_n}$, we can conclude that the cokernel of $F_M$ is killed by $\xi$.
\end{proof}

\begin{rem}\label{Remark:proof difference}
Our proof of Proposition \ref{Proposition:prismatic Dieudonne crystal of p-divisivle group} in the case where $A=\O_C$ is slightly different from that given in \cite{Anschutz-LeBras}.
For example, we do not use \cite[Proposition 14.9.4]{Scholze-Weinstein} (see the proof of \cite[Proposition 4.48]{Anschutz-LeBras}).
\end{rem}

Finally, we recall the following classification theorem for $p$-divisible groups given in \cite{Anschutz-LeBras}.
Let
$R$ be a complete regular local ring with perfect residue field $k$ of characteristic $p$.
Let
$(\mathfrak{S}, (\mathcal{E}))$
be a prism of Breuil--Kisin type, where $\mathfrak{S}:=W(k)[[t_1, \dotsc, t_n]]$,
with an isomorphism
$R \simeq \mathfrak{S}/\mathcal{E}$
which lifts $\id_k \colon k \to k$.

\begin{thm}[{Ansch\"utz--Le Bras \cite[Theorem 4.74, Theorem 5.12]{Anschutz-LeBras}}]\label{Theorem:classification theorem for p-divisible group}
\ 
\begin{enumerate}
    \item The contravariant functor
\[
\{ \, p\text{-divisible groups over} \ \Spec R \, \} \to \{ \, \text{prismatic Dieudonn\'e crystals on} \ (R)_\Prism \, \}
\]
defined by 
$
\mathcal{G} \mapsto \mathcal{E}xt^1_{(R)_\Prism}(\underline{\mathcal{G}}, \O_\Prism)
$
is an anti-equivalence of categories.
    \item The contravariant functor
\[
\{\,  p\text{-divisible groups over} \ \Spec R \, \} \to \{\,  \text{minuscule Breuil--Kisin modules over} \ (\mathfrak{S}, (\mathcal{E})) \, \}
\]
defined by 
$
\mathcal{G} \mapsto \mathrm{Ext}^1_{(\mathfrak{S}, (\mathcal{E}))_\Prism}(\underline{\mathcal{G}}, \O_\Prism)
$
is an anti-equivalence of categories.
\end{enumerate}
\end{thm}

\begin{proof}
    (1) This is a consequence of \cite[Theorem 4.74, Proposition 5.10]{Anschutz-LeBras}.

    (2) The assertion follows from (1) and Proposition \ref{Proposition:evaluation prismatic dieudonne crystal equivalence}.
    This result was already stated in \cite[Theorem 5.12]{Anschutz-LeBras}, and the proof was given in the case where $n \leq 1$.
\end{proof}

\section{Comparison with prismatic $F$-gauges}\label{Section:Comparison with prismatic $F$-gauges}

For the sake of completeness,
we discuss the relation between our prismatic $G$-$\mu$-displays and \textit{prismatic $F$-gauges} introduced by Drinfeld and Bhatt--Lurie (cf.\ \cite[1.8.1]{Drinfeld22}, \cite{BL}, \cite{BL2}, \cite[Definition 6.1.1]{BhattGauge}).
For simplicity,
we assume that
$\O_E=\Z_p$ throughout this section, and we restrict ourselves to the case where base rings $R$ are quasisyntomic.
In this case, Guo-Li in \cite{Guo-Li} study prismatic $F$-gauges over $R$ in a slightly different way, without using the original stacky approach.
Here we follow the approach employed in \cite{Guo-Li}.
In Section \ref{Subsection:Prismatic F-gauges in vector bundles}, we compare prismatic $F$-gauges in vector bundles with displayed Breuil--Kisin modules.
In Section \ref{Subsection:Prismatic G-F-gauges of type mu},
we introduce
\textit{prismatic $G$-$F$-gauges of type $\mu$}
and 
explain their relation to prismatic $G$-$\mu$-displays.

We work with
the category
$\mathrm{QSyn}$
of quasisyntomic rings
and
the full subcategory
$\mathrm{QRSPerfd} \subset \mathrm{QSyn}$
spanned by quasiregular semiperfectoid rings (see Section \ref{Subsection:Quasisyntomic rings}).

\subsection{Prismatic $F$-gauges in vector bundles}\label{Subsection:Prismatic F-gauges in vector bundles}

We recall the definition of prismatic $F$-gauges in vector bundles over quasisyntomic rings,
following \cite{Guo-Li}.

Let $S \in \qrsperfd$ be a quasiregular semiperfectoid ring.
By \cite[Proposition 7.10]{BS},
the category $(S)_\Prism$ admits an initial object
\[
(\Prism_S, I_S) \in (S)_\Prism.
\]
Moreover the bounded prism $(\Prism_S, I_S)$ is orientable.
We often omit the subscript and simply write $I=I_S$.
Following \cite[Definition 12.1]{BS}, we define
\[
\Fil^i_\mathcal{N}(\Prism_S):= \{ \, x \in \Prism_S \, \vert \, \phi(x) \in I^i\Prism_S \, \}
\]
for a non-negative integer $i \geq 0$.
For a negative integer $i < 0$, we set $\Fil^i_\mathcal{N}(\Prism_S)=\Prism_S$.
The filtration
$\{ \Fil^i_\mathcal{N}(\Prism_S) \}_{i \in \Z}$ is called the \textit{Nygaard filtration}.
We recall the following terminology from \cite[Section 5.5]{BhattGauge}.

\begin{defn}\label{Definition:Rees algebra for Nygaard filtration}
    The \textit{extended Rees algebra}
    $\Rees(\Fil^\bullet_\mathcal{N}(\Prism_S))$
    of the Nygaard filtration $\{ \Fil^i_\mathcal{N}(\Prism_S) \}_{i \in \Z}$ is defined by
    \[
    \Rees(\Fil^\bullet_\mathcal{N}(\Prism_S)):= \bigoplus_{i \in \Z} \Fil^i_\mathcal{N}(\Prism_S)t^{-i} \subset \Prism_S[t, t^{-1}].
    \]
    We view $\Rees(\Fil^\bullet_\mathcal{N}(\Prism_S))$ as a graded ring, where the degree of $t$ is $-1$.
    Let
    \[
    \tau \colon \Rees(\Fil^\bullet_\mathcal{N}(\Prism_S)) \to \Prism_S
    \]
    be the homomorphism of $\Prism_S$-algebras defined by $t \mapsto 1$.
    We consider the graded ring
    $\bigoplus_{i \in \Z} I^{i}t^{-i} \subset \Prism_S[1/I][t, t^{-1}]$.
    Let
    \[
    \sigma \colon \Rees(\Fil^\bullet_\mathcal{N}(\Prism_S)) \to \bigoplus_{i \in \Z} I^{i}t^{-i}
    \]
    be the graded homomorphism
    defined by
    $a_it^{-i} \mapsto \phi(a_i)t^{-i}$ for any $i \in \Z$.
    %(See also the beginning of Section \ref{Subsection:Displayed Breuil--Kisin modules} for the notation $I^{i}$.)
\end{defn}

\begin{rem}\label{Remark:sign convention grading of t}
For the grading of $\Rees(\Fil^\bullet_\mathcal{N}(\Prism_S))$, our sign convention is opposite to that of \cite{BhattGauge}, where the degree of $t$ is defined to be $1$.
Our grading is chosen to be consistent with the convention of \cite{Lau21}.
\end{rem}

\begin{defn}[Drinfeld, Bhatt--Lurie]\label{Definition:prismatic F-gauge in vector bundle over qrsp}
    Let $S \in \qrsperfd$.
    A \textit{prismatic $F$-gauge in vector bundles} over $S$ is a pair $(N, F_N)$ consisting of a graded $\Rees(\Fil^\bullet_\mathcal{N}(\Prism_S))$-module $N$ which is finite projective as a $\Rees(\Fil^\bullet_\mathcal{N}(\Prism_S))$-module, and an isomorphism
    \[
    F_N \colon (\sigma^*N)_0 \overset{\sim}{\to} \tau^*N
    \]
    of $\Prism_S$-modules.
    Here $(\sigma^*N)_0$ is the degree $0$ part of the graded $\bigoplus_{i \in \Z} I^{i}t^{-i}$-module
    $
    \sigma^*N.
    $
\end{defn}

Let
$
F\mathchar`-\mathrm{Gauge}^{\mathrm{vect}}(S)
$
be
the category of prismatic $F$-gauges in vector bundles over $S$.

\begin{rem}\label{Remark:degree 0 part equivalence}
    Let $M=\bigoplus_{i \in \Z} M_i$ be a graded $\bigoplus_{i \in \Z} I^{i}t^{-i}$-module.
    For any $i \in \Z$, we have a natural isomorphism
    $
    M_0 \otimes_{\Prism_S} I^{i}t^{-i} \overset{\sim}{\to} M_i
    $
    of $\Prism_S$-modules.
    It follows that
    the functor
    $
    M \mapsto M_0
    $
    from the category of graded $\bigoplus_{i \in \Z} I^{i}t^{-i}$-modules to the category of $\Prism_S$-modules is an equivalence, whose inverse is given by $L \mapsto L \otimes_{\Prism_S} (\bigoplus_{i \in \Z} I^{i}t^{-i})$.
    %Thus, giving an isomorphism
    %$F \colon (\sigma^*N)_0 \overset{\sim}{\to} \tau^*N$ as in Definition \ref{Definition:prismatic F-gauge in vector bundle} is equivalent to giving an isomorphism
    %\[
    %\widetilde{F} \colon \sigma^*N \overset{\sim}{\to} (\tau^*N) \otimes_{\Prism_S} (\bigoplus_{i \in \Z} I^{i}t^{-i})
    %\]
    %of graded $\bigoplus_{i \in \Z} I^{i}t^{-i}$-modules.
\end{rem}

\begin{rem}\label{Remark:higher display of Lau}
    The notion of prismatic $F$-gauges in vector bundles is closely related to the notion of (higher) displays in the sense of Lau \cite[Definition 3.2.1]{Lau21}.
    See Remark \ref{Remark:analogue of Lau's definitions gauge version} for more details.
\end{rem}

We collect some useful facts about graded $\Rees(\Fil^\bullet_\mathcal{N}(\Prism_S))$-modules.

\begin{rem}\label{Remark:degree i part is complete}
    Let $N=\bigoplus_{i \in \Z} N_i$ be a graded $\Rees(\Fil^\bullet_\mathcal{N}(\Prism_S))$-module which is finite projective as a $\Rees(\Fil^\bullet_\mathcal{N}(\Prism_S))$-module.
    Then each degree $i$ part $N_i$ is a direct summand of a $\Prism_S$-module of the form
    $\bigoplus^m_{j=1} \Fil^{i_j}_\mathcal{N}(\Prism_S)t^{-i_j}$, and in particular $N_i$ is $(p, I)$-adically complete.
    This follows from the following fact:
    For a graded ring $A$, a graded $A$-module $N$ is projective as an $A$-module if and only if $N$ is projective in the category of graded $A$-modules; see \cite[Lemma 3.0.1]{Lau21}.
\end{rem}

Let
\begin{equation}\label{equation:homomorphism rho}
    \rho \colon \Rees(\Fil^\bullet_\mathcal{N}(\Prism_S)) \to \Prism_S/\Fil^1_\mathcal{N}(\Prism_S)
\end{equation}
be the composition of the projection $\Rees(\Fil^\bullet_\mathcal{N}(\Prism_S)) \to \Prism_S$
with the natural homomorphism
$\Prism_S \to \Prism_S/\Fil^1_\mathcal{N}(\Prism_S)$.
The map $\rho$ is a ring homomorphism.
For an integer $n \geq 1$, we write
    \[
    \Prism^{\mathcal{N}}_{S, n}:=\Rees(\Fil^\bullet_\mathcal{N}(\Prism_S)) \otimes_{\Prism_S} \Prism_S/(p, I)^n.
    \]
Let
$
\overline{\rho} \colon \Prism^{\mathcal{N}}_{S, n} \to \Prism_S/(\Fil^1_\mathcal{N}(\Prism_S)+(p, I))
$
be the homomorphism induced by $\rho$.

\begin{lem}[{cf.\ \cite[Lemma 3.1.1, Corollary 3.1.2]{Lau21}}]\label{Lemma:Nakayama's lemma graded version}
    \ 
    \begin{enumerate}
        \item Let $M$ be a finite graded $\Prism^{\mathcal{N}}_{S, n}$-module.
        If $\overline{\rho}^*M=0$, then we have $M=0$.
        \item Let $M$ and $N$ be finite graded $\Prism^{\mathcal{N}}_{S, n}$-modules.
        Assume that $N$ is projective as a $\Prism^{\mathcal{N}}_{S, n}$-module.
        Then a homomorphism
        $f \colon M \to N$ of graded $\Prism^{\mathcal{N}}_{S, n}$-modules is an isomorphism if 
        $\overline{\rho}^*f \colon \overline{\rho}^*M \to \overline{\rho}^*N$
        is an isomorphism.
    \end{enumerate}
\end{lem}

\begin{proof}
    (1) By \cite[Lemma 4.28]{Anschutz-LeBras},
    the pair $(\Prism_S, \Fil^1_\mathcal{N}(\Prism_S))$ is henselian.
    In particular we have $\Fil^1_\mathcal{N}(\Prism_S) \subset \mathrm{rad}(\Prism_S)$.
    Using this fact, we can prove the assertion by the same argument as in the proof of \cite[Lemma 3.1.1]{Lau21}.

    (2) By (1), we see that $f$ is surjective.
    Since $N$ is projective as a graded $\Prism^{\mathcal{N}}_{S, n}$-module (Remark \ref{Remark:degree i part is complete}), we have
    $M \simeq N \oplus \Ker f$ as graded $\Prism^{\mathcal{N}}_{S, n}$-modules.
    Thus $\Ker f$ is a finite graded $\Prism^{\mathcal{N}}_{S, n}$-module such that $\overline{\rho}^*\Ker f=0$.
    By (1) again, we have $\Ker f=0$.
\end{proof}

\begin{cor}\label{Corollary:standard projective}
    Let $N=\bigoplus_{i \in \Z} N_i$ be a graded $\Rees(\Fil^\bullet_\mathcal{N}(\Prism_S))$-module with the following properties:
    \begin{enumerate}
        \item The degree $i$ part $N_i$ is $(p, I)$-adically complete for every $i \in \Z$.
        \item $N^n:=N/(p, I)^nN$ is a finite projective $\Prism^{\mathcal{N}}_{S, n}$-module for every $n \geq 1$.
    \end{enumerate}
    Then there exists a graded finite projective $\Prism_S$-module $L$ with an isomorphism
    \[
    L \otimes_{\Prism_S} \Rees(\Fil^\bullet_\mathcal{N}(\Prism_S)) \simeq N
    \]
    of graded $\Rees(\Fil^\bullet_\mathcal{N}(\Prism_S))$-modules.
    In particular $N$ is a finite projective $\Rees(\Fil^\bullet_\mathcal{N}(\Prism_S))$-module.
\end{cor}

\begin{proof}
    Since $\Prism_S$ is henselian with respect to both ideals $\Fil^1_\mathcal{N}(\Prism_S)$ and $(p, I)$,
    there exists a graded finite projective $\Prism_S$-module $L$ with an isomorphism
    \[
    L/(\Fil^1_\mathcal{N}(\Prism_S)+(p, I))L \overset{\sim}{\to} \overline{\rho}^*N^1
    \]
    of graded modules (by \cite[Tag 0D4A]{SP} or \cite[Theorem 5.1]{Greco}).
    This isomorphism lifts to a homomorphism
    \[
    f \colon L \otimes_{\Prism_S} \Rees(\Fil^\bullet_\mathcal{N}(\Prism_S)) \to N
    \]
    of graded $\Rees(\Fil^\bullet_\mathcal{N}(\Prism_S))$-modules (see Remark \ref{Remark:degree i part is complete}).
    By Lemma \ref{Lemma:Nakayama's lemma graded version}, the reduction modulo $(p, I)^n$ of $f$ is bijective for every $n$.
    Since degree $i$ parts of $L \otimes_{\Prism_S} \Rees(\Fil^\bullet_\mathcal{N}(\Prism_S))$ and $N$ are $(p, I)$-adically complete for every $i \in \Z$, it follows that $f$ is an isomorphism.
\end{proof}

For a bounded prism $(A, I)$, let
$
\mathrm{BK}_{\mathrm{disp}}(A, I)
$
be the category of displayed Breuil--Kisin modules over $(A, I)$ (see Definition \ref{Definition:displayed and minuscule Breuil-Kisin module}).
Prismatic $F$-gauges in vector bundles over $S$ can be related to displayed Breuil--Kisin modules over $(\Prism_S, I_S)$ as follows.

\begin{prop}\label{Proposition:F-gauges to F-crystals qrsp case}
    Let $S \in \qrsperfd$. There exists a fully faithful functor
    \[
        F\mathchar`-\mathrm{Gauge}^{\mathrm{vect}}(S) \to \mathrm{BK}_{\mathrm{disp}}(\Prism_S, I_S)
    \]
    which is compatible with base change along any homomorphism $S \to S'$ in $\qrsperfd$.
\end{prop}

\begin{proof}
    To each $(N, F_N) \in F\mathchar`-\mathrm{Gauge}^{\mathrm{vect}}(S)$, we attach a displayed Breuil--Kisin module $(M, F_M)$ over $(\Prism_S, I)$ as follows.
    Let $M:=\tau^*N$.
    The kernel of
    $\tau \colon \Rees(\Fil^\bullet_\mathcal{N}(\Prism_S)) \to \Prism_S$
    is generated by $t-1$, so that
    $M = N/(t-1)N$.
    It follows that the natural homomorphism
    $N_i \to M$
    of $\Prism_S$-modules is injective, whose image is denoted by $\Fil^i(M) \subset M$.
    We have $\Fil^{i+1}(M) \subset \Fil^i(M)$, and the corresponding map
    $N_{i+1} \to N_i$
    is given by $x \mapsto tx$.
    Moreover we have $M = \cup_i \Fil^i(M)$.
    Let $i$ be a small enough integer such that $\Fil^i(M)=M$.
    We define
    $\phi^*M \to M[1/I]$
    to be the following composition:
    \[
    \phi^*M = \phi^*\Fil^i(M) \overset{\sim}{\to} \phi^*N_i \to (\sigma^*N)_i  \overset{\sim}{\to} (\sigma^*N)_0 \otimes_{\Prism_S} I^it^{-i}
    \overset{F_N}{\to} M \otimes_{\Prism_S} I^it^{-i} \overset{t \mapsto 1}{\to} M[1/I].
    \]
    (For the isomorphism $(\sigma^*N)_i  \overset{\sim}{\to} (\sigma^*N)_0 \otimes_{\Prism_S} I^it^{-i}$, see Remark \ref{Remark:degree 0 part equivalence}.)
    This homomorphism  $\phi^*M \to M[1/I]$ is independent of the choice of $i$.
    Let
    $F_M \colon (\phi^*M)[1/I] \to M[1/I]$
    be the induced homomorphism.
        
    We shall prove that $(M, F_M)$ is a displayed Breuil--Kisin module over $(\Prism_S, I)$.
    By Corollary \ref{Corollary:standard projective},
    we may assume that
    \[
    N=L \otimes_{\Prism_S} \Rees(\Fil^\bullet_\mathcal{N}(\Prism_S))
    \]
    for some graded finite projective $\Prism_S$-module $L=\bigoplus_{j \in \Z} L^{(-1)}_j$.
    Then we have $M=L$ and
    \begin{equation}\label{equation:filtration on M}
        \Fil^i(M)= (\bigoplus_{j \geq i} L^{(-1)}_j) \oplus (\bigoplus_{j < i} \Fil^{i-j}_\mathcal{N}(\Prism_S)L^{(-1)}_j)
    \end{equation}
    for every $i \in \Z$.
    We set $L_j:=\phi^*L^{(-1)}_j$.
    By sending $t$ to $1$, 
    we obtain
    $
    (\sigma^*N)_0 \overset{\sim}{\to} \bigoplus_{j \in \Z} (L_j \otimes_{\Prism_S} I^{-j}),
    $
    and the isomorphism $F_N$ can be written as
    \[
    F_N \colon \bigoplus_{j \in \Z} (L_j \otimes_{\Prism_S} I^{-j}) \overset{\sim}{\to} M.
    \]
    Now
    $
    F_M \colon \bigoplus_{j \in \Z} L_j[1/I] \to M[1/I] 
    $
    is the base change of $F_N$.
    (In particular $F_M$ is an isomorphism.)
    Recall the filtration
    $\{ \Fil^i(\phi^*M) \}_{i \in \Z}$
    of $\phi^*M$ from Definition \ref{Definition:Breuil-Kisin module}.
    We see that
    $\Fil^i(\phi^*M) \subset \phi^*M$
    is the intersection of
    $\phi^*M=\bigoplus_{j \in \Z} L_j$
    with
    $\bigoplus_{j \in \Z} (L_j \otimes_{\Prism_S} I^{i-j})$, and thus
    \begin{equation}\label{equation:filtration on phiM}
    \Fil^i(\phi^*M)=(\bigoplus_{j \geq i} L_j) \oplus (\bigoplus_{j < i} I^{i-j}L_j).
    \end{equation}
    From this description, we see that $(M, F_M)$ is a displayed Breuil--Kisin module.
    (Moreover
    $\phi^*M=\bigoplus_{j \in \Z} L_j$
    is a normal decomposition in the sense of Definition \ref{Definition:normal decomposition}.)

    We have constructed a functor
    $F\mathchar`-\mathrm{Gauge}^{\mathrm{vect}}(S) \to \mathrm{BK}_{\mathrm{disp}}(\Prism_S, I_S)$.
    The full faithfulness of this functor follows from \cite[Corollary 2.53]{Guo-Li}.
    The important point is that the filtration $\{ \Fil^i(M) \}_{i \in \Z}$ of $M$ can be recovered from $(M, F_M)$.
    Namely, $\Fil^i(M)$ agrees with the inverse image of $M \otimes_{\Prism_S} I^{i}$
    under the composition
    \[
    M \overset{x \mapsto 1 \otimes x}{\longrightarrow} (\phi^*M)[1/I] \overset{F_M}{\longrightarrow} M[1/I].
    \]
    (Compare $(\ref{equation:filtration on M})$ with $(\ref{equation:filtration on phiM})$.)
    Since the full faithfulness of the functor also follows from
    Example \ref{Example:GLn displays},
    Example \ref{Example:prismatic F-gauge and GLn-gauge},
    and
    Proposition \ref{Proposition:functor from G-gauge to G-display} below (compare with the argument in the proof of Proposition \ref{Proposition:evaluation prismatic dieudonne crystal equivalence}), we omit the details here.
\end{proof}

The following result is essential for the definition of prismatic $F$-gauges in vector bundles over a quasisyntomic ring.

\begin{prop}\label{Proposition:quasisyntomic descent for prismatic F-gauge}
The fibered category
over
$\mathrm{QRSPerfd}^{\op}$
which associates to each
$S \in \mathrm{QRSPerfd}$
the category
$F\mathchar`-\mathrm{Gauge}^{\mathrm{vect}}(S)$
satisfies descent with respect to the quasisyntomic topology.
\end{prop}

\begin{proof}
    This was originally proved by Bhatt--Lurie (cf.\ \cite[Remark 5.5.18]{BhattGauge}).
    In \cite[Proposition 2.29]{Guo-Li}, this result was obtained by a slightly different method.
    We briefly recall the argument given in \cite[Proposition 2.29]{Guo-Li}.
    
    Let $S \to S'$ be a quasisyntomic covering in $\mathrm{QRSPerfd}$.
    By the proof of \cite[Proposition 2.29]{Guo-Li},
    the induced homomorphism
    $
    \Prism^{\mathcal{N}}_{S, n} \to \Prism^{\mathcal{N}}_{S', n}
    $
    is faithfully flat for every $n \geq 1$, and for a homomorphism $S \to S_1$ in $\mathrm{QRSPerfd}$,
    the following natural homomorphism of graded rings is an isomorphism:
    \[
    \Prism^{\mathcal{N}}_{S', n} \otimes_{\Prism^{\mathcal{N}}_{S, n}} \Prism^{\mathcal{N}}_{S_1, n} \overset{\sim}{\to} \Prism^{\mathcal{N}}_{S' \widehat{\otimes}_S S_1, n},
    \]
    where $S' \widehat{\otimes}_S S_1 \in \mathrm{QRSPerfd}$ is the $p$-adic completion of $S' \otimes_S S_1$.
    Using these results, we can prove that the natural functor from the category of prismatic $F$-gauges in vector bundles over $S$ to the category of prismatic $F$-gauges in vector bundles over $S'$ with a descent datum (with respect to $S \to S'$) is an equivalence.
    We only prove the essential surjectivity of the functor.
    
    Let $(N', F_{N'}) \in F\mathchar`-\mathrm{Gauge}^{\mathrm{vect}}(S')$ with a descent datum.
    By the results recalled in the previous paragraph and by faithfully flat descent,
    we see that for every $n \geq 1$,
    the graded $\Prism^{\mathcal{N}}_{S', n}$-module
    $N'/(p, I)^nN'$ with the descent datum arises from a graded $\Prism^{\mathcal{N}}_{S, n}$-module
    $
    N^n= \bigoplus_{i \in \Z} N^n_i
    $
    such that $N^n$ is finite projective as a $\Prism^{\mathcal{N}}_{S, n}$-module.
    Let
    $N_i:= \varprojlim_n N^n_i$
    and we define
    $N:=\bigoplus_{i \in \Z} N_i$,
    which is a
    graded
    $\Rees(\Fil^\bullet_\mathcal{N}(\Prism_S))$-module.
    We have $N/(p, I)^nN =N^n$ for every $n \geq 1$.
    By Corollary \ref{Corollary:standard projective}, we see that $N$ is a finite projective $\Rees(\Fil^\bullet_\mathcal{N}(\Prism_S))$-module.
    Moreover the natural homomorphism
    \[
    N \otimes_{\Rees(\Fil^\bullet_\mathcal{N}(\Prism_S))} \Rees(\Fil^\bullet_\mathcal{N}(\Prism_{S'})) \to N'
    \]
    is an isomorphism
    (as its reduction modulo $(p, I)^n$ is an isomorphism for every $n$).
    Since
    $(\Prism_S, I_S) \to (\Prism_{S'}, I_{S'})$
    is a faithfully flat map of $\O_E$-prisms and the $(p, I)$-adic completion of
    $\Prism_{S'} \otimes_{\Prism_S} \Prism_{S'}$
    is isomorphic to $\Prism_{S' \widehat{\otimes}_S S'}$ (see the proof of \cite[Proposition 2.29]{Guo-Li}),
    the isomorphism
    $F_{N'}$ descends to an isomorphism
    $F_N \colon (\sigma^*N)_0 \overset{\sim}{\to} \tau^*N$ by Proposition \ref{Proposition:flat descent for finite projective modules}.
    This shows that $(N', F_{N'})$ with the descent datum arises from $(N, F_N) \in F\mathchar`-\mathrm{Gauge}^{\mathrm{vect}}(S)$.
\end{proof}

Proposition \ref{Proposition:quasisyntomic descent for prismatic F-gauge}, together with \cite[Proposition 4.31]{BMS2}, shows that
the fibered category
$S \mapsto F\mathchar`-\mathrm{Gauge}^{\mathrm{vect}}(S)$
over
$\mathrm{QRSPerfd}^{\op}$
extends uniquely to 
a fibered category
\[
R \mapsto F\mathchar`-\mathrm{Gauge}^{\mathrm{vect}}(R)
\]
over
the category $\qsyn^{\op}$ that satisfies descent with respect to the quasisyntomic topology.

\begin{defn}[Drinfeld, Bhatt--Lurie, Guo--Li]\label{Definition:prismatic F-gauge in vector bundle over qsyn}
Let $R \in \qsyn$.
An object 
$N \in F\mathchar`-\mathrm{Gauge}^{\mathrm{vect}}(R)$
is called a \textit{prismatic $F$-gauge in vector bundles} over $R$.
For a homomorphism $R \to S$ with $S \in \qrsperfd$, the image of $N$ in $F\mathchar`-\mathrm{Gauge}^{\mathrm{vect}}(S)$ is denoted by $(N_S, F_{N_S})$.
\end{defn}

A fully faithful functor from
$F\mathchar`-\mathrm{Gauge}^{\mathrm{vect}}(R)$
to the category of prismatic $F$-crystals on $(R)_\Prism$ (in the sense of \cite{BS2}) is obtained in \cite[Corollary 2.53]{Guo-Li}.
More precisely, we have the following result.

\begin{prop}\label{Proposition:F-gauges to F-crystals}
    Let $R \in \qsyn$.
    There exists a fully faithful functor
    \begin{equation}\label{equation:functor from gauge to crystal}
        F\mathchar`-\mathrm{Gauge}^{\mathrm{vect}}(R) \to {2-\varprojlim}_{(A, I) \in (R)_{\Prism}} \mathrm{BK}_{\mathrm{disp}}(A, I).
    \end{equation}
    This functor is compatible with base change along any homomorphism $R \to R'$ in $\qsyn$.
\end{prop}

\begin{proof}
    By Proposition \ref{Proposition:quasisyntomic descent for prismatic F-gauge}, the left hand side of $(\ref{equation:functor from gauge to crystal})$ satisfies quasisyntomic descent.
    By \cite[Proposition 7.11]{BS} and Corollary \ref{Corollary:flat descent for dispalyed BK modules}, the right hand side of $(\ref{equation:functor from gauge to crystal})$ also satisfies quasisyntomic descent; see also the proof of \cite[Proposition 2.14]{BS2}.
    Thus, the assertion follows from Proposition \ref{Proposition:F-gauges to F-crystals qrsp case}.
\end{proof}

As in Section \ref{Subsection:BK module of type mu},
let
$\mu=(m_1, \dotsc, m_n)$
be a tuple of integers $m_1 \geq \cdots \geq m_n$.
Let $r_i \in \Z_{\geq 0}$ be the number of occurrences of $i$ in $(m_1, \dotsc, m_n)$.
We want to compare prismatic $F$-gauges in vector bundles with (displayed) Breuil--Kisin modules of type $\mu$ in the sense of Definition \ref{Definition:type of displayed BK module}.

\begin{rem}\label{Remark:quotient of Fil1}
Let $S \in \qrsperfd$.
By \cite[Theorem 12.2]{BS}, the Frobenius
$\phi \colon \Prism_S \to \Prism_S$
induces an isomorphism
\[
\Prism_S/\Fil^1_\mathcal{N}(\Prism_S) \overset{\sim}{\to} S.
\]
Using this, we regard the homomorphism
$\rho$
(see (\ref{equation:homomorphism rho})) as
$\rho \colon \Rees(\Fil^\bullet_\mathcal{N}(\Prism_S)) \to S$.
\end{rem}

\begin{defn}\label{Definition:type of prismatic $F$-gauge}
Let $R \in \qsyn$.
Let $N \in F\mathchar`-\mathrm{Gauge}^{\mathrm{vect}}(R)$.
We say that $N$ is \textit{of type $\mu$} if for any $R \to S$ with $S \in \qrsperfd$,
the degree $i$ part $(\rho^*N_S)_i$ of the graded $S$-module $\rho^*N_S$
is of rank $r_i$ for any $i \in \Z$.
\end{defn}

Let
    \[
    F\mathchar`-\mathrm{Gauge}_\mu(R) \subset F\mathchar`-\mathrm{Gauge}^{\mathrm{vect}}(R)
    \]
be the full subcategory spanned by those objects of type $\mu$.
The property of being of type $\mu$ can be checked locally in the quasisyntomic topology.
Thus
the fibered category
$R \mapsto F\mathchar`-\mathrm{Gauge}_\mu(R)$
over
$\qsyn^{\op}$
satisfies descent with respect to the quasisyntomic topology.

By construction, the functor $(\ref{equation:functor from gauge to crystal})$
induces a fully faithful functor
\begin{equation}\label{equation:functor from gauge to BK type mu}
    F\mathchar`-\mathrm{Gauge}_{\mu}(R) \to {2-\varprojlim}_{(A, I) \in (R)_{\Prism}} \mathrm{BK}_{\mu}(A, I)
\end{equation}
for any $R \in \qsyn$.
(Recall that $\mathrm{BK}_{\mu}(A, I)$ is the category of Breuil--Kisin modules over $(A, I)$ of type $\mu$.)
We will prove later that the functors $(\ref{equation:functor from gauge to crystal})$ and $(\ref{equation:functor from gauge to BK type mu})$ are equivalences if $R$ is a perfectoid ring or a complete regular local ring with perfect residue field $k$ of characteristic $p$; see Corollary \ref{Corollary:functor from gauge to BK perfectoid and regular} below.

\begin{ex}\label{Example:F-gauge weight [0, 1]}
    Let $R \in \qsyn$.
    Let
    $F\mathchar`-\mathrm{Gauge}^{\mathrm{vect}}_{[0, 1]}(R) \subset F\mathchar`-\mathrm{Gauge}^{\mathrm{vect}}(R)$
    be the full subcategory of those $N \in F\mathchar`-\mathrm{Gauge}^{\mathrm{vect}}(R)$
    such that for any homomorphism $R \to S$ with $S \in \qrsperfd$,
    we have $(\rho^*N_S)_i=0$ for all $i \neq 0, 1$.
    The functor $(\ref{equation:functor from gauge to crystal})$
induces a fully faithful functor
\[
F\mathchar`-\mathrm{Gauge}^{\mathrm{vect}}_{[0, 1]}(R) \to {2-\varprojlim}_{(A, I) \in (R)_{\Prism}} \mathrm{BK}_{\mathrm{min}}(A, I).
\]
The right hand side can be identified with the category of prismatic Dieudonn\'e crystals on $(R)_{\Prism}$; see Section \ref{Subsection:A remark on prismatic Dieudonne crystals}.
By \cite[Theorem 2.54]{Guo-Li},
the essential image of this functor is the full subcategory of admissible prismatic Dieudonn\'e crystals on $(R)_{\Prism}$.
If $R$ is a perfectoid ring or a complete regular local ring with perfect residue field $k$ of characteristic $p$, then any prismatic Dieudonn\'e crystal on $(R)_{\Prism}$ is admissible by \cite[Proposition 4.12, Proposition 5.10]{Anschutz-LeBras}, and hence the above functor is an equivalence in this case.
This fact also follows from Corollary \ref{Corollary:functor from gauge to BK perfectoid and regular}.
\end{ex}

\subsection{Prismatic $G$-$F$-gauges of type $\mu$}\label{Subsection:Prismatic G-F-gauges of type mu}

Let $G$ be a smooth affine group scheme over
$\Z_p$.
Let
$\mu \colon \G_m \to G_{W(k)}$
be a cocharacter where $k$ is a perfect field of characteristic $p$.
We introduce prismatic $G$-$F$-gauges of type $\mu$ in the same way as for prismatic $G$-$\mu$-displays.

We retain the notation of Section \ref{Section:display group}.
For the cocharacter $\mu$,
we have
the action
$(\ref{equation:action cocharacter adjoint})$ of $\G_m$ on $G_{W(k)}=\Spec A_G$.
Let
$
A_G= \bigoplus_{i\in\Z} A_{G, i}
$
be the weight decomposition.
We define
$
A^{(-1)}_{G, i}:=(\phi^{-1})^*A_{G, i}
$
where $\phi^{-1} \colon W(k) \to W(k)$ is the inverse of the Frobenius.
Since $(\phi^{-1})^*A_G=A_G$, we have
\[
A_G= \bigoplus_{i\in\Z} A^{(-1)}_{G, i}.
\]
Let
$\mu^{(-1)} \colon \G_m \to G_{W(k)}$
be the base change of $\mu$ along
$\phi^{-1}$.
Then 
$
A_G= \bigoplus_{i\in\Z} A^{(-1)}_{G, i}
$
is the weight decomposition with respect to the action of $\G_m$ induced by $\mu^{(-1)}$.

Let $S$ be a quasiregular semiperfectoid ring over $W(k)$.

\begin{defn}\label{Definition:gauge group}
    Let
    \[
    G_{\mu, \mathcal{N}}(S) \subset G(\Prism_S)
    \]
    be the subgroup consisting of homomorphisms
    $g^* \colon A_G \to \Prism_S$
    of $W(k)$-algebras such that
    $g^*(A^{(-1)}_{G, i}) \subset \Fil^i_\mathcal{N}(\Prism_S)$ for any $i \in \Z$.
    The group $G_{\mu, \mathcal{N}}(S)$ is called the \textit{gauge group}.
\end{defn}

\begin{rem}\label{Remark:gauge group as inverse image}
It follows from Lemma \ref{Lemma:Gmu iff condition} that
$G_{\mu, \mathcal{N}}(S) \subset G(\Prism_S)$
is the inverse image of the display group
$G_\mu(\Prism_S, I_S) \subset G(\Prism_S)$
under the homomorphism
$\phi \colon G(\Prism_S) \to G(\Prism_S)$.
\end{rem}

For a generator $d \in I_S$,
we have the homomorphism
\[
\sigma_{\mu, \mathcal{N}, d} \colon G_{\mu, \mathcal{N}}(S)  \to G(\Prism_S), \quad g \mapsto \mu(d)\phi(g)\mu(d)^{-1}
\]
by Remark \ref{Remark:gauge group as inverse image}.
Let $G(\Prism_S)_{\mathcal{N}, d}$ be the set
$G(\Prism_S)$ together with the following action of $G_{\mu, \mathcal{N}}(S)$:
\begin{equation}\label{equation:gauge group action}
    G(\Prism_S) \times G_{\mu, \mathcal{N}}(S) \to G(\Prism_S), \quad (X, g) \mapsto X \cdot g:=g^{-1}X\sigma_{\mu, \mathcal{N}, d}(g).
\end{equation}
For another generator
$d' \in I_S$
and the unique element $u \in \Prism^\times_S$ such that $d=ud'$,
the bijection $G(\Prism_S)_{\mathcal{N}, d} \to G(\Prism_S)_{\mathcal{N}, d'}$, $X \mapsto X\mu(u)$ is $G_{\mu, \mathcal{N}}(S)$-equivariant.
Then we set
\[
G(\Prism_S)_\mathcal{N} := \varprojlim_{d} G(\Prism_S)_{\mathcal{N}, d},
\]
which is equipped with a natural action of $G_{\mu, \mathcal{N}}(S)$.
Here $d$ runs over the set of generators $d \in I_S$.
Although $G(\Prism_S)_\mathcal{N}$ depends on $\mu$, we omit it from the notation.

\begin{rem}\label{Remark:qsyn sheaves and qrsp sheaves}
We recall some notation from \cite[Definition 2.9]{BS2}.
Let $R$ be a quasisyntomic ring.
Let $(R)_{\mathrm{qsyn}}$
(resp.\ $(R)_{\mathrm{qrsp}}$)
be the category of quasisyntomic rings $R'$
(resp.\ quasiregular semiperfectoid rings $R'$)
with a quasisyntomic map $R \to R'$.
We endow
both
$(R)^{\op}_{\mathrm{qsyn}}$
and $(R)^{\op}_{\mathrm{qrsp}}$
with the quasisyntomic topology.
Since
quasiregular semiperfectoid rings form a basis for $\qsyn$,
we may identify sheaves on $(R)^{\op}_{\mathrm{qsyn}}$
with
sheaves on $(R)^{\op}_{\mathrm{qrsp}}$.
On the site $(R)^{\op}_{\mathrm{qsyn}}$, we have the sheaves
$\Prism_{\bullet}$
and
$I_{\bullet}$
such that
\[
\Prism_{\bullet}(S)=\Prism_S \quad \text{and} \quad I_{\bullet}(S)=I_S
\]
for each $S \in (R)_{\mathrm{qrsp}}$.
\end{rem}

\begin{lem}\label{Lemma:gauge groups form a sheaf}
    Let $R$ be a quasisyntomic ring over $W(k)$.
    The functors
    \begin{align*}
    G_{\mu, \mathcal{N}} &\colon (R)_{\mathrm{qrsp}} \to \mathrm{Set}, \quad S \mapsto G_{\mu, \mathcal{N}}(S), \\
    G_{\Prism, \mathcal{N}} &\colon (R)_{\mathrm{qrsp}} \to \mathrm{Set}, \quad S \mapsto G(\Prism_S)_\mathcal{N}
    \end{align*}
    form sheaves with respect to the quasisyntomic topology.
\end{lem}

\begin{proof}
    As
    $\Prism_\bullet$
    is a sheaf, so is $G_{\Prism, \mathcal{N}}$.
    Since
    $I_\bullet$
    is a sheaf, it follows that the functor
    $S \mapsto \Fil^i_\mathcal{N}(\Prism_S)$
    forms a sheaf for any $i \in \Z$.
    This implies that $G_{\mu, \mathcal{N}}$ is a sheaf.
\end{proof}

We regard $G_{\mu, \mathcal{N}}$ and $G_{\Prism, \mathcal{N}}$ as sheaves on $(R)^{\op}_{\mathrm{qsyn}}$.
The sheaf
$G_{\Prism, \mathcal{N}}$
is equipped with an action of $G_{\mu, \mathcal{N}}$.

\begin{defn}[{Prismatic $G$-$F$-gauge of type $\mu$}]\label{Definition:Prismatic G-F-gauge of type}
Let $R$ be a quasisyntomic ring over $W(k)$.
A \textit{prismatic $G$-$F$-gauge of type $\mu$} over
    $R$ is a pair
    \[
    (\mathscr{Q}, \alpha_\mathscr{Q})
    \]
    where $\mathscr{Q}$ is a $G_{\mu, \mathcal{N}}$-torsor on $(R)^{\op}_{\mathrm{qsyn}}$ and $\alpha_\mathscr{Q} \colon \mathscr{Q} \to G_{\Prism, \mathcal{N}}$ is a $G_{\mu, \mathcal{N}}$-equivariant map.
    We say that
    $(\mathscr{Q}, \alpha_\mathscr{Q})$ is \textit{banal} if $\mathscr{Q}$ is trivial as a $G_{\mu, \mathcal{N}}$-torsor.
    When there is no possibility of confusion, we write $\mathscr{Q}$ instead of $(\mathscr{Q}, \alpha_\mathscr{Q})$.
An isomorphism of prismatic $G$-$F$-gauges of type $\mu$ over $R$ is defined in the same way as in Definition \ref{Definition:G mu display over oriented prisms}.
\end{defn}

Let
\[
G\mathchar`-F\mathchar`-\mathrm{Gauge}_\mu(R)
\]
be the groupoid of prismatic $G$-$F$-gauges of type $\mu$ over $R$.
For a homomorphism $f \colon R \to R'$ of quasisyntomic rings over $W(k)$, we have a base change functor
\[
f^* \colon G\mathchar`-F\mathchar`-\mathrm{Gauge}_\mu(R) \to G\mathchar`-F\mathchar`-\mathrm{Gauge}_\mu(R')
\]
defined in the same way as in Definition \ref{Definition:base change for G-displays}.

\begin{rem}\label{Remark:on Drinfeld's definition of G-gauge of type mu}
    A ``truncated analogue'' of the notion of prismatic $G$-$F$-gauges of type $\mu$ was introduced by Drinfeld in \cite[Appendix C]{Drinfeld23} for a $p$-adic formal scheme $\mathcal{X}$ which is formally of finite type over $\Spf \Z_p$, in terms of certain torsors on the \textit{syntomification} of $\mathcal{X}$ in the sense of Drinfeld and Bhatt--Lurie.
    It should be possible to define prismatic $G$-$F$-gauges of type $\mu$ over any $p$-adic formal scheme by using certain torsors on syntomifications, but we will not discuss this here\footnote{After this work was completed, and during the refereeing process, this has been carried out by Gardner--Madapusi--Mathew \cite{GMM}.}.
\end{rem}

\begin{rem}\label{Remark:analogue of Lau's definitions gauge version}
Let
$S$
be a quasiregular semiperfectoid ring over $W(k)$.
For a generator $d \in I_S$,
let
$
\sigma_d \colon \Rees(\Fil^\bullet_\mathcal{N}(\Prism_S)) \to \Prism_S
$
be the homomorphism defined by $a_it^{-i} \mapsto \phi(a_i)
d^{-i}$ for any $i \in \Z$.
Recall the homomorphism
    $\tau \colon \Rees(\Fil^\bullet_\mathcal{N}(\Prism_S)) \to \Prism_S$
    from Definition \ref{Definition:Rees algebra for Nygaard filtration}.
Similarly to the triple
    $
    (\Rees(I^\bullet), \sigma_d, \tau)
    $
    in Remark \ref{Remark:analogue of Lau's definitions display version},
    the triple
    \[
    (\Rees(\Fil^\bullet_\mathcal{N}(\Prism_S)), \sigma_d, \tau)
    \]
    is an analogue of a higher frame in the sense of Lau.
    The homomorphism $\tau$
    induces an isomorphism
    \[
    G(\Rees(\Fil^\bullet_\mathcal{N}(\Prism_S)))^0
    \overset{\sim}{\to} 
    G_{\mu, \mathcal{N}}(S)
    \]
    where $G(\Rees(\Fil^\bullet_\mathcal{N}(\Prism_S)))^0$ is the group of elements $g \in G(\Rees(\Fil^\bullet_\mathcal{N}(\Prism_S)))$ such that
    \[
    g^* \colon A_G=\bigoplus_{i\in\Z} A^{(-1)}_{G, i} \to \Rees(\Fil^\bullet_\mathcal{N}(\Prism_S))
    \]
    is a homomorphism of graded $W(k)$-algebras.
    Via this isomorphism, the homomorphism
    $\sigma_{\mu, \mathcal{N}, d}$
    agrees with the one induced by $\sigma_d$.
    Thus, the action $(\ref{equation:gauge group action})$ is consistent with the one considered in \cite[(5-2)]{Lau21}.
    
    Roughly speaking,
    prismatic $F$-gauges in vector bundles
    (resp.\ prismatic $G$-$F$-gauges of type $\mu$) over $S$ can be considered as
    displays
    (resp.\
    $G$-$\mu$-displays) over the ``higher frame''
    $
    (\Rees(\Fil^\bullet_\mathcal{N}(\Prism_S)), \sigma_d, \tau).
    $
    On the other hand,
    displayed Breuil--Kisin modules
    (resp.\ prismatic $G$-$\mu$-displays) over $(\Prism_S, I_S)$
    can be thought of as
    displays
    (resp.\ $G$-$\mu$-displays) over the ``higher frame''
    $
    (\Rees(I^\bullet_S), \sigma_d, \tau).
    $
    See also \cite[Section 3.7]{Lau21} where the relation between displays over higher frames and Frobenius gauges in the sense of Fontaine--Jannsen \cite[Section 2.2]{Fontaine--Jannsen} is discussed.
    %(See also \cite[Remark 3.4.6]{BhattGauge}.)
\end{rem}

Let us discuss the relation between
prismatic $F$-gauges in vector bundles of type $\mu$
and
prismatic
$\GL_n$-$F$-gauges of type $\mu$.

\begin{ex}\label{Example:gauge group GLn case}
Let
$\mu \colon \G_m \to \GL_{n, W(k)}$
be a cocharacter and let $(m_1, \dotsc, m_n)$ be the corresponding tuple of integers $m_1 \geq \cdots \geq m_n$ as in Section \ref{Subsection:BK module of type mu}.
We retain the notation of Section \ref{Subsection:BK module of type mu}.
Let
$
L_{W(k)} = \bigoplus_{j \in\Z}  L_{\mu, j}
$
be the weight decomposition with respect to the action of $\G_m$ on $L_{W(k)}=W(k)^n$ induced by $\mu$.
We set
$L^{(-1)}_{\mu, j}:=(\phi^{-1})^*L_{\mu, j}$.
By the decomposition $L_{W(k)} = \bigoplus_{j \in\Z}  L^{(-1)}_{\mu, j}$, we regard $L_{W(k)}$ as a graded module.
Let $S$ be a quasiregular semiperfectoid ring over $W(k)$.
Then,
via the isomorphism
\[
\GL_n(\Rees(\Fil^\bullet_\mathcal{N}(\Prism_S)))^0 \overset{\sim}{\to} (\GL_n)_{\mu, \mathcal{N}}(S)
\]
given in Remark \ref{Remark:analogue of Lau's definitions gauge version},
we may identify
$(\GL_n)_{\mu, \mathcal{N}}(S)$
with the group of graded automorphisms of
$
L_{W(k)} \otimes_{W(k)} \Rees(\Fil^\bullet_\mathcal{N}(\Prism_S)).
$
\end{ex}

\begin{ex}\label{Example:prismatic F-gauge and GLn-gauge}
Let the notation be as in Example \ref{Example:gauge group GLn case}.
Let $R$ be a quasisyntomic ring over $W(k)$.
We shall construct an equivalence
\begin{equation}\label{equation:gauge to GLn gauge}
    F\mathchar`-\mathrm{Gauge}_{\mu}(R)^{\simeq} \overset{\sim}{\to} \GL_n\mathchar`-F\mathchar`-\mathrm{Gauge}_\mu(R), \quad N \mapsto \mathscr{Q}(N)
\end{equation}
where $F\mathchar`-\mathrm{Gauge}_{\mu}(R)^{\simeq}$ is the groupoid of prismatic $F$-gauges in vector bundles of type $\mu$ over $R$.
Let $N \in F\mathchar`-\mathrm{Gauge}_{\mu}(R)$.
We consider the sheaf
\[
\mathscr{Q}(N) \colon (R)_{\mathrm{qsyn}} \to \mathrm{Set}
\]
sending $S \in (R)_{\mathrm{qrsp}}$ to the set of graded isomorphisms
\[
h \colon L_{W(k)} \otimes_{W(k)} \Rees(\Fil^\bullet_\mathcal{N}(\Prism_{S}))  \overset{\sim}{\to} N_S.
\]
Such an isomorphism $h$ exists locally in the quasisyntomic topology by (the proof of) Corollary \ref{Corollary:standard projective}.
By Example \ref{Example:gauge group GLn case},
the sheaf $\mathscr{Q}(N)$ then admits the structure of a $(\GL_n)_{\mu, \mathcal{N}}$-torsor.
Let $d \in I_S$ be a generator.
We fix an isomorphism $h$ as above.
Then we have the following isomorphisms
\[
(\sigma^*N)_0 \simeq \bigoplus_{j \in \Z} (L_{\mu, j} \otimes_{W(k)} I^{-j}_S) \simeq L_{\Prism_S}
\]
where the second isomorphism is given by $\mu(d)$.
We also have
$\tau^*N \simeq L_{\Prism_S}$.
Thus, the isomorphism $F_{N_S}$ gives an element
$\alpha(h)_d \in \GL_n(\Prism_S)=\GL_n(\Prism_S)_{\mathcal{N}, d}$.
The element $\alpha(h) \in \GL_n(\Prism_S)_{\mathcal{N}}$ corresponding to $\alpha(h)_d$ does not depend on the choice of $d$.
In this way, we get a $(\GL_n)_{\mu, \mathcal{N}}$-equivariant map
$\alpha \colon \mathscr{Q}(N) \to (\GL_n)_{\Prism, \mathcal{N}}$,
so that the pair $(\mathscr{Q}(N), \alpha)$ belongs to $\GL_n\mathchar`-F\mathchar`-\mathrm{Gauge}_\mu(R)$.
This construction gives the functor $(\ref{equation:gauge to GLn gauge})$.
Using quasisyntomic descent, one can check that this functor is an equivalence.
\end{ex}

We now compare prismatic $G$-$\mu$-displays with prismatic
$G$-$F$-gauges of type $\mu$.
We first note the following result.

\begin{prop}\label{Proposition:quasisyntomic descent for G mu displays}
Let $R$ be a quasisyntomic ring over $W(k)$.
The fibered category over
$
(R)^{\op}_{\mathrm{qsyn}}
$
which associates to each $R' \in (R)_{\mathrm{qsyn}}$
the groupoid
$
G\mathchar`-\mathrm{Disp}_\mu((R')_\Prism)
$
is a stack with respect to the quasisyntomic topology.
\end{prop}

\begin{proof}
    This follows from \cite[Proposition 7.11]{BS} and Proposition \ref{Proposition:flat descent of G display} by a standard argument.
    (See also the proof of \cite[Proposition 2.14]{BS2}.)
\end{proof}

\begin{prop}\label{Proposition:functor from G-gauge to G-display}
    Let $R$ be a quasisyntomic ring over $W(k)$.
    There exists a fully faithful functor
    \begin{equation}\label{equation:functor from G-gauge to G-display}
        G\mathchar`-F\mathchar`-\mathrm{Gauge}_\mu(R) \to G\mathchar`-\mathrm{Disp}_{\mu}((R)_{\Prism}).
    \end{equation}
    This functor is compatible with base change along any homomorphism $R \to R'$ in $\qsyn$.
\end{prop}

\begin{proof}
    It is clear that the left hand side of $(\ref{equation:functor from G-gauge to G-display})$ satisfies quasisyntomic descent.
    By Proposition \ref{Proposition:quasisyntomic descent for G mu displays}, the right hand side of $(\ref{equation:functor from G-gauge to G-display})$ also satisfies quasisyntomic descent.
    It thus suffices to construct, for each quasiregular semiperfectoid ring $S$ over $W(k)$,
    a fully faithful functor
    \begin{equation}\label{equation:functor from G-gauge to G-display banal}
    G\mathchar`-F\mathchar`-\mathrm{Gauge}_\mu(S)_{\mathrm{banal}} \to G\mathchar`-\mathrm{Disp}_{\mu}(\Prism_S, I_S)_{\mathrm{banal}}
    \end{equation}
    that is compatible with base change along any homomorphism $S \to S'$ in $\qrsperfd$.
    Here $G\mathchar`-F\mathchar`-\mathrm{Gauge}_\mu(S)_{\mathrm{banal}}$
    is the groupoid of banal prismatic $G$-$F$-gauges of type $\mu$ over $S$.
     By Remark \ref{Remark:quotient groupoid banal G-displays}, we may identify $G\mathchar`-\mathrm{Disp}_{\mu}(\Prism_S, I_S)_{\mathrm{banal}}$ with the groupoid
    $
    [G(\Prism_S)_{I_S}/G_\mu(\Prism_S, I_S)].
    $
    Similarly, we may identify $G\mathchar`-F\mathchar`-\mathrm{Gauge}_\mu(S)_{\mathrm{banal}}$ with the groupoid
    \[
    [G(\Prism_S)_{\mathcal{N}}/G_{\mu, \mathcal{N}}(S)]
    \]
    whose objects are the elements $X \in G(\Prism_S)_{\mathcal{N}}$ and whose morphisms are defined by
$
\Hom(X, X')=\{\, g \in G_{\mu, \mathcal{N}}(S) \, \vert \, X'\cdot g=X  \, \}.
$
The map $\phi \colon G(\Prism_S) \to G(\Prism_S)$ induces a map 
$\phi \colon G(\Prism_S)_{\mathcal{N}} \to G(\Prism_S)_{I_S}$ such that
for every $X \in G(\Prism_S)_{\mathcal{N}}$ and every $g \in G_{\mu, \mathcal{N}}(S)$, we have
$\phi(X \cdot g)=\phi(X) \cdot \phi(g)$,
where $\phi(g) \in G_\mu(\Prism_S, I_S)$ is the image of $g$ under the natural homomorphism $\phi \colon G_{\mu, \mathcal{N}}(S) \to G_\mu(\Prism_S, I_S)$.
Then we define the functor $(\ref{equation:functor from G-gauge to G-display banal})$ as
\begin{equation}\label{equation:functor from G-gauge to G-display banal quotient groupoid}
[G(\Prism_S)_{\mathcal{N}}/G_{\mu, \mathcal{N}}(S)] \to [G(\Prism_S)_{I_S}/G_\mu(\Prism_S, I_S)], \quad X \mapsto \phi(X).
\end{equation}
This functor is fully faithful.
Indeed, let $d \in I_S$ be a generator.
It suffices to prove that for all $X, X' \in G(\Prism_S)$, the map
\[
\{\, g \in G_{\mu, \mathcal{N}}(S) \, \vert \, g^{-1}X'\sigma_{\mu, \mathcal{N}, d}(g)=X  \, \} \to \{\, h \in G_\mu(\Prism_S, I_S) \, \vert \, h^{-1}\phi(X')\sigma_{\mu, d}(h)=\phi(X) \, \}
\]
defined by $g \mapsto \phi(g)$ is bijective.
(Recall that $\sigma_{\mu, \mathcal{N}, d}(g)=\mu(d)\phi(g)\mu(d)^{-1}$ and $\sigma_{\mu, d}(h)=\phi(\mu(d)h\mu(d)^{-1})$.)
One can check that the map
$h \mapsto X'\mu(d)h\mu(d)^{-1}X^{-1}$ gives the inverse of the above map.
Indeed, for an element $g \in G_{\mu, \mathcal{N}}(S)$ in the left hand side, we have
\[
X'\mu(d)\phi(g)\mu(d)^{-1}X^{-1} = X'\sigma_{\mu, \mathcal{N}, d}(g)X^{-1} = g.
\]
Similarly, for an element $h \in G_\mu(\Prism_S, I_S)$
in the right hand side, we have
\[
\phi(X'\mu(d)h\mu(d)^{-1}X^{-1})=\phi(X')\sigma_{\mu, d}(h)\phi(X)^{-1}= h.
\]

The proof of Proposition \ref{Proposition:functor from G-gauge to G-display} is complete.
\end{proof}

\begin{cor}\label{Corollary:functor from G-gauge to G-display perfectoid and regular}
Let $R$ be a perfectoid ring over $W(k)$ or a complete regular local ring over $W(k)$ with residue field $k$.
Then the functor $(\ref{equation:functor from G-gauge to G-display})$ is an equivalence:
\[
G\mathchar`-F\mathchar`-\mathrm{Gauge}_\mu(R) \overset{\sim}{\to} G\mathchar`-\mathrm{Disp}_{\mu}((R)_{\Prism}).
\]
\end{cor}

\begin{proof}
    Since we already know that this functor is fully faithful (Proposition \ref{Proposition:functor from G-gauge to G-display}), it suffices to prove the essential surjectivity.
    The assertion can be checked locally in the quasisyntomic topology.
    
    We first assume that $R$ is a perfectoid ring over $W(k)$.
    In this case, we have
    \[
    G\mathchar`-\mathrm{Disp}_{\mu}((R)_{\Prism}) \overset{\sim}{\to} G\mathchar`-\mathrm{Disp}_{\mu}(W(R^\flat), I_R).
    \]
    Since every $G$-$\mu$-display over $(W(R^\flat), I_R)$ is banal over a $p$-completely \'etale covering $R \to R'$ with $R'$ a perfectoid ring (Example \ref{Example:perfectoid ring etale morphism}), it suffices to prove that the functor
    $(\ref{equation:functor from G-gauge to G-display banal quotient groupoid})$
    given in the proof of Proposition \ref{Proposition:functor from G-gauge to G-display} is essentially surjective when $S=R$.
    This follows since $(\Prism_R, I_R) \simeq (W(R^\flat), I_R)$ and the Frobenius $\phi \colon W(R^\flat) \to W(R^\flat)$ is bijective.
    
    The case where $R$ is a complete regular local ring over $W(k)$ with residue field $k$ follows from the previous case
    since there exists a quasisyntomic covering $R \to R'$ with $R'$ a perfectoid ring by \cite[Proposition 5.8]{Anschutz-LeBras}.
\end{proof}

\begin{cor}\label{Corollary:functor from gauge to BK perfectoid and regular}
    The functors $(\ref{equation:functor from gauge to crystal})$ and $(\ref{equation:functor from gauge to BK type mu})$ are equivalences if $R$ is a perfectoid ring or a complete regular local ring with perfect residue field $k$ of characteristic $p$.
\end{cor}

\begin{proof}
    We need to prove that $(\ref{equation:functor from gauge to crystal})$ and $(\ref{equation:functor from gauge to BK type mu})$ are essentially surjective.
    For $(\ref{equation:functor from gauge to BK type mu})$, this follows from Corollary \ref{Corollary:functor from G-gauge to G-display perfectoid and regular} together with Example \ref{Example:GLn displays} and Example \ref{Example:prismatic F-gauge and GLn-gauge}.
    For $(\ref{equation:functor from gauge to crystal})$, we argue as follows.
    As in the proof of Corollary \ref{Corollary:functor from G-gauge to G-display perfectoid and regular}, it suffices to treat the case where $R$ is a perfectoid ring.
    Then we have
    \[
    {2-\varprojlim}_{(A, I) \in (R)_{\Prism}} \mathrm{BK}_{\mathrm{disp}}(A, I) \overset{\sim}{\to} \mathrm{BK}_{\mathrm{disp}}(W(R^\flat), I_R).
    \]
    For each $M \in \mathrm{BK}_{\mathrm{disp}}(W(R^\flat), I_R)$, there exists a $p$-completely \'etale covering
    $R \to R_1 \times \cdots \times R_m$
    with $R_1, \dotsc, R_m$ perfectoid rings
    such that for any $1 \leq i \leq m$,
    the base change
    $M_{(W(R^\flat_i), I_{R_i})}$ is of type $\mu$ for some $\mu$; see Example \ref{Example:perfectoid ring etale morphism} and Remark \ref{Remark:etale locally banal of type mu}.
    Since $(\ref{equation:functor from gauge to BK type mu})$ is essentially surjective, we can conclude that $(\ref{equation:functor from gauge to crystal})$ is also essentially surjective by using $p$-completely \'etale descent.
\end{proof}

\begin{rem}\label{Remark:gauge to G-BK modules}
     Let $R$ be a quasisyntomic ring over $W(k)$.
     For a bounded prism
     $(A, I) \in (R)_{\Prism}$, we defined the groupoid
     $G\mathchar`-\mathrm{BK}_{\mu}(A, I)$
     of $G$-Breuil--Kisin modules of type $\mu$ over $(A, I)$ in Section \ref{Subsection:G-BK module} and showed that it is equivalent to
     $G\mathchar`-\mathrm{Disp}_{\mu}(A, I)$
     in Proposition \ref{Proposition:G-displays and G-BK modules}.
     Thus the fully faithful functor $(\ref{equation:functor from G-gauge to G-display})$
     can be written as
     \[
     G\mathchar`-F\mathchar`-\mathrm{Gauge}_\mu(R) \to G\mathchar`-\mathrm{BK}_{\mu}((R)_{\Prism}):={2-\varprojlim}_{(A, I) \in (R)_{\Prism}} G\mathchar`-\mathrm{BK}_{\mu}(A, I).
     \]
     The essential image of this functor consists of those $\mathcal{P} \in G\mathchar`-\mathrm{BK}_{\mu}((R)_{\Prism})$ such that for some quasisyntomic covering $R \to S$ with $S$ a quasiregular semiperfectoid ring,
     the image $\mathcal{P}_{(\Prism_S, I_S)} \in G\mathchar`-\mathrm{BK}_{\mu}(\Prism_S, I_S)$ of $\mathcal{P}$ is a trivial $G_{\Prism_S}$-torsor, and via some trivialization
    $\mathcal{P}_{(\Prism_S, I_S)} \simeq G_{\Prism_S}$,
    the isomorphism $F_{\mathcal{P}_{(\Prism_S, I_S)}}$ is given by $g \mapsto Yg$ for an element $Y$ in
    \[
    \mu(d)\phi(G(\Prism_S)) \subset G(\Prism_S[1/I_S])
    \]
    where $d \in I_S$ is a generator.
    Therefore, we can simply define a prismatic $G$-$F$-gauge of type $\mu$ over $R$ as an object $\mathcal{P} \in G\mathchar`-\mathrm{BK}_{\mu}((R)_{\Prism})$ that satisfies the above condition.
     However, similarly to prismatic $G$-$\mu$-displays,
     it should be more technically convenient to work with the one introduced in Definition \ref{Definition:Prismatic G-F-gauge of type}.
     %We plan to study prismatic $G$-$F$-gauges of type $\mu$ in more detail in a future work.
    \end{rem}

We shall give an example which shows that the functors
$(\ref{equation:functor from gauge to crystal})$
and 
$(\ref{equation:functor from G-gauge to G-display})$
are not essentially surjective in general.
This also shows that there exists a non-admissible prismatic Dieudonn\'e crystal (see Example \ref{Example:F-gauge weight [0, 1]}).

Let $\O_C$ be the ring of integers of an algebraically closed nonarchimedean extension $C$ of $\Q_p$.
    Then the quotient $S=\O_C/p$ is a quasiregular semiperfectoid ring.
    The natural homomorphism
    $S \to \Prism_S/I_S$ is injective, and the Frobenius
$\phi \colon \Prism_S \to \Prism_S$
induces an isomorphism
$
\Prism_S/\Fil^1_\mathcal{N}(\Prism_S) \overset{\sim}{\to} S
$
(see \cite[Theorem 12.2]{BS}).
The Hodge--Tate comparison theorem for the conjugate filtration with respect to the natural homomorphism $\O_C \to S$ shows that $S \to \Prism_S/I_S$ is not surjective; see \cite[Section 12.1]{BS}.
We fix a generator $d \in I_S$.

\begin{ex}\label{Example:non-admissible prismatic Dieudonn\'e crystal}
    Let the notation be as above.
    We assume that $G=\GL_2$ and
    $\mu \colon \G_m \to \GL_2$ is the 1-bounded cocharacter defined by
$
t \mapsto \diag{(t, 1)}.
$
We choose an element $x \in \Prism_S$ whose image $\overline{x} \in \Prism_S/I_S$ is not contained in $S$.
Let $X \in G(\Prism_S)_{I_S}$ be the element such that $X_d=\begin{pmatrix}
1 & 0 \\
x & 1 \\
\end{pmatrix} \in G(\Prism_S)$.
We shall show that $X$ is not contained in the essential image of the functor (\ref{equation:functor from G-gauge to G-display banal quotient groupoid}).
If $X$ is contained in the essential image, then there are $Y \in G(\Prism_S)_{\mathcal{N}}$ and $g \in G_\mu(\Prism_S, I_S)$ such that
the equality $X \cdot g=\phi(Y)$ holds in $G(\Prism_S)_{I_S}$.
In particular, we see that $g^{-1}X_d$ belongs to the image of $\phi \colon G(\Prism_S) \to G(\Prism_S)$.
We write
\[
g=\begin{pmatrix}
g_{11} & g_{12} \\
g_{21} & g_{22} \\
\end{pmatrix} \in G_\mu(\Prism_S, I_S).
\]
By Proposition \ref{Proposition:BB isomorphism}, we have $g_{21} \in I_S$ and $g_{11}, g_{22} \in \Prism^{\times}_S$.
By computing the image of $g^{-1}X_d$ in $G(\Prism_S/I_S)$ and using that $g^{-1}X_d \in \phi(G(\Prism_S))$, it follows that
$\overline{x}/\overline{g}_{22}, 1/\overline{g}_{22} \in \Prism_S/I_S$ are contained in $S$.
We thus have $\overline{x} \in S$, which leads to a 
contradiction.
\end{ex}

\subsection*{Acknowledgements}
The author would like to thank
Kentaro Inoue,
Hiroki Kato,
Arthur-C\'esar Le Bras,
Kimihiko Li,
Samuel Marks,
and Alex Youcis for helpful discussions and comments.
The author is also grateful to Teruhisa Koshikawa for answering some questions on prisms.
Finally, the author would like to thank the referees for helpful comments.
The author especially would like to thank one of the referees for the suggestion that a discussion of the relation between prismatic $G$-$\mu$-displays and prismatic $F$-gauges should be included, and one of the referees for explaining an alternative description of $G$-$\mu$-displays, which simplifies some constructions given in the previous version of this paper.
The work of the author was supported by JSPS KAKENHI Grant Numbers 22K20332 and 24K16887.

\bibliographystyle{abbrvsort}
\bibliography{bibliography.bib}
\end{document}